\documentclass[a4paper, 11pt, cleardoublepage=empty,bibliography=totoc]{scrartcl}  
\usepackage[left=2.5cm, right=2.5cm, top=0.5cm, bottom=2.5cm]{geometry}
\usepackage[colorlinks,
pdfpagelabels,
pdfstartview = FitH,
bookmarksopen = true,
bookmarksnumbered = true,
linkcolor = black,
plainpages = false,
hypertexnames = false,
citecolor = black] {hyperref}
\usepackage{adjustbox}
\usepackage{amsmath,amssymb,mathrsfs,amsfonts}
\usepackage[dvipsnames]{xcolor}
\usepackage{hyperref}
\hypersetup{colorlinks,citecolor={NavyBlue}} 

\usepackage{bm}
\newcommand*{\B}[1]{\ifmmode\bm{#1}\else\textbf{#1}\fi}
\usepackage[english]{babel}
\usepackage{tikz-cd}
\usepackage{amstext}
\usepackage{pdfpages}
\usepackage{mathtools}
\usepackage{stmaryrd}
\usepackage{xcolor}
\usepackage{xfrac} 
\usepackage[utf8]{inputenc}
\usepackage{url} 
\usepackage{pgfplots}
\usepackage[autooneside=false,headsepline,markcase=noupper]{scrlayer-scrpage}
\usepackage{blindtext}
\usepackage{enumitem}
\usepackage{verbatim}
\usepackage{graphicx}
\usepackage{mathtools}
\usepackage{makeidx}
\usepackage{tikz}
\usetikzlibrary{babel}

\addto\captionsenglish{}

\usetikzlibrary{%
	matrix,%
	calc,%
	arrows%
}

\definecolor{mycolor}{RGB}{196,19,47}
\definecolor{mygray}{gray}{0.4}

\setkomafont{section}{\normalfont\LARGE\bfseries}
\setkomafont{subsection}{\normalfont\large\bfseries}
\setkomafont{pagehead}{\normalfont\bfseries}
\setkomafont{title}{\normalfont\bfseries\Large}

\usepackage{chngcntr}
\counterwithin{equation}{subsection}
\usepackage{amsthm}

\theoremstyle{remark}
\newtheorem{remark}[equation]{Remark}
\newtheorem*{remark_nn}{Remark}
\newtheorem*{bcov}{BCOV Theory}

\theoremstyle{definition}
\newtheorem{df}[equation]{Definition}
\newtheorem{cons}[equation]{Construction}
\newtheorem{obs}[equation]{Observation}
\newtheorem{Zusa}[equation]{Summary}
\newtheorem*{ass}{Assumption $(*)$}
\newtheorem{ctr}[equation]{Rule}

\newtheorem{ex}[equation]{Example}
\newtheorem{prop}[equation]{Proposition}
\newtheorem{lemma}[equation]{Lemma}

\theoremstyle{plain}

\newtheorem*{thm_nn}{Theorem}
\newtheorem{theorem}[equation]{Theorem}

\newtheorem{corollary}[equation]{Corollary}

\newtheorem{Conj}[equation]{Conjecture}
\newtheorem*{thm_nn_a}{Theorem A}

\newtheorem*{thm_nn_b}{Theorem B}

\newtheorem*{thm_nn_c}{Theorem C}
\newtheorem*{thm_nn_d}{Theorem D}
\newtheorem*{thm_nn_e}{Theorem E}
\newtheorem*{thm_nn_f}{Theorem F}
\newtheorem*{thm_nn_g1}{Theorem G.1}
\newtheorem*{thm_nn_g2}{Theorem G.2}

\newtheorem*{thm_nn_i}{Theorem I}
\newtheorem*{thm_nn_j}{Theorem J}
\newtheorem*{thm_nn_k}{Theorem K}
\newtheorem*{thm_nn_l}{Theorem L}

\newtheorem*{Claim}{Claim}

\newtheorem*{Conj_nn_h}{Conjecture H}

\clearpairofpagestyles 
\ohead{  \ifstr{\rightbotmark}{\leftmark}{}{\rightbotmark}}
\ihead{\leftmark}
\rohead*{\pagemark}
\rehead*{\pagemark}
\lehead{\headmark}
\lohead{\headmark}
\clearscrheadfoot
\cofoot[\pagemark]{\pagemark}

\makeatletter

\providecommand*{\rightbotmark}{\expandafter\@rightmark\botmark\@empty\@empty}
\makeatother
\makeindex

\usepackage{mathabx}

\usepackage{amssymb}

\begin{document}
\usetikzlibrary{arrows}
\tikzset{commutative diagrams/.cd,arrow style=tikz,diagrams={>=latex'}}
\pagenumbering{roman}

\begin{titlepage}
\centering
%\vspace*{0.7cm}
\includegraphics[scale=0.8]{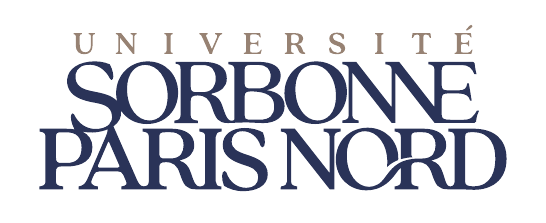}

\begin{center}
\vspace*{0.1cm}
\noindent {\Large {UNIVERSITÉ PARIS XIII -- SORBONNE PARIS NORD}} \\
\vspace*{0.2cm}
\noindent {\large \textsc {École doctorale Sciences, Technologies, Santé -- Galilée}} \\
\vspace*{1.2cm}
\hrulefill{}\\
\vspace*{0.5cm}
\noindent  {\huge\textbf{Théorie des champs de cordes à  partir \\[-5pt] des catégories de Calabi–Yau et \\ applications à la géométrie énumérative}}\\
\vspace*{0.3cm}
{\LARGE\textbf{Open-Closed String Field Theory from\\[-5pt] Calabi-Yau Categories and its\\ Applications to Enumerative Geometry}} \\

\vspace*{0.2cm}
\hrulefill{}\\
\vspace*{0.8cm}
\noindent {\large Thèse de Doctorat présentée par} \\
\vspace*{0.5cm}
\noindent {\huge  \textbf{Jakob Ulmer}} \\
\vspace*{0.5cm}
\noindent {\Large  \textbf{Laboratoire Analyse, Géométrie et Applications}} \\
\vspace*{1cm}
\noindent \large  {Pour l'obtention du grade de}\\
\vspace*{0.1cm}
\noindent \Large {\textbf{DOCTEUR EN MATHÉMATIQUES}}\\
\vspace*{1.1cm}
\noindent \large {Soutenue le 9 janvier 2026}\\
\vspace{0.1cm}
\noindent \large {Devant la commission d’examen formée de}\\
\vspace{0.4cm}

\large
\begin{tabular}{llll}
&CALAQUE Damien  & Université de Montpellier & Examinateur \\
&CĂLDĂRARU Andrei  & University of Wisconsin, Madison & Rapporteur \\
&GEORGIEVA Penka  & Sorbonne Université & Examinatrice \\ 
&GINOT Grégory  & Université Sorbonne Paris Nord & Directeur de thèse \\
&GWILLIAM Owen  & University of Massachusetts, Amherst  & Co-Directeur de thèse\\
&LI Si  & Tsinghua University  & Rapporteur \& Examinateur  \\
&OANCEA Alexandru  & Université de Strasbourg& Examinateur \\
&VALLETTE Bruno  & Université Sorbonne Paris Nord & Examinateur \\

\\ 
\end{tabular}
\end{center}

\end{titlepage}

\newpage
\titlepage

\pagestyle{plain}
\textbf{Abstract.} The overarching goal of this thesis is to develop categorical methods that connect enumerative geometry, as studied in mirror symmetry, with large $N$ gauge theories.
In the first part, we establish a relation between graph complexes, Calabi–Yau $A_\infty$-categories, and Kontsevich’s cocycle construction. The main result is a commutative square of shifted Poisson algebras, where one edge is given by the Loday–Quillen–Tsygan map, which we generalize to $A_\infty$-categories, and the other corners are constructed from graph complexes. We further describe a quantized version of this diagram using Beilinson–Drinfeld algebras. Our next main result is the construction of a formality $L_\infty$-morphism relating algebraic structures built from a Calabi–Yau category and one of its objects; this morphism depends on a splitting of the non-commutative Hodge filtration of given Calabi-Yau category. The involved algebraic structures conjecturally give a home to open-closed Gromov-Witten invariants when applied to the Fukaya category of a symplectic manifold and one of its Lagrangians, representing an object in this category. This generalizes the approach of categorical enumerative invariants from the closed to the open-closed setting, specifically the formality morphism used in the original construction. From a physics perspective, closed categorical enumerative invariants are encoded by the partition function of the associated closed string field theory (SFT). We explain how our open-closed morphism is an ingredient in quantizing the large $N$ open SFT associated to an object of a Calabi–Yau category. In the final part of this thesis, based on an algebraic approach to open and closed backreacted SFT, we propose ideas towards a categorical formulation of `Twisted Holography' at the level of partition functions, given as input a Calabi–Yau category and one of its objects. 
\\
\\
\\
\\
\\
\\
\textbf{Résumé.} Cette thèse a pour objectif de développer des méthodes catégoriques reliant la géométrie énumérative, telle qu’étudiée dans la symétrie miroir, aux théories de jauge à grand 
$N$. Dans une première partie, nous établissons une correspondance entre complexes de graphes, $A_\infty$-catégories de Calabi–Yau et construction de cocycles de Kontsevich. Le résultat principal est un carré commutatif d’algèbres de Poisson décalées, où une arête est donnée par l’application de Loday–Quillen–Tsygan, que nous avons
généralisé aux $A_\infty$-catégories, et les autres nœuds sont construits à partir de complexes de graphes. Nous démontrons également une version quantifiée de ce diagramme via les algèbres de Beilinson–Drinfeld.
Notre principal résultat suivant est la construction d’un $L_\infty$-morphisme de formalité reliant des structures algébriques construites à partir d’une catégorie de Calabi–Yau et de l’un de ses objets; ce morphisme dépend d’un choix de décomposition de la filtration de Hodge non commutative de la catégorie de Calabi–Yau considérée. Les structures algébriques impliquées sont conjecturalement liées aux invariants de Gromov–Witten ouverts-fermés, lorsqu’on les applique à la catégorie de Fukaya d’une variété symplectique et à l’un de ses lagrangiens, vu comme un objet de cette catégorie. Cela généralise l’approche des invariants énumératifs catégoriques du cadre fermé au cadre ouvert-fermés, en particulier le morphisme de formalité utilisé dans la construction d’origine. Du point de vue de la physique, les invariants énumératifs catégoriques fermés sont codés par la fonction de partition de la théorie des champs des cordes fermées (TCC) associée. Nous expliquons comment notre morphisme ouvert-fermés intervient dans la quantification de la TCC ouverte à grand $N$, associée à un objet d’une catégorie de Calabi–Yau.
Dans la dernière partie de cette thèse, en nous appuyant sur une approche algébrique de la TCC ouverte et fermée avec rétroaction, nous proposons des idées en vue d’une formulation catégorique de la `Twisted Holography', au niveau des fonctions de partition, en prenant comme données d’entrée une catégorie de Calabi–Yau et l’un de ses objets.
\newpage
\textbf{Zusammenfassung.} Ein übergeordnetes Ziel dieser Arbeit ist es kategorielle Methoden bereitzustellen, die enumerative Geometrie, wie in der Spiegel-Symmetrie untersucht, und $N\rightarrow \infty$ Eichtheorien  miteinander verbinden. Im ersten Teil dieser Dissertation stellen wir Zusammenhänge zwischen Graphkomplexen, Calabi-Yau $A_\infty$-Kategorien und der Kokettenkonstruktion von Kontsevich her. Das zentrale Ergebnis ist ein kommutatives Diagramm verschobener Poisson-Algebren, wobei eine seiner Kanten durch die Loday-Quillen-Tsygan-Abbildung, verallgemeinert auf $A_\infty$-Kategorien, gegeben ist und und die anderen Ecken aus Graphkomplexen konstruiert sind. Wir beschreiben eine quantisierte Version eines solchen Diagramms mittels Beilinson-Drinfeld Algebren. Unser nächstes zentrales Ergebnis ist die Konstruktion eines Formalitäts $L_\infty$-Morphismus, der algebraische Strukturen in Verbindung setzt, welche aus einer Calabi–Yau Kategorie und einem ihrer Objekte aufgebaut sind. Dieser Morphismus ist abhängig von einer Wahl einer Aufspaltung der nichtkommutativen Hodge-Filtration der gegebenen Calabi–Yau-Kategorie. Die erwähnten algebraischen Strukturen liefern vermutungsweise nach einen natürlichen Rahmen für offen-geschlossene Gromov–Witten-Invarianten wenn sie auf die Fukaya-Kategorie einer symplektischen Mannigfaltigkeit und eine ihrer Lagrangeschen Untermannigfaltigkeiten, die als Objekte in dieser Kategorie erscheinen, angewendet werden. Dies verallgemeinert die Idee der kategoriellen enumerative Invarianten vom rein geschlossenen zum offenen-geschlossenen Fall, insbesondere den ursprünglich verwendeten Formalitätsmorphismus.
Aus physikalischer Sicht werden geschlossene kategorielle enumerative Invarianten durch die Zustandssumme der zugehörigen geschlossenen Stringfeldtheorie (SFT) beschrieben. Wir zeigen, wie der von uns konstruierte offene-geschlossene Morphismus eine zentrale Rolle bei der Quantisierung der offenen $N\rightarrow \infty$ SFT spielt, die einem Objekt einer Calabi–Yau-Kategorie zugeordnet ist.
Im letzten Teil dieser Arbeit schlagen wir - ausgehend von einem algebraischen Zugang zur offenen und geschlossenen rückreagierten Stringfeldtheorie - einen Ansatz für eine kategoriellle Formulierung der `Twisted Holography' auf Ebene der Zustandssummen vor, wobei als Ausgangsdaten eine Calabi–Yau-Kategorie und eines ihrer Objekte dienen.

\newpage
\thispagestyle{empty}
\huge{\textbf{Acknowledgments}}
\vspace{0.5cm}

\normalsize{

I would like to express my gratitude to a number of persons who helped me in numerous ways to finish this PhD thesis.

\medskip
Thank you Grégory for welcoming me to Paris and agreeing to be my PhD supervisor, thank you for your positivity and the mathematical, administrative and emotional support these last three years. Merci!

\smallskip
Thank you Owen for the invitation to come to Amherst even before the start of my PhD, which was an important moment for me, your hospitality, the numerous discussions on zoom and during the other visits in Amherst, Paris and Bonn and your insistence on clarifying my ideas, which made me a better mathematician.   

\smallskip
I am very grateful to Andrei Căldăraru and Si Li to have agreed to referee this thesis and further to Damien Calaque, Penka Georgieva, Alexandru Oancea and Bruno Vallette for being part of the jury. 

\smallskip
I would like to thank the algebraic topology group of Paris 13 for their welcoming atmosphere, notably Christian Ausoni, Bruno Vallette and Geoffroy Horel and the administrative team of LAGA for their help, especially Yolande and her always friendly presence.

\smallskip
Next I would like to say thank you to Surya Raghavendran and Philsang Yoo for their early on interest in my work, which was a very valuable motivational factor, the support and the invitation to New Haven, Edinburgh, Beijing and Seoul. 

\smallskip
I am very grateful to Junwu Tu for the invitation to come to Shanghai in the summer of 2024, which has been a crucial moment for my work, and our other discussions in Paris, Stony Brook and online.

\smallskip
I would further like to thank Alastair Hamilton, Bruno Vallette, Lino Amorim and Vivek Shende for discussions and interest in my work, and Guillaume Laplante-Anfossi and Vivek Shende for the possibility to visit Odense and the hospitality.

\smallskip
I want to take the opportunity to thank the Fondation Sciences Mathématiques de Paris for cofunding my PhD through a Marie Sklodowska‑Curie Action. Furthermore I am thankful to the GDR Topologie Algébrique for hosting the nice yearly conferences in Nantes, Lille, Toulouse and Marseille, which allowed me to interact with the french homotopy theory community and to visit these cities. Generally I am very grateful for having had the opportunity to travel much and attend many conferences: I feel fortunate to been able to participate at the ‘Recent developments in higher genus curve counting' program at the Simons-Center in Stony Brook. I fondly remember the conferences in Luminy at CIRM, the winter schools in mathematical physics at Les Diablerets in the mountains of Switzerland and the CIME school in Cetraro in the south of Italy.

\smallskip
At a more practical level, I want to acknowledge, next to my nice desk in Paris 13 with the great view and wonderful tree next to it, all the libraries in Paris where I worked, at Sorbonne Nouvelle, at the École Normale Superiour, and at Institut Henri Poincaré, amongst others.

\medskip
On a more personal level: I had a very warm welcome to Paris in August 2022, which was mainly due to Jule, Verena and Paul. I hope that our paths will cross again many times! 

\smallskip
I keep in good memory  my first (freezing) colocation with Thibault and Pierre and am very thankful for our continued friendship.

\smallskip
At Paris 13 I found a kind community of PhD students and for that I am grateful to Elie, Neige, Wassim, Mohammed, Victor and Oisín and especially Hugo. The newer generation of PhD students are continuing this tradition. Thank you Nicolas, Dominik, Joëlle, Loth, David and all the others!

\smallskip
I am especially thankful to the people that began at LAGA at the same time as me and with whom I shared many great moments that made these last three years much better: thank you Jordan for your reliable optimism and our colocation in Boulevard de Magenta, thank you Francesca for your incredible social energy and warmth, thank you Noé for your political motivation and the parties in Montreuil. Thank you Maissâ for your deep kindness und alles Gute in Bielefeld! Thank you Guglielmo for your gentle presence and the visit in Pisa. Thank you Marie-Camille for all your ideas to make conferences much more fun; especially memorable is the vegetable soup in Schleiden. Thank you Abhigyan. Thank you Vincent for never giving up on our slightly awkward interactions! 

\thispagestyle{empty}
\smallskip
In Paris, outside of mathematics, I am very thankful for my friendship with Stefan, who was at times a much needed emotional pillar for me. Thank you Lili for our time together. Lieber Guillaume, ich hätte vor 10 Jahren niemals gedacht, dass ich einmal in Paris leben und promoviert haben würde! Ich bin sehr dankbar für unsere bisherige gemeinsame Zeit von Fulda, über Heidelberg, Berlin, Rom und Zürich nach Paris.   

\smallskip
I want to say thank you to my other friends scattered all around the world. Danke Céline für unsere schrullige Freundschaft. Danke Alex und Leon für die vielen Momente seit langer Zeit. Danke Simon für deine unbändige Energie und dass wir damals über string field theory sprachen, was rückblickend ganz wichtig für meine Arbeit war. Danke Greg für deine inspirierende Neugierde, danke Lukas, dass Du immer Zeit hast zu telefonieren und für mich da bist. Thanks Misha for our philosophical discussions and the hospitality in Amherst. Thank you Zimai and Atsu for the wonderful people you are. Dear Yangning, thank you for being here and being you. 

\smallskip
Danke Mama und Papa für Eure Fürsorge, Liebe und all das was Ihr mir mitgegeben habt. Danke Friedrich! Danke an meine ganze Familie, die mir immer einen Platz freihält.

\medskip
Finally, I am very thankful to have gotten the chance to live in Paris and to see it from banlieues to center, which made me understand a bit more about our world, more than mathematics. 

\bigskip
\raggedleft{Paris, 5th of January 2026}

}

\sloppy

%\newpage
%\thispagestyle{empty}
%\mbox{}
\newpage

\tableofcontents
\newpage
\section{Introduction and Summary}

\pagenumbering{arabic}
The results of this thesis are situated at the intersection of the areas of \begin{itemize}
    \item 
 homotopical algebra, which is the study of algebraic operations up to homotpy,
 \item mathematical physics, the study of ideas from physics - notably from string theory - using mathematical rigor, and
 \item enumerative geometry, or more narrowly the study of geometric object by probing them via maps from other objects, for instance from surfaces.
 \end{itemize}
 A prime example of such an intersection and a key motivation for this thesis lies in the subject of `\textbf{Enumerative and Homological Mirror Symmetry}' \cite{Kon94}, which will be the starting point from which on we situate the results obtained in this thesis. In the course of this introduction it will become clear that \textbf{string field theory} (SFT) respectively an algebraic avatar which applies to \textbf{Calabi-Yau categories} provides a good language to discuss these phenomena. In the last part, using this language, we try to make contact with `\textbf{Twisted Holography}' \cite{cg21}, a recent instance of an \textbf{intersection of ideas from physics and mathematics}. 
\subsection{Enumerative and Homological Mirror Symmetry}\label{eahms}

Physicists \cite{can91} observed, based on the study of two topological string theories, called the A-model and B-model, that there is a close relationship between enumerative invariants on the one hand obtained from a given symplectic manifold and on the other hand obtained from a complex manifold, which are additionally both Calabi-Yau manifolds and called mirror partners. In the symplectic case these numbers are called Gromov-Witten invariants. Physicists postulated that such a phenomena may hold in generally, ie. that such manifolds always come in pairs. This research area was from then referred to as \textit{(enumerative)} mirror symmetry. 

\medskip
Kontsevich's visionary approach \cite{Kon94} to mirror symmetry was:
 Let us not try to directly compare numbers obtained from these two geometric set-ups.
 Let us associate categories to the geometric context, concretely the so called Fukaya category associated to the symplectic manifold
 and the derived category of coherent sheaves associated to the complex manifold and then compare these categories.\footnote{Physically, the objects of these categories describe the set of branes on which open strings of the respective topological string theory can end.} The statement of \textit{homological} mirror symmetry is that these two categories should be equivalent in a suitable sense if the two geometric contexts are `mirror'.
 Further, the categories that we find are Calabi-Yau, which roughly means that their Hom-spaces are equipped with non-degenerate pairings. 

 \medskip
 
 However, now one may ask the question how it is possible to pass back from these categories to the numbers that we were originally interested in and \textbf{whether an analogue of Gromov-Witten invariants exists for general Calabi-Yau categories}.
$$\begin{tikzcd}
   \text{Enumerative Invariants}\arrow[bend left=40]{r}{\text{Kontsevich}}& \text{Calabi Yau Category}\arrow[bend left=40, swap, dotted]{l}{?}
\end{tikzcd}$$

A proposal to solve this question has been given in the works of Costello \cite{Cos05}, Căldăraru-Tu \cite{CaTu24}, inspired by Kontsevich's work. By now it is known that a Calabi-Yau category equivalently describes a topological string theory \cite{Kon94}, or in different terminology a topological conformal field theory \cite{Cos07a}, a version of a 2d topological field theory \cite{Lu09}. To recover actual numbers from a Calabi-Yau category we need an additional datum, a \emph{splitting of its non-commutative Hodge filtration}. A splitting of the Hodge filtration should roughly allow to extend the corresponding field theory to a natural compactification, inducing a cohomological field theory, compare~\cite{des22}.

\medskip
As an \textbf{upshot} \cite{Cos05, CaTu24} tell us that given a Calabi-Yau category $\mathcal{C}$ together with a splitting `s', which is canonically determined in the geometric contexts of mirror symmetry,\footnote{See sections 6.2 and 6.3 of \cite{AmTu22}. The fact that a splitting is canonically determined in the context of mirror symmetry may explain why their relevance was only discovered relatively late.} we recover an element called the \textbf{closed categorical Gromov-Witten potential} \begin{equation}\label{ccei}
\mathcal{Z}^{\mathcal{C},s}\in Sym\left(HH^*(\mathcal{C})\llbracket u\rrbracket \right)\llbracket\gamma\rrbracket.\end{equation}

We can understand $\mathcal{Z}^{\mathcal{C},s}$ just as a polynomial in variables given by powers of $u$ and the Hochschild cohomology classes of $\mathcal{C}$.
Roughly speaking\footnote{Indeed, to match the geometric invariants one may need to consider a family of categories and take a certain limit; see the comment just before remark 1.2 of \cite{CaTu24}. We ignore this point here. However see remark \ref{DepOnPar} related to the bulk-boundary deformation, which may address this.} the coefficients of $\mathcal{Z}^{\mathcal{C},s}$ in these variables should be the enumerative invariants that we are interested in.
The following conjecture by Amorim-Tu, likely going back to Kontsevich, and implicit in Costello \cite{Cos05}, phrases this expectations: 
\begin{Conj}[\cite{AmTu22}, conjecture 6.12]\label{catgeo}
Let $X$ be a compact symplectic manifold and denote $\mathcal{C}=Fuk(X)$ its Fukaya category.\footnote{We write  $\Lambda$ for its Novikov ring.}
Then there exist a canonical splitting s and an isomorphism $HH^*(\mathcal{C})\cong H^*(X,\Lambda)$ and for $\eta_i\in H^*(X,\Lambda)$ 
the coefficient of  $$(\eta_1u^{k_1})\cdots(\eta_nu^{k_n})\gamma^g\ \  \text{of}\ \ 
\mathcal{Z}^{\mathcal{C},s}\in Sym\left(HH^*(\mathcal{C})\llbracket u\rrbracket\right)\llbracket\gamma\rrbracket$$
is the Gromov-Witten invariant
$$\int_{\bar{\mathcal{M}}_{g,n}(X)}\psi^{k_1}ev_*(\eta_1)\dots \psi^{k_n}ev_*(\eta_n),$$
\end{Conj}
which `counts' $\bar{\mathcal{M}}_{g,n}(X)$, the space of stable pseudoholomorphic maps from a Riemann surfaces of genus $g$ with $n$ marked points to $X$, weighted by cohomology class and $\psi$-class insertions along the marked points. In fact for the case of $X=pt$ conjecture \ref{catgeo} is true:
\begin{theorem}[\cite{Tu21}]\label{JTT}
Considering the category with one object with endomorphisms given by the algebra of rational numbers, then we have $HH^*(\mathbb{Q})\cong \mathbb{Q}$,
there is a unique splitting $s$ and the coefficient of  $$u^{k_1}\cdots u^{k_n}\gamma^g\ \ \text{of}\ \ \mathcal{Z}^{\mathcal{C},s}\in Sym\left(\mathbb{Q}\llbracket u\rrbracket\right)\llbracket\gamma\rrbracket$$ is
$$\int_{\bar{\mathcal{M}}_{g,n}}\psi^{k_1}\dots \psi^{k_n},$$
\end{theorem}
ie. exactly the intersection numbers on the moduli space of stable Riemann surfaces $\bar{\mathcal{M}}_{g,n}$, also known as the Gromov-Witten invariants of a point.

\subsection{Closed Categorical Enumerative Invariants}\label{cceis}
Let us now explain the ideas of Costello, Căldăraru-Tu in more detail. In the section afterwards we will begin to situate the results obtained in this thesis in relation to their work.

First we introduce another algebraic object, called a \textbf{Beilinson-Drinfeld} (BD) \textbf{algebra}.\footnote{A closely related notion are Batalin-Vilkovisky algebras.} Those are meant to capture the structure of the observables of a quantum field theory formulated in the BV formalism \cite{BV81} - we comment on a more precise connection of the BD algebras appearing here to physics in remark \ref{DRDEA} at the end of this paragraph. A Beilinson-Drinfeld algebra $$(W,\cdot,d,\{\_,\_\})$$ consists of a multiplication, a differential, and a Lie bracket on a graded vector space $W$; those operations have to be related in an interesting way; see definition \ref{BD}. Given a Calabi-Yau\footnote{In fact we need to work with a strictification of that notion, called cyclic.} $A_\infty$-category $\mathcal{C}$ we can construct an \textbf{algebraic Beilinson-Drinfeld} algebra in which the numbers from \eqref{ccei} have a natural home: There is a Beilinson-Drinfeld structure 
\begin{equation}\label{cabd}
\mathcal{F}^c(\mathcal{C}):=\Big(Sym(CH_*(\mathcal{C})[ u^{-1}])\llbracket\gamma\rrbracket,\cdot\ ,d_{hoch}+uB+\gamma\Delta_c,\{\_,\_\}_c\Big),
\end{equation}
on the symmetric algebra of the Hochschild chains of $\mathcal{C}$, adjoined a formal variable $u$, and all adjoined a formal, genus counting parameter $\gamma$. Here $\Delta_c$ and $\{\_,\_\}_c$ are induced from the residue pairing and Connes boundary operator $B$. The prior is given by a map
\begin{equation}\label{Res}
   \langle\_,\_\rangle_{res}: CH_*(\mathcal{C})((u))\otimes CH_*(\mathcal{C})((u))\rightarrow k,
\end{equation}
itself induced from the so called Mukai pairing.
\begin{remark}
In \cite{CaTu24} and \cite{AmTu25} this BD algebra is denoted $\mathfrak{h}_\mathcal{C}$. We chose the notation $\mathcal{F}^c(\mathcal{C})$ following the initial description in \cite{Cos05} from its origin as a Fock space and to stay compatible with the notations of the following chapters. We apologize for possible confusions.  
\end{remark}
\smallskip
Further we introduce the abelian Beilinson-Drinfeld algebra 
\begin{equation}\label{abBD}
\mathcal{F}^c(\mathcal{C})^{Triv}:=\Big(Sym(CH_*(\mathcal{C})[u^{-1}])\llbracket\gamma\rrbracket,\cdot\ ,d_{hoch},0\Big).
\end{equation}
There is also a \textbf{geometric BD algebra}, called $M^c$, whose underlying graded vector space is roughly
 \begin{equation}\label{cms}
    \bigoplus_{\text{stable}} C_*(\mathcal{M}_{g,n}^{fr}).\end{equation}
Here $\mathcal{M}_{g,n}^{fr}$ denotes the moduli space of genus $g$ Riemann surfaces (not necessarily connected) with $n$ framed boundaries. Denoting by $\partial$ the $S^1$-equivariant singular chain differential, where the $S^1$ action(s) comes from rotating the framings, we have more precisely
\begin{equation}\label{cgbd}
M^c:=\Big(\bigoplus_{\text{stable}} C_*(\mathcal{M}_{g,n}^{fr})_{hS^1}\llbracket\gamma\rrbracket,\cdot\ ,\partial+\gamma\Delta_c,\{\_,\_\}_c\Big),
\end{equation}
 where both  $\Delta_c$ and $\{\_,\_\}_c$ are induced by twist sewing along framed boundaries and the multiplication is induced from disjoint union.
 
 Note that a Beilinson-Drinfeld algebra is a dg shifted Lie algebra by forgetting the multiplication. Further we restrict to suitably finite categories, called smooth; see  definition \ref{smanpr}. Then one can connect the BD algebras $M^c$ and $\mathcal{F}^c(\mathcal{C})$ as follows, which is a first key step in defining the closed categorical Gromov-Witten potential \eqref{ccei}:
\begin{theorem}[\cite{Cos05, CaTu24}]\label{cei}
     There is a `\textbf{representing' map} of induced dg shifted Lie algebras given a smooth cyclic $A_\infty$-category $\mathcal{C}$
   \begin{equation}\label{omgkbm}
   \rho_{\mathcal{C}}:M^c\rightarrow \mathcal{F}^c(\mathcal{C}).
\end{equation}
\end{theorem}
\smallskip
The next fundamental ingredient in the approach of Costello, Căldăraru-Tu are the so called \textbf{closed string vertices}, which are the unique element $SV^c\in M^c$ which satisfies (\cite{SZ94}, \cite{Cos05})
\begin{equation}\label{csvmce}
    (\partial+\Delta_c)SV^c+\frac{1}{2}\{SV^c,SV^c\}_c=0,
\end{equation}
in other words it is a Maurer-Cartan element, and $SV^c|_{\mathcal{M}_{0,3}^{fr}}=\frac{1}{3!}[pt]$; recalling that $\mathcal{M}_{0,3}^{fr}=pt$.

\medskip

We remind that we were interested to extract actual numbers (conjecturally enumerative invariants), given a Calabi-Yau category $\mathcal{C}$. To do so we need additionally a \textbf{splitting of the non-commutative Hodge filtration}; see definition \ref{splitting}. Roughly speaking Kaledin proved \cite{kal17} that given a smooth and proper $\mathbb{Z}$-graded category the $S^1$-action on Hochschild chains is homotopy equivalent to zero.
A splitting of the non-commutative Hodge filtration specifies \emph{how} this action is trivial. Given such a splitting one can show that the algebraic BD algebra \eqref{cabd} is almost trivial, which is the next key step in extracting the closed categorical Gromov-Witten potential \eqref{ccei}: \begin{theorem}[\cite{CaTu24}]\label{splmap}
Given a smooth cyclic $A_\infty$-category $\mathcal{C}$ and a splitting $s$ of the non-commutative Hodge filtration there is an $L_\infty$-quasi-isomorphism of dg shifted Lie algebras
\begin{equation}\label{cltrmo}
\mathcal{K}_s:\mathcal{F}^c(\mathcal{C})\rightsquigarrow {\mathcal{F}^c(\mathcal{C})}^{Triv},\end{equation}\end{theorem} 
recalling the abelian BD algebra from \eqref{abBD}. 

\medskip
Summarizing, it follows that the image of the string vertices (and their exponential) under the maps from the two previous theorems defines a closed element, which finally leads us to the closed categorical Gromov-Witten potential \eqref{ccei}:
\begin{df} Given a smooth cyclic $A_\infty$ category $\mathcal{C}$ and a splitting $s$ we denote\footnote{In fact we should localize at $\gamma$, see just before lemma 6.11 of \cite{CaTu24}. We omit this here and in the following for simplicity of presentation.} the \textbf{\emph{ closed categorical Gromov-Witten potential}} of $(\mathcal{C},s)$, the class of the exponential of the image of the closed string vertices \eqref{csvmce} under the maps \eqref{cltrmo} and \eqref{omgkbm}, by
\begin{equation}\label{cgw}
Z_{\mathcal{C},s}^c:=[e^{\mathcal{K}_s\rho(SV^c)/\gamma}\ ]\in Sym (HH_*(\mathcal{C})[ u^{-1}])\rrbracket\gamma\rrbracket.
\end{equation}
\end{df}
The name `closed categorical Gromov-Witten potential' for the element $Z_{\mathcal{C},s}^c$ is motivated by conjecture \ref{catgeo}.\footnote{To be more precise the element $Z_{\mathcal{C},s}^c$ uniquely determines the element $Z^{\mathcal{C},s}$ from equation \eqref{ccei}, which we called closed categorical Gromov-Witten potential before. We ignored this point here as these elements uniquely determine each other; they are dual under the residue pairing \eqref{Res}; see section 1.1 of \cite{CaTu24}.} 
\begin{remark}\label{DRDEA}
From a physics perspective the BD algebra $\mathcal{F}^c(\mathcal{C})$ describes the quantum observables of the \emph{free closed string field theory} (SFT) associated to the closed TCFT induced by $\mathcal{C}$.\footnote{We recall in section \ref{2dop} Costello's theorem~A \cite{Cos07a} which explains how a Calabi-Yau category induces a closed TCFT.} We justify this terminology as follows: for $\mathcal{C}=D^b_{dg}(X),$ for $X$ some CY manifold $\mathcal{F}^c(\mathcal{C})$ is equivalent (proposition \ref{propprop}) to the free quantum observables of \textbf{BCOV theory}, also called Kodaira-Spencer gravity, formulated in the BV formalism, as studied in \cite{CL15} using analytical and obstruction theoretic methods. Kodaira-Spencer gravity, respectively BCOV theory had been identified as the closed SFT of the B-model by physicists in \cite{BCOV94}.\footnote{Physicists suggest that there is a general method to assign to a string theory a yet other theory, called its string field theory. Parts of the results recalled and obtained in this thesis may be views as making this precise for topological string theories, those encoded by Calabi-Yau categories.}
These ideas let us to define the notion of \emph{interacting closed string field theory} associated to a Calabi-Yau category, see definition \ref{cisft}. We conjecture that for $\mathcal{C}=D^b_{dg}(X),$ for $X$ some CY manifold the interacting closed string field theory associated to $\mathcal{C}$ is equivalent to interacting BCOV theory; see conjecture \ref{conj_fo}, where we also collect some evidence. It is under this (conjectural) perspective that we understand the element $Z_{\mathcal{C},s}^c$ as the partition function of the interacting closed string field theory associated to the Calabi-Yau category~$\mathcal{C}$. 
\end{remark}
 One main accomplishment of Căldăraru-Tu \cite{CaTu24} is to give an explicit formula from which one can in principle compute the closed categorical GW potential \eqref{cgw}; see theorem 1.3 of \cite{CaTu24} and see also \cite{CaTu20} for a computation in low genus for an elliptic curve, confirming predictions from mirror symmetry. However, general computations seem to remain very difficult; compare section 9.2 of~\cite{CaTu24}. As mentioned, there exists an a priori different, analytic approach to closed SFT of the B-model, that is BCOV theory, developed by Costello-Li \cite{coli12, CL15}. In these works the authors show that (in particular) for elliptic curves there is a unique quantization of interacting BCOV theory\footnote{Strictly speaking they do so only for a variant of BCOV theory, which may access only some enumerative invariants. However Li shows in \cite{Li16}, using different methods, that a quantization does exists.} that  satisfies the dilaton equation. Using the complex conjugate splitting one can obtain numerical invariants, which they prove coincide with the GW-invariants of the mirror dual of the input elliptic curve \cite{li11}. Assuming a comparison result of algebraic and analytic descriptions of closed SFT for the B-model (along the lines of proposition \ref{propprop}) one could prove that the quantization obtained through the closed string vertices \eqref{cgw} dequantizes to the BCOV interaction term (see our conjecture \ref{conj_fo}) and satisfies the dilaton equation (compare footnote on page 43 of \cite{CaTu24}) an identification of analytically and algebraically obtained quantizations should follow by the uniqueness result of \cite{coli12}. This could be of great advantage, allowing to combine the explicit formulae from \cite{CaTu24} with the analytic methods of \cite{coli12, CL15}; we hope to further explore this in the future.
\begin{remark}
We restrict ourselves to Calabi-Yau categories of \emph{odd} dimension $d$ in the following sections, because it is then that we understand the quantized LQT map (see \eqref{lqt} later), even though in this section that was not necessary. Further, the `fermion trick' makes this assumption relatively mild.\footnote{We recall the `fermion trick' from appendix B of \cite{AmTu22}: to any smooth cyclic $A_\infty$-category of even dimension we can associate such a category of odd dimension whose categorical GW invariants coincide with the original one.} We will not indicate the Calabi-Yau dimension each time in this introduction. Similarly we will omit some shifts and will be imprecise about some degrees. Further there are some choices how to set things up: the appearing BD and underlying shifted Poisson algebras built from a cyclic category of dimension $d$ most naturally either have a $(d-2)$ or $(2d-5)$ shifted bracket, depending on that choice. We refer the interested reader to the later chapters of this thesis; we have suppressed this subtlety in the introduction.
\end{remark}
\subsection{Results Obtained in this Thesis}
We begin to situate the results obtained in this thesis as part of an open-(closed) analogue of the previous section.

Recall that so far the motivation was provided by (re)-understanding \textbf{closed Gromov-Witten invariants}, which count pseudo-holomorphic maps from Riemann surfaces with some interior marked points to symplectic manifolds. Physically, one may think of such a pseudo-holomorphic map as the process of closed strings (corresponding to the image of the punctures) propagating in the symplectic manifold and merging with each other. In string theory one can also consider \emph{open strings}, whose boundaries in this specific example have to end on Lagrangian submanifolds, but which can also propagate in a given symplectic manifold and possibly merge with each other. Thus one may wonder whether there is a theory of \textbf{open Gromov-Witten invariants}, counting pseudoholomorphic maps from Riemann surfaces with boundary and boundary marked points to symplectic manifolds, such that the free boundaries lie in some fixed Lagrangian submanifold(s).

Indeed, that is what physicists and symplectic topologists have begun to study; see eg. \cite{psw08}. However, these invariants are subtle to define as the moduli space of such maps has an intricate boundary behavior \cite{Liu02}. Recently, in \cite{EkSh25}, the authors use exactly some algebraic structures induced from this boundary behavior to show that open Gromov Witten invariants have a natural home in the skein module, in spirit sort of related to the ideas in the previous section~\ref{cceis}.
Already earlier Fukaya \cite{Fuk11} had studied a priori different algebraic structures which similarly exhibit a home for such invariants in genus zero, a higher genus version was perhaps first described in \cite{mo99}. Those structures can be phrased in the language of \emph{involutive Lie bialgebras}, also appearing in string topology; see eg. \cite{NaWi19} for a recent work.

\medskip
In any case, one may wonder whether there is an `{open}' analogue of the `{closed}' theory outlined in the previous section \ref{cceis}. We propose to begin such a story as follows, directly inspired by the previously cited works \cite{mo99, Fuk11}:\footnote{The relationship to the work \cite{EkSh25} remains more unclear, however; though a relationship though large $N$ Chern-Simons theory to the skein module seems very suggestive.} fixing a set of objects $\Lambda$ of a cyclic $A_\infty$-category $\mathcal{C}$, having in mind a set of Lagrangian submanifolds as objects of the Fukaya category, the \textbf{open algebraic analogue} of the closed algebraic BD algebra \eqref{cabd} is the  Beilinson-Drinfeld algebra 
\begin{equation}\label{OBDa}
\mathcal{F}^o(\Lambda)=\Big(Sym\big(Cyc^*(\Lambda)[-1]\big) \llbracket\gamma\rrbracket,\cdot\ ,d+\nabla+\gamma\delta,\{\_,\_\}_o\Big),\end{equation}
built as the symmetric algebra on a central extension of cyclic cochains $Cyc^*(\Lambda)$ of the full subcategory on $\Lambda$, adjoined a genus counting parameter $\gamma$. The algebraic structure is induced from a central extension of an involutive Lie bialgebra structure on cyclic cochains
\begin{equation}\label{auihgj}
\nabla:Cyc^*(\Lambda)\rightarrow Cyc^*(\Lambda)\otimes Cyc^*(\Lambda),
\end{equation}
\begin{equation}\label{tzucvb}
\delta: Cyc^*(\Lambda)\otimes Cyc^*(\Lambda)\rightarrow Cyc^*(\Lambda);
\end{equation}
further $d$ indicates the differential induced from the underlying chain complexes (\emph{not} the differential computing cyclic cohomology, which is not compatible with the co-bracket \eqref{auihgj} in general.) We refer to definition \ref{nc_pq} for further details.  

\smallskip
To \emph{construct the `open' analogue to the universal geometric closed BD algebra} \eqref{cms} from the previous section \ref{cceis} we introduce the chain complex\footnote{To be precise we consider chains with values in a certain line bundle but ignore this point here. Compare~\cite{Cos07a}.} $$\big(C_*(\mathcal{M}_{g,b,m_1,\cdots,m_b}^{\Lambda}),\partial\big),$$ built from the moduli space of genus $g$ Riemann surfaces (not necessarily connected) with  $b$ (unframed) boundaries, each with $m_i\in\mathbb{Z}_{\geq 0}$ embedded open boundaries,\footnote{Up to homotopy we may consider marked points on the boundary instead of embedded intervals.
This replacement can be done in a way compatible with the operation of gluing marked points respectively embedded intervals. Compare proposition 6.1.5 of \cite{Cos07a}.} the free boundaries labeled by the set $\Lambda$. Out of it one obtains an \textbf{open geometric} Beilinson-Drinfeld algebra
\begin{equation}\label{ogbd}
M^{o}_\Lambda:=\Big(\bigoplus_{\text{stable}} C_*(\mathcal{M}_{g,b,m_1,\cdots,m_b}^{\Lambda})\llbracket\gamma\rrbracket,\cdot\ ,\partial+\Delta_o,\{\_,\_\}_o\Big),\end{equation} where $\Delta_o$ and $\{\_,\_\}_o$ are induced by gluing along open boundaries with matching free boundary labels, the multiplication is induced from disjoint union. See \cite{HVZ08}.

\bigskip
Before we continue to explain the analogies to the previous section \ref{cceis} and begin to present our results  we recall from \cite{GGHZ21} the quantized Loday-Quillen-Tsygan map which is a map of Beilinson-Drinfeld algebras 
\begin{equation}\label{lqtintro}
LQT_N:\mathcal{F}^o(\lambda)\rightarrow \mathcal{O}^{pq}\big(\mathfrak{gl}_NEnd_\mathcal{C}(\lambda)\big), \end{equation} connecting for $\Lambda=\{\lambda\}$ the open BD algebra \eqref{OBDa} and the right-hand side, built from Lie algebra cochains of the cyclic chain complex $\mathfrak{gl}_NEnd_\mathcal{C}(\lambda)$ with its standard BD algebra structure.\footnote{We recall this around equation \eqref{c_pq}. The appearing $V$ there would here be the underlying cyclic chain complex of the cyclic $L_\infty$-algebra $\mathfrak{gl}_NEnd_\mathcal{C}(\lambda)$.} We elaborate on its meaning:
\begin{remark}\label{DRDZA}
For $N\rightarrow\infty$ the LQT map becomes a quasi-isomorphism \cite{LQ84}\cite{Tsy83}. From a physics perspective the open BD algebra $\mathcal{F}^o(\lambda)$ thus describes the large $N$ observables of the \emph{open string gauge theory on a stack of $N$ branes of $\lambda$}. We elaborate a bit more around equation \eqref{copex}. Compare further \cite{CL15} where this is related to \textbf{holomorphic Chern-Simons theory} in the case of the B-model using analytical methods to deal with the appearing infinite dimensional spaces. Holomorphic Chern-Simons theory had been identified as the open string field theory of the B-model in~\cite{wi95}. 
\end{remark}
\medskip
We now present the results of this thesis, in comparison to the theory outlined so far.
\subsubsection{Kontsevich's Cocycle Construction and the Quantized Loday-Quillen-Tsygan Theorem}\label{gvnlmss}
 Following results are based on the preprint \cite{Ul25a} and focus on an `open' analogue of the \textbf{representing map} \eqref{omgkbm} from the closed geometric BD algebra to the closed algebraic BD algebra and on \textbf{open} analogues of the \textbf{Maurer-Cartan element}, the string vertices~\eqref{csvmce}.

\bigskip
The definition of the open geometric BD algebra \eqref{ogbd} is quite abstract and we would like to have a combinatorial description to study it better, ie. to potentially find interesting Maurer-Cartan elements. A first result obtained in this thesis is the following:
consider the ribbon graph complex with leafs, whose free boundaries are decorated by the set $\Lambda$; see definition \ref{Ddrgc}. After adjoining a formal variable we have: 
\begin{thm_nn_a}[theorem \ref{rbg_comp}]
  The ribbon graph complex  $\mathcal{RG}_{\Lambda}\llbracket\gamma\rrbracket$ decorated by $\Lambda$ can be equipped with structure of BD algebra, induced by gluing leafs.
\end{thm_nn_a}
We denote this BD algebra (by abuse of notation) by the same letter. Results in the literature indicate that $\mathcal{RG}_{\Lambda}\llbracket\gamma\rrbracket$ gives a combinatorial description of the \textbf{geometric open BD algebra}
\begin{equation}\label{comdes}
\mathcal{RG}_{\Lambda}\llbracket\gamma\rrbracket\simeq M^0_\Lambda,
\end{equation}
which we did not directly explore in detail in this thesis, however.\footnote{Up to signs and possible gradings, this claim should follow from section 6.2 of \cite{Cos07a}, see also \cite{Cos07b} and may be straightforward to experts.} Using this combinatorial description we state the next result in this thesis, the `open' analogue of theorem \ref{cei}, which connected the closed geometric BD algebra with the closed algebraic BD algebra: 
\begin{thm_nn_b}[theorem \ref{Kc2.1}]
Given an odd dimensional cyclic $A_\infty$-category $\Lambda$ there is a \textbf{representing map} of \textbf{open} BD algebras
\begin{equation}\label{ijijij}
\begin{tikzcd}
\mathcal{RG}_{\Lambda}\llbracket\gamma\rrbracket\arrow{r}{\rho_{\Lambda}}&\mathcal{F}^{o}(\Lambda).
    \end{tikzcd}
    \end{equation}
    \end{thm_nn_b}
    
    \begin{remark}\label{kccc}
We explain how theorem B is a generalization of Kontsevich's cocycle construction from \cite{Kon92b}: let us denote by $\mathcal{M}_{g,b}$ the moduli space of connected Riemann surfaces of genus $g$ with $b$ marked points. It is a standard result that 
$$\mathcal{RG}_{m=0}^{con.}\simeq \bigoplus_{stable}C_*( \mathcal{M}_{g,b}),$$
where $\mathcal{RG}_{m=0}^{con.}$ stands for connected ribbon graphs without leafs. 
Given a cyclic $A_\infty$-algebra $A$ we restrict the map $\rho_A$ from \eqref{ijijij} to finite sums of connected graphs without leaves and its $\gamma=0$ part, denoted $\rho|_{m=0}$. 
The composite  
$$\begin{tikzcd}
     \bigoplus_{stable}C_*( \mathcal{M}_{g,b})\simeq\mathcal{RG}_{m=0}^{con.}\arrow{r}{\rho|_{m=0}}&k[\nu,\gamma]\arrow{r}{\nu=\gamma=1}&k
\end{tikzcd}$$
 coincides with Kontsevich's construction \cite{Kon92b}, ie. produces a cocycle on the moduli space $\mathcal{M}_{g,b}$, which is also a corollary of theorem B above. We are actually additionally weighting Kontsevich's map by $\gamma$ to the power of the genus of the respective ribbon graph.
\end{remark}
\bigskip
A natural question is whether we can find a universal geometric source  for the target of the LQT map from equation~\eqref{lqtintro}  $$LQT_N:\mathcal{F}^o(\lambda)\rightarrow \mathcal{O}^{pq}\big(\mathfrak{gl}_NEnd_\mathcal{C}(\lambda)\big),$$ resulting in a commutative diagram of BD algebras. In fact a more general result is true:
using the graph complex from definition \ref{hfig} and the fact that any cyclic chain complex (in particular any cyclic $L_\infty$-algebra $L$) induces a BD algebra $\mathcal{O}^{pq}(L)$, built from Lie algebra cochains\footnote{but endowed only with the differential induced from the underlying chain complex of the $L_\infty$-algebra - see around theorem \ref{c_pq} for details.} we can state the answer.
\begin{thm_nn_c}[theorems \ref{rbg_comp} and \ref{Kc3.1}]
  The  graph complex  $\mathcal{G}\llbracket\hbar\rrbracket$  can be equipped with the structure of a BD algebra, induced from gluing leafs. Given an odd dimensional cyclic $L_\infty$-algebra $L$ there is a \textbf{representing map} of BD algebras
$$
{\theta_{L}}:\ \mathcal{G}\llbracket\hbar\rrbracket\rightarrow\mathcal{O}^{pq}(L).
$$
   \end{thm_nn_c}
    \begin{remark}
        This map generalizes Penkava's `commutative' cocycle construction on the graph complex \cite{pe96} similarly as explained in remark \ref{kccc} on Kontsevich's construction.
    \end{remark}
\bigskip
A next result in this thesis is generalizing\footnote{After the writing of the relevant paper had been completed the author became aware of \cite{zeng23}, which deals with closely related phenomena.} the original LQT theorem from $A_\infty$-algebras to $A_\infty$-categories. This is natural as we generalized its source to arbitrary $A_\infty$-categories $\Lambda$, for which we have in mind some full subcategory on a choice of Lagrangians of the Fukaya category of a symplectic manifold. We recall that the original version of the LQT map connects $(Cyc^*_+(A),d_{cyc})$, cyclic cochains of an $A_\infty$-algebra $A$ computing its cyclic cohomology, and Lie algebra cochains $\big(C_*(\mathfrak{gl}_NA),d_{CE}\big)$ of the Lie algebra $\mathfrak{gl}_NA$ of $N\times N$ matrices with values in $A$. 
\begin{thm_nn_d}[see theorems \ref{LQT_c} and \ref{LQT_conj}]
Given a small $A_\infty$-category $\Lambda$ there is a map of dg-algebras 
\begin{equation}\label{ndgmgh}
\begin{tikzcd}
LQT_{\Lambda}:\ \Big(Sym\big(Cyc^*_+(\Lambda)[-1]\big),d_{cyc}\Big)\arrow{r}& \Big(C^*(\mathfrak{gl}_NA_\Lambda),d_{CE}\Big),
\end{tikzcd}
\end{equation}
which is functorial in $\Lambda$ with respect to $A_\infty$-functors. Here $\mathfrak{gl}_NA_\Lambda$ is what we call the commutator $L_\infty$-algebra with values in $N\times N$-matrices associated to $\Lambda$. If $\Lambda$ is unital then for $N\rightarrow \infty$ this map becomes a quasi-isomorphism
   $$ \begin{tikzcd}
        LQT_{\Lambda}:\ Sym\big(Cyc^*_+(\Lambda)[-1]\big)\arrow{r}{\sim}& \underset{N}{\varprojlim}\ C^*(\mathfrak{gl}_NA_\Lambda).\end{tikzcd}$$
\end{thm_nn_d}
To state our main next theorem we want to understand how the LQT map from theorem D is compatible with the BD algebra structure that one can construct from some modifications of its domain in case the category is cyclic; see equation \eqref{OBDa}. In fact this BD algebra depends only on the underlying collection of cyclic chain complexes,\footnote{A collection of cyclic chain complexes
$$V_B:=\big(V_{ij}, \langle\_,\_\rangle_{ij}\big)_{i,j\in B}$$
consists of a set $B$ and chain complexes $V_{ij}$ indexed by $i,j\in B$ together with symmetric non-degenerate pairings $\langle\_,\_\rangle_{ij}:V_{ij}\otimes V_{ji}\rightarrow k[-d]$ .
Note that Kontsevich-Soibelman \cite{KoSo24} refer to such an object as a quiver.} denoted $V_\Lambda$, that is not on the (higher) compositions. To make this dependence more clear, which will also be useful for the rest of this section, we change notation and set
     $$\mathcal{F}^{pq}(V_\Lambda):=\mathcal{F}^o(\Lambda),$$
     where the superscript $pq$ stands for \textbf{prequantum}.\footnote{By what we have in mind a quantization of the free theory.}
We note that to equip some modification of the codomain of the quantized LQT map with a BD algebra structure we need to restrict to \emph{suitably finite} cyclic $A_\infty$-categories. Then we can construct a quantized LQT map for multi-objects, which we do in theorem \ref{LQT_pq}, generalizing theorem 4.9 of \cite{GGHZ21}: given $V_\Lambda$ a collection of odd dimensional cyclic chain complexes where $\Lambda$ is finite there is a 2-weighted map of BD algebras\footnote{The appearing BD algebras live over different base rings, over $k\llbracket\gamma\rrbracket$ and $k\llbracket\hbar\rrbracket$ respectively. Speaking of a 2-weighted map of BD algebras means by definition \ref{weimap} that ${LQT(\gamma)=\hbar^2}$.}
    \begin{equation}\label{omghiadamnm}
        LQT^{pq}_N:\ \mathcal{F}^{pq}(V_\Lambda)\rightarrow \mathcal{O}^{pq}(\mathfrak{gl}_NV_\Lambda),
        \end{equation}
      functorial with respect to strict cyclic $A_\infty$-morphisms. 
\bigskip

Using the multi-object BD LQT map \eqref{omghiadamnm} we come to a first main result of this thesis, which connects the universal geometric open BD algebras with the quantized LQT map and the generalized Kontsevich cocycle construction. 
\begin{thm_nn_e}[theorem \ref{main_compa}]
Given an essentially finite cyclic odd dimensional $A_\infty$-category $\Lambda$ there is a commutative diagram of BD algebras relating the quantized LQT map with Kontsevich's cocycle construction and the universal geometric BD algebras built from (ribbon) graphs.
$$\begin{tikzcd}
    \mathcal{F}^{pq}(\Lambda)\arrow{r}{LQT^{pq}_N}&\mathcal{O}^{pq}(\mathfrak{gl}_NA_\Lambda)\\
\mathcal{RG}_{\Lambda}\llbracket\gamma\rrbracket\arrow{u}{\rho_\Lambda}&\mathcal{G}\llbracket\hbar\rrbracket\arrow{l}\arrow{u}{\theta}.
\end{tikzcd}$$
Here the right vertical maps is given by theorem C applied to the cyclic $L_\infty$-algebra $\mathfrak{gl}_NA_\Lambda$ and the upper horizontal map is 2-weighted, whereas the lower horizontal map is dual to such a map.
\end{thm_nn_e}

\bigskip

In the previous section \ref{cceis} a crucial ingredient in constructing the closed categorical Gromov-Witten potential \eqref{cgw} were the closed string vertices $$SV^c\in MCE(M^c)$$ from equation \eqref{csvmce}. It thus seems a natural question whether there exist an \textbf{analogous Maurer-Cartan element} of $\mathcal{RG}_{\Lambda}\llbracket\gamma\rrbracket$ (and similarity for $\mathcal{G}\llbracket\hbar\rrbracket$). However, this cannot be the case.\footnote{This would imply that any cyclic $L_\infty$-algebra admits a canonical quantization, which we know is not the case using obstruction theoretic ideas \cite{Co11}.} To understand this situation better we turn to introducing and examining smaller and bigger cousins of the (ribbon) graph complex which do carry such Maurer-Cartan elements.

Before we do that we recall the notion of a \textbf{shifted Poisson algebra} (see definition \ref{shPoA}), consisting of a multiplication on a graded vector space $P$ and a differential graded shifted Lie algebra structure satisfying the Leibnitz rule  $$(P,\cdot,d,\{\_,\_\}),$$
which is a sort of \textbf{classical limit of a Beilinson-Drinfeld algebra} $W$. This idea is formalized by the notion of a dequantization map (see definition \ref{dq})
$$p: W\rightarrow P.$$
To find Maurer-Cartan elements we introduce the (ribbon) tree complexes (definitions \ref{Ddrgc} and \ref{tc}) and the stable (ribbon) graph complexes (definition \ref{Ddrgc} and \ref{sG}), originally discussed in \cite{Kon92a} and \cite{geKa96}. The stable ribbon graph complex has a close relationship with the Deligne-Mumford compactification of the moduli space of Riemann surfaces, see \cite{Ha07, Ba10}. We prove how these complexes can be equipped with algebraic structure, how they are interrelated and that there are certain canonical Maurer-Cartan elements:
\begin{thm_nn_f}[theorems \ref{rbg_comp}, \ref{ag_dq} and section \ref{vsgzg}, \ref{RS_mce}, theorems \ref{sRG_mce}, \ref{S_mce} and \ref{sG_mce}]
Given a set $\Lambda$ the ribbon tree complex $\mathcal{RT}_{\Lambda}$ can be equipped with a shifted Poisson algebra structure and the stable ribbon graph complex 
$s\mathcal{RG}_{\Lambda}\llbracket\gamma\rrbracket$ with a BD algebra structure, induced from gluing leafs.
There is a commutative diagram
    \begin{equation}\label{icnam} \begin{tikzcd}
\mathcal{RT}_{\Lambda}&\mathcal{RG}_{\Lambda}\llbracket\gamma\rrbracket\arrow[swap]{l}{p}\arrow{r}{v}&
s\mathcal{RG}_{\Lambda}\llbracket\gamma\rrbracket\arrow[bend left=25]{ll}{p}
    \end{tikzcd},    
    \end{equation}
    where the $p$ maps are dequantization maps and $v$ is a map of BD algebras. Analogously the tree complex $\mathcal{T}$ can be equipped with a shifted Poisson algebra structure and the stable  graph complex 
$s\mathcal{G}_{B}\llbracket\hbar\rrbracket$ with a BD algebra structure. We have a commutative diagram

 \begin{equation}\label{ag_dq1} \begin{tikzcd}
\mathcal{T}&\mathcal{G}\llbracket\hbar\rrbracket\arrow[swap]{l}{p}\arrow{r}{v}&
s\mathcal{G}\llbracket\hbar\rrbracket\arrow[bend left=25]{ll}{p}
    \end{tikzcd},
    \end{equation}
   where the $p$ maps are dequantization maps and $v$ is a map of BD algebras. Furthermore there are maps $\pi$ from the respective graph to ribbon graph complexes, which respect the algebraic structures and are compatible as follows:
\begin{equation}\label{yangningduscht}
\begin{tikzcd}
\mathcal{T}\arrow{d}{\pi}&\mathcal{G}\llbracket\hbar\rrbracket\arrow{d}{\pi}\arrow[swap]{l}\arrow{r}&
s\mathcal{G}\llbracket\hbar\rrbracket\arrow{d}{\pi}\arrow[bend right=20]{ll}\\
\mathcal{RT}_{\Lambda}&\mathcal{RG}_{\Lambda}\llbracket\gamma\rrbracket\arrow[swap]{l}\arrow{r}&
s\mathcal{RG}_{\Lambda}\llbracket\gamma\rrbracket\arrow[bend left=20]{ll}.
    \end{tikzcd}
    \end{equation}

There are explicit non-zero \textbf{Maurer-Cartan elements}
$$D\in MCE(\mathcal{RT}_{\Lambda})\ \text{and}\ T\in MCE(\mathcal{T})$$
in the respective shifted Poisson algebras; further there are explicit non-zero
$$S\in MCE(s\mathcal{RG}_{\Lambda}\llbracket\gamma\rrbracket)\ \text{and}\ G\in MCE(s\mathcal{G}\llbracket\hbar\rrbracket)$$
in the respective  BD algebras. Those are all compatible under the maps from diagram \eqref{yangningduscht}.
\medskip
\end{thm_nn_f}
To explain the relevance of these specific Maurer-Cartan elements and their relationship with \textbf{quantization} we need to elaborate:
given a collection of cyclic chain complexes $V_\Lambda$ we denote  $$\mathcal{F}^{fr}(V_\Lambda):=\Big(Sym(Cyc^*_+(V_\Lambda)[-1]),\cdot,d,\{\_,\_\}_o\Big),$$
the `\textbf{free open observables}',
which is a shifted Poisson algebra; see around equation \eqref{fP} and note the differences to the object from equation \eqref{OBDa}. Further there is 
a dequantization map 
 $$p:\mathcal{F}^{pq}(V_\Lambda)\rightarrow \mathcal{F}^{fr}(V_\Lambda)$$
from the `\textbf{prequantum open observables}' to the `\textbf{free open observables}'. A well-known result (see lemma \ref{Ham1}) is that an odd dimensional cyclic $A_\infty$-category $\Lambda$ with underlying collection of cyclic chain complexes $V_\Lambda$ induces a Maurer-Cartan element \begin{equation}\label{Hamint}
I\in MCE(\mathcal{F}^{fr}(V_\Lambda)),
\end{equation} 
also called hamiltonian of the cyclic $A_\infty$-category.
\begin{df}\label{incqcyc}
 A \emph{quantization} of a cyclic $A_\infty$-category $\Lambda$ is a Maurer-Cartan element \begin{equation}\label{uzwq}
    I^q \in\mathcal{F}^{pq}(V_\Lambda)\ \ \text{such that}\ \ p(I^q)=I.
    \end{equation}
    \end{df}
\begin{remark}
Quantizations of cyclic $A_\infty$-categories  are sometimes called quantum $A_\infty$-categories.
Equivalently they may be described as a possibly curved (involutive bi-Lie$)_\infty$-algebra structure on a central extension of cyclic cochains of this category \cite{Ca16, CiFuLa20}.
\end{remark}

We can provide an interpretation for the Maurer-Cartan elements from theorem F as being the geometric analogues of the notion of hamiltonian of a cyclic $A_\infty$-category and of their quantizations:
\begin{thm_nn_g1}[see theorem \ref{comp_rg_al} and lemmas \ref{comp_1} and \ref{comp_2}]
Let $\Lambda$ be an odd dimensional cyclic $A_\infty$-category with underlying collection of cyclic chain complexes $V_\Lambda$.
Denote by ${I\in MCE(\mathcal{F}^{fr}(V_\Lambda))}$ its associated hamiltonian.
This datum induces the two left vertical maps of shifted Poisson algebras respectively of Beilinson-Drinfeld algebras in the diagram below.
Given a quantization of $\Lambda$, denoted $I^q\in MCE(F^{pq}(V_\Lambda))$, we further get the right vertical map of BD algebras, compatible as follows:
\begin{equation}\label{tot_nc}
\begin{tikzcd}
\mathcal{F}^{fr}(V_\Lambda)&\mathcal{F}^{pq}(V_\Lambda)\arrow{l}{p}&
\mathcal{F}^{pq}(V_\Lambda)\arrow[equal]{l}\\
\mathcal{RT}_{\Lambda}\arrow{u}{\rho_{I}}&\mathcal{RG}_{\Lambda}\llbracket\gamma\rrbracket\arrow[swap]{l}\arrow{r}\arrow{u}{\rho_{I}}&
s\mathcal{RG}_{\Lambda}\llbracket\gamma\rrbracket\arrow[bend left=20]{ll}{p}\arrow{u}{\rho_{I^q}}.
    \end{tikzcd}
    \end{equation}
    Furthermore for the \textbf{Maurer-Cartan elements} from theorem F we have that $$\rho_I(D)=I\ \text{and }\ \rho_{I^q}(S)=I^q.$$
    
\end{thm_nn_g1}

\medskip

There is a totally parallel story for cyclic $L_\infty$-algebras: indeed for a cyclic $L_\infty$-algebra $L$ we can introduce the notion of hamiltonian, which is given by a Maurer-Cartan element $I$ in a shifted Poisson algebra $\mathcal{O}^{fr}(V)$, built from the underlying cyclic chain complex $V$ of $L$; see around equation \eqref{shPoC}. There is a dequantization map, see definition \ref{fc_dq},
$$p:\mathcal{O}^{pq}(V)\rightarrow \mathcal{O}^{fr}(V).$$
A \emph{quantization} of an odd dimensional cyclic $L_\infty$-algebra $L$ is a Maurer-Cartan element \begin{equation}\label{luzwq}
    I^q \in\mathcal{O}^{pq}(V)\ \ \text{such that}\ \ p(I^q)=I.
    \end{equation}
These notions provide an interpretation of the other Maurer-Cartan elements in theorem F as being the geometric analogues of the notion of hamiltonian of a cyclic $L_\infty$-algebra and of their quantizations:
\begin{thm_nn_g2}[see theorem \ref{comp_gr_al} and lemmas \ref{comp_3}, \ref{comp_4}] Let $L$ be a cyclic $L_\infty$-algebra with underlying cyclic chain complex $V$.
Denote by $I\in MCE(\mathcal{O}^{fr}(V))$ its associated hamiltonian.
This datum induces the two left vertical maps of shifted Poisson algebras respectively of Beilinson-Drinfeld algebras in the diagram below.
Given a quantization of $L$, denoted $I^q\in MCE(\mathcal{O}^{pq}(V))$, we further get the right vertical map of BD algebras, compatible as follows:
$$\begin{tikzcd}
\mathcal{O}^{fr}(V)&\mathcal{O}^{pq}(V)\arrow{l}{p}&
\mathcal{O}^{pq}(V)\arrow[equal]{l}\\
\mathcal{T}^d\arrow{u}{\theta_{I}}&\mathcal{G}^d\llbracket\hbar\rrbracket\arrow[swap]{l}\arrow{r}\arrow{u}{\theta_{I}}&
s\mathcal{G}^d\llbracket\hbar\rrbracket\arrow[bend left=20]{ll}\arrow{u}{\theta_{I^q}}.
    \end{tikzcd}$$
  Furthermore for the Maurer-Cartan elements from theorem F we have that $$\theta_I(T)=I\ \text{and }\ \theta_{I^q}(G)=I^q.$$
       
\end{thm_nn_g2}
\bigskip

Our next theorem provides a very satisfying characterization of the twisted (ribbon) tree complexes as universal sources of domain and target of the classical LQT map, that is cyclic cochains of a cyclic $A_\infty$-category respectively Lie algebra cochains of a cyclic $L_\infty$-algebras:
\begin{thm_nn_i}[see theorems \ref{Main}, \ref{fde}]\label{comdiacla}
Given an essentially finite odd dimensional cyclic $A_\infty$-category $\Lambda$ there is a commutative diagram of shifted Poisson algebras
$$\begin{tikzcd}
    \Big(Sym\big(Cyc^*_+(\Lambda)[-1]\big),d_{cyc},\{\_,\_\}\Big)\arrow{r}{LQT_{}}&\Big(C_+^*(\mathfrak{gl}_NA_\Lambda),d_{CE},\{\_,\_\}\Big)\\
    \mathcal{RT}^{tw}_{\Lambda}\arrow{u}{}&\mathcal{T}^{tw}\arrow{l}{}\arrow{u}{},
\end{tikzcd}$$
functorial with respect to strict cyclic $A_\infty$-morphisms. The (ribbon) tree complexes are now \textbf{twisted} by the Maurer-Cartan elements from theorem F. The vertical right map is defined also for cyclic $L_\infty$-algebra not induced by a cyclic $A_\infty$-categories.
\end{thm_nn_i}
 It would be interesting to explore whether these universal sources induce some canonical cohomology classes and to find a different interpretation of those.

\begin{remark}
A natural question is how far the construction from theorem I can be made homotopically coherent.
One may try to generalize theorem I to proper Calabi-Yau categories, but allowing shifted Poisson $\infty$-algebras and maps of those.\footnote{
It is to be expected that proper Calabi-Yau categories can be strictified to a cyclic one. See theorem 10.2.2 of \cite{KoSo24} and \cite{ChLe10} for a proof in the algebra case and theorem 2.34 of \cite{AmTu22} for the category case, but where additionally smoothness, see definition \ref{smanpr}, is required.}
\end{remark}
\smallskip
It is an immediate question to find a quantized version of theorem I, which we tackle now: we denote the left upper corner of the diagram in theorem I by $\mathcal{F}^{cl}(\Lambda)$ and the right upper corner by $\mathcal{O}^{cl}(\mathfrak{gl}_NA_\Lambda),$ the `\textbf{classical observables}'. Given a quantization of a cyclic $A_\infty$-category ${I^q \in\mathcal{F}^{pq}(V_\Lambda)}$ we denote the twist of $\mathcal{F}^{pq}(V_\Lambda)$ by that Maurer-Cartan element by $\mathcal{F}^q(\Lambda)$ and we denote the twist of $\mathcal{O}^{pq}(\mathfrak{gl}_NA_\Lambda)$ by the determined quantization of the cyclic $L_\infty$-algebra by $\mathcal{O}^q(\mathfrak{gl}_NA_\Lambda)$, the `\textbf{quantum observables}'. 

With this we present a quantized version of theorem I, which is a main result in this thesis:
\begin{thm_nn_j}[see theorem \ref{LQT_q} and theorem \ref{Main}] Given a quantization of an essentially finite odd dimensional cyclic $A_\infty$-category, there is a commutative diagram of $BD$ algebras
$$
\begin{tikzcd}
    \mathcal{F}^q(\Lambda)\arrow{r}{LQT^q_N}&\mathcal{O}^q(\mathfrak{gl}_NA_\Lambda)\\
s\mathcal{RG}^{tw}_{\Lambda}\llbracket\gamma\rrbracket\arrow{u}&s\mathcal{G}^{tw}\llbracket\hbar\rrbracket\arrow{l}\arrow{u}{},
\end{tikzcd}$$
functorial with respect to strict cyclic $A_\infty$-morphisms of quantizations of cyclic $A_\infty$-categories. Furthermore for $N\rightarrow\infty$ the map $LQT^q_N$ becomes a quasi-isomorphism. Lastly there are dequantization maps to theorem I as follows:
 $$ \begin{tikzcd}[row sep=scriptsize, column sep=scriptsize]
& \mathcal{F}^{cl}(\Lambda)  \arrow[rr,"LQT^{cl}"] & & \mathcal{O}^{cl}(\mathfrak{gl}_NA_\Lambda)   \\ \mathcal{F}^{q}(\Lambda)\arrow[rr, "LQT^q"{xshift=-7pt}]  \arrow[ur] & & \mathcal{O}^{q}(\mathfrak{gl}_NA_\Lambda)\arrow[ur] \\
& \mathcal{RT}^{tw}_\Lambda\arrow[uu,dashed, "\rho_{I}^{tw}" {yshift=-10pt}]   & & \mathcal{T}^{tw}\arrow[ll,dashed]\arrow[ uu,"\theta_{I_N}^{tw}" {yshift=-8pt}]  \\
 s\mathcal{RG}^{tw}_{\Lambda}\llbracket\gamma\rrbracket \arrow[uu,"\rho_{I^q}^{tw}" {yshift=5pt}]\arrow[ur] & & s\mathcal{G}^{tw}\llbracket\hbar\rrbracket\arrow[ll]\arrow[ur]\arrow[uu,"\theta_{I^q_N}^{tw}" {yshift=-10pt}],\\
\end{tikzcd}
$$
that is the 4 maps between the front and back face are dequantization maps. 
\end{thm_nn_j}

\subsubsection{Open-Closed Circle-Action Formality Morphism}\label{nmsendd}
In the previous section \ref{gvnlmss} we analyzed in great details the `open analogue' of  the representing map \eqref{omgkbm} connecting the closed universal geometric BD algebra \eqref{cgbd} with the closed algebraic BD algebra \eqref{cBD}. This closed representing map was crucial to define the categorical Gromov-Witten potential \eqref{cgw} in section \ref{cceis}, for which we further used the  canonical Maurer-Cartan element in the closed geometric BD algebra \eqref{cgbd}, the string vertices \eqref{csvmce}.

\smallskip
We explained in the previous section that there is \textbf{no analogous Maurer-Cartan element} in the open universal geometric BD algebra from theorem B, which we quantified in theorem G.1. 

\smallskip
As a step towards defining an \textbf{open categorical Gromov-Witten potential}, we ask if it is anyways possible to obtain a \textbf{quantization} of a cyclic $A_\infty$-category. Based on the preprint \cite{Ul25}, we propose to consider the combined open-closed theory to provide an answer.

\medskip
Given a cyclic $A_\infty$-category $\mathcal{C}$ and a full subcategory $\Lambda\subseteq\mathcal{C}$ the `\textbf{open-closed observables}' are given by the tensor product of the closed BD algebra \eqref{cabd} and the open BD algebra \eqref{OBDa},\footnote{which exists, see lemma \ref{tensorBD}.} 
\begin{equation}\label{ocSFTBD}
\mathcal{F}^c(\mathcal{C})\otimes \mathcal{F}^o(\Lambda).\end{equation}
Similarly we can consider
$$\mathcal{F}^c(\mathcal{C})^{Triv}\otimes \mathcal{F}^o(\Lambda),$$
recalling the abelian BD algebra from \eqref{abBD}.
The next main result in this thesis is following open-closed version of the closed \textbf{trivializing morphism} from equation \eqref{cltrmo}:
\begin{thm_nn_k}[corollary \ref{maco}]
Given a cyclic $A_\infty$-category $\mathcal{C}$, a set of objects $\Lambda\subset \mathcal{C}$ and a splitting $s$ of the non-commutative Hodge filtration. Then there is an $L_\infty$ quasi-isomorphism
\begin{equation}\label{otrmo}
\Psi_s^{oc}:\ \mathcal{F}^{c}(\mathcal{C})\otimes \mathcal{F}^{o}(\Lambda)\rightsquigarrow \mathcal{F}^{c}(\mathcal{C})^{Triv}\otimes \mathcal{F}^{o}(\Lambda).
\end{equation}
\end{thm_nn_k}
\begin{remark}
The proof of this theorem reduces to proving an identity (theorem \ref{BVinf}) for the Taylor components of the purely closed $L_\infty$-morphism \eqref{cltrmo}. It would be interesting to have an operadic interpretation of this identity, perhaps in terms of $BV_\infty$-algebras \cite{GTV12}. 
\end{remark}
\begin{remark}
As is typical in this area of mathematical physics the morphism \eqref{otrmo} is defined in terms of graphs. To verify theorem K we prove an interesting identity (equation \ref{wfegtffv}) about the `moduli space of graphs', which is about relating graphs with varying number of vertices to each other, with the additional choice of a partition of the half-edges of a given vertex.
\end{remark}
\bigskip
We now explain how theorem K provides a building stone in understanding when and how to obtain a quantization of a given \emph{smooth} cyclic $A_\infty$-category, or for simplicity a given {smooth} cyclic $A_\infty$-algebra $End_\mathcal{C}(\lambda)$. 

For that we introduce the map $$\nu\langle ch(\lambda),\_\rangle_{res}:\mathcal{F}^c(\mathcal{C})\otimes \mathcal{F}^o(\lambda)\rightarrow \mathcal{F}^c(\mathcal{C})\otimes \mathcal{F}^o(\lambda),$$ the derivation induced by pairing the non-commutative Chern character\footnote{defined as the class of the unit of $End(\lambda)$ in negative cyclic chains} $ch(\lambda)$ via the residue pairing \eqref{Res}, further weighted by $\nu$ (the basis of $Cyc^0(\lambda)[-1]$ cyclic cochains of length zero).

\medskip
Our following results are valid for tuples $(\mathcal{C},\lambda)$ of smooth cyclic categories $\mathcal{C}$ and $\lambda\in\mathcal{C}$ satisfying a assumption $(*)$. Forthcoming work \cite{AmTu25} combined with the fundamental work \cite{Zw98,HVZ08} about the open-closed string vertices will show that this assumption is  satisfied for all such tuples; see section \ref{FOCS} for some more context.

\begin{ass}
Given a smooth cyclic $A_\infty$-category $\mathcal{C}$ and an object ${\lambda\in \mathcal{C}}$ there is an element $S^{oc}\in \mathcal{F}^c(\mathcal{C})\otimes \mathcal{F}^o(\lambda)$ that is a Maurer-Cartan element up an extra term coming form the Chern character of $\lambda$
\begin{equation}\label{Ass}
(d_{hoch}+uB+\gamma\Delta_{c}+d+\gamma\delta+\nabla)S^{oc}+\frac{1}{2}\{S^{oc},S^{oc}\}_{o}+\frac{1}{2}\{S^{oc},S^{oc}\}_{c}+\nu\langle ch(\lambda),S^{oc}\rangle_{res}=0,
\end{equation}
and the open tree level part\footnote{That is the projection to $Sym^0(\cdots)\otimes Sym^1(Cyc^*(\lambda)[-1])\gamma^0$.} of $S^{oc}$ is the hamiltonian $I$ of the cyclic $A_\infty$-algebra $End_\mathcal{C}(\lambda),$
\begin{equation}\label{HamfurIn}
I:=\sum_{k=1}^\infty\langle m^k(\_,\cdots,\_),\_\rangle\in Cyc^*(\lambda)[-1].
\end{equation}
\end{ass}
\medskip
If we assume that the Chern character is homotopically trivial we can `gauge away' the last summand of equation \eqref{Ass}, which is our proposition \ref{ki}:
    given a smooth cyclic $A_\infty$-category $\mathcal{C}$ and ${\lambda\in \mathcal{C}}$ that satisfy assumption $(*)$ and such that $ch(\lambda)=(d_{hoch}+uB)h$. Denote
    $$S^{oc,q}(\lambda,h):=e^{\nu\langle h,-\rangle_{M}}S^{oc}$$ Then we have that \begin{equation}\label{aiwigfmf}
        S^{oc,q}(\lambda,h)\in QME\big(\mathcal{F}^c(\mathcal{C})\otimes \mathcal{F}^o(\lambda)\big).
        \end{equation}
    Further if we change $h$ by an exact term $S^{oc,q}(\lambda,h)$ changes up to homotopy.
\begin{df}\label{iwdsbv}
We call a \emph{package} the datum of a smooth cyclic $A_\infty$-category $\mathcal{C}$ together with the choice of an object $\lambda\in\mathcal{C}$ of that category, the choice of a trivialization $(d_{hoch}+uB)h=ch(\lambda)$ and a splitting $s$   \begin{equation}\label{P}
\mathcal{P}_{\mathcal{C},\lambda,h,s}:=\{\mathcal{C},\lambda,h,s\}. \end{equation}
\end{df}
A package allows us to find the quantization we searched for: consider the map of dg Lie algebras $$\pi:\mathcal{F}^{c}(\mathcal{C})^{Triv}\otimes \mathcal{F}^{o}(\Lambda)\rightarrow Sym^0(\cdot)\gamma^0\otimes \mathcal{F}^o(\lambda)\cong \mathcal{F}^o(\lambda)$$
and denote the image of the gauged element \eqref{aiwigfmf} under the open-closed trivializing morphism from theorem K and the map $\pi$ by  $$S^{o,q}(\mathcal{C},\lambda,h,s):=(\pi\circ\Psi_s^{oc})S^{oc,q}(\lambda,h).$$
Then we can find an \textbf{`open' quantization}:
\begin{thm_nn_l}[theorem \ref{oq}]
   Given a package $\mathcal{P}_{\mathcal{C},\lambda,h,s}$ from definition \ref{iwdsbv} such that assumption $(*)$ holds then $S^{o,q}(\mathcal{C},\lambda,h,s)$ is a quantization  of the cyclic $A_\infty$-algebra $End_\mathcal{C}(\lambda)$. 
\end{thm_nn_l}
In section \ref{cceis} we described how one can extract from a solution to the closed Maurer-Cartan equation the closed categorical Gromov-Witten potential, which in turn allowed to extract actual numbers. \emph{Theorem L tells us under which conditions we can find a solution to the open Quantum Master Equation.} However, we have not understood how to extract numbers from such an element at the moment. Geometrically, in open Gromov-Witten theory one imposes some extra conditions on the Lagrangians in order to extract actual numbers: eg. in \cite{Fuk11} it is required that the Lagrangian has the homology of a sphere and in \cite{psw08} that the Lagrangian is the locus of an involution. One could ponder whether similar constraints for the categorical set-up allow to extract actual numbers. 

\medskip

\textbf{From a physical point of view}, however, theorem $L$ tells us when we can find a quantization of the classical large $N$ open string field theory on a stack of $N$ branes of type $\lambda$, implied by the quantized LQT map; see remark \ref{DRDZA}. Elaborating a bit, the field content of the open string field theory on a stack of $N$ branes of type $\lambda$ is described by the cyclic $L_\infty$-algebra
\begin{equation}\label{copex}
End_\mathcal{C}(\lambda^{\oplus N}).\end{equation}
Note that the object \eqref{copex} only exists in categories $\mathcal{C}$ that are additive. In this case we can further identify the cyclic $L_\infty$-algebra \eqref{copex} (or better its Lie-algebra cochains) with the target of the LQT-map \eqref{lqtintro}:
\begin{equation}\label{qwertzuio}
End_\mathcal{C}(\lambda^{\oplus N})\cong \mathfrak{gl}_NEnd_\mathcal{C}(\lambda).
\end{equation}
Not all the categories we are interested in are additive, eg. Fukaya categories. In this case the isomorphism \eqref{qwertzuio} allows us to view the algebra of $N\times N$ matrices with values in $End_\mathcal{C}(\lambda)$, which also exists in non-additive categories, as a convenient replacement for $End_\mathcal{C}(\lambda^{\oplus N})$.

\smallskip
 The perspective advocated in \cite{CL15,Ba10b, GGHZ21} is that a lift of the classical interaction term \eqref{HamfurIn}, which is the large $N$ limit of the classical interaction terms associated to the cyclic $L_\infty$-algebras \eqref{copex}, to a solution to the open Quantum Master Equation describes, within the Batalin-Vilkovisky formalism from theoretical physics and under the quantized LQT map, the quantum partition function of a large $N$ matrix model (the finite dimensional cousins of large $N$ gauge theories).\footnote{See the exciting recent work \cite{ha25a, ha25b} which allows to include these infinite-dimensional cases.} This was worked out in \cite{GGHZ21} for the Gaussian matrix model.

\begin{remark}
    In \cite{CL15}, page 3 it is argued that open string field theory in general and (large $N$) holomorphic Chern-Simons theory on $\mathbb{C}^3$ in particular does not admit a quantization by itself, in apparent\footnote{Note that we can't apply theorem L as Dolbeault forms on $\mathbb{C}^3$ do not verify the properness assumption. See however the recent work \cite{ha25a,ha25b} and also \cite{CL15}, which may allow to loosen the properness restriction at the cost of introducing more analysis.} contrast to what we claim here. However note that the quantizations we produce in theorem L have `contributions' from the closed sector by stable graphs with zero leaves and thus do depend on both the open and closed sectors, which is actually similar in spirit to the main idea of \cite{CL15}. We would like to examine this further in the future.
\end{remark}
 In the context of theorem L and motivated by the physics discussion we call \textbf{\emph{partition function of the large $N$ open string gauge theory on a stack of $N$ $\lambda$-branes}} 
the exponential of the Maurer-Cartan element ${S^{o,q}(\mathcal{C},\lambda,h,s)\in QME(\mathcal{F}^o(\lambda))},$
\begin{equation}\label{offpart}
Z^o(\mathcal{P}_{\mathcal{C},\lambda,h,s}):=e^{S^{o,q}(\mathcal{C},\lambda,h,s)/\gamma},
\end{equation}
which is a closed element with respect to the differential of the BD algebra.\footnote{This is a general fact about BD algebras; we further should suitably localize to allow power series in $\gamma$.}

\subsubsection{Towards Twisted Holography from Calabi-Yau Categories}\label{twhocy}
In the last part of the thesis we build upon the physics perspective laid out around equation \ref{copex} to conjecture a \emph{relationship between the {partition function of the large $N$ open string gauge theory on a stack of $N$ $\lambda$-branes} and a deformation of the closed categorical Gromov-Witten potential}, which we will introduce in the following. It is useful to recall the interpretation of the GW potential as the partition function of closed string field theory; see remark~\ref{DRDEA}. Note that despite the physics name these partition functions should be encoding enumerative invariants and indeed \emph{our conjecture aims to give a new viewpoint on known and proposed deep relationships} \cite{Kon92a, GoVa99, cg21}, which we explain at the very end of this section.

\smallskip
The ideas presented in this section are still in an early stage and have not appeared in print; for instance some shifting conventions are yet to be fixed.

\medskip
We recall that our set-up consists of a smooth cyclic $A_\infty$-category $\mathcal{C}$ together with the choice of `a brane`, an object $\lambda\in\mathcal{C}$ of that category.\footnote{We will work with pairs $(\mathcal{C},\lambda)$ satisfying assumption $(*)$ from equation \eqref{Ass}. As it will be shown that \emph{all} smooth cyclic $A_\infty$-category $\mathcal{C}$ with a choice of  $\lambda\in\mathcal{C}$ satisfy $(*)$ in the forthcoming work \cite{AmTu25} we do not mention this assumption each time in this section.} We referred to such datum with additional choice of a trivialization $(d_{hoch}+uB)h=ch(\lambda)$ and a splitting $s$ of the non-commutative Hodge filtration (definition \ref{splitting}) as a package \begin{equation}\label{holographypack}
\mathcal{P}_{\mathcal{C},\lambda,h,s}:=\{\mathcal{C},\lambda,h,s\}. \end{equation}
We call \emph{holographic} a pair consisting of a smooth cyclic $A_\infty$-category $\mathcal{C}$ together with an object $\lambda\in\mathcal{C}$, provided a (for the moment relatively idealistic) condition, see definition \ref{HolDef}.
In this case we will then also call $\mathcal{P}_{\mathcal{C},\lambda,h,s}$  a \textbf{holography package}.
\begin{remark}
It seems suggestive that this condition or a possible weakening of it may be satisfied if the endomorphisms of the object have the rational homology of a sphere, but investigating this has to be left to future work.    
\end{remark}
Given a holography package an immediate consequence\footnote{We imagine that one could relax the condition considerably on what it means to be holographic in our sense, which would on the other hand make this consequence less immediate.}  is that there exists a closed element
\begin{equation}\label{brclgwp}
Z^{c,br}(\mathcal{P}_{\mathcal{C},\lambda,h,s})\in \mathcal{F}^{c}(\mathcal{C})^{Triv}\llbracket\nu\rrbracket,\end{equation}
which we call \textbf{\emph{partition function of the closed string field theory associated to $\mathcal{C}$ backreacted in the direction of $\lambda$}}, see definition \ref{yangning}. Here $\mathcal{F}^{c}(\mathcal{C})^{Triv}\llbracket\nu\rrbracket$ is the abelian BD algebra from definition \ref{abBD}, extended $\nu$-linearly. 
\begin{remark}
The variable $\nu$ is a placeholder for the number of branes $N$. Also in physics the backreacted closed string field theory depends highly non-trivial on $N$ \cite{Ma99, GoVa99, cg21}.\footnote{In more detail, we recall that $\nu$ is the basis of cyclic cochains of length zero from the open sector, keeping track of the number of free boundaries under the map from theorem B, see also around theorem \ref{ocRep}. Under the LQT map at rank $N$ $\nu$ gets mapped to $N$.} 
\end{remark}
\smallskip
Very roughly speaking the partition function $Z^{c,br}(\mathcal{P}_{\mathcal{C},\lambda,h,s})$ is determined by some deformation of the closed categorical Gromov-Witten potential from \eqref{cgw} associated to the category $\mathcal{C}$ in the direction of $\lambda$. 
More precisely the element $Z^{c,br}(\mathcal{P}_{\mathcal{C},\lambda,h,s})$ is the large $N$ quantum partition function of a classical field theory; its defining $L_\infty$-algebra is obtained by twisting a \emph{curved} $L_\infty$-algebra on cyclic cochains of $\mathcal{C}$, whose curvature is determined by the brane $\lambda$; see equation \eqref{Krumm} below.

\smallskip
This seems to exactly mimic the heuristic how physicists understand the \textbf{backreaction} phenomena of closed string field theory, compare the beginning of section \ref{mimsss}.
In the discussion in remark \ref{sihdrdnls} below we make this more precise and compare this description closely with the treatment of backreaction of the closed B-model in Costello-Gaiotto's approach to `Twisted Holography'~\cite{cg21}. 

\bigskip
We come to the mentioned conjecture: in theoretical physics there is the enormously influential idea of holography, going back to the seminal work \cite{Ma99}. Roughly speaking holography proposes to exactly match a gravitational theory living on some bulk spacetime with a gauge theory on the boundary.
Often these two sides can be realized as a closed backreacted string field theory in the direction of some brane, respectively as an large $N$ open string field theory on a stack of this brane.
Thus holography proposes that in particular the partition functions of these two theories should be identified. 

Having in mind the interpretation of $Z^{c,br}(\mathcal{P}_{\mathcal{C},\lambda,h,s})$ from equation \ref{brclgwp} as the partition function of the \textbf{closed string field theory backreacted in direction $\lambda$} and the interpretation of $Z^o(\mathcal{P}_{\mathcal{C},\lambda,h,s})$ from equation \ref{offpart} as the \textbf{large $N$ open string field theory on a stack of $N$ branes of type $\lambda$} it seems natural to conjecture a version of \textbf{holography} in our set-up (which is about topological string theories, more precisely those induced by Calabi-Yau categories):
\newpage
\begin{Conj_nn_h}\label{conj_q}
Given a holography package (definition \ref{holographypack}) the partition function of the backreacted closed string field theory is equal to the partition function of the large $N$ open string gauge theory on a stack of $N$ $\lambda$-branes; or more precisely the element \eqref{brclgwp} is homologous to the element \eqref{offpart}, that is $$\iota_c(Z^{c,br}(\mathcal{P}_{\mathcal{C},\lambda,h,s}))\simeq \iota_o(Z^o(\mathcal{P}_{\mathcal{C},\lambda,h,s})),$$
under the natural inclusion maps 
$$\begin{tikzcd}
   & \mathcal{F}^{c}(\mathcal{C})^{Triv}\otimes \mathcal{F}^{o}(\lambda)&\\
   \mathcal{F}^{c}(\mathcal{C})^{Triv}\llbracket\nu\rrbracket\arrow{ur}{\iota_c}&&\mathcal{F}^{o}(\lambda)\arrow{ul}[swap]{\iota_o}.
\end{tikzcd}$$
\end{Conj_nn_h}
We view this conjecture as a main accomplishments of this thesis. Its formulation is a step in understanding a plethora of \textbf{extremely rich and important relations in enumerative geometry} \cite{Kon92a, GoVa99, cg21} (and proposals thereof from physics) from a unified mathematical perspective. See also remark \ref{DepOnPar} for ideas on a refined version of this conjecture, allowing both sides to depend on natural parameters; a key step is a formality statement, see theorem \ref{sddgnngi}. We elaborate on the relation of mentioned works to conjecture H at the end of this introduction under `applications'.

\medskip
Before that, to support our interpretation of $Z^{c,br}(\mathcal{P}_{\mathcal{C},\lambda,h,s})$ as the partition function of closed backreacted SFT, we comment on the treatment of backreaction in Costello-Gaiotto's approach to `Twisted Holography', notably their treatment of BCOV theory in section 4 of \cite{cg21}. A reader not so familiar with these ideas may want to skip the long remark \ref{sihdrdnls} and go directly to the proposed applications of conjecture H.

\smallskip
\begin{remark}\label{sihdrdnls}
To further support our interpretation of $Z^{c,br}(\mathcal{P}_{\mathcal{C},\lambda,h,s})$ we first observe the following: the genus zero part of the element $Z^{c,br}(\mathcal{P}_{\mathcal{C},\lambda,h,s})$ describes (or rather its preimage under the $L_\infty$ quasi-isomorphism \eqref{otrmo}) for every $N\in\mathbb{N}$ an $L_\infty$-structure on the chain complex $CC^*(\mathcal{C})[1]$, cyclic cochains of $\mathcal{C}$, that is on the fields of the closed string field theory, by setting $\nu=N$. We denote this $L_\infty$-algebra by $L_N^{br}(\mathcal{C},\lambda,h)$; see definition \ref{brlinalsb} for details. The $L_\infty$-structure $L_N^{br}(\mathcal{C},\lambda,h)$ is roughly speaking obtained by twisting an explicit \emph{curved} $L_\infty$-structure $L_N(\mathcal{C},\lambda)$ on cyclic cochains by a primitive $h$ of the Chern character; see remark \ref{dbegaj} and its preceding discussion. We explain in lemma \ref{isidfrz} that the curvature of the mentioned curved $L_\infty$-algebra $L_N(\mathcal{C},\lambda)$ is given (to leading order in $N$) by  \begin{equation}\label{Krumm}
N\langle ch(\lambda),\_\rangle_{res}\in CC^*(\mathcal{C})[1].
\end{equation}
We come to Costello-Gaiotto's treatment of backreaction in the closed B-model:
\begin{bcov}
 Costello-Gaiotto explain how the dg-Lie algebra of classical interacting BCOV theory on $\mathbb{C}^3$ (the classical closed SFT of the B-model, see sections \ref{scibcov}, \ref{inclsfts}, conjecture \ref{conj_fo})
\begin{equation}
\big(PV_c^{\bullet,\bullet}(\mathbb{C}^3)\llbracket u\rrbracket,\bar{\partial}+u\partial,[\_,\_]\big)
\end{equation}
gets modified to the curved dg-Lie algebra
\begin{equation}\label{brcrdglie}
\big(PV_c^{\bullet,\bullet}(\mathbb{C}^3)\llbracket u\rrbracket,N\delta_\mathbb{C},\bar{\partial}+u\partial,\{\_,\_\}\big)
\end{equation}
when introducing the brane $\lambda\in \mathcal{D}^{b}_{dg}(\mathbb{C}^3)$ determined by 
$\mathbb{C}\subset \mathbb{C}^3$. Twisting by a (in this situation essentially unique) Maurer-Cartan element, a Beltrami differential, they recover an uncurved dg-Lie algebra. They claim that this (uncurved) dg-Lie algebra describes (deformations of the complex structure of) the \emph{deformed conifold} (see section 4.1 of \cite{cg21}), which is the backreacted geometry of the closed string field theory on $\mathbb{C}^3$ in the direction of the brane $\lambda$. 
\end{bcov}
Inspired by Costello-Gaiotto's work we view (respectively call) the non-curved $L_\infty$-structure $L_N^{br}(\mathcal{C},\lambda,h)$ as the \emph{classical backreacted closed string field theory in the direction of $\lambda$}. See also conjecture \ref{diaestf} on the \emph{existence of backreacted Calabi-Yau categories} and its surrounding discussion for more context.

Basically we would like to realize the dg Lie algebra of Costello-Gaiotto describing the deformed conifold as $L_N^{br}(\mathcal{C},\lambda,h)$ for  $\mathcal{C}=\mathcal{D}^{b}_{dg}(\mathbb{C}^3)$, $\lambda$ given by $\mathcal{O}(\mathbb{C})^{\oplus N}$ for $\mathbb{C}\subset \mathbb{C}^3$ and $h$ given by their Mauer-Cartan element. We run into the problem that $\mathcal{D}^{b}_{dg}(\mathbb{C}^3)$ is not proper nor cyclic, thus we cannot apply our theory.\footnote{Perhaps working with compactly supported polyvector fields could resolve this. In another direction, to define the structure $L_N^{br}(\mathcal{C},\lambda)$ we are only using the genus zero part of the open-closed TCFT associated to $\mathcal{C}$; compare section \ref{2dop}. Results in the literature \cite{brdy23} seem to suggest that already the structure of a smooth Calabi-Yau induces such a structure, which $\mathcal{D}^{b}_{dg}(\mathbb{C}^3)$ would satisfy.} One can identify the underlying chain complex of \eqref{brcrdglie} and thus the underlying chain complex of the dg Lie algebra describing the deformed conifold with $CC^*\big(\mathcal{D}^{b}_{dg}(\mathbb{C}^3)\big)$ using an HKR isomorphism (see page 36 of~\cite{CL15}). Further, under this isomorphism the curvature term of \eqref{brcrdglie} should be identified with the term \eqref{Krumm}.
In any case, even if we could apply our theory, see footnote in this remark, an identification (in the sense of some $L_\infty$ quasi-isomorphism) of the whole dg Lie structures could be difficult.  

The partition function of the closed backreacted string field theory
$$Z^{c,br}(\mathcal{P}_{\mathcal{C},\lambda,h,s})\in \mathcal{F}^{c}(\mathcal{C})^{Triv}\llbracket\nu\rrbracket$$
from \eqref{brclgwp} defines a quantization of the classical backreacted closed string field theor(ies), defined by $L_N^{br}(\mathcal{C},\lambda,h)$. We note the intricate $N$-dependence, encoded by the $\nu$-parameter, which is also present in physics~ \cite{Ma99,cg21}.
\end{remark}

\bigskip
\textbf{Applications:} we now elaborate on the mentioned relations and dualities in enumerative geometry. We imagine that those could be formulated in the framework of conjecture H;\footnote{or perhaps in a refined version of the conjecture, depending on certain parameters; see remark \ref{DepOnPar}.} thus on the one hand they provide evidence for the validity of conjecture H and on the other hand could possibly be re-proven if conjecture H would be shown to be true.
\begin{itemize}
    \item 
The finite versions of large $N$ open string field theories are large $N$ matrix models. The prime example of such comes from the work of Kontsevich \cite{Kon92a}.
He considered the following integral over hermitian $N\times N$ matrices and proved that when $N \rightarrow \infty$ it can be written, suitably normalized, as a rational polynomial
in variables $tr(\Lambda^{-2i+1}), i\in\mathbb{N} $, where $\Lambda$ is just some parameter matrix.
$$\lim_{N\rightarrow \infty}\int_{\mathfrak{h}_N}e^{itr(X^3)}e^{tr(X\Lambda X)}dX\ \in\ \mathbb{Q}[tr(\Lambda^{-1}),tr(\Lambda^{-3}),tr(\Lambda^{-5}),..].$$
He showed that the coefficients of this polynomial are exactly the intersection numbers on the moduli space of Riemann surfaces. We have seen in theorem \ref{JTT} that we can recover the intersection numbers as the partition function of the closed string field theory, as proved by Tu \cite{Tu21}, when taking the category $\mathcal{C}$ to have one object with endomorphisms being the rational numbers. Thus strikingly Kontsevich's work tells us that \emph{open and closed invariants coincide}, even though the backreaction phenomena seems not to be visible.
\end{itemize}
At this point we wish to remark that the large $N$ matrix model representation of the intersection numbers provided a \emph{major advantage}. Those are amenable by different means as well, for instance formal PDE methods, allowing Kontsevich to prove that the intersection numbers are a logarithm of the $\tau$-function of the KdV-hierarchy, thus verifying Witten's conjecture \cite{wi93}.
\begin{itemize}
  \item
An analogous instance of a large $N$ matrix model computes r-spin intersection numbers \cite{BCEGF23}. Those numbers, also known as the FJRW invariants of $\frac{x^r}{1+r}$ are (after identifying them with the B-model LG invariants, known as Givental-Saito theory, by applying LQ mirror symmetry, proved in \cite{he15}) also conjecturally encoded by the closed Gromov-Witten potential of the Calabi-Yau category of matrix factorizations, see \cite{CaLiTu18} for a verification in low genus. Again, we do not see any backreaction phenomena here. 
\end{itemize}

Interestingly there is no $N$-dependence on the closed side for the two previous examples contrary to conjecture H; further there was no backreaction phenomena visible. This is different with the next example.

\begin{itemize}
\item
An important inspiration for conjecture H is the work of \cite{GoVa99}, see also \cite{EkSh25} for a more recent approach.  Gopakumar-Vafa propose that the partition function of large $N$ Chern-Simons on $S^3$, which is an infinite dimensional matrix model (a gauge theory) exactly encodes the Gromov-Witten invariants of the resolved conifold, described as the backreacted geometry of $S^3$ in the direction of the zero-section. The zero section is a Lagrangian of $T^*S^3$, that is an object $\lambda$ of $Fuk(T^*S^3)$. On the other hand the observables of large $N$ Chern-Simons theory can be described \cite{GGHZ21, ha25a, ha25b}\footnote{ See also \cite{CiVo23}, \cite{CiFuLa20} with applications to string topology.} by applying the open BD algebra $\mathcal{F}^o(\lambda)$ to the `cyclic' differential graded algebra of de Rham form on $S^3$, which computes $End_{Fuk(T^*S^3)}(\lambda)$, compare \cite{ab11}. We should note that in the quantization of large $N$ Chern-Simons theory a so called framing anomaly is present  \cite{as92, Kon92b,ia10}. The treatment of \cite{ha25a}  (following \cite{co07}) is modulo constants and does not address this point. The choice of a trivialization of the Chern character in theorem L seems very suggestive in capturing this phenomena. Note, however, that the smoothness assumption is not satisfied for de Rham forms, a comparison for the obtained quantizations may anyways be possible.

\item As mentioned before, Costello-Gaiotto \cite{cg21} propose and verify in some aspects a duality between the so called large $N$ $\beta\gamma$-system and the B-model of the deformed conifold. The latter is described as the backreaction of the closed string field theory of the Calabi-Yau category of coherent sheaves on $\mathbb{C}^3$ in the direction of the brane $\mathbb{C}\subset \mathbb{C}^3$, inducing an object $\lambda\in \mathcal{D}^{b}_{dg}(\mathbb{C}^3)$ of the bounded category of coherent sheaf. The B-model, BCOV invariants of the deformed conifold, are conjecturally computed by $Z^{c,br}(\mathcal{P}_{\mathcal{C},\lambda,h,s})$ given the choice of the Beltrami-differential for the primitive of $ch(\lambda)$ (see lemma 1 of \cite{cg21}) and the splitting determined by the complex conjugate; compare \cite{CaTu24}, page 9. On the other hand the large $N$ $\beta\gamma$-system is the open string field theory on the brane $\mathbb{C}\subset \mathbb{C}^3$; see appendix E of \cite{cg21}.

Costello-Gaiotto focus on local aspects of their proposed duality; conjecture H would be about the partition functions of large $N$ $\beta\gamma$-system and the partition function of the closed SFT of the B-model of the deformed conifold.
\end{itemize}

Costello-Gaiotto conjecture that their proposed duality describes a subsector of the celebrated $AdS_5/CFT_4$ correspondence in physics \cite{Ma99}.
    
\medskip
Notably in the two last examples listed there is an extremely rich local structure present on both sides (mathematically this can be captured by the notion of factorization algebras \cite{cogw21}), which is supposed to be matched by holography as well. For instance in \cite{cg21} the open side carries the structure of a vertex operator algebra. We contrast this with conjecture H, which is about the global structure. It would be highly interesting to eventually combine these two threads together.

\bigskip
Given the impact of mirror symmetry in the field of mathematics and its origin in physics one can wonder whether  holography - phenomenally influential in physics - could also eventually become a rich source of inspiration for mathematics. See \cite{coli16,co17,Gai24,gai25} for deep developments in this direction.      
\newpage
\section{Overview over the Chapters}
In section 3 we put some of the results of this thesis in a broader context by reviewing related works.
 \begin{itemize}
\item In section 3.1 we remind the reader of Costello's work \cite{Cos07a} on topological conformal field theories, the notions of modular, $S^1$-framed modular operads and cyclic operads and of the notion of totalization. 
\item In section 3.2 we comment on Barannikov's work on modular operads \cite{Ba10}.
\end{itemize}
The main part of this thesis is organized as follows:
\begin{itemize}
\item In section 4 we introduce some algebraic preliminaries. Section 4.1 is dedicated to BD and shifted Poisson algebras, section 4.2 to (cyclic) $L_\infty$-algebras and section 4.3 to  $A_\infty$-categories, Hochschild and cyclic homology. In section 4.4 we introduce Calabi-Yau structures and in section 4.5 splittings of the non-commutative Hodge filtration.
\item In section 5.1 and 5.2 we recall the definition of free closed SFT associated to a smooth cyclic $A_\infty$-category and define the notion of interacting closed SFT. In section 5.3 and 5.4 we compare these notions with BCOV theory when applied to the B-model, for the interacting case in terms of  conjecture \ref{conj_fo}.
\item In section 6 we introduce and recall shifted Poisson and BD structures induced from cyclic $A_\infty$-categories and their quantizations. We study carefully the effect of restricting to essentially finite cyclic $A_\infty$-categories.
\item In section 7.1 and 7.2 we recall shifted Poisson and BD structures induced from cyclic $L_\infty$-algebras and their quantizations. In section 7.3 we study a functor from $A_\infty$-categories to $A_\infty$-algebras; in section 7.4 we specialize to finite cyclic $A_\infty$-categories.
\item Section 8.1 is about setting up the LQT map for $A_\infty$-categories and proving the large $N$ statement. We explain how to incorporate the shifted Poisson and BD algebra structures in section 8.2.
\item In section 9.1 we study shifted Poisson and BD algebras built from various ribbon graph complexes; in section 9.2 we do the same but for various (ordinary) graph complexes. Further we identify certain non-trivial Maurer-Cartan elements. In section 9.3 we connect the previous two subsections.
\item In section 10 we explain the generalized Kontsevich's cocycle construction, both in the commutative and non-commutative setting. It connects the (ribbon) graphs with the algebraic structures induced from cyclic $A_\infty$-categories and cyclic $L_\infty$-algebras.
\item In section 11 we show how the LQT map, graph complexes and generalized Kontsevich cocycle construction fit together, and how quantum and classical setting are connected. 
\item Section 12 is about the open-closed circle action formality morphism, depending on a splitting of the non-commutative Hodge filtration.
\item Section 13 is about ideas how to phrase a `Twisted Holography' duality \cite{cg21} at the level of partition functions, given as input a Calabi-Yau category and one of its objects. 

\end{itemize}
\newpage
\newpage

\newpage
\section{Some Related Works}
In this section we wish to recall some related works and begin by \cite{Cos07a}. Even though we do not directly use its results in this thesis it provides the backbone of much of the theory we study. In particular it is crucial in proving theorem \ref{cei} about the closed representing map, which we motivate and explain with proposition \ref{ssgb}. Similar this theory should give a different view on theorem B from the introduction about the open representing map, which we indicate around proposition \ref{dglsgnm}. To do this we recall the notions of cyclic and modular operads. In the last section we comment on the relation to Barannikov's work \cite{Ba10}.
\subsection{Topological Conformal Field Theories}\label{2dop}
We begin by recalling the work of \cite{Cos07a}. 
Here, a crucial input is the symmetric monoidal category, depending on a set $\Lambda$ of `branes', \begin{equation}\label{OpSyMo}
\mathscr{O}_\Lambda^d:=C_*(\mathcal{M}^{o}_\Lambda,\mathcal{L}^d),
\end{equation}
from \cite{Cos07a}, definition 3.0.2, which  is given as follows. Its objects are tuples of tuples of $\Lambda$ like $$(\lambda_1^a,\lambda_1^b),\cdots, (\lambda_n^a,\lambda_n^b)$$ for $\lambda_i^x\in \Lambda,$ the monoidal structure $\cup$ induced by disjoint union. The Hom spaces of $\mathscr{O}_\Lambda^d$ are built from chains on the moduli space of (not necessarily connected) stable Riemann surfaces with boundary
and \emph{ordered} intervals embedded into the boundary, the free boundaries decorated by objects of a set $\Lambda$ and with coefficients in a certain line bundle $\mathcal{L}^d$. One imposes the condition that there is at least one free boundary on each boundary component. The composition is induced by gluing together surfaces along open boundaries with matching boundary conditions, as shown in the figure below.
\begin{remark}
    We emphasize that here we are considering ordered embedded intervals, which is crucial for the symmetric monoidal structure. The open BD algebra \eqref{ogbd} that we considered before was built from the same moduli space, only that we considered \emph{unordered} embedded intervals. One goal of this section is to explain how the BD algebra \eqref{ogbd} is related to the category $\mathscr{O}_\Lambda^d$. 
\end{remark}

\begin{figure}[h]
\centering
\includegraphics[]{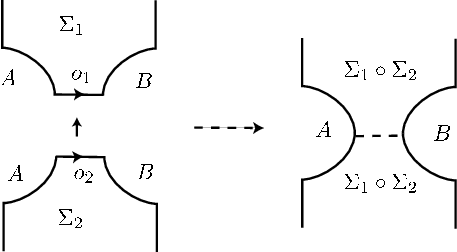}
\caption{An example of composition in $\mathcal{O}_\Lambda$, from \cite{Cos07a}}
\end{figure}
\begin{df}[\cite{Cos07a}, section 5]
An open topological conformal field theory (TCFT) is a symmetric monoidal functor 
$$F:(\mathscr{O}_\Lambda^d,\cup)\rightarrow (Ch,\otimes),$$
which is split.
\end{df}
\begin{theorem}[\cite{Cos07a}, theorem A.1]\label{OTCFTcyc}
There is an equivalence of categories between the category of dimension $d$ cyclic $A_\infty$-category $\mathcal{C}$ with set of objects $B$ and the category of open TCFT's 
$$F:(\mathscr{O}_\Lambda^d,\cup)\rightarrow (Ch,\otimes).$$

\end{theorem}
The open TCFT $F_\mathcal{C}$ corresponding under above theorem to a cyclic $A_\infty$-category $\mathcal{C}$ satisfies
\begin{equation}\label{tensf}
F_\mathcal{C}\big((\lambda_1^a,\lambda_1^b),\cdots, (\lambda_n^a,\lambda_n^b)\big)=Hom_\mathcal{C}(\lambda_1^a,\lambda_1^b)\otimes\cdots \otimes Hom_\mathcal{C}(\lambda_n^a,\lambda_n^b). \end{equation}
In a similar way Costello \cite{Cos07a} defines another symmetric monoidal category (definition~3.0.2)
 \begin{equation}\label{ClSyMo}
\mathscr{C}^d:=C_*(\mathcal{M}^{c},\mathcal{L}^d),
\end{equation}
its object given by finite sets (think a disjoint union of circles) and
its Hom-spaces  built from chains on the moduli space of stable Riemann surfaces with ordered framed boundary and with coefficients in a certain line bundle. One imposes the condition that there is always at least one incoming framed boundary.

Similarly one defines a symmetric monoidal category (definition 3.0.2 of \cite{Cos07a})
\begin{equation}\label{OpClSyMo}
\mathscr{OC}^d_\Lambda:=C_*(\mathcal{M}^{oc}_\Lambda,\mathcal{L}^d),
\end{equation}
by basically combining the previous ones (see section 1.2 of \cite{Cos07a}). Here one imposes that there is always at least one free boundary or one incoming closed boundary.
\begin{df}[\cite{Cos07a}, section 5]\label{octcft123}
A closed TCFT is a symmetric monoidal functor 
$$F:(\mathscr{C}^d,\cup)\rightarrow (Ch,\otimes)$$
which is split.
An open-closed topological conformal field theory TCFT is a symmetric monoidal functor 
$$F:(\mathscr{OC}_B^d,\cup)\rightarrow (Ch,\otimes)$$
which is split.
\end{df}
\begin{theorem}[theorem A.1 and A.2, \cite{Cos07a}]\label{KanExt}
Given an open TCFT, corresponding under theorem \ref{OTCFTcyc} to a cyclic $A_\infty$-category $\mathcal{C}$, there is an associated universal open-closed TCFT $\Phi$ such that given an object of $\mathscr{C}^d$ (ie. a set $C$) we have $$\Phi(C)\simeq CH_*(\mathcal{C})^{|C|},$$
where $CH_*(\mathcal{C})$ denotes the Hochschild chains of the category $\mathcal{C}.$
We also get an induced closed TCFT with the same property.
\end{theorem}

\subsection{Cyclic and Modular Operads}\label{cycmodop}
In this section 3.2 we work only with a $\mathbb{Z}/2\mathbb{Z}$-grading as we have not checked the degrees and some shifts carefully and take for simplicity $\Lambda=\{*\}$. We nevertheless wished to include this chapter as it provides a more conceptual explanation for eg. theorem B (as we explain at the end of section \ref{opensec}) and a heuristic  for theorem \ref{cei} (as we explain at the end of section \ref{cloSec}).

We recall from \cite{geKa96} the notion of cyclic and of modular operad and from \cite{DoShVaVa24}, definition 1.15 the notion of an odd shifted modular operad.
\begin{ex}
For $d$ odd, considering the connected parts of the Hom-spaces of $\mathscr{O}^d$ with zero outgoing intervals defines an odd shifted modular operad;
the partial composition and contraction maps are defined by gluing with the disc with two outgoing intervals, see section 6 of \cite{Cos07a} and also the treatment of \cite{HVZ08}. By abuse of notation we denote this odd shifted modular operad also by $\mathscr{O}^d$.
The genus zero part of $\mathscr{O}^d$ defines an odd shifted cyclic operads. 
\end{ex}
An open TCFT induces in particular an algebra $End(V)_n:=\big(V^\vee[-1]\big)^{\otimes n}$ over the odd shifted modular operad; the partial composition and contraction maps (recalling definition 1.15 of \cite{DoShVaVa24}) are defined by pairing with the inverse of the non-degenerate pairing coming from the image of the disk with two outgoing intervals; see lemma 7.3.4 of \cite{Cos07a} for comparison. 
\begin{remark}\label{dasistGG}
In fact the same statements are not true for    
$\mathscr{OC}^d$ and $\mathscr{C}^d$. If we consider a surface with two incoming and no outgoing framed boundaries nor free boundaries or embedded intervals and we glue together these two incoming framed boundaries we would land outside of $\mathscr{C}^d$, as we required positive number of inputs. For $\mathscr{O}^d$ the situation is different because gluing two open intervals always results in a free boundary. The requirement of positive input for closed TCFTs is crucial to prove theorem \ref{KanExt} above. The closed TCFT induced from a cyclic $A_\infty$-category can in general not be extended to zero inputs. However the results of \cite{Lu09} indicate that this should be possible if we additionally ask our category to be \emph{smooth}. See also later around diagram \ref{smooRes} later.
\end{remark}

Next, we recall from \cite{DoShVaVa24}, definition 2.1 the notion of odd shifted dg $\Delta$ Lie algebras:
\begin{df}
An \emph{odd shifted dg  $\Delta$-Lie algebra} consists of a chain complex $(\mathfrak{g},d)$, a degree $1$ linear operator 
$\Delta : \mathfrak{g} \rightarrow \mathfrak{g}$, and a degree $1$ symmetric product $\{\_ , \_\} : \mathfrak{g}^{\odot 2} \rightarrow \mathfrak{g}$
satisfying the following relations:
\[\Delta^2=0 \ , \qquad  d \Delta+\Delta d =0\ ,
\quad d \ \text{and} \ \Delta\ \ \text{are derivations of }\ \ \{\_ ,\_\}\ ,
\]
and we ask that the shifted Jacobi identity for  \ $\{\_ ,\_\}$ holds.
\end{df}

Further in lemma 2.8 of \cite{DoShVaVa24}  a functor is introduced
\begin{equation}\label{total}
Tot:\text{odd shifted modular operads}\rightarrow \text{odd shifted dg $\Delta$ Lie algebras}
\end{equation}
which restricts to 
\begin{equation}\label{cyctotal}
Tot_{cyc}:\text{odd shifted cyclic operads}\rightarrow \text{dg odd shifted Lie algebras.}
\end{equation}
In fact we could have factored some of our constructions in this thesis through shifted dg $\Delta$ Lie algebras by the following functor, as we will see momentarily.
 \begin{lemma}
     Given an odd shifted dg $\Delta$ Lie algebras
$(\mathfrak{g},d,\Delta ,\{\_ , \_\})$ defining
$$BD(\mathfrak{g}):=Sym(\mathfrak{g})\llbracket\gamma\rrbracket,$$
where $\gamma$ is a formal variable of even degree we equip $BD(\mathfrak{g})$ with a BD algebra structure  
$$\Big(BD(\mathfrak{g}),\cdot,d+\gamma\Delta+\gamma\delta_e,\{\_ , \_\} \Big)$$
    where $d$ and $\Delta$ denote the differentials extended as derivations and $\delta_e$ the Chevalley-Eilenberg differential of the Lie bracket, further $\{\_ , \_\}$ denoting the Lie bracket extended $\gamma$-linearly and according to the Leibniz rule. We find a functor
\begin{align}\label{BDfSL}
Fr^\gamma:\text{odd shifted dg $\Delta$ Lie algebras}&\rightarrow \text{BD algebras}\\
\mathfrak{g}&\mapsto BD(\mathfrak{g})\nonumber,
\end{align}
whose $\gamma=0$ part gives a functor
$$Fr:\text{odd shifted dg Lie algebras}\rightarrow \text{shifted Poisson algebras}.
    $$
  \end{lemma}
  \begin{remark}

  Compare also \cite{djpp22}, where a functor
$$\text{odd shifted modular operads with connected sum}\rightarrow \text{BD algebras}$$
is constructed, which should be a finer construction than to just consider the free algebra on the totalization as in lemma \ref{BDfSL}.  
\end{remark}
\subsubsection{Open Sector\label{opensec}}
As remarked at the beginning of the last section given an open TCFT of odd dimension we get an algebra over the shifted modular operad $\mathscr{O}^d\rightarrow End(V)$. We can apply the totalization functor \ref{total}, followed by the functor from lemma \ref{BDfSL}, which results in a map of BD algebras
\begin{equation}\label{absthe}
Fr^\gamma(Tot(\mathscr{O}^d))\rightarrow Fr^\gamma(Tot(End(V)))
\end{equation}

It is easy to see that $$Fr^\gamma(Tot(\mathscr{O}^d))=M^0,$$
where $M^0$ denotes the BD algebra \eqref{ogbd} for $\Lambda=\{*\}.$
Further we can identify 
$$Tot(End(V))=\Big(Sym(V[-1]^\vee),d,\Delta_o,\{\_,\_\}_o\Big),$$
where $d$ is the derivation induced from the underlying differential of the chain complex $(V,d)$ and the BV Laplacian $\Delta$ and the shifted Poisson bracket $\{\_,\_\}$ are the standard ones on Lie algebra cochains of a complex with non-degenerate pairing; see around \ref{c_pq}. Then it follows that
$$Fr^\gamma(Tot(\mathscr{O}^d))=\Big(Sym(Sym(V[-1]^\vee))\llbracket\gamma\rrbracket,\cdot,d+\gamma\Delta+\gamma\delta_e,\{\_,\_\}_o\Big),$$
where $\delta_e$ denotes the Chevalley-Eilenberg differential associated to $\{\_,\_\}_o$. The multiplication gives a map of BD algebras
$$
\Big(Sym(Sym(V[-1]^\vee))\llbracket\gamma\rrbracket,\cdot,d+\gamma\Delta_o+\gamma\delta_e,\{\_,\_\}_o\Big)\rightarrow \Big(Sym(V[-1]^\vee)\llbracket\gamma\rrbracket,\cdot,d+\gamma\Delta_o,\{\_,\_\}_o\Big)\cong \mathcal{O}^{pq}(V).$$
Summarizing, the map \eqref{absthe} composed with multiplication induces\footnote{We ignored here the difference between the formal variables $\hbar$ and $\gamma$; compare section \ref{slqtq}. } a map, denoted  
\begin{equation}\label{wwdaf}
j:\ M^0\rightarrow \mathcal{O}^{pq}(V).
\end{equation}
It is not difficult to see that we can actually, by the LQT map of rank 1 (recalling the definition of its domain from \ref{OBDa})
$$LQT_1:\ \mathcal{F}^{pq}(V)\rightarrow \mathcal{O}^{pq}(V),$$
 factor the map \eqref{wwdaf} through a map 
$$
M^0\rightarrow \mathcal{F}^{pq}(V),
$$
since we get an induced symmetric ordering of cyclic orderings of the tensors in the image of the map $j$ from \eqref{wwdaf}, coming from that of symmetric ordering of cyclic orderings of the embedded intervals into the boundaries of a surface.

Precomposing with the quasi-isomorphism of BD algebras mentioned in~\eqref{comdes}
$$M^0\simeq \mathcal{RG}\llbracket\gamma\rrbracket$$
we have shown the following:
\begin{prop}\label{dglsgnm}
Given an open TCFT $F$, for simplicity with $\Lambda=\{*\}$, so that $F(*,*)=V$  there is an induced map of BD algebras
\begin{equation}\label{simonstehauf}\tilde{\rho}_F:\ \mathcal{RG}\llbracket\gamma\rrbracket\rightarrow\mathcal{F}^{pq}(V),
\end{equation}
induced from the totalization functor \ref{total}.
 \end{prop}

\begin{Claim}[`Conjecture']\label{ibssm}
 The map \eqref{simonstehauf}, given an open TCFT, coincides with the map
$$\rho_A:\ \mathcal{RG}\llbracket\gamma\rrbracket\rightarrow \mathcal{F}^{pq}(V)$$
  from theorem B applied to the cyclic $A_\infty$-algebra $A$ corresponding to the open TCFT under theorem~\ref{OTCFTcyc}. 
\end{Claim}
Up to signs and possible gradings, this claim should follow from Costello's work, eg. proposition 6.2.1 of \cite{Cos07a} and generally section 6.2 of \cite{Cos07a}, see also \cite{Cos07b} and may be straightforward to experts.
We did not check the details on signs and gradings and thus formulated the above as a conjecture.
\subsubsection{Closed Sector: $S^1$-framed modular operads}\label{cloSec}
We recall the notion of an $S^1$-framed modular operads from \cite{Tu19}, also called $S^1$ -equivariant modular operads, there. Using this notion, we explain in this section the origin of the fundamental closed representing map from the universal closed BD algebra to the algebraic closed BD algebra from theorem \ref{cei}.
\begin{lemma}[sections 2.4 and 5.1 of \cite{Tu19}]
There is a functor 
\begin{equation}\label{S1tot}
Tot_{S^1}:\ \text{$S^1$-framed modular operads}\rightarrow \text{dg $\Delta$ shifted Lie algebras},\end{equation}
whose underlying chain complex is given by taking the homotopy $S^1$-quotient with respect to an $S^1$-action; for the algebraic structure we refer to sections 2.4 and 5.1 of \cite{Tu19}).
\end{lemma}

Recalling the definitions from section 3.1, let us say extended closed (or extended open-closed) TCFT if we do not impose any positive input condition.
Not imposing the positive input condition we do recover an $S^1$-framed modular operad denoted $\mathscr{C}^d$ by collecting together the Hom-spaces of $\mathscr{C}^d$ with zero incoming and only outgoing framed boundaries, see example 2.3 of \cite{Tu19} (similarly for $\mathscr{OC}^d$ by collecting together the Hom-spaces of $\mathscr{OC}^d$ with zero incoming closed and zero outgoing open and only outgoing closed boundaries and incoming open boundaries); compare in contrast remark \ref{dasistGG}.
Applying the functors \eqref{S1tot} and \eqref{BDfSL} we have
$$Fr^\gamma(Tot_{S^1}(\mathscr{C}^d))=M^c,$$
the BD algebra from \eqref{cgbd}, see \cite{Tu19}. Further given an extended closed  TCFT we find an algebra over the $S^1$-framed modular operad $\mathscr{C}^d$. Let's assume that a closed TCFT induced from a cyclic $A_\infty$-category $\mathcal{C}$ by theorem \ref{KanExt} extends to an extended closed  TCFT, which we denote $\Phi$.
It then follows (basically by definition) that 
$$Tot_{S^1}(\Phi)=\Big(Sym(CH_*(\mathcal{C})[u^{-1}]),d_{Hoch}+uB,\Delta_c,\{\_ , \_\}_c\Big),$$
compare the notions from the BD algebra \ref{cabd}. Further it follows that
$$Fr^\gamma(Tot_{S^1}(\Phi))=\Big(Sym(Sym(CH_*(\mathcal{C})[u^{-1}]))\llbracket\gamma\rrbracket,\cdot,d_{Hoch}+uB+\gamma\Delta_c+\gamma\delta_e,\{\_ , \_\}_c\Big).$$
The multiplication induces a map of BD algebras
\begin{align*}
\Big(Sym(Sym(CH_*(\mathcal{C})[u^{-1}]))&\llbracket\gamma\rrbracket,\cdot,d_{Hoch}+uB+\gamma\Delta_c+\gamma\delta_e,\{\_ , \_\}_c\Big)\\
&\rightarrow \Big(Sym(CH_*(\mathcal{C})[u^{-1}])\llbracket\gamma\rrbracket,\cdot,d_{Hoch}+uB+\gamma\Delta_c,\{\_ , \_\}_c\Big)\cong\mathcal{F}^c(\mathcal{C}).
\end{align*}
Thus we have shown:
\begin{prop}\label{ssgb}
    Given a closed TCFT induced from a cyclic $A_\infty$-category $\mathcal{C}$ (by theorem \ref{KanExt}) which extends to zero inputs there is a map of BD algebras
    $$M^c\rightarrow \mathcal{F}^c(\mathcal{C}),$$
    induced from the functor \ref{S1tot}.
\end{prop}
Of course, the idea of proposition \ref{ssgb} is already well-known; see for instance the published version of \cite{Cos05}, section 3.4. As remarked before, the claim of Lurie \cite{Lu09} suggest that given a \emph{smooth} cyclic $A_\infty$-category one can extend the induced TCFT's from theorem \ref{KanExt} to zero inputs. As far as the author is aware, this has not yet been explicitly constructed. However, the idea is used to central effect in theorem \ref{cei} from the introduction, which we here spell out a bit more: 
\begin{thm_nn}[\cite{CaTu24}]
Given a smooth cyclic $A_\infty$-category $\mathcal{C}$ there is a map in the homotopy category of dg Lie algebra
$$M^c\rightarrow   \mathcal{F}^c(\mathcal{C}),$$
given by a zig-zag of dg Lie algebras
\begin{equation}\label{smooRes}
\begin{tikzcd}
 M^c\arrow{d}{\simeq}\arrow[dotted]{r}&   \mathcal{F}^c(\mathcal{C})\arrow{d}{\simeq}\\
 \widehat{M^c}\arrow{r}&\widehat{\mathcal{F}^c}(\mathcal{C}).
\end{tikzcd}
\end{equation}
\end{thm_nn}
The dg-Lie algebras $\widehat{M^c}$ and $\widehat{\mathcal{F}^c}(\mathcal{C})$ are resolutions of $M^c$ and  $\mathcal{F}^c(\mathcal{C})$ (the latter only if $\mathcal{C}$ is smooth), see section 6.1 of \cite{CaTu24} for details.

Amorim-Tu \cite{AmTu25} are currently writing up an open-closed version (theorem~\ref{ocRep}) of diagram \eqref{smooRes}.

\subsubsection{Barannikov's work}
Theorem 1 of \cite{Ba10}  (see also the very good exposition \cite{bra12}, section 8.2.1.) describes, given a modular operad $\mathcal{P}$, a bijection between: \begin{itemize}
    \item 
 algebras on a chain complex $V$ over an operad built from the Feynman transform of $\mathcal{P}$, 
 \item  Maurer-Cartan elements in a dg shifted Lie algebra built out of $V$ and $\mathcal{P}$.
 \end{itemize}
In section 10 of \cite{Ba10} the author constructs an explicit modular operad $\mathbb{S}[t]$.
In theorem 2 of loc.cit. he shows that the Feynman transform of $\mathbb{S}[t]$ is, after applying a totalization functor similar as in the previous subsection, isomorphic to the stable ribbon graph complex $s\mathcal{RG}$, at least when restricting to the part without leaves.
Furthermore, restricting to the cyclic operad part of  $\mathbb{S}[t]$ recovers in the similar way $\mathcal{RT}$, see eg. \cite{geKa96} 5.9 and section 9 of \cite{Ba10}.
The dg shifted Lie algebra he describes, associated to $\mathbb{S}[t]$ and a chain complex $V$, is basically $\mathcal{F}^{pq}(V)$ considered in this thesis and its cyclic part reduces to $\mathcal{F}^{fr}(V)$, see around equation 6.1 of loc.cit.
\begin{remark_nn}
One could expect that the outer part of diagram \eqref{tot_nc} can be obtained as a corollary of Barannikov's result after applying a totalization functor, analogous to the previous subsection.
It is interesting to note the existence of Maurer-Cartan elements in the induced dg-Lie algebras.
\end{remark_nn}
A similar construction may be available for the commutative side:
That is one may ponder the existence of a modular operad, analogous to $\mathbb{S}[t]$, which is a modular extension of $Com$ and whose Feynman transform gives $s\mathcal{G}$.
The LQT map may then be derived as induced from a Morita invariance map and a map from $\mathbb{S}[t]$ to this  modular generalization of $Com$, which would naturally explain our compatibility result, eg. theorem E.

\newpage
\section{Algebraic Preliminaries}
In this section we introduce the fundamental notions that we will work with: Shifted Poisson and Beilinson-Drinfeld algebras, (cyclic) $L_\infty$-algebras and their cohomology theories, (Calabi-Yau) $A_\infty$-categories, Hochschild and cyclic homology and splittings of the non-commutative Hodge filtration.

For the rest of this thesis we fix following conventions. We denote by $k$ a field of characteristic 0, unless remarked otherwise. Furthermore we work in the category of chain complexes over $k$, denoted $Ch$, and stick to the homological convention.
The suspension $[\_]$ of a chain complex $V$ is defined in such a way that if $V$ is concentrated in degree zero, then V[1] is concentrated in degree $-1$.
\subsection{Shifted Poisson Algebras and Beilinson-Drinfeld Algebras}
We follow section 1.4.2 of \cite{CPTTV17} for the following definition.
\begin{df}\label{shPoA}
An \emph{$r$-shifted Poisson differential graded algebra} is a differential graded, not necessarily unital, k-algebra such that 
\begin{itemize}
    \item 
 the differential has degree 1
 \item it has a Poisson bracket of degree $r$, that is it is a Lie bracket of degree $r$ that fulfills the Leibniz rule with respect to the product and is compatible with the differential.
 \end{itemize}

A \emph{map of shifted Poisson dg algebras} is both a map of the underlying dg-algebras and the shifted Lie algebras.    

We denote by $(Poiss_{r})$ the category whose objects are $r$-shifted Poisson algebras and whose morphisms are maps of $r$-shifted Poisson algebras.
\end{df}

 \begin{remark} 
 Note that by forgetting the multiplication a shifted Poisson dg algebra forgets to a dg shifted Lie algebra.
 When we say Maurer-Cartan element (MCE) of a shifted Poisson dg algebra we refer to a MCE of this dg shifted Lie algebra.
 We can \textit{twist} both this dg shifted Lie algebra as well as the overlying shifted Poisson algebra by such a MCE.
\end{remark}

We have in mind that $r$ denotes an odd integer\footnote{Since then we understand the quantized LQT map.} in the following, even though we don't need to require that at this point. We make following definitions:

\begin{df}\label{BD}
An \emph{$r$-twisted Beilinson-Drinfeld ($BD$) algebra} is a graded commutative unital algebra over the ring $k\llbracket \mu\rrbracket$ of formal power series in a variable $\mu$, which has cohomological degree $1-r$ such that 
\begin{itemize}
    \item it is endowed with a Poisson bracket $\{\_,\_\}$ of degree $r$, which is $k\llbracket \mu\rrbracket$-linear
    \item 
    it is endowed with an $k\llbracket\mu\rrbracket$-linear map $d$ of degree 1 that squares to zero 
    \item
    and such that
\begin{equation}\label{BDrel}
d(a\cdot b)=d(a)\cdot b+(-1)^{|a|}a\cdot d(b)+\mu \{a,b\}.\end{equation}
\end{itemize}
\end{df}

\begin{remark}
We will also encounter $r$-twisted $BD$ algebras whose multiplication has even degree $m$.
Here we require the formal variable to have degree $1-r+m$.
Note that we could always reduce to the previous definition by performing a shift.\footnote{We do not do so since then the quantized LQT map would have non zero degree.}
We do not want to introduce additional notation for this case. 
\end{remark}
     
\begin{remark}
In \cite{CPTTV17}, section 3.5.1, the notion of a $BD_d$ algebra is introduced.
It is not clear to the author what is the relationship of this notion and ours, in general.
However, for $r=1$ an $r$-twisted BD algebra is the same as a $BD_0$ algebra.
\end{remark}

\begin{remark}
Since in this thesis we will only deal with $r$-twisted $BD$-algebras we will sometimes refer to them as simply BD algebras as well. 
\end{remark}

\begin{remark}
Note that by forgetting the multiplication a BD algebra forgets to a dg shifted Lie algebra.
When we say Maurer-Cartan element of a BD algebra we refer to a MCE of this dg shifted Lie algebra.
Note that we can \textit{twist} not only this dg shifted Lie algebra, but also the overlying BD algebra by such a MCE.
\end{remark}
We will only see maps of BD algebras which have the same dimension.
However, we will encounter maps of BD algebras which do not directly respect the algebra structure, but do so up to a certain weighting and live over different base rings:
%\begin{df}[Weak Map of BD algebra]\label{wBD}
 %   Let A be a dimension s $BD$ algebra over a  $k$-algebra $R$ and formal variable $\gamma$ and B be a dimension $s$ $BD$ algebra over a  k-algebra $S$ and formal variable $\hbar$. A weak map of BD algebras is given by a degree 0 k-linear map $$f:A\rightarrow B$$ and a k-linear map $$g:R[\gamma]\rightarrow S[\hbar]$$ where $f$ is a map of the underlying shifted dg k-Lie algebras (but not necessarily a map of algebras over the bigger base rings!) and satisfies additionally : If $a,b\in A$ then \begin{equation}\label{wBD}g(\gamma)\cdot f\{a,b\}=f(\gamma\cdot \{a,b\}).\end{equation} 
%\end{df}
%\begin{remark}
 %   Equivalent condition, given that A and B are BD algebras, is that
 %   $$df(a\cdot b)=f(da\cdot b)+(-1)^{|a|}f(a\cdot db)+g(\gamma)\cdot\{f(a),f(b)\},$$
  %  that is, it intertwines the BD relation in an interesting way.
 %   Note that this equation is trivially satisfied if g is additionally a map of rings and $f: A\rightarrow g^*B$ is a map of $R[\gamma]$-algebras. We will encounter examples where these conditions are not met, but that are still weak maps of BD algebras.
%\end{remark}}
\begin{df}\label{weimap}
Let A be an $s$-twisted $BD$ algebra over $k\llbracket\gamma\rrbracket$ and B be an $s$-twisted $BD$ algebra over $k\llbracket\hbar\rrbracket$.
Given $t\in \mathbb{N}_{}$ we say that a \emph{$t$-weighted map of BD algebras} is given by a map of algebras
$$g:k\llbracket\gamma\rrbracket\rightarrow k\llbracket\hbar\rrbracket$$
such that 
\begin{equation}\label{pow}
    g(\gamma)=\hbar^t 
    \end{equation}
and a map of $k\llbracket\gamma\rrbracket$ dg shifted Lie algebras $$f:A\rightarrow g^*B$$ such that 
\begin{equation}\label{mult}
f(a_1\cdot \ldots \cdot a_k)=\hbar^{(t-1)(k-1)}f(a_1)\cdot \ldots \cdot f(a_k).\end{equation}
\end{df}
\begin{remark}
    For $t=1$ we recover the standard definition of a map of BD algebras.
    We indicate this by simply dropping the adjective weighted. 
\end{remark}
\begin{remark}
    Composition of a $t$-weighted map of BD algebras and an $s$-weighted map of BD algebras is an $s\cdot t$-weighted map of BD algebras. 
\end{remark}
\begin{df}
  Denote by $(\text{BD}^{r-tw})$ the category whose objects are $r$-twisted BD algebras and whose morphisms are maps of BD algebras.
\end{df}
\begin{remark}
Note that \eqref{pow} is a necessary condition allowing us to extend g to a map
$$\tilde{g}:k(\!(\gamma)\!)\rightarrow k(\!(\hbar)\!)$$
where $\tilde{g}(r\cdot \gamma^{-k})=g(r)g(\gamma)^{-k}.$
We further get an induced map
$$\tilde{f}:A(\!(\gamma)\!)\rightarrow \tilde{g}^*B(\!(\hbar)\!).$$
Condition \eqref{mult} implies that for any $x\in A$ we have 
  \begin{equation}\label{int_rel}
 \tilde{f}e^{x/\gamma}=\hbar^{1-t} e^{f(x)/\hbar}. 
 \end{equation} 
This guarantees commutativity of 
$$\begin{tikzcd}
MCE(A)\arrow{r}{f}\arrow{d}{exp(-/\gamma)}&MCE(B)\arrow{d}{exp(-/\hbar)}\\
Z(A(\!(\gamma)\!))\arrow{r}{\hbar^{t-1}\tilde{f}}&Z(B(\!(\hbar)\!))
\end{tikzcd},$$
where MCE refers to the Maurer-Cartan set of the respective dg shifted Lie algebra and Z refers to the cycles of the underlying chain complex.\footnote{As is well known, we can replace $Z$ by homology if we replace MCE by its quotient by the gauge group action. } 
For $t=1$ commutativity of the diagram is a standard fact about BD algebras, for general $t$ it follows from equation~\eqref{int_rel}.
\end{remark}
\begin{df}\label{dq}
Let $P$ be an r-shifted Poisson algebra and $R$ be an $r$-twisted BD algebra. 
We call 
$$p:R\rightarrow P$$
 a \emph{dequantization map} if it is a  degree zero map of the underlying chain complexes and of shifted Lie algebras.
\end{df}
\begin{remark}
    Note that we do not ask that this map respects the products, which is what we will see in our examples.
\end{remark}

\begin{lemma}\label{tensorBD}
    Given two $r$-twisted BD algebras $A$, $B$ the tensor product of their underlying graded vector spaces $A\otimes B$ can be equipped with the structure of an $r$-twisted BD algebra in a natural way, extending the initial ones.
\end{lemma}
To define the shifted Poisson bracket on the tensor product we use both the respective algebra multiplications, denoted $m_A$ and $m_B$, and the original shifted Poisson brackets. Precisely we have that 
\begin{align}\label{tenbra}
\{a_1\otimes b_1,a_2\otimes b_2\}=(-1)^{|a_2|(|b_1|+r)}m_A(a_1,a_2)&\otimes\{b_1,b_2\}_B\\
&+(-1)^{|b_1|(|a_2|+r)}\{a_1,a_2\}_A\otimes m_B(b_1,b_2).\nonumber
\end{align}

\begin{proof}
This is the same argument described in definition 2.7 and proposition-definition 2.2 of \cite{legu14}, following \cite{man99}, section 5.8.1. Note however that we follow the convention of remark 2.3 (2) of \cite{legu14}.  
\end{proof}

 We make some slightly ad-hoc definitions, meant as a first approach to a homotopy theory of BD algebras.
 \begin{df}\label{dgBD}
 An \emph{r-twisted dg-BD algebra} is a r-twisted BD algebra $A$, additionally equipped with a degree 1 differential $D:A\rightarrow A$, which is a derivation with respect to the product and commutes with respect of the BD differential.
\end{df}
Given such an r-twisted dg BD algebra $(A,D)$ the homology $H_*(A,D)$ is naturally a BD algebra.
 \begin{df}\label{qiBD}
We say that two r-twisted dg-BD algebra $(A,D_1)$ and $(B,D_2)$ are \emph{quasi-isomorphic} if there is an isomorphism of BD algebra $H_*(A,D_1)\rightarrow H_*(B,D_2).$
     \end{df}
   \begin{remark}\label{Induzierung}
       There is a functor from r-twisted dg BD algebra to r-twisted BD algebra which is the identity on the underlying graded shifted Poisson algebra; the BD differential of the BD algebra is given by the sum of the initial BD differential of the dg-BD algebra plus the differential of the dg-BD algebra. 
   \end{remark}

\subsection{(Cyclic) $L_\infty$-Algebras and Lie Algebra Homology}
We recall the following notions, basically citing the succinct presentation of \cite{wi11}:

Let $(V,d)$ be a chain complex. We denote the symmetric algebra by $Sym V=\bigoplus_{n\geq 0} V^{\otimes n} / I$, where $I$ is the two-sided ideal generated by relations $x\otimes y -(-1)^{|x||y|}y\otimes x$. The product will be denoted by $\odot$. For example, the expressions $x_1\odot \cdots \odot x_n :=[x_1\otimes \cdots \otimes x_n]$ generate $Sym^nV$ as a vector space. Let $Sym^+ V:= \bigoplus_{n\geq 1} Sym^n V$, with grading given by $|x_1\odot \cdots \odot x_n|=\sum_j|x_j|$. This space carries the structure of a graded cocommutative dg coalgebra without counit, with comultiplication given by
\begin{equation}\label{coal_def}
\Delta( x_1 \odot \cdots \odot x_n) = \sum_{\substack{I\sqcup J=[n] \\ |I|,|J|\geq 1}} \epsilon(I,J) \bigodot_{i\in I} x_i \otimes \bigodot_{j\in J} x_j.
\end{equation}
Here $\epsilon(I,J)$ is the sign of the ``shuffle'' permutation bringing the elements of $I$ and $J$ corresponding to \emph{odd} $x$'s into increasing order. Note that $\epsilon(I,J)$ implicitly depends on the degrees of the $x$'s. It is straightforward to check that the differential $d$, extended from the chain complex to the symmetric algebra via the Leibniz rule, is compatible with the comultiplication $\Delta$. 

\begin{df}
Let $(\mathcal{C},\Delta,d)$, $(\mathcal{C}',\Delta',d')$ be dg coalgebras. A linear map $\mathcal{F}\in Hom_k(\mathcal{C},\mathcal{C}')$ is called \emph{degree $k$ morphism of coalgebras} if $\Delta' \circ \mathcal{F} = (\mathcal{F}\otimes \mathcal{F})\circ \Delta$ and it commutes with the underlying differentials. A linear map $Q\in Hom_k(\mathcal{C},\mathcal{C})$ is called a \emph{degree $k$ coderivation} on $\mathcal{C}$ if  $\Delta \circ Q = (Q\otimes 1 + 1\otimes Q)\circ \Delta$ and it commutes with the underlying differential. Here we use the Koszul sign rule, e.g., $(1\otimes Q)(x\otimes y) = (-1)^{k|x|}x\otimes Qy$ etc.
%\[
%\Delta_1 Q x = \sum_{(x)} Qx'\otimes x'' + (-1)^{k|x'|}x'\otimes Qx''
%\]
%using Sweedler notation.
\end{df}

\begin{remark}\label{nütz1}
Any coderivation $Q$ on $Sym^+V$ (coalgebra morphism $\mathcal{F}:Sym^+V\to Sym^+W$) is uniquely determined by its composition with the projection $Sym^+V\to Sym^1V=V$, respectively (${Sym^+W\to Sym^1W=W}$). The restriction to $Sym^nV$ of this composition will be denoted by ${Q_n\in Hom(Sym^kV,V)}$ ($\mathcal{F}_n\in Hom(Sym^kV,W)$) and called the $n$-th ``Taylor coefficient'' of $Q$ ($\mathcal{F}$).
\end{remark}
\begin{df}\label{coalgebra}
An $L_\infty$-algebra structure on a chain complex $(L,d)$ is a degree 1 coderivation $l$ on $Sym^+(L[1])$ such that $l^2=0$ and that $[l,d]=0$. A morphism of $L_\infty$ algebras $${F}:(L,l)\rightsquigarrow (L',l')$$ is a degree 0 coalgebra morphism $C({F}):S^+(L[1])\to Sym^+((L')[1])$ such that $C({F})\circ l=l'\circ C({F})$ and that commutes with the underlying differentials.
\end{df}

In components, the $L_\infty$-relations read
\[
\sum_{\substack{I\sqcup J=[n] \\ |I|,|J|\geq 1} } \epsilon(I,J) l_{|J|+1}(l_{|I|}(\bigodot_{i\in I} x_i) \odot \bigodot_{j\in J} x_j) = 0.
\]
\begin{remark}
    
In fact we assume that $l_1=0$ since we can just absorb it into the differential $d$ of the underlying chain complex.

 \end{remark}  
\begin{ex}\label{dgla}
Let $(\mathfrak{g},d,\{\_,\_\})$ be a differential graded Lie algebra. Then the assignments $l_1(x)=d x$, $l_2(x,y)=(-1)^{|x|}\{x,y\}$, $Q_n=0$ for $n=3,4,..$ define an $L_\infty$-algebra structure on $\mathfrak{g}$. 
Here and everywhere in this thesis $|x|$ is the degree wrt. the grading on the coalgebra.
\end{ex}
\begin{ex}\label{wmg}
An $L_\infty$-morphism $\mathcal{F}$ between dglas $\mathfrak{g}$, $\mathfrak{g}'$ has to satisfy the relations
\begin{multline}
\label{equ:dglaLinfty}
l_1'\mathcal{F}_n(x_1\odot \cdots \odot x_n) + \frac{1}{2} \sum_{\substack{I\sqcup J=[n] \\ |I|,|J|\geq 1} }\epsilon(I,J)
l_2'(\mathcal{F}_{|I|}(\bigodot_{i\in I} x_i) \odot \mathcal{F}_{|J|}( \bigodot_{j\in J} x_j) )
= \\ =
\sum_{i=1}^n \epsilon(i,1,\dots,\hat{i},\dots, n) \mathcal{F}_n( l_1(x_i)\odot x_1\odot \cdots\odot\hat{x}_i \odot \cdots \odot x_n)
+ \\ +
\frac{1}{2}
\sum_{i\neq j}^n \epsilon (i,j,\dots,\hat{i},\dots,\hat{j},\dots, n) \mathcal{F}_{n-1}( l_2(x_i\odot x_j)\odot x_1\odot\cdots\odot \hat{x}_i \odot \cdots \odot\hat{x}_j \odot \cdots x_n).
\end{multline}
Here the factor $\epsilon(i,j,1,..,\hat{i},..,\hat{j},.., n)$ is the sign of the permutation on the \emph{odd} $x$'s that brings $x_i$ and $x_j$ to the left and analogously for $\epsilon(i,1,\dots,\hat{i},\dots, n)$.

\end{ex}

\begin{df}\label{redcecom}
    Given an \emph{$L_\infty$-algebra} $L$
    we call the resulting chain complex  $$C_*^+(L):=\big(Sym_+(L[1]),d+l\big)$$ the (reduced) Chevalley-Eilenberg chains of $L$.
    We further adopt the notation ${C_*(L):=\big(Sym(L[1]),d+l\big)},$ the (non-reduced) Chevalley-Eilenberg chains.
    \end{df}
 We define $$C^*(L):=Hom(C_*(L),k),$$ the Chevalley-Eilenberg cochains of the $L_\infty$-algebra $L$ and denote $C^*_+(L):=Hom(C_*^+(L),k).$
Recall from remark \ref{nütz1} that an $L_\infty$-morphism 
can be entangled in saying that there are maps 
 $$\phi_n: L_1^{\otimes n}\rightarrow L_2$$
 of certain degree and satisfying relations. 
 For instance $\phi_1$ is just a chain map. 
\begin{df}
    An \emph{$L_\infty$-quasi-isomorphism} is an $L_\infty$-morphism such that $\phi_1$ is a quasi-isomorphism. 

 We say that an $L_\infty$-morphism is strict if $\phi_n=0$ for all $n\geq 2$.     
 \end{df}
    
\begin{df}
    A \emph{cyclic $L_\infty$-algebra} of dimension $d$ is an $L_\infty$-algebra $L$ together with a symmetric non-degenerate pairing 
    $$\langle\_,\_\rangle:L\otimes L\rightarrow k[-d]$$
    such that $S:=\langle d+l\_,\_\rangle\in C^*(L),$ i.e. such that it is symmetrically invariant. 
    Note that in this case $S$ is an element of degree $3-d$. We denote $I:=\langle l\_,\_\rangle.$
\end{df}
We limit ourselves to strict cyclic $L_\infty$-morphisms:
\begin{df}
    A \emph{cyclic strict $L_\infty$-morphism} is a strict $L_\infty$-morphism such that 
    $$\langle F_1(x),F_1(y)\rangle=\langle x,y\rangle.$$ 
\end{df}
\begin{remark}
 Note that a cyclic $L_\infty$-morphism is automatically injective.  
\end{remark}
\begin{remark}
 One should be able to give a more abstract definition of a cyclic $L_\infty$-algebra and $L_\infty$-morphisms, compare the $A_\infty$-case later.
 See for instance section 5.2 of \cite{GG15}, in the context of $L_\infty$-spaces.
\end{remark}

\subsection{$A_\infty$-Categories and Cyclic Homology}
An $A_\infty$-category is a homotopically better behaved generalization of the notion of a dg-category, further relevant since Fukaya categories of symplectic manifolds are \emph{not} dg categories, but do carry an $A_\infty$-category structure:

Let $B$ be a set and $V_B:=\{V_{ij}\}_{i,j\in B}$ be a set of chain complexes.
Then its bar-construction\footnote{In the sense that it is the Bar complex of a dg-category with zero compositions.} is defined by 
\begin{equation}\label{bar}BarV_B:=\bigoplus_{n=1}^\infty \bigoplus_{\lambda_0,\cdots,\lambda_n\in B}V_{\lambda_{0},\lambda_{1}}[1]\otimes V_{\lambda_{1},\lambda_2}[1]\otimes \cdots \otimes V_{\lambda_{n-1},\lambda_n} [1]
\end{equation}
and it is naturally a dg coalgebra $(BarV_B,d)$ whose comultiplication, see e.g. \cite{KoSo24}, example 2.1.6, we don't give an extra symbol.
\begin{df}
A \emph{small $A_\infty$-category} $\mathcal{C}$ is given by a set of objects $Ob(\mathcal{C})$, a set of chain complexes  $\{V_{\lambda_i,\lambda_j}\}_{\lambda_i,\lambda_j\in Ob(\mathcal{C})}$ and a collection of maps
\begin{equation*}
m_{n}\colon V_{\lambda_0,\lambda_{1}}[1]\otimes V_{\lambda_{1},\lambda_{2}}[1]\otimes\ldots\otimes V_{\lambda_{n-1},\lambda_n}[1] \to V_{\lambda_0,\lambda_{n}}[1]
\end{equation*} 
for all $n\geq 1$ and all tuples of objects $(\lambda_0,\lambda_1,\cdots,\lambda_n).$
We can uniquely extend those to a coderivation 
$$m_\mathcal{C}:BarV_{Ob(\mathcal{C})}\rightarrow BarV_{Ob(\mathcal{C})}$$
and we require that $$m_\mathcal{C}^2=0$$ and that it commutes with the internal differential $d$ of the chain complex $(BarV_{Ob(\mathcal{C})},d)$. 
\end{df}
In this case we denote $V_{\lambda_i,\lambda_j}=:Hom_\mathcal{C}(\lambda_i,\lambda_j)$ and 
$(BarV_{Ob(\mathcal{C})},d+m_\mathcal{C})=:(Bar\mathcal{C},d+m_\mathcal{C}).$ 
\begin{remark}
We will only consider small $A_\infty$-categories in this paper and thus drop this adjective from now on.    
\end{remark}
\begin{df}
We say that an $A_\infty$-category $\mathcal{C}$ is \emph{unital} if there is an element 
$1_{i}\in Hom_\mathcal{C}(\lambda_i,\lambda_i)$ for all $\lambda_i\in Ob(\mathcal{C})$ such that 
$m_2(g,1_i)=g$ for all $g\in Hom_\mathcal{C}(\lambda_k,\lambda_i)$, 
$\lambda_k\in Ob(\mathcal{C})$ and $m_2(1_i,f)=f$ for all $f\in Hom_\mathcal{C}(\lambda_i,\lambda_l)$, $\lambda_l\in Ob(\mathcal{C})$
and if in $m_n(a_1,a_2,\cdots)$ any entry is one of the $1_{i}$ then $m_n(a_1,a_2,\cdots)=0$ for $n\geq 3$.
\end{df}
\begin{remark} 
We can somewhat unpack the definition of an $A_\infty$-category:
\begin{itemize}
\item   
The fact that $m_\mathcal{C}^2=0$ is  equivalent to the infinite list of relations, one for each $n \geq 1$, 
\begin{equation}\label{A_inf_re}
    \sum_{\substack{r,t\geq 0, s\geq 1\\r+s+t=n}} m_{r+1+t}(id^{\otimes r} \otimes m_{s} \otimes id^{\otimes t}) = 0.
\end{equation}
 Here $id$ denotes ambiguously the identity on whichever Hom-space it is applied to.
 \item Using this it is easy to see that in the case that $m_\mathcal{C}$ is only non-zero on $Bar^2V_B$,
 i.e. on the summand $n=2$, we recover the definition of a (not necessarily unital) differential graded category; where $f\circ g=m_2(g,f)$ for composable $f,g.$
     
\item In fact we assume that $m_1=0$ for the rest of this article since we can just absorb it into the differential of the underlying chain complexes. 
\end{itemize}
\end{remark}

\begin{df}
A \emph{functor between two $A_\infty$-categories} $F:\mathcal{C}\rightarrow \mathcal{D}$ is given by a map of sets
$$F: Ob(\mathcal{C})\rightarrow Ob(\mathcal{D})$$
and maps
    $$F_n: Hom_\mathcal{C}({\lambda_0,\lambda_1})[1]\otimes Hom_\mathcal{C}({\lambda_1,\lambda_2})[1]\otimes \cdots \otimes Hom_\mathcal{C}({\lambda_{n-1},\lambda_n})[1]\rightarrow Hom_\mathcal{D}(F(\lambda_0),F(\lambda_n))[1],$$
for all $n\geq 1$ and all tuples of objects $(\lambda_0,\lambda_1,\cdots,\lambda_n).$ 
We can uniquely extend this to a map of coalgebras, which we require to determine a chain map $$F_*:\  (Bar(\mathcal{C}),d+m_\mathcal{C})\rightarrow (Bar\mathcal{D}),d+m_\mathcal{D}).$$ 
\end{df}
We denote by $A_\infty cat$ the category whose objects are $A_\infty$-categories and with morphisms given by $A_\infty$-functors.

\begin{remark}
    Again we can unpack the requirement that $F_*$ intertwines the differentials. 
    Equivalently it means that
    \begin{equation}\label{F_inf_rel}
        \sum_{r+s+t=n} F_{r+1+t}(id^{\otimes r} \otimes m_{\mathcal{C},s} \otimes id^{\otimes t}) = \sum_{i_1+\ldots+i_l=n}  m_{\mathcal{D},l}(F_{i_1} \otimes F_{i_2} \otimes \ldots \otimes F_{i_l})
    \end{equation}
for $n\geq 1$ and where $m_1=d$. 
Notably $F_1$ induces maps of chain complexes on the Hom spaces. 
\end{remark}
\begin{df}
 We say that an $A_\infty$-functor $F:\mathcal{C}\rightarrow \mathcal{D}$ of unital $A_\infty$-categories is \emph{unital} if $F_1(1_{i})=1_{F(i)}$ and $F_n(\cdots,1_{i},\cdots)=0$ for all $n\geq 2$. \end{df}
\begin{df}
   We call an $A_\infty$-functor 
   $$F:\mathcal{C}\rightarrow \mathcal{D}$$
   \emph{strict} if $F_n=0$ for all $n\geq 2$. 
\end{df}

 Given an unital $A_\infty$-functor $F: \mathcal{C}\rightarrow \mathcal{D}$ taking homology induces a functor of ordinary categories $H_*F: H_*\mathcal{C}\rightarrow H_*\mathcal{D}.$  

\begin{df}\label{equiv}
We call an unital $A_\infty$-functor $F:\mathcal{C}\rightarrow \mathcal{D}$ an \emph{equivalence} if the induced functor 
$$H_*F_1: H_*\mathcal{C}\rightarrow H_*\mathcal{D}$$
of ordinary categories is an equivalence.
\end{df}

Let us recall some further notions some of which we, strictly speaking, won't need; thus we are relatively fast in doing so:
\par
Given an $A_\infty$-category $\mathcal{C}$ there is dg-category of $\mathcal{C}$-bimodules as described in \cite{Ga19}, section 3.1 and references therein.
Further in loc. cit. it is explained that there are certain objects called the diagonal bimodule $\mathcal{C}_\Delta$, respectively the linear dual diagonal bimodule $\mathcal{C}_\Delta^\vee$ in this category.
Then one can define 
\begin{df}
$$CH_*(\mathcal{C}):=\mathcal{C}_\Delta\otimes^{\mathbb{L}}_{\mathcal{C}^e}\mathcal{C}_\Delta\in Ch,$$
 given by the derived tensor product in the category of $\mathcal{C}$-bimodules (see e.g. \cite{Ga19}, section 3.4).
 Its homology is called \emph{Hochschild homology}.
\end{df}
\begin{remark}
There are explicit complexes that compute this derived tensor product, which is one way to see that there are maps of chain complexes, for all objects $x$ of $\mathcal{C}$
$$\iota_x: Hom_\mathcal{C}(x,x)\rightarrow CH_*(\mathcal{C}).$$
\end{remark}

Let us come back to the notion of Hochschild homology.
One of its important properties is that there is an $S^1$-action on Hochschild chains (\cite{co85}, in our context eg. \cite{Ga23}, section 3.2).
By taking its homotopy $S^1$ orbits we arrive at
\begin{df}
   Given an $A_\infty$-category we call the homology of $$CC_*(\mathcal{C}):=CH_*(\mathcal{C})_{hS^1}\in Ch$$ its \emph{cyclic homology}. A concrete model is given as follows: Denote by $u$ a formal variable of degree $-2$. There is a chain map of degree 1, called Connes operator, $$B:CH_*(\mathcal{C})\rightarrow CH_*(\mathcal{C}),$$ on the graded vector space $CH_*(\mathcal{C})$ on which the differential $d_{Hoch}$ computes Hochschild homology; it encodes the $S^1$-action. $B$ commutes with the Hochschild differential. The following chain complex computes the homotopy $S^1$ orbits, that is cyclic homology: 
 $$CC_*(\mathcal{C})\simeq\big(CH_*(\mathcal{C})[ u^{-1}], d_{Hoch}+uB\big).$$
\end{df}
\begin{remark}\label{cph}
Thus just by abstract reasons we  get a map $\pi:CH_*(\mathcal{C})\rightarrow CC_*(\mathcal{C}).$ \end{remark}
Finally we want to consider a smaller model computing the cyclic (co)homology of an $A_\infty$-category:
\begin{cons}
    Let $\mathcal{C}$ be an $A_\infty$-category and consider the map given by cyclic permutation:
    $$\begin{tikzcd}
     Hom(\lambda_0,\lambda_1)[1]\otimes Hom(\lambda_1,\lambda_2)[1]\otimes \cdots \otimes Hom(\lambda_{n},\lambda_0)[1]\arrow{d}{t}&\ni&a_{01}\otimes a_{12}\otimes a_{23}\otimes \cdots\otimes a_{no}\arrow[mapsto]{d} \\  Hom(\lambda_1,\lambda_2)[1]\otimes Hom(\lambda_2,\lambda_3)[1]\otimes \cdots \otimes Hom(\lambda_{0},\lambda_1)[1]&\ni&(-1)^* a_{12}\otimes a_{23}\otimes\cdots\otimes a_{no}\otimes a_{01}
    \end{tikzcd}$$
    Here $(-1)^*$ comes from the Koszul sign rule.
    We extend $t$ to  $Bar\mathcal{C}$ by zero on the subspaces that are not of the above form, i,e, where $\lambda_n\neq \lambda_0$ and denote it by the same letter.
    \end{cons}
    \begin{lemma}
Given an $A_\infty$-category the differential $m_\mathcal{C}+d$ on $Bar\mathcal{C}$ restricts to ${ker(1-t)}$.
    \end{lemma}
\begin{proof}
        This is a standard result, which follows in the same way as in lemma 2.1.1 of \cite{Lo13}.
\end{proof}
    Finally we have to include a degree shift by 1: 
\begin{theorem}[`Cyclic Words']
The cyclic chain complex computes cyclic homology $$(Cyc_*^+(\mathcal{C}),m_\mathcal{C}+d):=\big(ker(1-t)[-1],m_\mathcal{C}+d\big)\simeq CC_*(\mathcal{C})$$ 
\end{theorem}
\begin{proof}
    This is well known.
    For the algebra case see \cite{Lo13}, theorem 2.1.5, in our generality eg. around definition 2.5 and remark 2.6 of \cite{AmTu22}.
\end{proof}
% If confused again for grading convention, check Loday, page 177. We follow that convention.
\begin{df}\label{llekwr}
  The cochain complex $$(Cyc^*_+(\mathcal{C}),m_\mathcal{C}+d):=Hom(Cyc_*^+(\mathcal{C}),k)$$ with the induced differential is called \emph{cyclic cochain complex} and its homology is referred to as the \emph{cyclic cohomology} of $\mathcal{C}$. 
\end{df}

\begin{remark}
    We think of a cyclic cochain as a cyclic word in the dual of the Hom-spaces with matching `boundary condition'.
    We adopt the notation 
    $$(a_1a_2\cdots a_n)\in Cyc^*_+(\mathcal{C})[-1],$$
where by definition $a_i\in Hom(\lambda,\mu)^\vee[-1]$ and $a_{i+1}\in Hom(\mu,\nu)^\vee[-1]$ for $\lambda,\mu,\nu\in Ob(\mathcal{C})$.
\end{remark}

Lastly we introduce important properties of $A_\infty$-categories:

\begin{df}\label{smanpr}
    We say that an $A_\infty$-category $\mathcal{C}$ is \emph{smooth} if $\mathcal{C}_\Delta$ is a compact, ie. a perfect object in the dg category of $\mathcal{C}$-bimodules.

    We say that an $A_\infty$-category $\mathcal{C}$ is \emph{proper} if $Hom_\mathcal{C}(x,y)\in Ch$ is perfect for all $x,y\in Ob(\mathcal{C}).$ 
\end{df}
\begin{remark}
Compare also \cite{BrDy19} ex. 2.8 and ex. 2.9 for a more abstract phrasing, in the setting of dg categories, which shows that the notions smooth and proper are complementary in a certain sense.
\end{remark}

\subsection{Calabi-Yau $A_\infty$-Categories}
For the purposes of this thesis we consider $A_\infty$-categories with an extra structure, which most generally may be phrased as requiring a (smooth and) proper Calabi-Yau structure. We  work with a convenient strictification of such a structure, called cyclic and leave it to further investigations how to obtain more homotopically coherent results. 
Compare \cite{choSa12} \cite{BrDy21}   for relevant results in this direction and \cite{AmTu22} for an answer to such investigations in the purely `closed' setting. 

In any case, a cyclic $A_\infty$-category is canonically a strong proper Calabi-Yau category and thus a weak one, see around corollary 2.26 of \cite{AmTu22}. A weak proper Calabi-Yau category together with a splitting of the non-commutative Hodge filtration (a central ingredient in this paper, which we introduce later in definition \ref{splitting}) naturally induce a strong proper Calabi-Yau structure, see lemma 3.16 of \cite{AmTu22}. Lastly in characteristic zero one can strictify a strong proper Calabi-Yau category to a cyclic one, see \cite{KoSo24} theorem 10.7. 
Thus we may pretend that all these notions are roughly equivalent for our purposes.\footnote{However, the notion of cyclic $A_\infty$-category is not preserved under general $A_\infty$-functors (whereas the one of proper Calabi-Yau categories is).} 
\begin{df}
Given $\mathcal{C}$ a proper $A_\infty$-category we 
call a choice of quasi-isomorphism of $\mathcal{C}$-bimodules
$$\mathcal{C}_\Delta \simeq\mathcal{C}_\Delta^\vee[-d]$$
a \emph{weak proper Calabi-Yau} structure of dimension $d$.
\end{df}
\begin{remark}[See rmk. 49 of \cite{Ga23}]
A weak proper Calabi-Yau structure of dimension $d$, also called weak right Calabi-Yau structure, may be equivalently described by saying that there is a map of complexes 
$$tr: CH_*(\mathcal{C})\rightarrow k[-d]$$
such that the induced map 
$$\begin{tikzcd}
    Hom_\mathcal{C}(x,y)\otimes Hom_\mathcal{C}(y,x)\arrow{r}{m^2_{x,y}}& Hom_\mathcal{C}(x,x)\arrow{r}{\iota_x}& CH_*(\mathcal{C})\arrow{r}{tr}& k[-d]\end{tikzcd}$$
is non-degenerate on homology, for all objects $x,y\in Ob\mathcal{C}$.
If we additionally ask that the trace map comes from a map from cyclic homology along the map $\pi$ from remark \ref{cph}, this is called a strong proper Calabi-Yau structure.
\end{remark}
\begin{remark}
There is (are) also the notion(s) of a (weak and strong) smooth Calabi-Yau, also called (weak and strong) left Calabi-Yau structure on a smooth $A_\infty$-category.
See definitions 12 and 13 of \cite{Ga23}. In fact if a category is both smooth and proper, the notions of smooth and proper Calabi-Yau structures coincide; see proposition 6.10 of \cite{GaPeSh15}. 
\end{remark}

\begin{df}\label{cycAinfcat}
    A \emph{cyclic $A_\infty$-category} of degree $d$, denoted $(\mathcal{C},\langle\_,\_\rangle_\mathcal{C})$, is an
    $A_\infty$-category $\mathcal{C}$ together with a  symmetric non-degenerate chain map
    $$\langle\_,\_\rangle_{\mu\lambda}:\  Hom_\mathcal{C}(\mu,\lambda)\otimes Hom_\mathcal{C}(\lambda,\mu)\rightarrow k[-d]$$
    for all $\mu,\lambda\in Ob(\mathcal{C})$ such that 
    \begin{equation}\label{cycl}S_\mathcal{C}:=\langle (m+d)\_,\_\rangle\in Cyc^*_+(\mathcal{C})[-1],
    \end{equation}
    i.e. is cyclically symmetric.
    \end{df}
    Note that in this case $S_\mathcal{C}$ is of degree $3-d$.
    We further denote by 
    $$I_\mathcal{C}:=\langle m\_,\_\rangle\in Cyc^*_+(\mathcal{C})[-1],$$
    which has the same degree.

\begin{remark}\label{cy_cyc}
    See remark 9 of \cite{Ga23} for the fact that any cyclic $A_\infty$-category is canonically a strong (and thus weak) proper Calabi-Yau.
    Further it is explained that in characteristic zero one can strictify a strong proper Calabi-Yau category to a cyclic one.
    However, the notion of cyclic $A_\infty$-category is not preserved under general $A_\infty$-functors, whereas the one of proper Calabi-Yau categories is.
\end{remark}
Without giving details we state the result from definition 2.15 of \cite{AmTu22} that there is a functor
$$\Omega^{2,cl}: (A_\infty Cat)^{op} \rightarrow Ch,$$
which associates to an $A_\infty$-category a chain complex called its non-commutative 2-forms.
As described in definition 2.23 of loc. cit. a cyclic structure $\langle\_,\_\rangle_\mathcal{C}$ on an 
$A_\infty$-category $\mathcal{C}$ induces an element $\rho_\mathcal{C}\in\Omega^{2,cl}(\mathcal{C}).$

\begin{df}\label{cycAi}
   Given two cyclic $A_\infty$-categories $(\mathcal{C},\langle\_,\_\rangle_\mathcal{C})$ and $(\mathcal{D},\langle\_,\_\rangle_\mathcal{D})$
   a  \emph{cyclic $A_\infty$-functor} is an $A_\infty$-functor $F:\mathcal{C}\rightarrow \mathcal{D}$ such that

   $$F^*\rho_{\mathcal{D}}=\rho_{\mathcal{C}},$$
   where $\rho_{\mathcal{D}}$ respectively $\rho_{\mathcal{C}}$ refers to the non-commutative 2-form associated to the cyclic structure by the previous remark.
\end{df}
\begin{remark}\label{caf}
    This definition should reduce to definition 5.3 (2) of \cite{HaLa08} in the case of a category with one object.
    Note that the authors use the term symplectic where we used the term cyclic.
    See equation 13 of \cite{AmTu22} for an explicit  equivalent description of definition \ref{cycAi}.
    In particular one has 
    \begin{equation}\label{nfn}
        \langle F_1(x),F_1(y)\rangle_\mathcal{D}=\langle x,y\rangle_\mathcal{C}.
        \end{equation}
\end{remark}
 Denote by  $(cyc_dA_\infty cat)_{st}$ the (ordinary) category whose objects are cyclic $A_\infty$-categories of degree $d$ and whose morphisms are cyclic strict $A_\infty$-functors.

\begin{df}\label{ColV}
    We say 
    $$V_B:=(V_{ij},\langle\_,\_\rangle_{ij})_{i,j\in B}$$
    is a \emph{collection of cyclic chain complexes} if it is a cyclic $A_\infty$-category $\mathcal{C}$ with $m_\mathcal{C}=0$ and $B=Ob(\mathcal{C})$. 
\end{df}

\subsection{Splitting of non-commutative Hodge filtration}
\begin{df}\label{splitting}
    Let $\mathcal{C}$ be an $A_\infty$-category. A \emph{splitting of the non-commutative Hodge filtration} is a chain map
$$s:(CH_*(\mathcal{C}),d_{Hoch})\rightarrow (CH_*(\mathcal{C})\llbracket u\rrbracket,d_{Hoch}+uB)$$
splitting the canonical projection 
$$\pi:(CH_*(\mathcal{C})\llbracket u\rrbracket),d_{Hoch}+uB)\rightarrow(CH_*(\mathcal{C}),d_{Hoch}),$$
that is we require $s\circ\pi=id$.
\end{df}
\begin{remark}
Let $\mathcal{C}$ be a $\mathbb{Z}$-graded smooth and proper $A_\infty$-category. By \cite{kal17} the non-commutative Hodge to de Rham spectral sequence degenerates and converges to negative cyclic homology, as explained eg. in the proof of theorem 5.45 of \cite{She19}. Because we are working over a field of characteristic zero this implies that there is a (non-canonical) identification of negative cyclic homology
\begin{equation}\label{noncan}
HC_*^-(\mathcal{C})\cong HH_*(\mathcal{C})\llbracket u\rrbracket.\end{equation}
\end{remark}
A splitting of the non-commutative Hodge filtration provides in particular the choice of such an isomorphism. Indeed, extending a splitting $s$ $u$-linearly gives a map 
$$(CH_*(\mathcal{C})\llbracket u\rrbracket,d_{Hoch}) \rightarrow (CH_*(\mathcal{C})\llbracket u\rrbracket,d_{Hoch}+uB),$$
which we can think of as a formal power series of endomorphisms of Hochschild chains whose first term is invertible (in fact it is the identity, by the splitting condition), thus the whole formal power series is invertible. Thus this map induces an isomorphism \eqref{noncan}.
\newpage
\section{Closed String Field Theory}\label{wsuaderinn}
In this section we introduce the classical and quantum algebras of observables of closed SFT. We first recall the quantum algebra of observables of \emph{free} closed SFT \cite{Cos05, CaTu24}, which we used already in the introduction. Then we define the classical algebras of observables of interacting closed SFT, which is novel in this thesis. We recall BCOV theory \cite{coli12} and present some comparisons to the previous algebraic theory of SFT, for the classical interacting part in the form of a conjecture.

For completeness we also mention the much later chapter \ref{mimsss} on backreacted closed string field theory. 
\subsection{Free Quantum Closed String Field Theory from CY Categories}\label{fcsfts}
\begin{df}\label{cBD}
Let $\mathcal{C}$ be a cyclic $A_\infty$-category of degree $d$. Then we call \emph{free quantum observables of the closed string field theory associated to $\mathcal{C}$} the $(2d-5)$-twisted Beilinson-Drinfeld over $k\llbracket\gamma\rrbracket$ (recalling definition \ref{BD})
$$\mathcal{F}^{c}(\mathcal{C}):=\Big(Sym\Big(CH_*(\mathcal{C})[u^{-1}][d-2]\Big)^{-*}\llbracket\gamma\rrbracket,d_{Hoch}+uB+\gamma\Delta_c,\{\_,\_\}_c\Big).$$ 
In detail, the map $d_{Hoch}+uB+\gamma\Delta_c$ is a differential and $\{\_,\_\}_c$ is a $(2d-5)$ shifted Poisson bracket, which together with the natural multiplication satisfy the BD relation \ref{BDrel}. More precisely the algebraic operations are defined exactly as in \cite{AmTu22}, that is  $\mathcal{F}^{c}(\mathcal{C}):=\mathfrak{h}_\mathcal{C}$, (see \cite{AmTu22}, page 35 for the definition of the latter) just that we carry the shifts with us and we further ignore the variable $\lambda$ from there. Further the notation $-*$ means that we reverse the grading of what is inside the brackets, which is why eg. $|d_{Hoch}|=1,$ here, despite the homological convention. See also section 4.1 of \cite{CaTu24} for a more detailed description of the algebraic operations in the case of a category with one object. 
\end{df}
\begin{remark}\label{dgsmnggw2}
This algebraic structure was originally described in \cite{Cos05}, but with slightly different grading conventions. Here we chose to follow \cite{CaTu24}, but had to perform yet different shifts to be compatible with the open sector (that is notably for section \ref{octosa}). We should remark that theorem \ref{cei} is formulated using different shifts. There one uses
$$\widetilde{\mathcal{F}^{c}(\mathcal{C})}:=Sym\Big(CH_*(\mathcal{C})[u^{-1}][d]\Big)\llbracket\gamma\rrbracket[1-2d],$$
see \cite{CaTu24} for the algebraic structures and especially appendix A for the $\mathbb{Z}$-grading.  One may try to formulate theorem \ref{cei} also with the conventions for $\mathcal{F}^{c}(\mathcal{C})$, but we haven't done this,~yet. Because of this we also need to formulate definition \ref{cisft} below with respect to the grading set-up for~$\widetilde{\mathcal{F}^{c}(\mathcal{C})}$.
\end{remark}

\begin{df}\label{cBD_t}
   Let $\mathcal{C}$ be a cyclic $A_\infty$-category of degree $d$. We denote by $\mathcal{F}^{c}(\mathcal{C})^{triv}$ the $(2d-5)$-twisted BD algebra which has the same underlying graded algebra as $\mathcal{F}^{c}(\mathcal{C}),$ but with zero bracket and whose differential is just induced from the derivation $d_{Hoch}+uB$.
\end{df}
\begin{df}\label{cBD_Tr}
   Let $\mathcal{C}$ be a cyclic $A_\infty$-category of degree $d$. We denote by $\mathcal{F}^{c}(\mathcal{C})^{Triv}$ the $(2d-5)$-twisted BD algebra which has the same underlying graded algebra as $\mathcal{F}^{c}(\mathcal{C}),$ but with zero bracket and whose differential is just induced from the derivation $d_{Hoch}$.
\end{df}

\subsection{Classical Interacting Closed String Field Theory from CY Categories}\label{inclsfts}
Denote by $S^c_{g=0}$ the genus zero part of the closed string vertices, recalling definition \eqref{cms} and equation \eqref{csvmce}. Let $\mathcal{C}$ be a smooth cyclic $A_\infty$-category. Then it follows from the genus zero part of equation \eqref{csvmce} and theorem \ref{cei} that $$(d_{hoch}+uB+\{\rho(S^c_{g=0}),\_\}_c)^2=0,$$ ie. it defines a differential on the shifted Poisson algebra $\Big(\widetilde{\mathcal{F}^c(\mathcal{C})}|_{\gamma=0},\{\_,\_\}_c\Big).$
Recalling that $$\widetilde{\mathcal{F}^c(\mathcal{C})}|_{\gamma=0}=Sym\Big(CH_*(\mathcal{C})[u^{-1}][d]\Big)$$  the differential equivalently defines a co $L_\infty$-algebra on the graded vector space ${CH_*(\mathcal{C})[u^{-1}][d+1]}$, which induces by dualizing an $L_\infty$-algebra structure on 
    \begin{equation}\label{cych}
        \Big(CH_*(\mathcal{C})[ u^{-1}][(d+1)]^\vee,d_{hoch}+uB\Big),
    \end{equation}
     a chain complex computing cyclic cohomology of $\mathcal{C}$ (shifted by $-d-1$). We denote this $L_\infty$-algebra by $L(\mathcal{C}).$
     \begin{df}\label{cisft}
     Given a smooth cyclic $A_\infty$-category $\mathcal{C}$
     we call \emph{classical interacting closed string field theory} associated to $\mathcal{C}$ the datum of the dg shifted Poisson algebra
$$\Big(\widetilde{\mathcal{F}^c(\mathcal{C})}|_{\gamma=0},\cdot, d_{hoch}+uB+\{\rho(S^c_{g=0}),\_\}_c, \{\_,\_\}_c\Big).$$
     and $L(\mathcal{C})$, the $L_\infty$-algebra \eqref{cych}, the prior being its observables and the latter describing its space of fields. 
 \end{df}    
\begin{remark}\label{bgbt}
    The shifted Poisson structure $\{\_,\_\}_c$ is not \emph{not} induced from a non-degenerate odd pairing on the space of fields \eqref{cych}, which is what usually happens in the BV approach to (quantum) field theories. See section 1.4 of \cite{CL15} and \cite{BuYo17} for a general discussion of  
   degenerate shifted Poisson structures.
\end{remark}

\subsection{BCOV Theory}
    In \cite{coli12} and \cite{CL15} Costello-Li introduce BCOV theory, which is a version of Kodaira-Spencer gravity when formulated in the BV formalism. We recall in proposition \ref{propprop} how BCOV theory is related to the objects from the previous section \ref{fcsfts}; see also section 4 of \cite{coli12}.
    \subsubsection{Quantum Free BCOV Theory}
    
    Let $X$ be a smooth compact Calabi-Yau manifold. The central object is the following chain complex built from polyvector fields \begin{equation}\label{poly1}
        \mathcal{E}_{BCOV}(X):=(PV^{\bullet,\bullet}(X)\llbracket u\rrbracket[2],\bar{\partial}+u\partial).
    \end{equation}
    The abelian $L_\infty$-algebra $\mathcal{E}_{BCOV}(X)[-1]$ describes the fields of (perturbative) free BCOV theory. We want to equip
    $$\mathcal{O}\big(\mathcal{E}_{BCOV}(X)\big):=\widehat{Sym}\big(\mathcal{E}_{BCOV}(X)^\vee\big)\llbracket\gamma\rrbracket,$$
    the space of completed symmetric powers of (strong) linear functionals\footnote{Here we use the grading convention that for $V$ a graded vector space $(V^\vee)_n=(V_{-n})^\vee$.} on the chain complex \eqref{poly1}, with the structure of a BD algebra. Because of the infinite dimensional nature of the objects appearing one has to use regularization techniques, which we do following the theory developed in \cite{Co11}. For that we fix a Kähler metric $g$ on $X$. Then we can define for every $L>0$ a square zero map 
    $$\Delta_L:\mathcal{O}\big(\mathcal{E}_{BCOV}(X)\big)\rightarrow \mathcal{O}\big(\mathcal{E}_{BCOV}(X)\big),$$
which is roughly defined by contracting with an element $K_L^g\in PV(X)\otimes PV(X)$ and applying the map $\partial$ on powers of $u=0$, zero otherwise; see section 3.1 of \cite{coli12} for details. By imposing the BD relation \eqref{BDrel} $\Delta_L$ determines a shifted Poisson bracket $\{\_,\_\}_L$. Thus we can formulate:
\begin{df}\label{uvdwekm}
Let $X$ be a smooth compact Calabi-Yau manifold ($L>0$ and $g$ a Kähler metric). Then the BD algebra 
    \begin{equation}\label{phycloBCOV}
  \Big( \mathcal{O}\big(\mathcal{E}_{BCOV}(X)\big),\bar{\partial}+u\partial+\gamma\Delta_L,\{\_,\_\}_L\Big)
    \end{equation}
    are \emph{the (regularized) quantum observables of free BCOV theory}. 
\end{df}
\begin{remark}
   Definition \ref{uvdwekm} is independent from the choice of a Kähler metric up to homotopy. See the discussion around theorem 3.6.2 of \cite{coli12} for a precise statement.
\end{remark}
\begin{remark}\label{nenr}
    There is a natural pairing on polyvector fields, given by 
    \begin{equation}\label{polypair}
\langle \mu,\lambda\rangle=\int_XTr(\mu\wedge \lambda),
 \end{equation}
 for $\mu,\lambda\in PV^{\bullet,\bullet}(X)[2]$ and where for $\alpha \in PV^{\bullet,\bullet}(X)[2]$ $$Tr(\alpha)=(\omega\vee \alpha)\wedge \omega$$
 with $\omega$ the holomorphic volume form of $X$, see section 2.2 of \cite{coli12}. For $L=0$ the element $K_L^g$ is given by the inverse of the non-degenerate pairing \eqref{polypair}, but it is of distributional nature (see section 4.3 of \cite{coli12}). The elements $K_l^g$ are homologous to it. 
\end{remark}
\begin{remark}
    As in remark \ref{bgbt} the shifted Poisson bracket $\{\_,\_\}_L$ is not induced from a (smooth functions of) non-degenerate pairing(s) on the (fibers over X of the) space of fields $\mathcal{E}_{BCOV}(X)$. If we however just pretend that this  is the case (and take $L=0$, see the previous remark \ref{nenr}) and denote by the  $\langle\_,\_\rangle_{im}$ the induced pairing on fields from this imagined inverse we recover the original kinetic term of the BCOV theory from \cite{BCOV94}
    \begin{equation}
    \int_XTr(\mu\wedge\bar{\partial}\partial^{-1}\mu)=\langle(\bar{\partial}+u\partial)\mu,\mu\rangle_{im}
     \end{equation}
    
    where $\mu\in PV^{1,1}(X)[2].$
\end{remark}

  We recall the notions of dg-BD algebras and of quasi-isomorphism of dg-BD algebras from definition \ref{dgBD} and \ref{qiBD}. The dg-BD algebra (note the extra comma)
  \begin{equation}\label{dgBD1}
\Big(  \mathcal{O}\big(\mathcal{E}_{BCOV}(X)\big),\bar{\partial}+u\partial, \gamma\Delta_L,\{\_,\_\}_L\Big)\end{equation}
  and for $\mathcal{C}$ a smooth cyclic $A_\infty$-category the dg-BD algebra
  \begin{equation}\label{dgBD2}
    \Big(Sym\Big(CH_*(\mathcal{C})[u^{-1}][d-2]\Big)^{-*}\llbracket\gamma\rrbracket,d_{Hoch}+uB,\gamma\Delta_c,\{\_,\_\}_c\Big)
     \end{equation}
  induce the BD algebras \eqref{phycloBCOV} respectively $\mathcal{F}^{c}(\mathcal{C})$ under the functor from dg-BD algebras to BD algebras from remark \ref{Induzierung}.
\begin{prop}\label{propprop}
    Let $X$ be a smooth compact Calabi-Yau manifold, so that its category of bounded coherent sheaves $\mathcal{C}=D^b_{dg}(X)$ is a smooth proper Calabi-Yau category. The dg-BD algebra from \eqref{dgBD1}  is quasi-isomorphic to the dg BD algebra  from \eqref{dgBD2} for $\mathcal{C}=D^b_{dg}(X)$.\footnote{Or rather a cyclic model of the latter via theorem 2.37 of \cite{AmTu22}, which we refer to for the independence of this choice.} 
\end{prop}    
\begin{proof}[Sketch of Proof]
    
This follows because we can identify the respective dg-BD algebras as Fock spaces associated to symplectic chain complexes, see section 7 of \cite{Cos05} respectively section 4 of \cite{coli12}.\footnote{In fact the two cited sources differ by the fact that in \cite{Cos05} the symplectic vector space is used, whereas in \cite{coli12} the dual symplectic vector space is used. In the identification \eqref{llD} we are implicitly reversing this, again. See also remark \ref{nenr}.}
Those symplectic chain complexes are quasi-isomorphic, thus implying the claim: The underlying chain complexes of these symplectic chain complexes are given by the left and right most object of diagram \eqref{llD} below
 \begin{equation}\label{llD}
 \begin{tikzcd}
 \hspace{-0.7cm}
\Big(CH_*(D^b_{dg}(X))((u))[d-2],d_{H}+uB\Big)\arrow{r}{\simeq}&\Big(H^*(X)((u))[d-2],0\Big)&\Big(PV^{\bullet,\bullet}(X)((u))[2],\bar{\partial}-u\partial\Big)\arrow[swap]{l}{\simeq}.
      \end{tikzcd}
     \end{equation} 
 The first map is the HKR map (see \cite{kon03}), the second map is described in proposition 5.2.1 of \cite{coli12}. The compatible symplectic structures of the symplectic chain complexes are given by following diagram (commuting on homology with respect to the differentials in diagram \ref{llD}).
 \begin{equation}
    \begin{tikzcd}\label{komkom}
        CH_*(D^b_{dg}(X))((u))[d-2]^{\otimes 2} \arrow{r}{}\arrow{d}&k[2d-6]\arrow[equal]{d}\\
         H^*(X)((u))[d-2]^{\otimes 2} \arrow{r}{}&k[2d-6]\arrow[equal]{d}\\
         PV^{\bullet,\bullet}(X)((u))[2]^{\otimes 2}\arrow{r}{}\arrow{u}&k[2d-6],
 \end{tikzcd}
   \end{equation}
where the upper most horizontal map is given by the higher residue pairing induced from the Mukai pairing (see section 7.1 of \cite{Cos05}), the second horizontal map is given by the standard integration pairing on  $H^*(X)$ for a compact manifold times the residue pairing (see proposition 5.2.1 of \cite{coli12}) and the last horizontal map is induced from pairing \ref{polypair} times the residue pairing (see section 4.3 of \cite{coli12}). The fact that the upper square commutes on homology is documented in \cite{tu24}, the fact that the second square commutes is proposition 5.2.1 of~\cite{coli12}. 
\end{proof}
\subsubsection{Classical Interacting BCOV Theory}\label{scibcov}
The original interaction term of classical BCOV theory from \cite{BCOV94} was defined as 
\begin{equation}\label{sesb}
I_{cubic}:=\int_XTr(\mu^3),
\end{equation}
recalling the pairing from equation \eqref{polypair} and for $\mu\in PV^{1,1}(X)[2]$ (that is for the ordinary, not (anti) ghost/field content and not taking into account the gravitational descendants encoded by the $u$ parameter).
However, it is not clear how to extend the functional \eqref{sesb} to positive powers of $u$ in $(PV^{\bullet,\bullet}(X)\llbracket u\rrbracket[2]$.
On the other hand the functional \eqref{sesb} induces a Lie algebra structure on $PV^{\bullet,\bullet}(X)[1]$ via the square zero vector field $$\{I_{cubic},\_\}_{L=0}:\mathcal{O}\big(\mathcal{E}_{BCOV}(X)\big)|_{\gamma=0}\rightarrow \mathcal{O}\big(\mathcal{E}_{BCOV}(X)\big)|_{\gamma=0}$$  which is well-defined for $L=0$, see section 5.9.4 of \cite{Co11}).
We can see that the resulting Lie structure on $PV^{\bullet,\bullet}(X)[1]$ exactly comes from the standard Schouten bracket $[\_\_]_S$ on polyvector fields, see eg. section 2.1 of \cite{coli12}.
We can extend the  Schouten bracket $u$-linearly to $(PV^{\bullet,\bullet}(X)\llbracket u\rrbracket[1]$. This leads Costello-Li to the following definition.
\begin{df}\label{cibcov}
 \emph{Classical interacting BCOV theory formulated in the BV formalism} is the datum of the $L_\infty$-algebra
 \begin{equation}\label{ezd}
     \Big(PV^{\bullet,\bullet}(X)\llbracket u\rrbracket[1],\bar{\partial}+u\partial,[\_\_]_S\Big)
 \end{equation}
 and the shifted Poisson algebra
 \begin{equation}
 \Big( \mathcal{O}\big(\mathcal{E}_{BCOV}(X)\big)|_{\gamma=0},\bar{\partial}+u\partial+d_{CE},\{\_,\_\}_{L=0}\Big),
 \end{equation}
 where $d_{CE}$ is the differential induced from the $L_\infty$-algebra structure \eqref{ezd}.
\end{df}
Costello-Li provide an alternative description of definition \ref{cibcov}. They note that there is an element 
\begin{equation}
I_{BCOV}\in \mathcal{O}\big(\mathcal{E}_{BCOV}(X)\big)|_{\gamma=0}\end{equation}
which satisfies the classical master equation
$$(\bar{\partial}+u\partial) I_{BCOV}+\frac{1}{2}\{I_{BCOV},I_{BCOV}\}_{L=0}=0$$
or in other words it gives a square zero vector field
$$\Big(\bar{\partial}+u\partial+\{I_{BCOV},\_\}_{L=0}\Big)^2=0,$$
which induces an $L_\infty$-structure on the complex underlying \eqref{ezd}.
Costello-Li further prove that there is an $L_\infty$-quasi isomorphism between the  $L_\infty$-algebra determined by the square zero vector field $\bar{\partial}+u\partial+\{I_{BCOV},\_\}_{L=0}$ and the $L_\infty$-algebra \eqref{ezd}, see section 6 of \cite{coli12}, also around section 8 of \cite{CL15}.

We expect the structures discussed in this subsection to be related to the structures from subsection \ref{inclsfts}, which we make precise via following conjecture.

\begin{Conj}\label{conj_fo}
The dg-Lie algebra \eqref{ezd} is $L_\infty$-quasi-isomorphic to the $L_\infty$-algebra of classical interacting closed SFT from definition \ref{cisft} applied to the smooth proper Calabi-Yau category ${\mathcal{C}=D^b_{dg}(X)}$ of quasi-coherent sheafs.\footnote{Or better a cyclic model of the latter via theorem 2.37 of \cite{AmTu22}.}
\end{Conj}
Let us make some remarks supporting this conjecture. We can identify the underlying chain complexes of the $L_\infty$-algebras featuring in conjecture \ref{conj_fo} via 
\begin{equation}\label{dstabi}
\Big(CH_*(\mathcal{C})[u^{-1}][(d+1)]^\vee,d_{hoch}+uB\Big)\simeq \Big(CH^*(\mathcal{C})\llbracket u\rrbracket[(d+1)],d_{hoch}+uB\Big)\simeq \Big(PV^{\bullet,\bullet}(X)\llbracket u\rrbracket[1],\bar{\partial}+u\partial\Big),
\end{equation}
where we first use the higher residue pairing and then the HKR isomorphism (see eg. section 3.3 of \cite{CL15}). Next we make some remarks about how also the higher brackets of the $L_\infty$-algebras featuring in conjecture \ref{conj_fo} should be related:
We recall that the $L_\infty$-structure \eqref{ezd} is (non-trivially) equivalently described via the square zero vector field $$\bar{\partial}+u\partial+\{I_{BCOV},\_\}_{L=0}:\mathcal{O}\big(\mathcal{E}_{BCOV}(X)\big)|_{\gamma=0}\rightarrow \mathcal{O}\big(\mathcal{E}_{BCOV}(X)\big)|_{\gamma=0},$$
whereas the $L_\infty$-structure \eqref{cisft} is defined via the square zero vector field 
$$d_{hoch}+uB+\{\rho(S^c_{g=0}),\_\}_c:\widetilde{\mathcal{F}^c(\mathcal{C})}|_{\gamma=0}\rightarrow \widetilde{\mathcal{F}^c(\mathcal{C})}|_{\gamma=0}.$$
We should be able to identify under the diagram \eqref{dstabi} the shifted Poisson structures $\{\_,\_\}_{L=0}$ from definition \ref{ezd} respectively $\{\_,\_\}_c$ from definition \ref{cych}, similarly for the differentials $$\bar{\partial}+u\partial:\mathcal{O}\big(\mathcal{E}_{BCOV}(X)\big)|_{\gamma=0}\rightarrow \mathcal{O}\big(\mathcal{E}_{BCOV}(X)\big)|_{\gamma=0}$$ respectively $$d_{hoch}+uB:\widetilde{\mathcal{F}^c(\mathcal{C})}|_{\gamma=0}\rightarrow \widetilde{\mathcal{F}^c(\mathcal{C})}|_{\gamma=0},$$ along the lines of proposition \ref{propprop}.
Then the conjecture reduces to showing that under such an identification the Maurer-Cartan element $$I_{BCOV}\in\mathcal{O}\big(\mathcal{E}_{BCOV}(X)\big)|_{\gamma=0}$$ is equivalent to the Maurer-Cartan element $$S^c_{g=0}\in\widetilde{\mathcal{F}^c(\mathcal{C})}|_{\gamma=0}$$ for $\mathcal{C}=D^b_{dg}(X)$ (or again better a cyclic model of the latter via theorem 2.37 of \cite{AmTu22}). Following theorem provides strong evidence that these two elements are indeed equivalent.
\begin{theorem}[theorem 6.6 of \cite{AmTu22}]\label{linojunwu}
Given a smooth Frobenius associative algebra $A$, which induces a commutative Frobenius algebra structure on $HH_*(A)$ or equivalently a closed 2d TCFT with structure maps 
$$\omega^A_{g,n}:HH_*(A)^{\otimes n}\rightarrow k,$$
there is an essentially unique splitting $s^{can}$ and
we have (recalling the splitting formality morphism \ref{cltrmo})
\begin{equation}\label{ssvl}
(\mathcal{K}_s\circ\rho_A)(S^c_{g=0})(\lambda_1u^{k_1},\cdots, \lambda_nu^{k_n})=\omega(\lambda_1,\cdots,\lambda_n)\cdot\int_{\widebar{\mathcal{M}_{0,n}}}\psi^{k_1}\cdots \psi^{k_n}.
\end{equation}
    \end{theorem}
We should compare this with the formula from definition 2.10.1 of \cite{coli12} for 
${I_{BCOV}\in\mathcal{O}\big(\mathcal{E}_{BCOV}(X)\big)|_{\gamma=0}}$ given by 
\begin{equation}\label{axig}
I_{BCOV}(\mu_1u^{k_1},\cdots, \mu_nu^{k_n})=\mu_1\wedge\cdots\wedge\mu_n\int_{\widebar{\mathcal{M}_{0,n}}}\psi^{k_1}\cdots \psi^{k_n},
\end{equation}
furthermore in this context there is a unique splitting provided by the complex conjugate, see section 6.3 of \cite{AmTu22}.
\begin{remark}
Let us recall sections \ref{2dop} and \ref{cycmodop} where we explain how open-closed TCFT's induce the open and closed BD algebras and shifted Poisson algebras considered here. This makes clear that a relevant question related to conjecture \ref{conj_fo} and more precisely related to comparing the elements \eqref{ssvl} and \eqref{axig} via above discussion is whether we can construct some (analytic version) of a (genus zero) open-closed TCFT (recalling this notion from definition \ref{octcft123}) such that it's purely open sector is described by the smooth `cyclic' dg-algebra of Dolbeault forms on X possibly with values in some vector bundles generating $D^b_{dg}(X)$ and whose purely closed part is described by the `Frobenius' algebra of polyvector fields with the pairing \eqref{polypair}. Then we could hope to apply some version of theorem \ref{linojunwu} \footnote{which we cannot directly apply due to the infinite dimensional nature of these objects} (such that the algebra of Dolbeault froms would take the place of the algebra $A$ from theorem \ref{linojunwu} and such that the algebra of polyvector fields would correspond to $HH_*(A)$ from theorem \ref{linojunwu}), which would then imply conjecture \ref{conj_fo}.

In fact such an analytic version of an open TCFT has been constructed in \cite{co07b} in a rather general fashion; we would be interested in applying this theory to example 3.2 from \cite{co07b}.  
\end{remark}

\begin{remark}
For further related questions we refer to section 8 of \cite{CL15}, the paper \cite{Tu19} and also to proposition 8.2 of \cite{HVZ08}. 
\end{remark}
In \cite{li11, Li16} quantum BCOV theory on the elliptic curve is constructed. Its $L\rightarrow \infty$ limit defines correlation functions\footnote{Since an elliptic curve is compact the $L\rightarrow \infty$ limit of the regularized BV Laplacian and the associated shifted Poisson bracket is zero. Thus the ${L\rightarrow \infty}$ limit of the obtained solution to the QME defines mentioned correlation functions; see section 2.3 of \cite{li11}.  It would be interesting to compare this phenomena with how a splitting of the non-commutative Hodge filtration is used to trivialize the BV structure in the algebraic set-up; see around equation \eqref{cltrmo}. }, which are shown to coincide with the Gromov-Witten invariants of the mirror elliptic curve. This presents one striking application of the circle of ideas around string field theory to the topic of mirror symmetry, as laid out in the introduction in section \ref{eahms}.

\newpage
\section{Large $N$ Open String Field Theory}\label{ncw}
In this section we explain how to obtain shifted Poisson and BD algebra structures from cyclic $A_\infty$-categories and their quantizations. We will see later in section \ref{lqtsection} that we can interpret those as describing the observables of large $N$ open string field theory, compare also around equation \eqref{lqtintro} from the introduction. In fact in order to do that we should restrict to what we call essentially finite cyclic $A_\infty$-categories, which we introduce in the last part of this section.
\subsection{Shifted Poisson Structure from Cyclic $A_\infty$-Categories} 
Given a collection of (cyclic) chain complexes $V_B$, see definition \ref{ColV}, we have that $\big(Cyc^*_+(V_B)[-1],d\big)$ is a cochain complex; this is just definition \ref{llekwr} applied to the $A_\infty$-category with zero (higher) compositions built from $V_B$.
Recall that we adopted the notation 
$$(a_1a_2\cdots a_n)\in Cyc^*_+({V_B})[-1],$$
where by definition $a_i\in Hom(\lambda,\mu)^\vee[-1]$ and $a_{i+1}\in Hom(\mu,\nu)^\vee[-1]$ for $\lambda,\mu,\nu\in B$.
\begin{cons}
   
 We define
 $$\{\_,\_\}:Cyc^*_+({V_B})[-1]\otimes Cyc^*_+({V_B})[-1]\rightarrow Cyc^*_+({V_B})[-1]$$
 by 
 \begin{equation}\label{bracket}
 \{(a_1a_2\cdots a_n),(b_1b_2\cdots b_m\}= \sum_{i=1}^m\sum_{j=1}^n\langle a_j,b_i\rangle^{-1}(-1)^*(a_{i+1}\cdots a_na_1\cdots a_{i-1}b_{j+1}\cdots b_mb_1\cdots b_{j-1})
 \end{equation}
for $n+m>2$ and zero otherwise, $(-1)^*$ being determined by the Koszul rule. 
Here 
\begin{equation*}
   \langle a_j,b_i\rangle^{-1}=
    \begin{cases}
      \langle a_j,b_i\rangle_{\lambda\mu}^{-1}, & \text{if}\ a_j\in Hom(\lambda,\mu)^\vee[-1]\ \text{and}\ b_i\in Hom(\mu,\lambda)^\vee[-1]\ \text{for some $\mu,\lambda\in Ob(\mathcal{C})$} \\
      0, & \text{otherwise}
    \end{cases}
      \end{equation*}
      \end{cons}
\begin{lemma}
   Let $V_B$ be a collection of cyclic chain complexes of degree $d$. Then 
   $$(Cyc^*_+(V_B)[-1],d,\{\_,\_\})$$
   is a $(d-2)$-shifted dg Lie algebra. 
   \end{lemma}
\begin{proof}
The fact that this is a shifted Lie algebra is proven in \cite{Che10}, Definition-Lemma 2.
For the algebra case see \cite{Ha07}, proposition 2.8 where a notation closer to here is used.
The fact that the differential is compatible with the Lie structure follows from the fact that the cyclic structure is by definition closed with respect to this differential.
\end{proof}
Recall the notion of cyclic $A_\infty$-category from definition \ref{cycAinfcat}.
\begin{lemma}\label{Ham1}

   A cyclic $A_\infty$-category structures on $V_B$ induces a Maurer-Cartan elements of $(Cyc^*_+(V_B)[-1],d,\{\_,\_\})$ of degree $(3-d)$.
    \end{lemma}
   
\begin{proof}
This is a standard result, see eg. 4.4.3 of \cite{da24}.
Recall that we can describe a cyclic $A_\infty$-category structure on $V_B$ as a degree $-1$ 
coderivation on $BarV_B$ that squares to zero and commutes with $d$.
We get an induced differential $m_\mathcal{C}$ on $Cyc^*_+(V_B)$ that squares to zero and commutes with the internal differential $d$.
This fact then implies that $I_\mathcal{C}=\langle m_\mathcal{C}\_,\_\rangle$ satisfies 
$$dI_\mathcal{C}+\frac{1}{2}\{I_\mathcal{C},I_\mathcal{C}\}=0,$$
ie. it is a Maurer-Cartan element.
\end{proof}
 \begin{remark}
We call such Maurer-Cartan elements also classical interaction term or $A_\infty$-hamiltonian.
    \end{remark}
 We recall that for $V$ a chain complex we use the notation $SymV:=\bigoplus_{n=0}V^{\otimes n}/S_n.$ 
 It is a standard fact that for $\mathfrak{g}$ a dg n-shifted Lie algebra $Sym(\mathfrak{g})$ 
 becomes an n-shifted Poisson dg algebra by extending the bracket according to the Leibniz rule and the differential as a derivation.
 Thus from the previous lemma it follows:
\begin{theorem}\label{fP}
    Let $V_B$ be a collection of cyclic chain complexes of degree $d$.
    Then `the observables of the free\footnote{a name motivated by physics, where a theory coming from a quadratic action functional is referred to as free.} theory' 
    $$\mathcal{F}^{fr}(V_B):=Sym(Cyc^*_+(V_B)[-1])$$
    are a $(d-2)$-shifted Poisson dg-algebra.

\end{theorem}

\begin{df}\label{F_cl} 
Given a cyclic $A_\infty$ category $\mathcal{C}$ of degree $d$  denote by
    $$\mathcal{F}^{cl}(\mathcal{C}):=\mathcal{F}^{fr}(V_{Ob\mathcal{C}})^{tw}$$
    the $(d-2)$-shifted Poisson dg-algebra obtained by twisting $\mathcal{F}^{fr}(V_{Ob\mathcal{C}})$ by the Maurer-Cartan element
    determined by cyclic $A_\infty$-structure through lemma \ref{Ham1}.
    
\end{df}
Here $cl$ stands for classical.\footnote{ a natural name from the physics perspective would also be interactive.}
\begin{lemma}\label{cov}
We find a functor
$$(cyc_dA_\infty cat)_{st}\rightarrow   (Poiss_{d-2})^{op},$$
$$(F:(\mathcal{C},\langle\_,\_\rangle_\mathcal{C})\rightarrow (\mathcal{D},\langle\_,\_\rangle_\mathcal{D}))\mapsto(F^*:\mathcal{F}^{cl}(\mathcal{D})\rightarrow \mathcal{F}^{cl}(\mathcal{C}))$$
\end{lemma}

\begin{proof}
Since we assume that $F$ is cyclic we have commutativity of following diagram.
$$\begin{tikzcd}
\hspace*{-1cm}
    Hom_\mathcal{D}(F(i),F(j))^\vee\otimes Hom_\mathcal{D}(F(j),F(i))^\vee[-2d]\arrow{d}{F^*_{i,j}\otimes F^*_{j,i}}&Hom_\mathcal{D}(F(i),F(j))\otimes Hom_\mathcal{D}(F(j),F(i))\arrow{l}\arrow{r}{\langle\_,\_\rangle_\mathcal{D}}&k[-d]\\
    Hom_\mathcal{C}(i,j)^\vee\otimes Hom_\mathcal{C}(j,i)^\vee[-2d]&Hom_\mathcal{C}(i,j)\otimes Hom_\mathcal{C}(j,i)\arrow{l}\arrow[swap]{ru}{\langle\_,\_\rangle_\mathcal{C}}\arrow{u}{F_*^{ij}\otimes F_*^{ji}}
\end{tikzcd}$$
Here the horizontal two maps are defined in the standard way using the respective pairings, see section 3.1 of \cite{GGHZ21}.
Since we assume that these are non-degenerate we can invert the vertical maps and ponder commutativity of following diagram.
$$\begin{tikzcd}
\hspace*{-1cm}
    Hom_\mathcal{D}(F(i),F(j))^\vee\otimes Hom_\mathcal{D}(F(j),F(i))^\vee[-2d]\arrow{d}{F^*_{i,j}\otimes F^*_{j,i}}\arrow{r}&Hom_\mathcal{D}(F(i),F(j))\otimes Hom_\mathcal{D}(F(j),F(i))\arrow{r}{\langle\_,\_\rangle_\mathcal{D}}&k[-d]\\
    Hom_\mathcal{C}(i,j)^\vee\otimes Hom_\mathcal{C}(j,i)^\vee[-2d]\arrow{r}&Hom_\mathcal{C}(i,j)\otimes Hom_\mathcal{C}(j,i)\arrow[swap]{ru}{\langle\_,\_\rangle_\mathcal{C}}\arrow{u}{F_*^{ij}\otimes F_*^{ji}},
\end{tikzcd}$$
where now the horizontal maps are the inverses of the previous vertical maps.
Indeed, this diagram is also commutative, which follows by injectivity of the maps $F_*^{ji}$ and commutativity of the previous diagram.
This implies that for all $F(i),F(j)\in Ob(\mathcal{D})$ and $a\in Hom_\mathcal{D}(F(i),F(j))^\vee$ and $b\in Hom_\mathcal{D}(F(j),F(i))^\vee$ we have 
$\langle F^*_{ij}(a),F^*_{ji}(b)\rangle_{\mathcal{C}}^{-1}=\langle a,b\rangle_{\mathcal{D}}^{-1}.$
As we further assume that  $F$ is strict it is now straightforward to verify that the induced map 
$F^*:\mathcal{F}^{cl}(\mathcal{D})\rightarrow \mathcal{F}^{cl}(\mathcal{C})$ respects the shifted Poisson algebra structures, using definition \eqref{bracket}.   
\end{proof}

\begin{lemma}
Given a cyclic strict $A_\infty$-morphism $F:\mathcal{C}\rightarrow \mathcal{D}$
we have that 
$$F^*(I_\mathcal{D})=I_\mathcal{C}$$
for the associated cyclic $A_\infty$ hamiltonians.
\end{lemma}
\begin{proof}
This is well known, see eg. proposition 4.6 of \cite{da24}.
Concretely this can be verified using the explicit characterizations of cyclic $A_\infty$-functors,
equation \eqref{nfn} and equation \eqref{F_inf_rel}.
\end{proof}
\begin{remark}\label{befu}
    In the case of a cyclic $A_\infty$-algebra definition 2.7 and equation 2.5 of \cite{Ha07}
    provide an alternative description of the shifted Poisson bracket. Using this alternative 
    description lemma 5.5 of \cite{HaLa08} implies that the assignment of the shifted Poisson
    algebra $F^{cl}(A)$ to a cyclic $A_\infty$ algebra $A$ is also functorial with respect to 
    cyclic $A_\infty$-isomorphisms, not just strict ones. The authors use the so called non-commutative calculus for doing so.
    See section 4.5 of \cite{da24} for a sketch of such a non-commutative calculus for the category case. 
    We imagine that using this one should be able to prove a similar functoriality result also for categories.
    See also our remark \ref{caf} for a comparison of nomenclature used. 
    \end{remark}
    \begin{remark}\label{dar}
Similarly lemma 5.5 of \cite{HaLa08} implies that a cyclic $A_\infty$-isomorphism sends associated
$A_\infty$-hamiltonian to associated $A_\infty$-hamiltonian. We imagine that analogous comments to
above remark apply.
\end{remark}

Note that this discussion carries over to the full subcategory of $(cyc_dA_\infty cat)_{st}$ on
collections of cyclic chain complexes, ie. explains functoriality of theorem \ref{fP}.
  
\subsection{Beilinson-Drinfeld Structure and Quantization of Cyclic $A_\infty$-Categories}
\begin{df}
    Given a collection of cyclic chain complexes $V_B$ we extend the previous dg shifted Lie
    algebra structure to
    \begin{equation}
    \label{ce}Cyc^*(V_B):=Cyc^*_+(V_B)\Pi (\prod_{\lambda\in B}\nu_\lambda k)[1], 
    \end{equation}
    where the vector space $\prod_{\lambda\in B}\nu_\lambda k$ spanned by the $\nu_\lambda$ is central and we further redefine the bracket on words of length one by 
    \begin{equation*}
    \{(a),(b)\}=
    \begin{cases}
      \langle a,b\rangle^{-1}\nu_\lambda, & \text{if}\ a\in Hom(\lambda,\lambda)^\vee[-1]\ \text{and}\ b_i\in Hom(\lambda,\lambda)^\vee[-1]\ \text{for some $\lambda\in Ob(\mathcal{C})$} \\
      0, & \text{otherwise.}
    \end{cases}
 \end{equation*}

\end{df}
\begin{remark}
   The image of the bracket of words of length one can contain infinite sums of $\nu_\lambda$ over
   different $\lambda$. Our definition is well defined since we are using the direct product over
   $B$ in the central extension. 
\end{remark}
\begin{cons}\label{nabla}
Given $V_B$ we define 
$$\nabla:Cyc^*(V_B)[-1]\rightarrow \big(Cyc^*(V_B)[-1]\otimes Cyc^*(V_B)[-1]\big)_{S_2}$$
by 
$$\nabla(a_1a_2\cdots a_n)= \sum_{1\leq i<j\leq n}\langle a_i,a_j\rangle^{-1}(-1)^*(a_{i+1}\cdots  a_{j-1})\otimes(a_{j+1}\cdots a_na_1\cdots a_{i-1}).$$
Here 
$$\langle a_i,a_j\rangle^{-1}\in k$$
is defined as before.
We define the value of the summand where $j=i+1$ and if $a_i\in Hom(\beta,\lambda)^\vee[-1]$ and 
$a_j=a_{i+1}\in Hom(\lambda,\beta)^\vee[-1]$ to be 
$$\langle a_i,a_{i+1}\rangle^{-1}_{\beta\lambda}\ \nu_\lambda\otimes(a_{i+2}\cdots a_na_1\cdots a_{i-1}),$$
otherwise zero. If additionally $n=2$ we define the value by 
$$\langle a_i,a_j\rangle^{-1}_{\mu\lambda}\ \nu_\lambda\otimes\nu_\mu$$
if $a_i\in Hom(\mu,\lambda)^\vee[-1]$ and $a_j=a_{i+1}\in Hom(\lambda,\mu)^\vee[-1]$, again
otherwise zero. 
Additionally $\nabla$ is zero on the vector space spanned by the $\nu$'s.
\end{cons}
We refer to definition 2.11 of \cite{Ha07} for the notion of an involutive bi Lie algebra.
\begin{lemma}\label{ibl}
Given $V_B$ a collection of cyclic chain complexes of degree $d$ we have that
    $(Cyc^*(V_B),d,\{\_,\_\},\nabla)$ is an involutive bi Lie dg-algebra of degree $(d-2)$.
\end{lemma}
\begin{proof}
    This is proved in theorem 7 of \cite{Che10}. Note however that the author does not consider
    the central extension. One can check directly that the result remains true. For the one object
    case this is also explained in section 3.2 of \cite{GGHZ21}, where this central extension is included.
\end{proof}

From now on let us \emph{restrict to odd $d$}: 
Introducing $\gamma$, a formal variable of cohomological degree $6-2d$, we denote 
$$\mathcal{F}^{pq}(V_B):=Sym\big(Cyc^*(V_B)[d-4]\big) \llbracket\gamma\rrbracket[3-d]$$
where $V_B$ is a collection of cyclic chain complexes of degree $d$. 
Denote by $\delta$ the Chevalley-Eilenberg differential extended to $\mathcal{F}^{pq}(V_B)$ from the Lie algebra chain
complex of the odd Lie algebra $Cyc^*(V_B)$. 
Denote by $d$ and $\nabla$ the differentials extended
from $Cyc^*(V_B)$ as derivations to $\mathcal{F}^{pq}(V_B)$ and by $\{\_,\_\}$
the shifted Poisson bracket extended from the odd Lie bracket on $Cyc^*(V_B)$.
\begin{theorem}\label{nc_pq} 
Let $V_B$ be a collection of cyclic chain complexes of degree $d$. Then 
    $$\Big(\mathcal{F}^{pq}(V_B) ,d+\nabla+\gamma\delta,\{\_,\_\}\Big)$$
    is a $(d-2)$ twisted Beilinson-Drinfeld algebra over $k\llbracket\gamma\rrbracket$. $\mathcal{F}^{pq}(V_B)$ has multiplication of even degree $(d-3)$.
\end{theorem}
The superscript $pq$ stands for prequantum.\footnote{A better name should have been derived from
`quantization of free theory', which is what $\mathcal{F}^{pq}(V)$ is.}
\begin{proof}
     The result follows as in lemma 3.4 of \cite{Ha07} from our lemma \ref{ibl} and accounting for
     the shifts made. 
     Note that these shifts and the degree of $\gamma$ precisely give the correct
     degrees required in our definition of an $r$-twisted BD algebra, here for $s=(d-2)$.
\end{proof}
Given a strict cyclic $A_\infty$-morphism $F: V_B\rightarrow V_C$ we define a map 
$$F^*: Cyc^*(V_C)\rightarrow Cyc^*(V_B)$$
by setting 
$$F^*(\nu_\lambda)=\sum_{\mu\in F^{-1}(\lambda)}\nu_{\mu}.$$
One can check directly that this is compatible with the extension of the odd Lie bracket. 
Furthermore, by arguments analogous to the discussion of lemma \ref{cov}, it follows that $F^*$ intertwines the co Lie brackets.
We denote by the same letter the extension of this map to 
    $$F^*: \mathcal{F}^{pq}(V_C)\rightarrow \mathcal{F}^{pq}(V_B),$$
   which is a map of BD algebras over $k\llbracket\gamma\rrbracket$. As an upshot we find
\begin{lemma}
There is a functor from the full subcategory of collections of cyclic chain complexes within the category  
$(cyc_dA_\infty Cat)_{st}$ to the category of $(d-2)$ twisted BD algebras by the assignment 
$$\mathcal{F}^{pq}:\ (F: V_B\rightarrow V_C)\mapsto \big(F^*: \mathcal{F}^{pq}(V_C)\rightarrow \mathcal{F}^{pq}(V_B)\big).$$
 
\end{lemma}
\begin{df}\label{nc_dq}
Define  $$p:  \mathcal{F}^{pq}(V_B)\rightarrow Sym^1(Cyc^*_+(V_B)[d-4])[3-d]\cong Cyc^*_+(V_B)[-1]\rightarrow \mathcal{F}^{fr}(V_B),$$
where the first map is the projection on the `$Sym^1$, $\gamma=0$ and $\nu=0$ part'.
The last map is just the inclusion.
\end{df}
 It is straightforward to check that $p$ is a dequantization map in the sense of definition \ref{dq}. 
 For example we have $p(\gamma\delta)=0=p(\nabla).$ 
\begin{df}\label{qcat}
Given a collection of cyclic chain complexes $V_B$ and $I\in MCE (\mathcal{F}^{fr}(V_B))$, we say that 
$$I^q\in MCE (\mathcal{F}^{pq}(V_B))$$
of degree $(3-d)$ is a \emph{quantization of $I$} if $p(I^q)=I.$ 
\end{df}
\begin{remark}
Note that more generally one may ask a quantization to be a differential $d^q$ on the underlying
shifted Poisson algebra of $\mathcal{F}^{pq}(V)$ satisfying the BD relation and such $p\circ d^q=d^{cl}\circ p.$
Thus we are more restrictive, imposing $d^q=d+\nabla+\gamma\delta+\{I^q,\_\}$. 
We ask that a quantization $I^q$ of the interacting theory $\mathcal{F}^{cl}$ is compatible with the standard quantization of the free theory, i.e. with $\mathcal{F}^{pq}(V)$.
\end{remark}
\begin{remark}
  Quantizations of cyclic $A_\infty$-categories  are sometimes called quantum $A_\infty$-categories.
Equivalently a Maurer-Cartan element $I^q \in\mathcal{F}^{pq}(V_\Lambda)$ may be described as a possibly curved (involutive bi-Lie$)_\infty$-algebra structure on a central extension of cyclic cochains of this category.  Compare work by Campos-Merkulov-Willwacher \cite{Ca16}, section 5 and also work by Cieliebak-Fukaya-Latschev \cite{CiFuLa20}, section 12 and work by Naef-Willwacher \cite{NaWi19}, section 7. In \cite{sla24} Slawinski has explicitly analyzed the lift of the $A_\infty$-structure on the Fukaya category of an elliptic curve to a quantum $A_\infty$-structure by using a certain combinatorial model.
However, these authors don't consider the central extension, which is crucial for us.
\end{remark}
\begin{remark}
Further we restrict to the case where $I^q$ has zero constant terms, that is its components in the ground field and the vector space spanned by the $\nu$'s are zero.
\end{remark}
Given $F:V_B\rightarrow V_C$ a cyclic strict $A_\infty$-morphism between collections of cyclic chain complexes and an $I^q\in MCE (\mathcal{F}^{pq}(V_C))$ then 
$$F^*I^q\in MCE (\mathcal{F}^{pq}(V_B)).$$

Indeed this follows as we saw that $F^*$ is compatible with the odd Lie bracket,
thus also with the induced Chevalley Eilenberg differential, as well as the odd co Lie bracket.

Thus we can define a category which we denote 
$$(\text{q-cyc}_dA_\infty cat)_{st},$$
whose objects are pairs of a cyclic $A_\infty$-category $\mathcal{C}$ together with an
$I^q\in MCE (\mathcal{F}^{pq}(V_{Ob\mathcal{C}}))$ such that $p(I^q)=I_\mathcal{C}$.
Its morphisms are cyclic strict $A_\infty$-morphisms $F:\mathcal{C}\rightarrow \mathcal{D}$ such that $F^*I^q_\mathcal{D}=I^q_\mathcal{C}$. 
Note that there is a forgetful functor to $(cyc_dA_\infty Cat)_{st}$ induced by the dequantization map $p$.

\begin{remark}
 In the literature such elements $I^q$ have also been referred to as quantum $A_\infty$-categories. 
\end{remark}
\begin{df}

Given an object of $(\text{q-cyc}_dA_\infty cat)_{st}$ define 
$$\begin{tikzcd}\label{F_q}
\mathcal{F}^q(\mathcal{C}):=\mathcal{F}^{pq}(V_B)^{tw}\end{tikzcd}$$
as the $(d-2)$ twisted BD algebra obtained by twisting $\mathcal{F}^{pq}(V_B)$ by the MC element $I^q$.
The assignment \ref{F_q} defines a functor 
$$\mathcal{F}^q:\ (\text{q-cyc}_dA_\infty cat)_{st}\rightarrow (\text{BD}^{(d-2)tw})^{op},$$
where on morphisms we use $F^*$.
\end{df}
\begin{remark}
Note that this notation for $\mathcal{F}^q(\mathcal{C})$ is ambiguous and we need to have the choice of an $I^q$ in mind.
\end{remark}

In this work we did not examine how the discussion so far can be refined when considering homotopy equivalent solutions to the quantum master equation.
Compare section 6 of \cite{GGHZ21}.

\subsection{Essentially Finite Cyclic $A_\infty$-Categories}
Recall that given a cyclic $A_\infty$-category $\mathcal{C}$ we defined the central extension of our previous shifted Lie algebra by 
$$Cyc^*(V_B):=Cyc^*_+(V_B)\times (\prod_{\lambda\in B}\nu_\lambda k[1]).$$
As we will see in section \ref{LQT_r}, under the `quantized' LQT map at level $N$, each $\nu_\lambda$ will be send to the actual number $N$.
However, if we allow infinite sums of $\nu_\lambda$'s this would be ill-defined.
Thus we restrict ourselves as follows:
\begin{df}\label{essfin}
We call a cyclic unital $A_\infty$-category $\mathcal{C}$ that admits a finite skeleton $Sk\mathcal{C}$, that is a subcategory with finitely many objects whose inclusion defines an equivalence, 
an \emph{essentially finite cyclic $A_\infty$-category}.
\end{df} 
\begin{remark}
Note that for an essentially finite cyclic $A_\infty$-category the inclusion 
$$Sk\mathcal{C}\rightarrow\mathcal{C}$$
is strict and cyclic.
\end{remark}

We define a category, denoted
$$(\text{e.f cyc}_dA_\infty cat)_{st},$$
as the full subcategory of $(cyc_dA_\infty cat)_{st}$ given by essentially finite cyclic $A_\infty$-categories.

In practise we will work with a slightly more complicated, though equivalent category:
\begin{df}
We denote by 
 $$(\text{e.f}^{+} \text{cyc}_dA_\infty cat)_{st},$$
 the category  whose objects are essentially finite cyclic unital $A_\infty$-categories together with the choice of a finite skeleton and together with the inclusion 
$$Sk\mathcal{C}\rightarrow\mathcal{C},$$
and whose morphisms are diagrams as follows 
    \begin{equation}\label{fwde}
    \begin{tikzcd}
   & \mathcal{C}\arrow{r}{F} &\mathcal{D}&\\
   Sk\mathcal{C}\arrow{ur}{\simeq}&  &&Sk\mathcal{D}\arrow{ul}[swap]{\simeq},
\end{tikzcd}
\end{equation}
where $F$ is a strict cyclic $A_\infty$-functor.
\end{df}
Note that the obvious functor 
\begin{equation}\label{obvFun}
  (\text{e.f}^{+} \text{cyc}_dA_\infty cat)_{st}\rightarrow (\text{e.f cyc}_dA_\infty cat)_{st}  
\end{equation}
is an equivalence, which is why we will treat the two appearing categories as interchangeable. We will later see \eqref{undf} that different choices of an inverse of \eqref{obvFun}, that is different choices of a finite skeleta for each essentially finite category, are related by natural quasi-isomorphism, once we localized appropriately.

Let us denote by $\mathcal{D}Poiss_{(d-2)}$ the localization of the category of 
$(d-2)$ shifted Poisson algebras with respect to quasi-isomorphisms.
We (re)define previous functors:
\begin{lemma}\label{nattr1}
    Sending $(\mathcal{C},Sk\mathcal{C})$ an essentially finite cyclic $A_\infty$ category of degree $d$ to $\mathcal{F}^{cl}(Sk\mathcal{C})$ defines a functor 
    $$\mathcal{F}^{cl}_{fin}:\ (\text{e.f cyc}_dA_\infty cat)_{st}\rightarrow (\mathcal{D}Poiss_{(d-2)})^{op}.$$ 
    \end{lemma}
This follows directly from definition \ref{fwde} and \eqref{fff} below.
Further by \eqref{undf} we will see that this functor is independent of the choice of skeleton, up to isomorphism.
To see these facts we first remark the standard result 
\begin{lemma}
    
Cyclic (co)homology is invariant under $A_\infty$-equivalence.
\end{lemma}
\begin{proof} 
   As a sketch, by Proposition 7.3.6 and Lemma 7.4.1 of \cite{Cos07a} it follows that Hochschild (co)homology is invariant under $A_\infty$-equivalence.
   By arguments as in Loday, corollary 2.2.4, it follows that also cyclic (co)homology is invariant.
\end{proof}
 Thus given $(\mathcal{C},Sk\mathcal{C})$ an essentially finite cyclic $A_\infty$-category we get a map of shifted Poisson algebras
    \begin{equation}\label{fff}
        \mathcal{F}^{cl}(\mathcal{C})\rightarrow \mathcal{F}^{cl}(Sk\mathcal{C}),
\end{equation}
    whose underlying map of chain complexes is a quasi-isomorphism.
    This follows from the lemma before, as the inclusion $Sk\mathcal{C}\rightarrow\mathcal{C}$ is in particular an $A_\infty$-equivalence.
    
    Furthermore \eqref{fff} defines a natural quasi-isomorphism as follows 

$$\begin{tikzcd}
(\text{e.f cyc}
    _dA_\infty cat)_{st} \arrow[rr, "\mathcal{F}^{cl}_{fin}", bend left=25, ""{name=U, below}]
\arrow[rd," forget ", bend right=10]
&&\text{$(\mathcal{D}Poiss_{(d-2)})^{op}$}
\\
&(cyc_dA_\infty cat)_{st}\arrow[ur, "\mathcal{F}^{cl}", bend right=10]\arrow[Rightarrow, to=U, ""]&
\end{tikzcd}$$
 
    which we remark to justify our redefinition.
Regarding the choice of finite skeleton of a cyclic $A_\infty$-category we remark that
given a cyclic $A_\infty$-category $\mathcal{C}$ with two choices of finite skeleta $(\mathcal{C},Sk\mathcal{C}_1)$, $(\mathcal{C},Sk\mathcal{C}_2)$ we get a zig-zag of shifted Poisson algebras
   \begin{equation}\label{undf}
        \begin{tikzcd}
        &\mathcal{F}^{cl}(\mathcal{C})\arrow{dl}[swap]{\simeq}\arrow{dr}{\simeq}&\\
        \mathcal{F}^{cl}(Sk\mathcal{C}_1)&&\mathcal{F}^{cl}(Sk\mathcal{C}_2),
    \end{tikzcd}
    \end{equation}
    where the underlying maps of chain complexes are quasi-isomorphisms. 
    
    We make similar (re)definitions for the quantized context: 
 \begin{df}\label{qesfq}
 Let $\mathcal{C}$ be an essentially finite cyclic $A_\infty$-category.
 Denote by $I$ the  $A_\infty$-hamiltonian associated to $\mathcal{C}$.
 We say that  
 $$ I^q\in MCE(\mathcal{F}^{pq}(V_{Ob{\mathcal{C}}}))$$
 is \emph{a quantization of $\mathcal{C}$}
if $$ p(I^q)=I,$$ recalling the dequantization maps p from definition \ref{nc_dq}.
\end{df}
\begin{df}
   We define a category, denoted  
   $$(\text{q.e.f. cyc}_dA_\infty cat)_{st}.$$
    Its objects are quantizations of essentially finite cyclic $A_\infty$-categories.
    The morphisms are given by morphisms of the underlying  cyclic $A_\infty$-categories  
    \begin{equation}
        F:\ \mathcal{C}\rightarrow \mathcal{D}
\end{equation}
such that additionally  $F^*I^q=J^q$ for the respective quantizations.
\end{df}
In practice we will work with a slightly more complicated, though equivalent category:

\begin{df}
   We define a category, denoted  
   $$(\text{q.e.f}^+ \text{cyc}_dA_\infty cat)_{st}.$$
    An objects of this category is given by \begin{itemize}
        \item an essentially finite cyclic $A_\infty$-category
        \item together with the choice of a finite skeleton and with the data of inclusion 
$$\Phi: Sk\mathcal{C}\rightarrow\mathcal{C},$$
\item
denoting by $(I_f,I)$ the pair of $A_\infty$-hamiltonian associated to $Sk\mathcal{C}$ and the one associated to $\mathcal{C}$, we further ask for a pair of 
 $$I^q_f\in MCE(\mathcal{F}^{pq}(V_{Ob{Sk\mathcal{C}}}))\ \text{and} \ I^q\in MCE(\mathcal{F}^{pq}(V_{Ob{\mathcal{C}}}))$$
such that $$p(I^q_f)=I_f\ \text{and}\ \ p(I^q)=I,$$ (recalling the dequantization maps p from definition \ref{nc_dq}) and we have that 
$$\Phi^*I^q=I^q_f.$$
\end{itemize}
The morphisms are given by morphisms in $(\text{e.f}^+ \text{cyc}_dA_\infty cat)_{st}$, that is diagrams like 
    \begin{equation}\label{tnc}
        \begin{tikzcd}
   & \mathcal{C}\arrow{r}{F} &\mathcal{D}&\\
   Sk\mathcal{C}\arrow{ur}{\simeq}&  &&Sk\mathcal{D}\arrow{ul}[swap]{\simeq},
\end{tikzcd}
\end{equation}
such that additionally  $F^*I^q=J^q$ for the respective quantizations.
\end{df}
Note that the obvious functor 
\begin{equation}\label{qobvFun}
  (\text{q.e.f}^{+} \text{cyc}_dA_\infty cat)_{st}\rightarrow (\text{q.e.f cyc}_dA_\infty cat)_{st}  
\end{equation}
is an equivalence, which is why we will treat the two appearing categories as interchangeable.

Note that given an element of $(\text{q.e.f}^{+} \text{cyc}_dA_\infty cat)_{st}$ the induced map 
\begin{equation}\label{BDqi}
\Phi^*:\mathcal{F}^q(\mathcal{C})\rightarrow \mathcal{F}^q(Sk\mathcal{C})
\end{equation}
of BD algebras is a quasi-isomorphism.
Indeed, consider the complete and exhaustive filtration of $\mathcal{F}^q(\mathcal{C})$ described in example 6.3 of \cite{GGHZ21}.
The map $\Phi^*$ is by definition a filtered map.
Then the claim follows from the from the Eilenberg-Moore comparison theorem (theorem 5.5.11 of \cite{wei94}) as on the first page of the spectral sequence associated to the filtrations the map $\Phi^*$ induces a quasi-isomorphism, as remarked in \eqref{fff}.

We denote by $\mathcal{D}\text{BD}^{(d-2)tw}$ the category obtained by localizing $\text{BD}^{(d-2)tw}$ with respect to quasi-isomorphisms.
Since \eqref{BDqi} is an quasi-isomorphism and by definition \eqref{tnc} we observe that
 sending a quantization of an essentially finite cyclic $A_\infty$-category to $\mathcal{F}^q(Sk\mathcal{C})$ defines a functor 
    \begin{equation}\label{redef1}
    \mathcal{F}^q_{fin}:\ (\text{q.e.f. cyc}_dA_\infty cat)_{st}\rightarrow (\mathcal{D}\text{BD}^{(d-2)tw})^{op}. 
    \end{equation}

This redefinition makes sense as follows: Given a quantization of an essentially finite cyclic $A_\infty$-category $(\mathcal{C},Sk\mathcal{C})$ we have by \ref{BDqi} a map of BD algebras
    $$\mathcal{F}^{q}(\mathcal{C})\rightarrow \mathcal{F}^{q}(Sk\mathcal{C}),$$
    whose underlying map of chain complexes is a quasi-isomorphism. This defines a natural quasi-isomorphism  
$$\begin{tikzcd}
(\text{q.e.f.cyc}_dA_\infty )_{st} \arrow[rr, "\mathcal{F}^{q}_{fin}", bend left=25, ""{name=U, below}]
\arrow[rd," forget ", bend right=10]
&&(\mathcal{D}\text{BD}^{(d-2)tw})^{op}
\\
&(\text{q.cyc}
    _dA_\infty cat)_{st}\arrow[ur, "\mathcal{F}^{q}", bend right=10]\arrow[Rightarrow, to=U, ""]&
\end{tikzcd}$$
Regarding the dependence on how a quantum $A_\infty$-category is essentially finitely presented we note:
Given a cyclic $A_\infty$-category $\mathcal{C}$ which is finitely presented in two ways $(\mathcal{C},Sk\mathcal{C}_1)$, $(\mathcal{C},Sk\mathcal{C}_2)$ and with respective quantizations that coincide on $\mathcal{C}$ we get a zig-zag of BD algebras
   \begin{equation}\label{undfq}
    \begin{tikzcd}
        &\mathcal{F}^{q}(\mathcal{C})\arrow{dl}[swap]{\simeq}\arrow{dr}{\simeq}&\\
        \mathcal{F}^{q}(Sk\mathcal{C}_1)&&\mathcal{F}^{q}(Sk\mathcal{C}_2),
    \end{tikzcd}
    \end{equation}
    where the underlying maps of chain complexes are quasi-isomorphisms.
    Thus the functor \eqref{redef1} does not depend on the choice of skeleton and compatible quantization, up to isomorphism.
    \newpage
\section{From (Cyclic) $A_\infty$-Categories to 
(Cyclic) $L_\infty$-Algebras}\label{cw}
In this section we recall the standard construction of shifted Poisson and BD algebras given a cyclic $L_\infty$-algebra or a quantization of such as input. Next we study a functor from $A_\infty$-categories to $A_\infty$-algebras, which lands in $L_\infty$-algebras when composed with the standard commutator functor. We examine how this functor behaves when additionally considering cyclic $A_\infty$-categories, for which we restrict ourselves to essentially finite categories.
\subsection{Shifted Poisson Structure}
A cyclic chain complex $V$ is by definition a cyclic $L_\infty$-algebra for which $l=0$; see definition \ref{coalgebra}.
Given a cyclic chain complex $V$ we denote 
$$\mathcal{O}^{fr}(V):=C^*_+(V),$$
recalling definition \ref{redcecom} of the Chevalley-Eilenberg complex. 
By extending the inverse of the pairing of $V$
$$\langle\_,\_\rangle^{-1}:V^\vee[-1]\otimes V^\vee[-1]\rightarrow k[d-2],$$
    according to the Leibniz rule to the symmetric algebra underlying the Lie algebra cochains we can define a $(d-2)$-shifted Poisson structure.
    The differential is compatible with the shifted Poisson bracket because $\langle\_,\_\rangle$ is closed with respect to the differential.
    Summarizing (for details see section 3.1 of \cite{GGHZ21}) we have
\begin{theorem}\label{shPoC}
  Let $(V,d,\langle\_,\_\rangle)$ be a cyclic chain complex of degree $d$. Then 
  $$\big(\mathcal{O}^{fr}(V),d,\{\_,\_\}\big)$$
  is a $(d-2)$-shifted Poisson dg-algebra.
  \end{theorem} 
  
\begin{lemma}\label{ham_l}
     A cyclic $L_\infty$-structures on a cyclic chain complex $V$ induces a Maurer-Cartan elements of $\mathcal{O}^{fr}(V)$ of degree $(3-d)$. 
  \end{lemma}
  \begin{proof}
      This is a standard result and follows as in the $A_\infty$-case.
      The Maurer-Cartan element is given by $I_L:=\langle l\_,\_\rangle.$
  \end{proof}
  \begin{df}
      Given $L$ a cyclic $L_\infty$-algebra with underlying cyclic chain complex of degree d denote by $\mathcal{O}^{cl}(L)$ the $(d-2)$-shifted Poisson algebra obtained by twisting $\mathcal{O}^{fr}(V)$ by the MCE determined by lemma \ref{ham_l}. 
  \end{df}
Denote by $(cyc_dL_\infty alg)_{st}$ the category whose objects are cyclic $L_\infty$-algebras of degree $d$ and whose morphisms are cyclic strict $L_\infty$-morphisms.
\begin{lemma}\label{nattr2}
    We find a functor 
    $$(cyc_dL_\infty Alg)_{st}\rightarrow (Poiss_{(d-2)})^{op}$$
    $$\big(\phi:(L_1\rightarrow L_2)\big)\mapsto \big(\phi^*:\mathcal{O}^{cl}(L_2)\rightarrow \mathcal{O}^{cl}(L_1)\big).$$ 
\end{lemma}
\begin{proof}
Indeed, this can be verified as in the $A_\infty$-case. See lemma \ref{cov}.    
\end{proof}

\subsection{Beilinson-Drinfeld Structure}
From now on let us restrict again to $d$ odd.
We introduce $\hbar$, a formal variable of cohomological degree $3-d$.
For $(V,d,\langle\_,\_\rangle)$ a cyclic chain complex denote 
$$\mathcal{O}^{pq}(V):=C^*(V)\llbracket\hbar\rrbracket,$$
the non-reduced Chevalley-Eilenberg complex adjoined $\hbar$.
Eg. following \cite{GGHZ21}, we can define a BD differential $d^q$ on $\mathcal{O}^{pq}(V)$ as 
    $$d^q=d+\hbar \Delta.$$
    Indeed, by imposing the BD-relation \eqref{BD} and $\hbar$-linearity we only need to specify the values of $d^q$ on $C^2(V)$, which leaves us to specify the value of $\Delta$ on $C^2(V).$
    It is given by  $\Delta(a\cdot b)=\langle a,b\rangle^{-1}$ for $a,b\in L^\vee[-1]$.
    This is compatible with defining $\{\_,\_\}$ by $\hbar$-linearly extending the bracket on $\mathcal{O}^{fr}(V)$.
    Summarizing we have
\begin{theorem}\label{c_pq}
    For $V$ a cyclic chain complex 
    $$\mathcal{O}^{pq}(V):=\big(C^*(V)\llbracket\hbar\rrbracket,d^q,\{\_,\_\}\big)$$
    is a $(d-2)$ twisted BD algebra over $k$.  
\end{theorem}
Again, analogously to the $A_\infty$-case it follows:
\begin{lemma}
     The assignement 
$$(\phi: V_1\rightarrow V_2)\mapsto (\phi^*:\mathcal{O}^{pq}(V_2)\rightarrow \mathcal{O}^{pq}(V_1))$$
defines a functor from the full subcategory of cyclic chain complexes of degree $d$ of 
$(cyc_dL_\infty alg)_{st}$ to the category of (d-2) twisted BD algebras.
\end{lemma}
\begin{df}\label{fc_dq} 
Define 
    $$p: \mathcal{O}^{pq}(V)\rightarrow \mathcal{O}^{fr}(V)$$ 
    by setting $\hbar$ and the $n=0$ factor of the Chevalley-Eilenberg cochains to zero. 
    It is easy to verify that this defines a dequantization map in the sense of definition \ref{dq}.
\end{df}
\begin{df} We say that an
    $I^q\in MCE(\mathcal{O}^{pq}(V))$ of degree $(3-d)$ is a \emph{quantization} of $I\in MCE(\mathcal{O}^{fr}(V))$ if $p(I^q)=I$.
\end{df}
If $\phi: V_1\rightarrow V_2$ is a strict cyclic $L_\infty$-morphism of cyclic chain complexes and if $I^q\in MCE(\mathcal{O}^{pq}(V_2))$ then $\phi^*I^q\in MCE(\mathcal{O}^{pq}(V_1))$.
This allows us to define a category, which we denote by $$(\text{q-cyc}_dL_\infty alg)_{st}$$ whose objects are pairs of a cyclic $L_\infty$-algebra $L$ with underlying chain complex V and $I^q\in MCE(\mathcal{O}^{pq}(V))$ such that $p(I^q)=I$.
Its morphisms are cyclic strict $L_\infty$-morphisms 
$$\phi:\ L_1\rightarrow L_2$$
such that $\phi^*I^q_2=I^q_1.$ 
\begin{df}\label{O_q}
    Given a cyclic $L_\infty$ algebra $L$ denote by $V$ its underlying chain complex and by $I\in MCE(\mathcal{O}^{fr}(V))$ its associated $L_\infty$-hamiltonian.
    Assume we are given a quantization $I^q\in MCE(\mathcal{O}^{pq}(V)),$ then we define
    $$\mathcal{O}^{q}(L):=\mathcal{O}^{pq}(V)^{tw}$$
    as the twist of $\mathcal{O}^{pq}(V)$ by the MCE $I^q.$
    This assignment defines a functor 
    $$\mathcal{O}^q: (\text{q-cyc}_dL_\infty Alg)_{st}\rightarrow (\text{BD}^{(d-2)tw})^{op}.$$
    \end{df}
\begin{remark}
    Note that the notation $\mathcal{O}^{q}(L)$ is ambiguous and we implicitly have to keep in mind an $I^q\in MCE(\mathcal{O}^{pq}(V))$.
\end{remark}

\subsection{Commutator $L_\infty$-algebras of $A_\infty$-categories}\label{CLS}
We recall the standard definition (eg definition 2.8 of \cite{GGHZ21})
of the functor 
\begin{equation}\label{Cmt}
[\_,\_]:A_\infty Alg\rightarrow L_\infty Alg,\end{equation}
 called the commutator functor, which generalizes the construction of the commutator Lie algebra out of an associative algebra.

In this section we wish to generalize the domain of this functor to $A_\infty$-categories.
We do so by defining a functor 
$$\ A_\infty cat\rightarrow A_\infty Alg.$$
 $$(F:\mathcal{C}\rightarrow \mathcal{D})\mapsto (\tilde{F}:A_\mathcal{C}\rightarrow A_\mathcal{D}).$$
  On objects the functor is defined as follows:
 \begin{cons}\label{obj}
    Let $\mathcal{C}$ be an $A_\infty$-category.
    Recall that this defines the datum of a degree $-1$ coderivation $m_\mathcal{C}$ on its bar complex $Bar \mathcal{C}$ that squares to zero.
\begin{itemize}
     \item We set
 \begin{equation}\label{sqc}
    A_\mathcal{C}:=\bigoplus_{(\lambda,\mu)\in Ob(\mathcal{C})^{\times2}}Hom_\mathcal{C}(\lambda,\mu). 
    \end{equation}
\item We define an $A_\infty$-algebra structure on $A_\mathcal{C}$ as follows.
Define maps
$$ \begin{tikzcd}
m^n: (A_\mathcal{C}[1]^{\otimes n})\arrow{r}{\iota_n}& Bar \mathcal{C}\arrow{r}{m_\mathcal{C}}&Bar \mathcal{C}[1]\arrow{r}{\pi[1]}& A_\mathcal{C}[2]\end{tikzcd}.$$
 Here $\iota_n$ denotes the projection onto length $n$ tensor factors of the form in the bar complex - that is those with `matching boundary conditions', except for the first and last one.\footnote{See equation~\eqref{bar}.}
 Further 
 $$\pi:Bar \mathcal{C}\rightarrow A_\mathcal{C}[1]$$
 denotes projection to tensors of $Bar\mathcal{C}$ of length one.
 Its image is by definition $A_\mathcal{C}[1]$. 
 \item
 We extend the maps $m^n$ as a coderivation to $Bar(A_\mathcal{C})$, which we denote by $m_{A_\mathcal{C}}$ and which has degree $-1$.
 Thus it remains to show that $m_{A_\mathcal{C}}^2=0$ to prove that this defines an $A_\infty$-algebra structure on ${A_\mathcal{C}}$.
 \item We observe that the $\iota_n$ maps induce a chain map 
 $$\iota:Bar(A_\mathcal{C})\rightarrow Bar\mathcal{C}.$$
 Note that this map has a `canonical' complement $W$ to its kernel, given by tensors with`matching boundary conditions, except the first and last one'.
 By definition $\iota|_W$ is injective.
\item It follows by the definition of $m_{A_\mathcal{C}}$ that $\iota\circ m_{A_\mathcal{C}}=m_\mathcal{C}\circ\iota$.
Thus we have that 
$$\iota \circ m_{A_\mathcal{C}}^2=m_\mathcal{C}^2\circ\iota=0.$$
 It holds that $m_{A_\mathcal{C}}(W)\subseteq W.$ Thus since $\iota$ is injective on $W$, it follows that ${m_{A_\mathcal{C}}^2|_W=0}$. 
 However, by definition also $m_{A_\mathcal{C}}|_{ker(\iota)}=0$, thus also $m_{A_\mathcal{C}}^2|_{ker(\iota)}=0.$ 
 Since $W$ is complementary to ${ker(\iota)}$ this together implies  $$m_{A_\mathcal{C}}^2=0$$ in general.
 \end{itemize}
 \end{cons}
 \begin{df}
  Given $\mathcal{C}$ an $A_\infty$-category denote by 
  $\big(A_\mathcal{C},d+m_{A_\mathcal{C}}\big)$ the thus constructed $A_\infty$-algebra, where $d$ denotes the internal differential.
 \end{df}
 Next we define the  functor on morphisms:
 \begin{cons}\label{func}
  Let  $F: \mathcal{C}\rightarrow \mathcal{D}$ be a functor of $A_\infty$-categories which by definition gives a map of coalgebras
  $$Bar\mathcal{C}\rightarrow  Bar\mathcal{D}.$$
  \begin{itemize}
      \item 
  We define 
  $$\tilde{F}_*: Bar(A_\mathcal{C})\rightarrow Bar(A_\mathcal{D})$$
  by extending 
$$  \begin{tikzcd}    
\tilde{F}_*: Bar(A_\mathcal{C})\arrow{r}{\iota_\mathcal{C}}& Bar \mathcal{C}\arrow{r}{F_*}& Bar \mathcal{D} \arrow{r}{\pi_\mathcal{D}}& A_\mathcal{D}[1]\end{tikzcd}$$
  as a map of coalgebras. 
  It has degree zero. 
  We check that $\tilde{F}_*$ intertwines the coderivation $m_{A_\mathcal{C}}$ with $m_{A_\mathcal{D}}$:
  \item
We have a commutative diagram of chain complexes
 \begin{equation}\label{iota_func}
\begin{tikzcd}
    Bar \mathcal{C}\arrow{r}{F_*}  &Bar \mathcal{D}\\
    Bar(A_\mathcal{C})\arrow{u}{\iota_\mathcal{C}}\arrow{r}{\tilde{F}_*}&Bar(A_\mathcal{D})\arrow{u}{\iota_\mathcal{D}},
 \end{tikzcd}
 \end{equation}
 which follows by construction.
 \item
 Thus we have that 
 $$\iota_\mathcal{D}\circ\tilde{F}_*\circ m_{A_\mathcal{C}}={F}_*\circ\iota_\mathcal{C}\circ m_{A_\mathcal{C}}={F}_*\circ m_\mathcal{C}\circ\iota_\mathcal{C}=m_\mathcal{D}\circ{F}_*\circ\iota_\mathcal{C}=m_\mathcal{D}\circ\iota_\mathcal{D}\circ\tilde{F}_*=\iota_\mathcal{D}\circ m_{A_\mathcal{D}}\circ\tilde{F}_*,$$
 
 where we first used commutativity of above diagram, then that $\iota_\mathcal{C}$ intertwines the respective coderivations.
 Then we used that $F_*$ intertwines the coderivations, next again commutatvity of above diagram and lastly that $\iota_\mathcal{D}$ intertwines the respective coderivations.
 Thus
 $$ \iota_\mathcal{D}\circ\tilde{F}_*\circ m_{A_\mathcal{C}}=\iota_\mathcal{D}\circ m_{A_\mathcal{D}}\circ\tilde{F}_*.$$
\item
We have 
$$\tilde{F}_*\circ m_{A_\mathcal{C}}|_{ker(\iota_\mathcal{C})}=0=m_{A_\mathcal{D}}\circ\tilde{F}_*|_{ker(\iota_\mathcal{C})},$$
which follows by definition. 
On the other hand we have that 
$$\tilde{F}_*\circ m_{A_\mathcal{C}}(W_\mathcal{C})\subseteq W_\mathcal{D}\ \text{and}\ \ m_{A_\mathcal{D}}\circ\tilde{F}_*(W_\mathcal{C})\subseteq W_\mathcal{D}.$$  
Since $\iota_\mathcal{D}|_{W_\mathcal{D}}$ is injective we conclude from the previous line and the previous bullet point that 
$$m_{A_\mathcal{D}}\circ\tilde{F}_*=\tilde{F}_*\circ m_{A_\mathcal{C}}.$$
\item Unwinding the definitions one can also show that $\widetilde{G\circ F}=\tilde{G}\circ \tilde{F}.$
\end{itemize}
\end{cons}
\begin{lemma}\label{Commu}
Constructions \ref{obj} and \ref{func} allow us to define 
$$(\tilde\_): A_\infty cat\rightarrow A_\infty Alg.$$
 $$(F:\mathcal{C}\rightarrow \mathcal{D})\mapsto (\tilde{F}:A_\mathcal{C}\rightarrow A_\mathcal{D}).$$
\end{lemma}
\begin{remark}
Note that the chain map $\iota$ is not a map of coalgebras.
However, one may ponder whether it is induced from a morphism in a certain bi-module category, see section 2.2.11 of \cite{Lo13}.
\end{remark}

Let us make observations that we will need in the next section.
By restricting $\iota_\mathcal{C}$ to cyclic words we get an induced map:
\begin{equation}\label{iota}
\iota_\mathcal{C}: (Cyc_*^+(A_\mathcal{C})[1],d+m_{A_\mathcal{C}})\rightarrow (Cyc_*^+(\mathcal{C})[1],d+m_{\mathcal{C}}).
\end{equation}
Additionally to the functor \eqref{Cmt} we also get a map 
\begin{equation}\label{f}
(C_*^+(L),d+l)\rightarrow (Cyc_*^+(A)[1],d+m_{A}),
\end{equation}
where $L$ is the commutator $L_\infty$-algebra of $A$,
see equation 2.6 of \cite{GGHZ21}.
Both these maps are functorial with respect to $A_\infty$-functors.

\subsection{Cyclic Commutator $L_\infty$-Algebras}\label{ccla}
Given a cyclic $A_\infty$-category we would like to associate to it a cyclic $A_\infty$-algebra, generalizing the previous section.
The relatively strict notions we use mean that we restrict ourselves to essentially finite cyclic $A_\infty$-categories to do so.
\begin{cons}\label{assgn1}
    Let $\big((\mathcal{C},\langle\_,\_\rangle_{\mathcal{C}}),(Sk\mathcal{C},\langle\_,\_\rangle_{Sk\mathcal{C}})\big)$ be an essentially finite cyclic $A_\infty$-category.
    Then by the previous section ${A_{Sk\mathcal{C}}}$ is an $A_\infty$-algebra.
    Furthermore we define a symmetric non-degenerate chain map
    $$\langle\_,\_\rangle_{A_{Sk\mathcal{C}}}:\ A_{Sk\mathcal{C}}\otimes A_{Sk\mathcal{C}}\rightarrow k[-d]$$
    by 
    recalling that 
    $A_{Sk\mathcal{C}}=\bigoplus_{(\lambda,\mu)\in Ob(Sk\mathcal{C})^{\times2}}Hom_{Sk\mathcal{C}}(\lambda,\mu)$ and by linearly extending
 \begin{equation}\label{sodef}
    \langle x,y\rangle_{A_{Sk\mathcal{C}}}=
    \begin{cases}
      \langle x,y\rangle_{Sk\mathcal{C}}, & \text{if}\ x\in Hom_{Sk\mathcal{C}}(\lambda,\mu) \text{and}\ y \in Hom_{Sk\mathcal{C}}(\mu,\lambda)\ \text{for some}\  \mu,\lambda\in Sk\mathcal{C} \\
      0, & \text{otherwise.}
    \end{cases}
  \end{equation}
 Non-degeneracy follows by non-degeneracy of the original pairing and since the cardinality of $Ob(Sk\mathcal{C})$ is finite.
 By noting that 
 $$\langle m_n(\_,\cdots,\_),\_\rangle_{A_{Sk\mathcal{C}}}=\iota_\mathcal{C}^*(\langle m_n(\_,\cdots,\_),\_\rangle_{Sk\mathcal{C}})\in Cyc^*_+(A_{Sk\mathcal{C}}) $$
 we conclude that we indeed defined a cyclic $A_\infty$-algebra. 
\end{cons}
We explain how to associate a morphism of cyclic $A_\infty$-algebras to a strict cyclic functor of essentially finite $A_\infty$-categories:
\begin{cons}
 Given $F$ a strict cyclic $A_\infty$-functor of essentially finite $A_\infty$-categories
 \begin{equation}\label{haha}
        \begin{tikzcd}
   & \mathcal{C}\arrow{r}{F} &\mathcal{D}&\\
   Sk\mathcal{C}\arrow{ur}{\simeq}&  &&Sk\mathcal{D}\arrow{ul}[swap]{\simeq},
\end{tikzcd}
\end{equation}
we note that by defining $\widebar{\mathcal{D}}$ to be the full subcategory of the union of $Sk\mathcal{D}$ and the image of $Sk\mathcal{C}$ we obtain a commutative diagram

 \begin{equation}\label{hihi}
        \begin{tikzcd}
   & \mathcal{C}\arrow{rr}{F} &&\mathcal{D}&\\ 
   &&\widebar{\mathcal{D}}\arrow{ur} &&\\
   Sk\mathcal{C}\arrow{uur}{\simeq}\arrow{urr}&  &&&Sk\mathcal{D}\arrow{uul}[swap]{\simeq}\arrow{ull},
\end{tikzcd}
\end{equation}
 where all the arrows are strict cyclic $A_\infty$-functors.
 Further $\widebar{\mathcal{D}}$ also has only finitely many objects and the functor
 $$Sk\mathcal{D}\rightarrow \widebar{\mathcal{D}}$$
 is an equivalence.
 This is because the induced homology functor is full and faithful by definition.
 Further it is essentially surjective since every object of $\mathcal{D}$ is isomorphic to an object in $Sk\mathcal{D}$, thus in particular every object of $\widebar{\mathcal{D}}$ as a full subcategory of $\mathcal{D}$.
 Summarizing, given a strict cyclic functor of essentially finite cyclic $A_\infty$-categories as \eqref{hihi}, we obtain a roof
\begin{equation}\label{wbd}
    \begin{tikzcd}
&\widebar{\mathcal{D}} &\\
   Sk\mathcal{C}\arrow{ur}&  &Sk\mathcal{D}\arrow{ul}[swap]{\simeq},
\end{tikzcd}
\end{equation}
where all functors are strict and cyclic and injective on objects.
 \end{cons}
 Let us denote by $W$ the class of morphisms of cyclic $A_\infty$-algebras inducing isomorphisms
 on cyclic cohomology.\footnote{Since we are working over a characteristic zero field we could equivalently consider cyclic homology or Hochschild homology.}
 We find
\begin{lemma}\label{iknm}
    The assignment \ref{assgn1} extends to a functor
\begin{equation}\label{cyccomm}
(\tilde\_)_{fin}:\ (\text{e.f. cyc}
    _dA_\infty cat)_{st}\rightarrow (cyc_dA_\infty Alg)_{st}[W^{-1}]
     \end{equation}
     sending equivalences to isomorphisms.  
\end{lemma}
In particular, see \eqref{ssvr}, this functor is up to isomorphism independent from the choice of a skeleton.
\begin{proof}
 Given a strict cyclic functor of cyclic $A_\infty$-categories with finitely many objects which is additionally injective on objects we obtain
 a strict cyclic $A_\infty$-algebra morphism of the induced cyclic $A_\infty$-algebras (from construction \ref{assgn1}), as can be deduced easily.
 Furthermore given a functor which is an equivalence, we already saw that it induces an isomorphism on cyclic cohomology.
 We will see later in the proof of theorem \ref{LQT_conj} that the cyclic cohomology of an $A_\infty$-category and its associated cyclic $A_\infty$-algebra are isomorphic.
 Thus from diagram \eqref{wbd} it follows that the assignment on morphisms underlying \eqref{cyccomm} is well defined.
 It remains to verify functoriality.
 Given 
$$\begin{tikzcd}
    \mathcal{C}\arrow{r}{G}&\mathcal{D}\arrow{r}{H}&\mathcal{E}\\
     Sk\mathcal{C}\arrow{u}{\simeq}&Sk\mathcal{D}\arrow{u}{\simeq}&Sk\mathcal{E}\arrow{u}{\simeq}
\end{tikzcd}$$
strict cyclic functors of essentially finite cyclic $A_\infty$-categories. Then we need to compare the images of the diagram

\begin{equation*}
    \begin{tikzcd}
&\widetilde{\mathcal{E}} &\\
   Sk\mathcal{C}\arrow{ur}{G\circ H}&  &Sk\mathcal{E}\arrow{ul}[swap]{\simeq},
\end{tikzcd}\end{equation*}
and the diagram
\begin{equation*}
    \begin{tikzcd}
&\widebar{\mathcal{D}} &&\widebar{\mathcal{E}}&\\
   Sk\mathcal{C}\arrow{ur}&  &Sk\mathcal{D}\arrow{ul}{\simeq}\arrow{ur}&&Sk\mathcal{E}\arrow{ul}[swap]{\simeq},
\end{tikzcd}\end{equation*}
under the assignment \ref{cyccomm}.
Let us denote by $\mathcal{E}^*$ the full subcategory of $\mathcal{E}$ on the union of $Sk\mathcal{E}$ and the images of $Sk\mathcal{D}$ and $Sk\mathcal{C}$.
Then the images of both diagrams before are, in the localized category, equal to 
\begin{equation*}
    \begin{tikzcd}
&\mathcal{E}^* &\\
   Sk\mathcal{C}\arrow{ur}{G\circ H}&  &Sk\mathcal{E}\arrow{ul}[swap]{\simeq},
\end{tikzcd}\end{equation*}
 which proves that they are equal, which proves functoriality
\end{proof}
\begin{remark}
 Further, diagram \eqref{hihi} induces a natural quasi-isomorphism between the functor \ref{cyccomm}, composed with forgetting the cyclic structure,
 and the functor induced by sending an essentially finite $A_\infty$-category $(\mathcal{C},Sk\mathcal{C})$ to $\mathcal{C}$ and then applying the functor \ref{Commu}.
\end{remark}
Lastly, given an cyclic $A_\infty$-category $\mathcal{C}$ which is finitely presented in two ways
$(\mathcal{C},Sk\mathcal{C}_1)$, $(\mathcal{C},Sk\mathcal{C}_2)$ we get a roof of cyclic $A_\infty$-algebras
   \begin{equation}\label{ssvr}
   \begin{tikzcd}
&A_{\widebar{\mathcal{C}}}\arrow{dl}\arrow{dr}&\\
    A_{Sk\mathcal{C}_1}&&A_{Sk\mathcal{C}_2},
    \end{tikzcd}
    \end{equation}
   where the arrows are in $W$, which tells us that the choice of finite presentation in lemma \ref{cyccomm} does not  matter.\par
We will see a bit later in definition \ref{vfnh} that we can also associate to a quantizaiton of an essentially finite cyclic $A_\infty$-category a quantization of a cyclic $A_\infty$-algebra.
If we denote by $U$ the class of morphisms inducing an isomorphism on the BD algebras associated to a cyclic $A_\infty$-algebra, see definition \ref{F_q}, then we find
\begin{lemma}
 The assignment from lemme \ref{iknm} extends to a functor
\begin{equation}\label{qcyccomm}
(\tilde\_)_{fin}:\ (\text{q.e.f. cyc}
    _dA_\infty cat)_{st}\rightarrow (q. cyc_dA_\infty Alg)_{st}[U^{-1}]
    \end{equation}
 sending equivalences of quantum $A_\infty$-categories to isomorphisms.  
\end{lemma}
Further this functor is  independent up to isomorphism from the choice of a finite presentation and a compatible quantization.
\begin{proof}
    Indeed, we can argue in exactly the same way as in the previous lemma.
    The only additional argument that we need is that if a map of the BD algebras considered here induces a quasi-isomorphism when applying the dequantization functor, see definition \ref{nc_dq},
    then that the map of BD algebras itself is already a quasi-isomorphism.
    This follows by a spectral sequence argument associated to a suitable filtration, compare e.g. the arguments around \ref{qc12}.
\end{proof}
    We finish this subsection by recalling the standard functors
   \begin{equation}\label{cAL}
       (cyc_dA_\infty Alg)_{st}\rightarrow  (cyc_dL_\infty Alg)_{st}
   \end{equation}
   and
   \begin{equation}\label{qcAL}
   (q.cyc_dA_\infty Alg)_{st}\rightarrow  (q.cyc_dL_\infty Alg)_{st}
    \end{equation}
which descend to the localized versions if one the right hand side we invert morphisms inducing a quasi-isomorphism on the associated shifted Poisson algebras (see definition \ref{nattr1}),
the class of those denoted $V$, respectively on the associated shifted BD-algebras (see definition \ref{O_q}), the class of those denoted $P$.
This is a standard result and follows by functoriality eg. of proposition 4.5 of \cite{GGHZ21}:
We are in characteristic zero which means that we can chose a chain level map inducing an inverse in homology to a given quasi-isomorphism.  
\newpage
\section{Loday-Quillen-Tsygan Theorem}\label{lqtsection}
In this section we construct a Loday-Quillen-Tsygan map for $A_\infty$-categories and prove that it becomes a quasi-isomorphism in the large $N$-limit. We further study variants of this maps for cyclic $A_\infty$-categories and their quantizations, which preserves the shifted Poisson and BD algebra structures touched upon in earlier sections.

Let us begin by recalling the following simple facts: For that we denote by $M_N(k)$ $N\times N$-matrices with values in the ground field $k$. 
\begin{df}
Let $A$ be an $A_\infty$-algebra. One can define another $A_\infty$-algebra $M_NA$ as follows:\begin{itemize}
    \item 
Set $$  M_NA=M_N(k)\otimes A.$$ We define  structure maps by multi-linearly extending 
   $$m_n^{M_NA}(M_1u_1,\cdots ,M_nu_n):=(M_1\cdot\ldots\cdot M_n)\otimes m_n^A(u_1,\cdots,u_n).$$
    Using equation \ref{A_inf_re} it is straightforward to verify that this indeed defines an $A_\infty$-algebra.
  \item  
    Given $F:A_1\rightarrow A_2$ an $A_\infty$-morphism define a morphism
    $M_NF:M_NA_1\rightarrow M_NA_2$ by multi-linearly extending 
    $$(M_NF)_n(u_1M_1,\cdots u_nM_n):=(M_1\cdot\ldots\cdot M_N)F_n(u_1,\cdots,u_n),$$
    where $F_n$ are the structure maps of $F$.
    Using equation \ref{F_inf_rel} it is easy to verify that this defines an $A_\infty$-morphism.
    \end{itemize}
\end{df} 
%\begin{remark}
 %   We note that for $\mathcal{C}$ essentially finite we have that $M_N\mathcal{C}$ is essentially finite as well, where the necessary map/functor is provided by functoriality of $M_N$. Verifying that the provided map is fully faithful is straightforward and also essential surjectivity can be checked easily. 
%\end{remark}
\begin{remark}\label{cycM}
    As another comment that we will need in the next section we remark that given $(A,\langle\_,\_\rangle^{A})$ a cyclic $A_\infty$-algebra we have that $(M_NA,\langle\_,\_\rangle^{M_NA})$ is a cyclic $A_\infty$-algebra as well. Here we define 
    $$\langle\_,\_\rangle_{\mu\lambda}^{M_NA}:M_N A\otimes M_n A\rightarrow k[d],$$
    by 
    $$\langle M_1u_1,M_2u_2\rangle^{M_NA}=tr(M_1M_2)\langle u_1,u_2\rangle.$$
    Using the cyclic invariance of the trace it is easy to verify that this defines a cyclic $A_\infty$-algebra, ie condition \ref{cycl} is satisfied.
\end{remark}
\subsection{For $A_\infty$-Categories}
We recall the functor from section \ref{CLS} assigning to an $A_\infty$-category $\mathcal{C}$ an $A_\infty$-algebra $A_\mathcal{C}$, see equation \ref{sqc} for its explicit form.
\begin{df}\label{LQT_d}
Let $\mathcal{C}$ be an $A_\infty$-category. Then we define the map
$$\begin{tikzcd}
LQT_N: Sym( Cyc_+^*(\mathcal{C})[-1])\arrow{r}{\iota_\mathcal{C}^*}& Sym (Cyc_+^*(A_\mathcal{C})[-1])\arrow{r}{LQT_N}& C^*_+(\mathfrak{gl}_NA_\mathcal{C})
\end{tikzcd}$$
as the composition of the map induced by equation \ref{iota} and the standard LQT map applied to the $A_\infty$-algebra $A_\mathcal{C}$, see eg. section 2.4 of \cite{GGHZ21}. In particular $\mathfrak{gl}_NA_\mathcal{C}$ denotes the commutator $L_\infty$-algebra of $M_NA_\mathcal{C}$, the $A_\infty$-algebra of $A_\mathcal{C}$ with values in $N\times N$-matrices.
\end{df}
\begin{theorem}\label{LQT_c}
The LQT map from definition \ref{LQT_d} is a map of dg-algebras, that is we have 
$$LQT_N\circ(d+m_\mathcal{C})=(d+l_N)\circ LQT_N.$$ Further this map is functorial in $\mathcal{C}$ under $A_\infty$-functors.
 \end{theorem}
 \begin{proof}
The fact that this is a map of dg-algebras follows by the results around remark \ref{iota} and the standard results on the LQT map, see eg. 2.4 of \cite{GGHZ21}. Similarly it is functorial under $A_\infty$-functors implied by diagram \ref{iota_func} and the standard results on the LQT map, see eg again 2.4 of \cite{GGHZ21}.
\end{proof}
We have a commutative diagram of dg algebras     
$$\begin{tikzcd}
   Sym_+( Cyc_+^*(\mathcal{C})[-1])\arrow[dddrr,"LQT_N",swap]\arrow[dddr,"LQT_{N+1}", swap]\arrow[ddd]\arrow[dashed,dddrrr]&&&\\
   \\ \\
   (\cdots)\arrow[r,dashed]&C^*_+(\mathfrak{gl}_{N+1}A_\mathcal{C})\arrow{r}\arrow{r}&C^*_+(\mathfrak{gl}_NA_\mathcal{C})\arrow[dashed,r]&(\cdots)
\end{tikzcd}$$
Commutativity of the diagram follows from the same diagram for the one object LQT map, see equation 2.8 of \cite{GGHZ21}.
\begin{theorem}\label{LQT_conj}
     If $\mathcal{C}$ is a unital $A_\infty$-category then 
    $$Sym_+( Cyc_+^*(\mathcal{C})[-1])\simeq \varprojlim_{N\to\infty}  C^*_+(\mathfrak{gl}_NA_\mathcal{C}) $$

\end{theorem}
%\begin{theorem}[Definition: LQT Map]
    
%We denote by $$\begin{tikzcd}
%LQT_N: Sym( Cyc_+^*(\mathcal{C})[-1])\arrow{r}{tr^*}& Sym (Cyc_+^*(M_N\mathcal{C})[-1])\arrow{r}{f}& Sym (C^*_+(\mathfrak{gl}_NL_\mathcal{C}))\arrow{r}{mult.}& C^*_+(\mathfrak{gl}_NL_\mathcal{C})\end{tikzcd}$$ the composition of the map induced by trace map, the map induced by \ref{f} and lastly the multiplication map\footnote{which we for convenience set to be zero on $Sym^0(\cdots)$.}. By the results \ref{Mor} and \ref{f} 
%\end{theorem}
%$$\begin{tikzcd}
%A_\infty cat \arrow[rr, "\mathcal{F}^{cl}", bend left=25, ""{name=U, below}]
%\arrow[rd," com_N ", bend right=10]
%&&\textit{(d-2)-}PoissAlg
%\\
%&L\arrow[ur, "\mathcal{O}^{cl}", bend right=10]\arrow[Rightarrow, from=U, "LQT^{cl}"]&
%\end{tikzcd}$$

\begin{proof}
The one object LQT theorem applied to the unital $A_\infty$-algebra $A_\mathcal{C}$ implies that the second map in the definition of the LQT map for categories is a quasi-isomorphism for $N\rightarrow \infty$ (see thm 2.11 of \cite{GGHZ21} for the LQT theorem for $A_\infty$-algebras). Thus it remains to show that $\iota_\mathcal{C}^*$ is a quasi-isomorphism for unital categories. Since we are in characteristic zero latter statement is equivalent to showing that the map \ref{iota} on cyclic chains induces a quasi-isomorphism. Note that this map is induced from a map on Hochschild chains $$\begin{tikzcd}\label{qwert}
\iota: CH_*(A_\mathcal{C})\rightarrow CH_*(\mathcal{C})\end{tikzcd}$$ and by corollary 2.2.3 of \cite{Lo13} the previous statement is equivalent\footnote{To be precise corollary 2.2.3 is only valid for maps on Hochschild chains induced by algebra maps, which $\iota$ is not. However it is straightforward to verify that $\iota$ still satisfies the required functoriality properties to make the argument there work.} to showing that this map \ref{qwert} on Hochschild chains is a quasi-isomorphism. Indeed, this seems to be known to experts under the slogan that both composable and non-composable Hochschild chains of a category compute its Hochschild homology. One way to prove this should be theorem 1.2.13 of \cite{Lo13} applied to $A_\mathcal{C}$ and the separable algebra $E$ spanned as a k-vector space by the original units of $\mathcal{C}$. As this statement is only provided for actual algebras we give a short alternative proof of above slogan, basically using similar ideas: We first note that we have a decomposition of chain complexes
$$\big(CH_*(A_\mathcal{C}),d_{Hoch}\big)=(ker(\iota),\partial)\oplus (ker(\iota)^\bot,\partial),$$
where $ker(\iota)^\bot$ denotes composable Hochschild chains, that is linear combinations of tensors with matching boundary conditions. Since $\iota$ is surjective it suffices to show that $H_*(ker(\iota),\partial)=0$ to prove the statement. Thus let $x\in ker(\iota)$ st. $\partial x=0$. Denote by $V_{\lambda\mu}^k$ the subspace of Hochschild chains of $A_\mathcal{C}$ such that at the k-th position for the first time the boundary conditions do not match and that those are $\lambda\neq\mu\in Ob(\mathcal{C})$. As a k-vector spaces we have
$$ker(\iota)=\oplus_{k\in \mathbb{N}}\bigoplus_{\lambda\neq\mu}V_{\lambda\mu}^k.$$

Thus we deduce that also 
$
\partial x|_{V_{\lambda\mu}^k}=0.$
Let us denote by $ e_\mu\otimes^k:V_{\lambda\mu}^k\rightarrow V_{\lambda\mu}^k$ the map given by inserting the unit of $Hom_\mathcal{C}(\mu,\mu)$ into the k-th position. We then have also 
\begin{equation}\label{123}e_\mu\otimes^k\partial x|_{V_{\lambda\mu}^k}=0.
\end{equation}
Consider $$\Tilde{x}:=\sum_k\sum_{\lambda\neq\mu}e_\mu\otimes^k x|_{V_{\lambda\mu}^k}.$$ We then have 
$$\partial \Tilde{x}=\sum_k\sum_{\lambda\neq\mu}e_\mu\otimes^k\partial x|_{V_{\lambda\mu}^k}+x.$$
Indeed, the summands of the Hochschild differential\footnote{See eg. equation 7 of \cite{She19} for the formula.} always give zero when a unit is involved, except for the $m^2$ term, which produces the second summand in above equality when applied to the inserted unit and the tensor factor to its right. Further the signs work out since $|e_\mu|=0.$ By equation \ref{123} we conclude
$$\partial \Tilde{x}=x,$$
which proves that $H_*(ker(\iota),\partial)=0$ and thus $$HH_*(A_\mathcal{C})=HH_*(\mathcal{C}).$$ \end{proof}
\subsection{The Quantized LQT Map for Cyclic $A_\infty$-Categories}\label{slqtq}
In this section we explain that the LQT map connects the `free and classical observables' from the commutative and non-commutative world that we saw in the previous sections. Further, slight modifications of the LQT map also connect the `prequantum and quantum observables' from the previous sections.

To start, given $V_B$ a collection of cyclic chain complexes of degree d where $B$ is finite and denoting $\mathfrak{gl}_NV_B$ its associated commutator cyclic chain complex we have
\begin{df}
Denote by $$LQT^{fr}_N:\ \mathcal{F}^{fr}(V_B)\rightarrow \mathcal{O}^{fr}(\mathfrak{gl}_NV_B)$$ the LQT map from definition \ref{LQT_d}, given as the composition of the original LQT map applied to $\widetilde{(V_B)}$, the cyclic chain complex of $V_B$ under the functor \ref{cyccomm}.
\end{df}
\begin{theorem}\label{LQT_f} Given $V_B$ a collection of cyclic chain complexes of degree d where $B$ is finite
    $$LQT^{fr}_N:\ \mathcal{F}^{fr}(V_B)\rightarrow \mathcal{O}^{fr}(\mathfrak{gl}_NV_B),$$ 
     is a map of (d-2) shifted Poisson dg algebras, functorial with respect to strict cyclic $A_\infty$-morphisms. 
\end{theorem}

\begin{proof}
 This follows directly from corollary \ref{ibc} further below and  thm. 4.7 of \cite{GGHZ21}, the previously studied LQT map for cyclic $A_\infty$-algebras.  
\end{proof}
\begin{cons}
We want to define a map, which we call the pre-quantum LQT map, as follows
\begin{equation}
\label{LQT_r}
LQT^{pq}_N:\ Sym\big({Cyc}^{*+d-4}(V_B)\big)\llbracket\gamma\rrbracket[3-d]\rightarrow C^*(\mathfrak{gl}_NV_B)\llbracket\hbar\rrbracket
\end{equation} 
$$
(a_{1_1}\cdots a_{n_1})\dots (a_{1_m}\cdots a_{n_m}) \nu_{\lambda_1}^{b_1}\dots \nu_{\lambda_r}^{b_r} \gamma^s\mapsto Tr\big(a_{1_1}(\_)\cdots a_{n_1}(\_)\big)\dots Tr\big(a_{1_m}(\_)\cdots a_{n_m}(\_)\big)\hbar^{m+b-1+2s}N^{b} 
$$
Here $\lambda_i\in B$ and $b=\sum^r_{i=1}b_i\in \mathbb{N}$, further we understand the right hand side as polynomials on $\mathfrak{gl}_NV_B[1]$, adjoined $\hbar$. 
\end{cons}
\begin{remark}Note that this assignment does not define a map of algebras: We have that $$LQT^{pq}_N(\nu^2)=\hbar N^2\ \  \text{but}\ \ LQT^{pq}_N(\nu)\cdot LQT^{pq}_N(\nu)= N^2.$$
However, we will see that it is a weighted map of BD algebras\footnote{to be precise it is a weighted BD map on $Sym^{>0}$, but we ignore this subtlety.} in the sense of definition \ref{weimap}.
\end{remark}
More formally, we can define the $LQT^{pq}_N$ as the composition of various maps:

$$LQT^{pq}_N:=m^{pq}\circ f^{pq}\circ tr^{pq}\circ Sym(\iota_\mathcal{C}^*),$$
We introduce these in the following, but first state the central, though immediate result:
\begin{lemma}\label{asgh}
For $\mathcal{C}$ an essentially finite cyclic $A_\infty$-category the map
    $${\iota_\mathcal{C}}^*:(Cyc^*(Sk\mathcal{C}),d+m_{Sk\mathcal{C}})\rightarrow (Cyc^*(A_{Sk\mathcal{C}}),d+m_{A_{Sk\mathcal{C}}}),$$
    induced by \ref{iota} and where we send $\nu_\lambda\mapsto \nu$ for all $\lambda\in Ob(Sk\mathcal{C}),$
    is a quasi-isomorphism of involutive shifted Lie bialgebras.
\end{lemma}
\begin{remark}Note that here it is essential that the cardinality of $ObSk\mathcal{C}$ is finite.
\end{remark}
\begin{proof}
We saw already in the proof of \ref{LQT_conj} that $\iota_\mathcal{C}^*$ is a quasi-isomorphism. Since $\iota_\mathcal{C}^*$ is just the inclusion of composable cyclic words into general cyclic words it follows immediately that it is a map of involutive shifted Lie bialgebras, recalling how we defined the cyclic structure on $A_{Sk\mathcal{C}}$, see definition \ref{sodef}.    
\end{proof}
\begin{corollary}\label{ibc}
   For $\mathcal{C}$ an essentially finite cyclic $A_\infty$-category we find a quasi-isomorphism of (d-2) shifted Poisson algebras
   $$\begin{tikzcd}
    \mathcal{F}^{cl}(\mathcal{C})\arrow{r}{\simeq}&\mathcal{F}^{cl}(Sk\mathcal{C})\arrow{r}{\iota^*}&\mathcal{F}^{cl}(A_{Sk\mathcal{C}}),\end{tikzcd}$$
   induced by the map from lemma \ref{asgh} and the map \ref{fff}.
\end{corollary}
\begin{df}\label{vfnh}
Given a quantization $I^q$ of an essentially finite cyclic $A_\infty$-category $\mathcal{C}$ denote by $(A_{Sk\mathcal{C}},I^q_f)$ the quantization of the cyclic $A_\infty$-algebra under \ref{cyccomm}, determined by $I^q_f:=\iota^*I^q,$ where $\iota^*$ is the map from corollary \ref{ibc}.
\end{df}
\begin{corollary}\label{qc12}
  For $\mathcal{C}$ a quantization of an essentially finite cyclic $A_\infty$-category we have a quasi-isomorphism of (d-2) twisted BD algebras
   $$\begin{tikzcd}
    \mathcal{F}^{q}(\mathcal{C})\arrow{r}{\simeq}&\mathcal{F}^{q}(Sk\mathcal{C})\arrow{r}{\iota^*}&\mathcal{F}^{q}(A_{Sk\mathcal{C}}),\end{tikzcd}$$
   induced by the map from lemma \ref{asgh} and the \ref{BDqi}.
\end{corollary}
\begin{proof}
Lemma \ref{ibc} implies that we have a map of twisted BD algebras. It remains to argue why this map is a quasi-isomorphism. This follows by considering the spectral sequence associated to the filtration by powers of $\gamma$, $\nu$ and length of symmetric words, which is complete and exhaustive (see example 6.3 of \cite{GGHZ21}). Note that here it is essential that we consider formal power series in $\gamma$. The map from this corollary induces an isomorphism on the first page of the spectral sequence, eg by lemma \ref{asgh}. By the Eilenberg-Moore comparison theorem, eg theorem 5.5.11 of \cite{wei94}, the result follows.   
\end{proof}
Let us introduce the other maps in defining the prequantum LQT map, which are just slight variants of the ones in \cite{GGHZ21}, taking care of the shifts needed for the $\mathbb{Z}$-grading we included:
%\subsubsection{Trace Map}
%Given $V_B$ a collection of cyclic chain complexes degree d, B not necessarily finite, so is
%$$M_NV_B:=\{M_NV_{ij}\}_{i,j\in B}$$ by remark \ref{cycM}.
\begin{itemize}
\item
We recall the trace map from \cite{GGHZ21} section 4.2, inducing for every cyclic chain complex $V$ a map of chain complexes
%$$tr^{fr}:\ (Sym\big(Cyc^{*-1}_+(V_B)\big),d)\rightarrow (Sym\big(Cyc^{*-1}_+(M_NV_B)\big),d)$$
%and also 
$$ (Cyc^*(V),d)\rightarrow (Cyc^*(M_NV),d),$$
where we further sent $\nu$ to $N\nu$. By suitably shifting and extending $\gamma$-linearly this further induces a map denoted
\begin{equation}\label{tracemap}
tr^{pq}:\ Sym\big(Cyc^{*+d-4}(V)\big)\llbracket\gamma\rrbracket[3-d]\rightarrow Sym\big(Cyc^{*+d-4}(M_NV)\big)\llbracket\gamma\rrbracket[3-d].\end{equation} 

   \item 
%As before the map \ref{f} induces
%$$f^{fr}:\ (Sym\big(Cyc^{*-1}_+(M_NV_B)\big),d)\rightarrow \big(Sym\big(C^*_+(\mathfrak{gl}_NV_B)\big),d\big).$$
We recall the map from our remark \ref{f} which allows us to define \begin{equation*}
\big(Cyc^*(M_NV)[-1],d\big)\rightarrow \big(C^*(\mathfrak{gl}_NV),d\big),\end{equation*}
by further sending $\nu\mapsto 1$.  Shifting both sides by $(d-3)$, applying the symmetric algebra functor, extending $\gamma$-linearly and then shifting by $(3-d)$ this induces a map of chain complexes
$$f^{pq}:\ (Sym\big(Cyc^{*+d-4}(M_NV)\big)\llbracket\gamma\rrbracket[3-d],d)\rightarrow \big(Sym\big(C^*(\mathfrak{gl}_NV)[d-3]\big)\llbracket\gamma\rrbracket[3-d],d\big).$$
 \item

There is also a map of degree zero
$$ \big(Sym\big(C^{*}(\mathfrak{gl}_NV)[d-3])\llbracket\gamma\rrbracket,d\big)\rightarrow \big(C^*(\mathfrak{gl}_NV)[d-3]\llbracket\hbar\rrbracket,d\big).
$$ First it sends $\gamma$ to $\hbar^2$, it applies multiplication, ie. sends a symmetric word in n letters of the algebra $C^*(\mathfrak{gl}_NV)[d-3]$ whose multiplication has degree $(d-3)$ to the product of these words, but multiplied by $\hbar^{n-1}$. Note that this map thus has degree 0. Again shifting by $(3-d)$ induces the map, denoted
\end{itemize}
$$m^{pq}: \big(Sym\big(C^*(\mathfrak{gl}_NV)[d-3])\big)\llbracket\gamma\rrbracket[3-d],d\big)\rightarrow \big(C^*(\mathfrak{gl}_NV)[d-3]\llbracket\hbar\rrbracket[3-d],d\big)
\cong \big(C^*(\mathfrak{gl}_NV)\llbracket\hbar\rrbracket,d\big).$$

These maps preserves the underlying algebraic structure as follows:
\begin{lemma}\label{tr}
For every cyclic chain complex $V$ of degree d
$$m^{pq}\circ f^{pq}\circ tr^{pq}:\mathcal{F}^{pq}(V)\rightarrow \mathcal{O}^{pq}(\mathfrak{gl}_NV)$$ is a weighted map of (d-2) twisted BD algebras and which is functorial with respect to strict cyclic morphisms. 
\end{lemma}
\begin{proof}
    This follows in exactly the same way as in \cite{GGHZ21}, noting that the shifts made do not change the argument in any way. More precisely the trace map is a map of (d-2) twisted BD algebras by lemma 4.6 of loc.cit. The composition of the other two maps are a weighted map of (d-2) twisted BD algebras by prop. 4.5 of loc.cit. To be precise it is a 2-weighted map, recalling the definition \ref{weimap} of a weighted map of BD algebras and examining the definition of the multiplication map above.
\end{proof}
\begin{df}
For $V_B$ a collection of cyclic chain complexes of degree d and $B$ finite denote 
$$LQT^{pq}_N:=m^{pq}\circ f^{pq}\circ tr^{pq}\circ Sym\iota^*:\ \mathcal{F}^{pq}(V_B)\rightarrow \mathcal{O}^{pq}(\mathfrak{gl}_NV_B).$$

\end{df}
It is a simple check to verify that this coincides with the explicit definition \ref{LQT_r}.
\begin{theorem}\label{LQT_pq}
Given $V_B$ a collection of cyclic chain complexes of degree d and $B$ finite the map
    $$LQT^{pq}_N:\ \mathcal{F}^{pq}(V_B)\rightarrow \mathcal{O}^{pq}(\mathfrak{gl}_NV_B),$$ 
    is a weighted map of (d-2) twisted BD algebras, functorial with respect to strict cyclic $A_\infty$-morphisms of essentially finite collections of cyclic cochain complexes.
\end{theorem}
\begin{proof}
Indeed this follows directly from corollary \ref{qc12} and lemma \ref{tr}. 
\end{proof}
Furthermore given $B$ finite and a collection of cyclic chain complexes $V_B$  the prequantum and free LQT map fit with the dequantization maps together in following commutative diagram
\begin{equation}\label{pqdfr}
\begin{tikzcd}
\mathcal{F}^{pq}(V_B)\arrow{r}{ LQT^{pq}_N}\arrow{d}{p}& \mathcal{O}^{pq}(\mathfrak{gl}_NV_B)\arrow{d}{p}\\
\mathcal{F}^{fr}(V_B)\arrow{r}{ LQT^{fr}_N}& \mathcal{O}^{fr}(\mathfrak{gl}_NV_B), \end{tikzcd}\end{equation}
which is a straightforward check.
As a corollary of theorems \ref{LQT_f} and \ref{LQT_pq} we have following two theorems:

\begin{theorem}\label{LQT_cl}
    Given an essentially finite cyclic $A_\infty$-category we denote by $\mathfrak{gl}_NA_{Sk\mathcal{C}}$ the associated cyclic $L_\infty$-algebra. The LQT map is a map of (d-2)-shifted Poisson algebras
    $$LQT^{cl}: F^{cl}(\mathcal{C})\rightarrow \mathcal{O}^{cl}(\mathfrak{gl}_NA_{Sk\mathcal{C}}).$$
    Further this is functorial with respect to  strict cyclic $A_\infty$-morphisms of essentially finite  $A_\infty$-categories. That is, it defines a natural transformation between the functor
    $$ (\text{e.f cyc}
    _dA_\infty cat)_{st}\rightarrow (\mathcal{D}Poiss_{(d-2)})^{op},$$
     given by definition
    \ref{nattr1} and the composite functor  $$ (\text{e.f cyc}
    _dA_\infty cat)_{st}\rightarrow(cyc_dA_\infty Alg)_{st}[W^{-1}]\rightarrow(cyc_dA_\infty Alg)_{st}[W^{-1}]\rightarrow\cdots$$
    $$\cdots\rightarrow(cyc_dL_\infty Alg)_{st}[V^{-1}] \rightarrow (\mathcal{D}Poiss_{(d-2)})^{op},$$
    determined by
    \eqref{cyccomm}, the assignment \ref{cycM}\footnote{The fact that the functor defined by this assignment descends to the localized versions follows from Morita invariance of cyclic cohomology, ie. thm. 2.10 of \cite{GGHZ21}.} and \ref{cAL}\footnote{See there for the definition of V.} and \ref{nattr2}. These functors and the natural transformation are in fact independent up to isomorphism from the choice of a skeleton.
\end{theorem}
Because of this independence we sometimes omit the notation for the choice of skeleton, which is what we also did for the domain of the $LQT^{cl}$ map.
\begin{proof}
    An essentially finite cyclic $A_\infty$-category has underlying  collection of cyclic chain complexes over a finite set $B$. Thus we have the free LQT map $$LQT^{fr}_N:\ \mathcal{F}^{fr}(V_B)\rightarrow \mathcal{O}^{fr}(\mathfrak{gl}_NV_B).$$ The cyclic $A_\infty$-category structure induces a Maurer-Cartan element in $\mathcal{F}^{fr}(V_B)$, further the commutator cyclic $L_\infty$-algebra structure induces a Maurer-Cartan element in $\mathcal{O}^{fr}(\mathfrak{gl}_NV_B)$ and $LQT^{fr}$ maps the prior one to the latter one. Twisting the $LQT^{fr}$ map by these compatible MCE is again a map of (d-2)-shifted Poisson algebras by general properties of the twisting procedure. When forgetting the shifted Poisson structures this map reduces to the standard LQT map applied to the underlying $A_\infty$-category, by construction.

    The functoriality statement follows by functoriality of the one object LQT map and functoriality of the map induced by \ref{iota}.

    The independence up to isomorphism for the choice of a skeleton follows for the first functor from \ref{undf}. The independence up to isomorphism for the choice of a skeleton for the second functor is induced from \ref{cyccomm}.
\end{proof}
\begin{theorem}\label{LQT_q}
   i) A quantization of an essentially finite cyclic $A_\infty$-category induces a quantization of the associated cyclic $L_\infty$-algebra with values in $N\times N$-matrices, for each $N\in\mathbb{N}$. There is a weighted map of (d-2) twisted BD algebras
    $$LQT^{q}: \mathcal{F}^{q}(\mathcal{C})\rightarrow \mathcal{O}^{q}(\mathfrak{gl}_NA_{Sk\mathcal{C}}).$$
  ii)  Further this is functorial with respect to strict cyclic $A_\infty$-morphisms of quantum essentially finite  $A_\infty$-categories in the sense that it defines a natural transformation between the functor
    $$ (\text{q.e.f. cyc}_dA_\infty cat)_{st}\rightarrow (\mathcal{D}\text{BD}^{(d-2)tw})^{op},$$ 
     given by definition
    \ref{F_q} and the composite functor  $$ (\text{q.e.f cyc}
    _dA_\infty cat)_{st}\rightarrow(q.cyc_dA_\infty Alg)_{st}[U^{-1}]\rightarrow(q.cyc_dA_\infty Alg)_{st}[U^{-1}]\rightarrow\cdots$$
$$\cdots\rightarrow (q. cyc_dL_\infty Alg)_{st}[P^{-1}] \rightarrow (\mathcal{D}BD^{(d-2)tw})^{op}$$
     determined by \eqref{qcyccomm}, by \ref{cycM}\footnote{The fact that this functor descend to quantized version follows from the trace map of lemma \ref{tr}. Further that this functor descends to the localized version follows from Morita invariance of cyclic cohomology, theorem 2.10 of \cite{GGHZ21} and the filtration spectral sequence, see proof of part iii) of  theorem \ref{LQT_q}.} and by  \ref{qcAL}\footnote{See there for the definition of P.} and \eqref{O_q}. These functors and the natural transformation are in fact independent from the choice of a skeleton, up to isomorphism.

  iii)    We have that 
    
\begin{equation}\label{lNq}
\mathcal{F}^{q}(\mathcal{C})\simeq \varprojlim_{N\to\infty}\mathcal{O}^{q}(\mathfrak{gl}_NA_{Sk\mathcal{C}})    .\end{equation}
\end{theorem}
Because of the independence from the choice of skeleton we sometimes omit the notation for the choice of such, which is what we also did for the domain of the $LQT^{q}$ map.
\begin{proof}
A quantization of an essentially finite cyclic $A_\infty$-category has underlying collection of cyclic chain complexes over a finite set $B$. Thus we have the prequantum LQT map $$LQT^{pq}_N:\ \mathcal{F}^{pq}(V_B)\rightarrow \mathcal{O}^{pq}(\mathfrak{gl}_NV_B).$$
The datum of the quantization induces a Maurer-Cartan element in $\mathcal{F}^{pq}(V_B)$. Since the prequantum LQT map is a weighted map of BD algebras it sends a Maurer-Cartan element to a Maurer Cartan element. By diagram \ref{pqdfr} it follows that the so determined Maurer-Cartan element is a quantization of the cyclic commutator $L_\infty$-algebra. Twisting the $LQT^{pq}$ map by these compatible MCE is again a weighted map of (d-2) twisted BD algebras, which we denote by $LQT^q$. 

The independence up to isomorphism for the choice of a skeleton follows for the first functor from \ref{undfq}. The independence up to isomorphism for the choice of a skeleton for the second functor is induced from \ref{qcyccomm}.

The large $N$-statement follows again by considering the compete and exhaustive filtration from the proof of corollary \ref{qc12} and the one given by powers of $\hbar$. The LQT maps at each level $N$, which are filtered, see diagram 7.1 of \cite{GGHZ21}, induce the map \ref{lNq}. On the first page of the spectral sequence associated to the filtrations this map induces a quasi-isomorphism, by \ref{LQT_conj}. By the Eilenberg-Moore comparison theorem (theorem 5.5.11 of \cite{wei94}) the claim follows.
\end{proof}

\newpage

\section{Various Graph Complexes}
In this section we introduce various graph complexes, equip those with shifted Poisson and Beilinson-Drinfeld structures, we explain how they are connected and we identify certain Maurer-Cartan elements in these algebras. 
 \subsection{Ribbon Graph Complexes}\label{rbgs}
We begin by introducing stable ribbon graphs, originally introduced in \cite{Kon92a}, their simpler cousins ribbon graphs and ribbon trees and by giving names to some of their features. 
Afterwards we equip complexes built from these graphs with shifted Poisson and Beilinson-Drinfeld structures, induced by gluing leafs. 
Lastly we identify certain Maurer-Cartan elements for these structures.

\begin{df}
A \emph{stable ribbon graph} $\Gamma$ is given by the data of
\begin{itemize}
\item a finite set $H$, called the set of half edges,
\item $V(\Gamma)$ a partition  of $H$, called the set of vertices. We denote the cardinality of a vertex $v$ by $|v|$ and the total number of vertices by $v(\Gamma)$.
\item an involution $\sigma_1$ acting on $H$ whose fixed points are called leafs, the set of those denoted $L(\Gamma)$ and whose 2-cycles are called internal edges, the set of those denoted $E(\Gamma)$.
We further denote the total number of internal edges by $e(\Gamma)$ and the total number of leafs by $l(\Gamma)$.
\item for each vertex $v$, a partition $C(v)$ of its half edges. We denote the  cardinality of such a partition at a vertex $v$ by $c(v)$ and by $c(\Gamma)$ the sum of those over all vertices.
\item a cyclic ordering for the half edges in each such partition,
\item a function $(g,b):V(\Gamma)\rightarrow \mathbb{N}_{\geq 0}\times \mathbb{N}_{\geq 0}$ called the genus and boundary defect
such that for any vertex $v$ if $g(v)=b(v)=0$ and $c(v)=1$ we have $|v|\geq 3$.
\end{itemize}

\end{df}
\begin{df}\label{faces}
Given a stable ribbon graph $\Gamma$ denote $\sigma_0: H\rightarrow H$ the permutation which defines the cyclic ordering of the elements in the partitions $C(v)$ at each vertex $v$.
Furthermore recall the permutation $\sigma_1: H \rightarrow H$, whose cycles are the interior edges respectively leaves.
We define $$\sigma_b:=\sigma_0^{-1}\circ\sigma_1.$$ The cycles of $\sigma_b$ are called the \emph{faces} of $\Gamma$, we denote the number of faces of a graph by $f(\Gamma)$.
    \end{df}
 Following remark 4.9 of \cite{Ha07} we make   
    \begin{df}\label{armg}
 The \emph{arithmetic genus} of a stable ribbon graph is defined by 
 $$g(\Gamma)=1-v(\Gamma)+\frac{1}{2}(e(\Gamma)+c(\Gamma)-f(\Gamma))+\sum_{v\in V(\Gamma)}g_v$$
    \end{df}  
We call the number of leafs in a given cycle defining a face the number of open boundaries in that face.

        A \emph{ribbon graph} is a stable ribbon graph such that $g(v)=b(v)=0$ for all vertices and thus $c(v)$=1 for all vertices.
        That is, we  recover graphs, possibly with leaves, together with a cyclic ordering on the half edges at each vertex. 
   
    \begin{remark}
        Note that for ribbon graphs we have the standard $2-2g(\Gamma)=
        v(\Gamma)-e(\Gamma)+f(\Gamma)).$
    \end{remark}
    \begin{ex}
  The ribbon graph given by the figure $8$ has 4 half edges, 1 vertex,  and 3 faces.     
    \end{ex}
    
        A \emph{ribbon tree} is a ribbon graph that has $f(\Gamma)=1$ and $g(\Gamma)=0$.
    
    \begin{df}\label{fBd}
Note that the half edges belonging to a given face of a stable ribbon graph $\Gamma$ which are leafs inherit a cyclic ordering.
We call a tuple of consecutive leafs in that cyclic ordering $(l_{i},l_{i+1})$ a \emph{free boundary} of that face.
Denote by $\partial(\Gamma)$ the set of free boundaries of a given stable ribbon graph. 
    \end{df}
\begin{ex}
    The ribbon tree given by the figure $+$ has four free boundaries.
\end{ex}
\begin{df}\label{Dec}
Let $B$ be a set, which we call the set of boundary conditions.
 A $B$-colored stable ribbon graph is a stable ribbon graph $\Gamma$ together with a map $c:\partial(\Gamma)\rightarrow B$. Furthermore we actually ask for a map $(g,b):V(H)\rightarrow \mathbb{N}_{\geq 0}\times \mathbb{N}_{\geq 0}^B$, that is a boundary defect $b_\lambda(v)\in\mathbb{N}_{\geq 0}$ for each element $\lambda$ of $B$ at each vertex $v$.
\end{df}  

        Denote the set of faces with zero open boundaries by $F_0(\Gamma)$ and its cardinality by $f_0(\Gamma).$ We call 
        $$f_{tot}(\Gamma):=f_0+(\Gamma)\sum_{\lambda\in B,v\in V(\Gamma)}b_\lambda(v)$$
        the total number of empty faces.

    An orientation of a (B-clolored) stable ribbon graph is an ordering of its internal edges.
   
\begin{remark}
    In the literature this may be better known as a twisted orientation.
    An orientation traditionally refers to an ordering of the edges and vertices.
    Since we only consider the latter we stick to the simpler name of orientation and hope that it does not cause confusion.
\end{remark}

    Two oriented B-colored stable ribbon graphs are called isomorphic if there is a bijection between their set of half edges that respects all other structures.

    The notions of a $B$-coloring, orientation and isomorphism extend to ribbon graphs and ribbon trees.

We define various chain complexes using the previous definitions. For that we introduce the following rule - slightly generalizing definition 4.3 of \cite{Ha07} to B-colored stable ribbon graphs, which we will use to define a differential.
\begin{df}[Contracting an Edge] \label{cont}
Let $\Gamma$ be an oriented $B$-colored stable ribbon graph and let $e\in E(\Gamma)$ be an edge.
We define the oriented B-colored stable ribbon graph $\Gamma/e$ to be the graph obtained by contracting this edge according to the following rules:
\begin{enumerate}
\item \label{item_contractedge}
Suppose that $e$ is not a loop, that is it joins distinct vertices $v_1,v_2\in V(\Gamma)$.
The half-edges of these vertices are partitioned into cycles, so the endpoints of $e$ lie in distinct cycles $c_1\subset v_1$ and $c_2\subset v_2$.
By contracting the edge $e$, the vertices $v_1$ and $v_2$ become joined, and the cycles $c_1$ and $c_2$ coalesce to form a new cycle with a naturally defined cyclic ordering.
The genus and boundary defects for the vertices $v_1$ and $v_2$ are added to give the defects for the new vertex created by joining $v_1$ and $v_2$.
The orientation is defined in an obvious way.
All the other combinatorial structures elsewhere on the graph are left alone.

Note that when both $c_1$ and $c_2$ each consist of a single half-edge, they define a face with one free boundary on it, which by definition is colored by an element of $B$.
In this case $c_1$ and $c_2$ do not coalesce when contracting the relevant edge, but instead are deleted.
The boundary defect at the new vertex is defined to be the sum of the boundary defects of $v_1$ and $v_2$ plus one, with respect to the element by which the face was colored.
If, furthermore, $c_1$ and $c_2$ are the only cycles of $v_1$ and $v_2$, then the edge $e$ cannot be contracted.

\item \label{item_contractloop1}
Now suppose that $e$ is a loop, in which case both its endpoints lie in a single vertex $v$.
Suppose furthermore, that they join distinct cycles $c_1,c_2\subset v$.
By contracting $e$, these cycles coalesce to form a single cycle as before.
In so doing the genus defect of $v$ increases by one.
No other combinatorial structures are changed.

As before, care must be taken when both $c_1$ and $c_2$ consist of a single half-edge.
Again these half-edges together define a face with one free boundary on it, which is by definition colored by an element of $B$.
In this case $c_1$ and $c_2$ are annihilated and \emph{both} the genus \emph{and} the boundary defect corresponding to that element are increased by one.

\item \label{item_contractloop2}
Finally, suppose that $e$ is again a loop, but that now both of its endpoints lie in the same cycle $c$ contained in some vertex $v$.
By contracting this loop the cycle $c$ splits up into two cycles $c_1$ and $c_2$, with naturally defined cyclic orderings.
All the other combinatorial structures remain unchanged.

Again, care must be taken with this definition when the endpoints of $e$ lie next to each other in the cyclic ordering- in this case they again define at least one face with one free boundary on it, which is colored by an element of $B$.
In this case, the cycle $c$ does not split up, but the boundary defect corresponding to that element is increased by one.
Furthermore, if the cycle $c$ consists of just the two half-edges of $e$ (and thus actually defines two faces, colored by an element each), then the cycle $c$ is annihilated and the boundary defect actually increases by one for each of these elements.
Finally, if in the lattermost situation the vertex has no other cycles than $c$, then the loop $e$ actually cannot be contracted at all.
\end{enumerate}
\end{df}
\begin{remark}\label{totfac}
    Note that the property being a ribbon tree is preserved under this differential, whereas being a ribbon graph is not.
    The differential preserves the arithmetic genus of a stable ribbon graph.
    It further preserves the 'total number of empty faces' $f_{tot}(\Gamma)$ of a stable ribbon graph, in sense of definition \ref{faces}.
\end{remark}

\begin{df}\label{Ddrgc}
Denote by $(s\mathcal{RG}_{c,\bullet}^B,\partial)$ the complex whose underlying vector space is freely generated by
isomorphism classes of connected B-colored oriented stable ribbon graphs, modulo the relation that reversing the orientation on a stable ribbon graph is equivalent to multiplying by $(-1)$.
This complex is graded by the number of internal edges.
The differential $\partial$ is given by summing over all possible contractions of the edges:
\[ \partial(\Gamma):=\sum_{e\in E(\Gamma)} \Gamma/e.\]
Note that some edges cannot be contracted, in which case the corresponding term in the sum is defined to be zero.

  Denote by $(\mathcal{RG}_{c,\bullet}^B,\partial)$ the complex given by the subspace (which is not sub-complex) of connected ribbon graphs, which uniquely becomes a chain complex by requiring that the projection map
  $$s\mathcal{RG}_{c,\bullet}^B\rightarrow \mathcal{RG}_{c,\bullet}^B$$
  is a map of chain complexes.

    Denote by $(\mathcal{RT}_{c,\bullet}^B,\partial)$ the sub-complex of $s\mathcal{RG}_{c,\bullet}^B$ given by ribbon trees.
\end{df}
We could have equivalently defined $\mathcal{RT}_{c,\bullet}^B$ as the sub-complex of $\mathcal{RG}_{c,\bullet}^B$
given by ribbon trees, which can be verified directly.
That is, following diagram of chain complexes commutes
   \begin{equation}\label{inc_proj}
   \begin{tikzcd}
(\mathcal{RT}_{c,\bullet}^B,\partial)\arrow[bend right=25]{rr}\arrow{r}&(\mathcal{RG}_{c,\bullet}^B,\partial)&
(s\mathcal{RG}_{c,\bullet}^B,\partial)\arrow{l}.
    \end{tikzcd}
    \end{equation}
Let us now further fix an odd integer $d$ and $\gamma$, a formal variable of degree $6-2d$.
By taking the dual of the previous complexes we arrive at the definitions we will mainly work with:
 \begin{df}\label{sRG}
The \emph{d-shifted stable ribbon graph complex colored by $B$} is defined to be the cochain complex 
$$s\mathcal{RG}^d_B\llbracket\gamma\rrbracket:=\Big(Sym\big( {s\mathcal{RG}^{B}_{c,\bullet}}^\vee[2d-6]\big)\llbracket \gamma\rrbracket[3-d],d\Big).$$ 
Its differential $d$ is given by the one induced from $(s\mathcal{RG}_{c,\bullet}^B,\partial)$. 
 \end{df}   
 \begin{df} The \emph{d-shifted ribbon graph complex colored by $B$} is defined to be the cochain complex
 $$\mathcal{RG}^d_B\llbracket\gamma\rrbracket:=\Big(Sym\big( {\mathcal{RG}^{B}_{c,\bullet}}^\vee[2d-6]\big)\llbracket \gamma\rrbracket[3-d],d\Big),$$
 its differential $d$ induced from  $(\mathcal{RG}^{B}_{c,\bullet},\partial)$
 \end{df}
 \begin{df}\label{rtc} The \emph{d-shifted ribbon tree complex colored by $B$} is defined to be the cochain complex 
 $$\mathcal{RT}^d_B:=\Big(Sym\big( {\mathcal{RT}^{B}_{c,\bullet}}^\vee[d-3]\big),d\Big),$$
 its differential induced from  $(\mathcal{RT}^{B}_{c,\bullet},\partial)$. 
 \end{df}
 \begin{remark}
   We refer to an element of eg. $s\mathcal{RG}^d_B\llbracket\gamma\rrbracket$ as a not necessarily connected stable ribbon graph.
   Indeed, note that above three graph complexes have a basis given by symmetric powers of $B$-colored connected stable ribbon graphs, ribbon graphs, respectively ribbon trees.
   In this basis the respective differential $d$ expands edges in all possible ways,
   thereby possibly also changing the other properties of the graph, ie. joining and splitting cycles and changing boundary and genus defect in the case of stable ribbon graphs.
 \end{remark}
 In the following section we will make extensive use of these bases:
 \subsubsection{Shifted Poisson Structure}
 We equip the graph complexes from the previous section with shifted Poisson structures:
\begin{cons}
    Define 
    $$\{\_,\_\}:{s\mathcal{RG}^{B}_{c,\bullet}}^\vee\otimes {s\mathcal{RG}^{B}_{c,\bullet}}^\vee\rightarrow {s\mathcal{RG}^{B}_{c,\bullet}}^\vee$$
    by 
    $$\{\Gamma_1,\Gamma_2\}=\sum_{ (*)}{\Gamma_1\cup_{l_1,l_2} \Gamma_2}.$$
    \begin{itemize}
        \item 
    Here $(*)$ indicates that we are summing over all isomorphism classes of stable ribbon graphs, suggestively denoted ${\Gamma_1\cup_{l_1,l_2} \Gamma_2}$, whose (un-oriented) underlying stable ribbon graph is obtained by gluing $\Gamma_1$ and $\Gamma_2$ together along one of their respective leafs, thus creating a new internal edge.

  \item The orientation of $\Gamma_1\cup_{l_1,l_2} \Gamma_2$ is defined by taking first the given ordering of internal edges of $\Gamma_1$, then the newly created internal edge, and then the given ordering of internal edges of $\Gamma_2$.
    \end{itemize}
\end{cons} 
Note that
  $\Gamma_1\cup_{l_1,l_2} \Gamma_2$ has $f(\Gamma_1) +f(\Gamma_2)-1$ faces.  

\begin{lemma}
Given $B$ a set 
$\big((s\mathcal{RG}^{B}_{c,\bullet})^\vee,d,\{\_,\_\}\big) $
is a 1-shifted dg Lie algebra.
\end{lemma}
\begin{proof}
$\{\_,\_\}$ is of degree 1 since we are creating a new internal edge.
The graded anti-symmetry follows from the fact that we quotient out by the relation that reversing the orientation is equivalent to multiplying by $(-1)$ and the definition of the bracket.
The shifted Jacobi-identity, ie. the fact that the bracket is a derivation with respect to the bracket can be verified easily.
Furthermore it is straightforward to see that the bracket and the differential together define a shifted dg Lie algebra.    
\end{proof}
\begin{remark}\label{nütz}
    In the same way we can define a shifted Lie bracket on ribbon graphs and ribbon trees such that the maps induced by diagram \ref{inc_proj} respect these brackets.
\end{remark}
  By first shifting the Lie bracket and then extending it according to the Leibniz rule to the symmetric algebra we obtain
  a bracket $\{\_,\_\}$ on $\mathcal{RT}^d_{B}$. Then, considering the previous lemma, it follows by standard arguments that: 
\begin{theorem}\label{RS}
For every set $B$ the ribbon tree graph complex from definition \ref{rtc} 
$$\big(\mathcal{RT}^d_{B},-d,\{\_,\_\},\cdot\big)$$
 is a $(d-2)$-shifted Poisson dg algebra. 
 \end{theorem}

 By abuse of notation we will denote this $(d-2)$-shifted Poisson dg algebra by $\mathcal{RT}^d_{B}$.

Define
 $$D:=\sum_{n>2}\sum_{(\#)}{D_{n,C}},$$
where $(\#)$ indicates that we are summing over all isomorphism classes of B-colored ribbon trees with one vertex, no internal edges and n leaves.
We denote  by ${D_{n,C}}$ a such determined isomorphism class. 
 \begin{theorem}\label{RS_mce}
 The element
$$D \in \mathcal{RT}^d_B$$
is a Maurer-Cartan element.
\end{theorem}
\begin{proof}
This will follow from remark \ref{prak_1} and theorem \ref{sRG_mce}.
Indeed, we wil see that $D$ is the image of a Maurer-Cartan element under a map of dg shifted Lie algebras.
\end{proof}

    We denote by ${\mathcal{RT}^{d,tw}_{B}}$ the $(d-2)$-shifted Poisson dg algebra
    obtained by twisting the $(d-2)$-shifted Poisson dg algebra ${RT}^d_{B}$ by the MCE $D$.

\subsubsection{Beilinson Drinfeld Structure}
\begin{cons}
   Again making use of our basis, define a map 
   $$\nabla:{s\mathcal{RG}^{B}_{c,\bullet}}^\vee\rightarrow {s\mathcal{RG}^{B}_{c,\bullet}}^\vee$$
by 
$$\nabla\Gamma=\sum_{ (\#)}\circ_{l_1,l_2}\Gamma.$$
    \begin{itemize}
        \item 
    Here $(\#)$ indicates that we are summing over all isomorphism classes of stable ribbon graphs whose underlying (un-oriented) stable ribbon graph is obtained by gluing  together two leafs of $\Gamma$ which lie in the same face.
   
 \item The orientation of $\circ_{l_1,l_2}\Gamma$ is defined by taking first the new internal edge and then the given ordering of internal edges of $\Gamma$.
    \end{itemize}
\end{cons}

    Note that in case $(\#)$ the stable ribbon graph $\circ_{l_1,l_2}\Gamma$ has one more face as compared to $\Gamma$, the genus stays the same, however.

One easily verifies:
\begin{lemma}
    We have $[d,\nabla]=0$.
\end{lemma}
We extend $\nabla$ to a degree 1 map by first shifting it, then as a derivation and then shifting again: 
$$\nabla: s\mathcal{RG}_B^d\rightarrow s\mathcal{RG}_B^d
.$$
\begin{cons}
    Define a map $$\delta_{i}:{s\mathcal{RG}^{B}_{c,\bullet}}^\vee\rightarrow {s\mathcal{RG}^{B}_{c,\bullet}}^\vee$$
   by $$\delta_{i}\Gamma=\sum_{  (\%)}\circ_{l_1,l_2}\Gamma.$$
    \begin{itemize}
        \item 
    Here $(\%)$ indicates that we sum over all isomorphism classes of graphs $\circ_{l_1,l_2}\Gamma$ whose
    underlying (unoriented) stable ribbon graph is obtained by gluing together two leafs of $\Gamma$ that lie in different faces. 
    
    \item The orientation of $\circ_{l_1,l_2}\Gamma$ is defined by taking first the new internal edge and then the given ordering of internal edges of $\Gamma$.
    \end{itemize}
\end{cons}

    Note that the stable ribbon graph $\circ_{l_1,l_2}\Gamma$ obtained
    by gluing together two leafs of $\Gamma$ that lie in the same face has one less face compared to $\Gamma$ and the genus increased by one.

It is easily verified that:
\begin{lemma}
    We have $[d,\delta_{i}]=0$.
\end{lemma}

We extend $\delta_{i}$ to a degree 1 map by first shifting it, then as a derivation and then shifting again: 
$$\delta_{i}: s\mathcal{RG}_B^d\rightarrow s\mathcal{RG}_B^d.$$  

Denote by $\delta_e$ the Chevalley-Eilenberg differential on 
$Sym\big( (s\mathcal{RG}^{B}_{c,\bullet})^\vee[(2d-6)]\big)$
associated to the $(2d-5)$-shifted Lie bracket $\{\_,\_\}$ on 
$(s\mathcal{RG}^{B}_{c,\bullet})^\vee[(2d-6)]$.
Denote by the same latter the induced degree $(2d-5)$ map  
$$\delta_e: s\mathcal{RG}^d_{B}\llbracket\gamma\rrbracket\rightarrow s\mathcal{RG}^d_{B}\llbracket\gamma\rrbracket.$$

\begin{lemma}\label{rbc}
    We have $[d,\delta_e]=(\nabla+\delta_i)^2=[\nabla+\delta_i,\delta_e]=0$ as maps on $s\mathcal{RG}^d_{B}\llbracket\gamma\rrbracket.$
\end{lemma}
\begin{proof}
 The fact that the first commutator is zero follows since $\delta_e$ is the Chevalley-Eilenberg differential to $\{\_,\_\}$, which was compatible with $d$.
 
Note that $\nabla+\delta_i$ is given by the sum over all possible self-gluing - independent of a condition on the endpoints of the leafs.
Then the fact that the second term is zero follows since for $\Gamma$ a stable ribbon graph $(\delta_i+\nabla)^2(\Gamma)$ is given by a sum of connected stable ribbon graphs in which each two summands have the same underlying unoriented stable ribbon graph, but their orientation differs by exactly swapping the first two internal edges.
Since we mod out by the relation that reversing the multiplication is equivalent to multiplying by (-1) the claim follows.

The fact that the third term is zero follows by similar arguments.
\end{proof}
Thus we can infer 
\begin{theorem}\label{sRGt}
The B-colored stable ribbon graph complex defined as 
$$\big(s\mathcal{RG}^d_{B}\llbracket\gamma\rrbracket,\cdot, -d+\nabla+\delta_i+\gamma\delta_e,\{\_,\_\}\big)$$
is a $(d-2)$-twisted BD algebra over $k\llbracket\gamma\rrbracket$ whose multiplication has even degree $(3-d)$. 
\end{theorem}
\begin{proof}
 Indeed the fact that $-d+\nabla+\delta_i+\gamma\delta_e$ defines a map of degree 1 that squares to zero follows from the previous two lemmas.
 The BD relation follows since $\nabla$, $\delta_i$ and $d$ are derivations and by the definition of $\delta_e$ as the Chevalley-Eilenberg differential associated to $\{\_,\_\}$. 
\end{proof}

    By abuse of notation we also denote by $s\mathcal{RG}^d_{B}\llbracket\gamma\rrbracket$ above $(d-2)$-twisted BD algebra. 

\begin{df}\label{deq_rib}
Denote by 
\begin{equation}\label{ncgdq}
p:s\mathcal{RG}^d_{B}\llbracket\gamma\rrbracket\rightarrow {\mathcal{RT}^{d}_{B}}\end{equation}
the composite 
$$Sym\big( {s\mathcal{RG}^{B}_{c,\bullet}}^\vee[2d-6]\big)\llbracket \gamma\rrbracket[3-d]\rightarrow Sym^1\big( {s\mathcal{RG}^{B}_{c,\bullet}}^\vee[2d-6]\big)[3-d]\cong {s\mathcal{RG}^{B}_{c,\bullet}}^\vee[d-3]$$
$${s\mathcal{RG}^{B}_{c,\bullet}}^\vee[d-3]\rightarrow {\mathcal{RT}^{B}_{c,\bullet}}^\vee[d-3]\rightarrow Sym\big({\mathcal{RT}^{B}_{c,\bullet}}^\vee[d-3]\big).$$
Here the first map in the first line is the projection on symmetric words of length one and the $\gamma=0$ part.
The first map in the second line is the one induced by the curved map of \ref{inc_proj} and the last map is just the inclusion. 
\end{df}
\begin{lemma}
    Given a set B
$$p:s\mathcal{RG}^d_{B}\llbracket\gamma\rrbracket\rightarrow {\mathcal{RT}^{d}_{B}} $$
    is a dequantization map in the sense of definition \ref{dq}.
\end{lemma}
\begin{proof}
The fact that $p$ intertwines the shifted Poisson bracket follows from remark \ref{nütz}.
The fact that $p$ intertwines the BD differential with the differential of ${\mathcal{RT}^{d}_{B}}$ can be verified directly:
$\nabla$ and $\delta_i$ get send to zero under the projection to ribbon trees since the prior produces a graph with more than one face and since the latter produces a graph with positive genus.
Lastly $\gamma\delta_e$ gets send to zero since it involved a positive power of $\gamma$.
\end{proof}
Define
 $${S}:=\sum_{n>0}\sum_{(\#)}{S}_{n,C},$$
where the second sum is over isomorphism classes of a B-colored stable ribbon graphs with no edges, one vertex and $n$ leaves and we denote a such determined class by ${S}_{n,C}$. 
\begin{theorem}\label{sRG_mce}
The element $${S}\in s\mathcal{RG}^d_{B}\llbracket\gamma\rrbracket$$
is a Maurer-Cartan element.
\end{theorem}
\begin{proof}
This will follow from theorem \ref{GALt} and theorem \ref{sG_mce}.
Indeed we will see that $S$ is the image of a Maurer-Cartan element under a map of dg-shifted Lie algebras.\end{proof}
\begin{remark}\label{prak_1}
    We have that $p(S)=D$, where $p$ is the dequantization map from \ref{ncgdq}.
    Thus indeed theorem \ref{sRG_mce} implies theorem \ref{RS_mce}.
\end{remark}
\begin{df}\label{sRG_tw}    
Denote by $s\mathcal{RG}^{d,tw}_{B}\llbracket\gamma\rrbracket$ the $(d-2)$-twisted BD algebra obtained by twisting $s\mathcal{RG}^{d}_{B}\llbracket\gamma\rrbracket$ by the Maurer-Cartan element $S$.
\end{df}
Twisting the map \ref{ncgdq} by the compatible Maurer-Cartan elements $D$ respectively $S$ gives an induced map 
$$p:s\mathcal{RG}^{d,tw}_{B}\llbracket\gamma\rrbracket\rightarrow \mathcal{RT}^{d,tw}_{B}\llbracket\gamma\rrbracket.$$
 Note that we can restrict the operators $\nabla$, $\delta_i$ and $\delta_e$ to ribbon graphs. Exactly as in theorem \ref{sRGt} we can deduce that this defines a BD algebra on the ribbon graph complex. Furthermore diagram \ref{inc_proj} induces a map of BD algebras
$$v:\mathcal{RG}^d_{B}\llbracket\gamma\rrbracket\rightarrow s\mathcal{RG}^d_{B}\llbracket\gamma\rrbracket$$
and a dequantization map, defined exactly as in \ref{deq_rib}
$$p:\mathcal{RG}^d_{B}\llbracket\gamma\rrbracket\rightarrow \mathcal{RT}^d_{B}.$$
Summarizing, we have 
\begin{theorem}\label{rbg_comp}
The ribbon graph complex $\big(\mathcal{RG}^d_{B}\llbracket\gamma\rrbracket,\cdot,-d+\nabla+\delta_i+\gamma\delta_e,\{\_,\_\}\big)$ is a $(d-2)$-twisted BD algebra. Furthermore diagram \ref{inc_proj} induces a diagram

 \begin{equation} \begin{tikzcd}
\mathcal{RT}^d_{B}&\mathcal{RG}^d_{B}\llbracket\gamma\rrbracket\arrow[swap]{l}{p}\arrow{r}{v}&
s\mathcal{RG}^d_{B}\llbracket\gamma\rrbracket\arrow[bend left=25]{ll}{p}
    \end{tikzcd},
    \end{equation}
    where the arrows pointing to $\mathcal{RT}^d_{B}$ denote dequantization maps and the third one denotes a map of $(d-2)$-twisted BD algebras.
\end{theorem}
\begin{remark}
    Note that the Maurer-Cartan element $S\in s\mathcal{RG}^d_{B}\llbracket\gamma\rrbracket$ does not come from an element in $\mathcal{RG}^d_{B}\llbracket\gamma\rrbracket$ under the map $v$.
\end{remark}
\subsection{(Ordinary) Graph Complexes}\label{grs}
In this section we basically repeat the constructions from the previous section, but for stable graphs, introduced in \cite{geKa96} section 2.8, their simpler cousins graphs and trees. That is we equip complexes built from these graphs with shifted Poisson and Beilinson-Drinfeld structures, induced by gluing leafs. Lastly we identify certain Maurer-Cartan elements for these structures.
\begin{df}\label{sgrdf} A \emph{stable graph} is given by the datum of
    \begin{itemize}
    \item a finite set $H$, called the set of half edges,
    \item $V(\Gamma)$ a partition  of $H$, called the set of vertices; we denote the cardinality of a vertex v by $|v|$ and the total number of vertices of a graph by $v(\Gamma)$,
    \item an involution $\sigma_1$ acting on $H$ whose fixed points are called leafs, the set of those denoted $L(\Gamma)$ and whose 2-cycles are called internal edges, the set of those denoted $E(\Gamma)$. We further denote the total number of internal edges by $e(\Gamma)$ and the total number of leafs by $l(\Gamma)$.
    \item a function $o:V(\Gamma)\rightarrow \mathbb{Z}_{\geq 0}$
    such that for $o(v)=0$ we have $|v|\geq 3$ and for $o(v)=1$ we have $|v|\geq1$ called loop defect .
    \end{itemize}
\end{df}
\begin{df}\label{betti}
    The \emph{first Betti number} $bt(\Gamma)$ of a connected stable graph $\Gamma$ is defined by  $$bt(\Gamma)=1-v(\Gamma)+e(\Gamma)+\sum_{v\in V(\Gamma)}o(v).$$
\end{df}

A \emph{graph} is a stable graph for which $o(v)=0$ for all vertices, ie. in particular all vertices are at least three-valent.
A connected \emph{tree} is a connected graph $\Gamma$ for which $bt(\Gamma)=0$, that is, it has no loops.
An orientation of a stable graph is an ordering of its edges.
Two oriented stable graphs are isomorphic if there is a bijection between their set of half-edges that respects all the other structures.
\begin{df}[Contracting an Edge]\label{cont_g}
Let $\Gamma$ be an oriented stable graph and let $e\in E(\Gamma)$ be an  edge.
We define the oriented stable graph $\Gamma/e$ to be the graph obtained by contracting this edge according to the following rules:
\begin{enumerate}
    \item Suppose that e is not a loop, that is it connects distinct vertices $v_1,v_2\in V(\Gamma)$. By contracting the edge the vertices $v_1$ and $v_2$ become joined.
    The loop defect for the vertices $v_1$ and $v_2$ are added to give the loop defect for the new vertex created by joining $v_1$ and $v_2$.
    The orientation is defined in the obvious way.
    All the other combinatorial structures elsewhere in the graph are left alone. 
    \item
    Now suppose that $e$ is a loop, that is both its endpoints lie in the same vertex $v$.
    Then $\Gamma/e$ is defined by just deleting this loop, but by that increasing the loop defect of $v$ by one.
\end{enumerate}
\end{df}
Note that the property being a tree is preserved under this differential, whereas being a graph is not.

Denote by $(s\mathcal{G}_{c,\bullet},\partial)$ the complex whose underlying vector space is freely generated by isomorphism classes of connected stable graphs, modulo the relation that reversing the orientation on a stable graph is equivalent to multiplying by (-1).
It is graded by the number of internal edges.
The differential $\partial$ is given by summing over all possible contractions of edges:
    $$\partial(\Gamma):=\sum_{e\in E(\Gamma)}\Gamma/e.$$
    Denote by $(\mathcal{G}_{c,\bullet},\partial)$ the complex given by the subspace (which is not a subcomplex) of connected graphs, which uniquely becomes a chain complex by requiring that the projection map
$$s\mathcal{G}_{c,\bullet}\rightarrow \mathcal{G}_{c,\bullet}$$
is a map of chain complexes.   

Denote by $(\mathcal{T}_{c,\bullet},\partial)$ the sub-complex of $s\mathcal{G}_{c,\bullet}$  given by connected trees.

We could have equivalently defined $(\mathcal{T}_{c,\bullet},\partial)$ as the sub-complex of $(\mathcal{G}_{c,\bullet},\partial)$ given by  trees, which can be verified directly. That is, following diagram of chain complexes commutes
   \begin{equation}\label{inc_proj_2} \begin{tikzcd}
(\mathcal{T}_{c,\bullet},\partial)\arrow{r}\arrow[ bend right=25]{rr}&(\mathcal{G}_{c,\bullet},\partial)&
(s\mathcal{G}_{c,\bullet},\partial)\arrow{l}.
    \end{tikzcd}
    \end{equation}

Let $d$ be an odd integer and $\hbar$ be a formal variable of degree $(3-d)$.
\begin{df}\label{sG}
    The \emph{d-shifted stable graph complex} is defined to be the cochain complex 
$$s\mathcal{G}^d\llbracket\hbar\rrbracket:=\Big(Sym\big( s\mathcal{G}_{c,\bullet}^\vee[d-3]\big)\llbracket \hbar\rrbracket,d\Big).$$ 
Its differential $d$ is given by the one induced from $(s\mathcal{G}_{c,\bullet},\partial)$.  
\end{df}
 \begin{df}\label{hfig} The \emph{d-shifted graph complex} is defined to be the cochain complex $$\mathcal{G}^d\llbracket\hbar\rrbracket:=\Big(Sym\big( \mathcal{G}_{c,\bullet}^\vee[d-3]\big)\llbracket \hbar\rrbracket,d\Big),$$
 its differential $d$ induced from  $(\mathcal{G}_{c,\bullet},\partial)$
 \end{df}
 \begin{df}\label{tc} The \emph{d-shifted tree complex} is defined to be the cochain complex 
 $$\mathcal{T}^d:=\Big(Sym\big( \mathcal{T}_{c,\bullet}^\vee[d-3]\big),d\Big),$$
 its differential induced from  $(\mathcal{T}_{c,\bullet},\partial)$. 
 \end{df}
 We refer to an element of eg. $s\mathcal{G}^d\llbracket\hbar\rrbracket$ as a not necessarily connected stable graph.
 Indeed, note that above three graph complexes have a basis given by symmetric powers of connected stable graphs, graphs, respectively trees.
 In this basis the differential expands edges, thereby possibly also changing the other properties of the graph, ie changing the loop defect in the case of stable graphs.

 As before we will make extensive use of these 'dual' bases to define further algebraic structures:
  \subsubsection{Shifted Poisson Structure}
\begin{cons}\label{iwjs}
    Define 
    $$\{\_,\_\}:s\mathcal{G}_{c,\bullet}^\vee\otimes s\mathcal{G}_{c,\bullet}^\vee\rightarrow s\mathcal{G}_{c,\bullet}^\vee$$
    by 
    $$\{\Gamma_1,\Gamma_2\}=\sum_{(*)}\Gamma_1\cup_{l_1,l_2} \Gamma_2.$$
    \begin{itemize}

\item Here $(*)$ indicates that we are summing over all isomorphism classes of stable graphs, suggestively denoted ${\Gamma_1\cup_{l_1,l_2} \Gamma_2}$, whose underlying (un-oriented) stable graph is obtained by gluing together two leafs of $\Gamma_1$ and $\Gamma_2$, thus creating a new internal edge\footnote{or by gluing together two leafs of $\Gamma_2$ and $\Gamma_1$}.

\item The orientation of $\Gamma_1\cup_{l_1,l_2} \Gamma_2$ is defined by taking first the given ordering of internal edges of $\Gamma_1$,
then the new internal edge, and then the given ordering of internal edges of $\Gamma_2$.
    \end{itemize}
\end{cons} 
\begin{lemma}
    $\big(s\mathcal{G}_{c,\bullet}^\vee,d,\{\_,\_\}\big)$ is a 1-shifted dg Lie algebra.
\end{lemma}
\begin{proof}
This follows analogously to the stable ribbon graph case.  
\end{proof}
\begin{remark}\label{nütz_2}
    In the same way we can define a shifted Lie bracket on  graphs and  trees such that the maps induced by \ref{inc_proj_2} respect these brackets.
\end{remark}
  In the same way as for ribbon graphs (theorem \ref{RS}) we deduce
\begin{theorem}\label{S}
The tree graph complex from definition \ref{tc} 
$$\big(\mathcal{T}^d,-d,\{\_,\_\},\cdot\big)$$
 is a $(d-2)$-shifted Poisson dg algebra. 
 \end{theorem}
 Recall that $\{\_,\_\}$ is given by first shifting the Lie bracket and then extending it as a according to the Leibniz rule to the symmetric algebra.

 We will by abuse of notation also denote this dg $(d-2)$-shifted Poisson algebra by $\mathcal{T}^d$.

Denote by $T_n$ the class of the tree with one vertex, no internal edges and $n$ leaves. We define 
 $$T:=\sum_{n>2}T_n$$
 \begin{theorem}\label{S_mce}
 The element
$$T \in \mathcal{T}^d$$
is a Maurer-Cartan element.
\end{theorem}
\begin{proof}
   This will follow from theorem \ref{sG_mce} and remark \ref{prak_2}. 
   Again, we will be able to express $T$ as the image of a Maurer-Cartan element under a map of dg-shifted Lie algebras.
\end{proof}

    Denote by $\mathcal{T}^{d,tw}$ the $(d-2)$-shifted Poisson dg algebra obtained by twisting the $(d-2)$-shifted Poisson dg algebra $\mathcal{T}^d$ by the Maurer-Cartan element $T\in\mathcal{T}^d$.

\subsubsection{Beilinson Drinfeld Structure}
\begin{cons}
   Define a map 
   $$\Delta_i:s\mathcal{G}_{c,\bullet}^\vee\rightarrow s\mathcal{G}_{c,\bullet}^\vee$$
by 
$$\Delta_i(\Gamma)=\sum_{(\#)}\circ_{l_1,l_2}\Gamma.$$
    \begin{itemize}

 \item Here $(\#)$ indicates that we sum over all isomorphism classes of stable graphs $\circ_{l_1,l_2}\Gamma$ whose
 underlying (un-oriented) stable graph is obtained by gluing two leafs of $\Gamma$ together, thus creating a new internal edge.
    
   \item The orientation of $\circ_{l_1,l_2}\Gamma$ is defined by taking first the new internal edge and then the given ordering of internal edges of $\Gamma$.
    \end{itemize}
\end{cons}
    Note that the betti number changes as $bt(\circ_{l_1,l_2}\Gamma)=bt(\Gamma)+1$.

We easily verify
\begin{lemma}
    We have $[d,\Delta_i]=0$.
\end{lemma}

We extend $\Delta_i$ to a degree 1 map by first shifting it and then as a derivation to
$$\Delta_i: s\mathcal{G}^d\llbracket\hbar\rrbracket\rightarrow s\mathcal{G}^d\llbracket\hbar\rrbracket
.$$
Denote by $\Delta_e$ the Chevalley-Eilenberg differential on $Sym\big( s\mathcal{G}_{c,\bullet}^\vee[d-3]\big)$
associated to the odd Lie bracket $\{\_,\_\}$ on $s\mathcal{G}_{c,\bullet}^\vee[d-3]$.
Denote by the same latter the induced degree (d-2) map  
$$\Delta_e: s\mathcal{G}^d\llbracket\hbar\rrbracket\rightarrow s\mathcal{G}^d\llbracket\hbar\rrbracket.$$
\begin{lemma}
    We have $[d,\Delta_e]=(\Delta_i)^2=[\Delta_i,\Delta_e]=0$ on $s\mathcal{G}^d\llbracket\hbar\rrbracket.$
\end{lemma}
\begin{proof}
 These facts follows analogously to the stable ribbon graph case, see lemma \ref{rbc}.
\end{proof}
Thus we can infer 
\begin{theorem}\label{BD_G}
The stable graph complex from definition \ref{sG} 
$$\big(s\mathcal{G}^d\llbracket\hbar\rrbracket,\cdot, -d+\Delta_i+\hbar\Delta_e,\{\_,\_\}\big)$$
is a $(d-2)$-twisted BD algebra. 
\end{theorem}
\begin{proof}
 Indeed the fact that $-d+\Delta_i+\hbar\Delta_e$ defines a differential of degree 1 that squares to zero follows from the previous lemmas.
 The BD relation follows since $\Delta_i$ and $d$ are derivations and by the definition of $\Delta_e$ as the Chevalley Eilenberg differential associated to $\{\_,\_\}$. 
\end{proof}

    By abuse of notation we also denote by $s\mathcal{G}^d\llbracket\hbar\rrbracket$ this (d-2 twisted) BD algebra. 

\begin{df}\label{deq_sg}
Denote by 
\begin{equation}\label{cgdq}
p:s\mathcal{G}^d_{B}\llbracket\hbar\rrbracket\rightarrow {\mathcal{T}^{d}_{B}}
\end{equation}
the composite 
$$Sym\big( s\mathcal{G}_{c,\bullet}^\vee[d-3]\big)\llbracket \hbar\rrbracket\rightarrow Sym^1\big( s\mathcal{G}_{c,\bullet}^\vee[d-3]\big)\cong s\mathcal{G}_{c,\bullet}^\vee[d-3]\rightarrow \mathcal{T}_{c,\bullet}^\vee[d-3]\rightarrow Sym\big(\mathcal{T}_{c,\bullet}^\vee[d-3]\big).$$
Here the first map is the projection on symmetric words of length one and the $\hbar=0$ part.
The third map in the second line is the one induced by the curved map of diagram \ref{inc_proj_2} and the last map is just the inclusion. 
\end{df}
\begin{lemma}
    For every set $B$ 
$$p:s\mathcal{G}^d_{B}\llbracket\hbar\rrbracket\rightarrow {\mathcal{T}^{d}_{B}}$$
defines a dequantization map in the sense of definition \ref{dq}.
\end{lemma}
\begin{proof}
The fact that $p$ intertwines the shifted Poisson bracket follows from remark \ref{nütz_2}.
The fact that $p$ intertwines the BD differential with the differential of ${\mathcal{T}^{d}}$ can be verified directly:
$\Delta_i$ gets send to zero under the projection to ribbon trees since it produces a graph with betti number bigger or equal to one.
Trivially $\hbar\Delta_e$ gets send to zero since it involved a positive power of $\hbar$.
\end{proof}

 Denote by ${G}_{n,g}$ the isomorphism class of the stable ribbon graph with one vertex of loop defect $g$, no internal edges and $n$ leafs.
 Define
 $${G}:=\sum_{n>0}\sum_{g\geq0}G_{n,g}.$$
\begin{theorem}\label{sG_mce}
The
 element 
$${G}\in s\mathcal{G}^d\llbracket\hbar\rrbracket$$
is a Maurer-Cartan element.
\end{theorem}
\begin{proof}
Let us denote by $X$ a representative of a stable graph with two vertices and one internal edge connecting these and
by $Y$ a representative of a stable graph that has one vertices and one internal edge (that is thus a loop).
 We have by definition that $$dG_{n,g}=\sum_{(i)} X+\sum_{(ii)} Y. $$   
 Here $(i)$ denotes the sum over all isomorphism classes of stable  graphs $X$ such that when contracting that internal edge according to rule (\ref{cont_g}-1.) the resulting graph is isomorphic to $G_{n,g}$.
 Further $(ii)$ denotes the sum over all isomorphism classes of stable graphs $Y$ such that when contracting that loop according to rule (\ref{cont_g}-2.) we recover $G_{n,g}$.
 On the other hand we have, also by definition \ref{iwjs}, and if $G_{n_1,g_1}\neq G_{n_2,g_2}$ that $$\{G_{n_1,g_1},G_{n_2,g_2}\}=\sum_{(a)}  X $$ 
and the same if we swap the entries of the bracket.
Here $(a)$ indicates that the sum is over all isomorphism classes of graphs $X$ such that when cutting its one edge the resulting stable graph is the disjoint union of $G_{n_1,g_1}$ and $G_{n_2,g_2}$.
Further we have  
$$\{G_{n_1,g_1},G_{n_1,g_1}\}=2\sum_{(a)}  X $$
where $(a)$ stands for the analogous indexing set from before. Then we have that 
$$\Delta_i(G_{n,g})=\sum_{(b)}Y$$
where $(b)$ indicates that the sum is over all isomorphism classes of stable graphs $Y$ such that when cutting its one loop the resulting stable graph is $G_{n,g}$. Trivially it follows that $$\Delta_e(G_{n,g})=0.$$ Next we observe, with the same notation, that 
$$\sum_{n>0} \sum_{g\geq 0}\ \sum_{(i)}X=\sum_{Iso.Cls(X)}X=\frac{1}{2} \sum_{n_1,n_2>0} \sum_{g_1,g_2\geq0}\ 
\sum_{(a)}  X$$
and that
$$\sum_{n>0}\ \sum_{g\geq0}\ \sum_{(ii)}Y=\sum_{Iso.Cls(Y)}Y= \sum_{n>0} \sum_{g\geq0}\ \sum_{(b)}  Y.$$
Here the sum in the middle is over all isomorphism classes of type $X$ and $Y$ respectively, for all $n>0$ and $g\geq0$.
We can regroup these sums in the middle in two different ways, respectively (namely $(i)$ and $(a)$ respectively $(ii)$ and $(b)$).
Summing over the indexing set of this regrouping of the terms in that group thus equals the original sum.

Taking the previous arguments together it follows that
$$dG=\frac{1}{2}\{G,G\}+\Delta_iG+\hbar\Delta_eG,$$
which is what we wanted to show.
\end{proof}
\begin{remark}\label{prak_2}
Under the dequantization map from definition \ref{deq_sg} we have that $p(G)=T$.   
\end{remark}
    Denote by  $s\mathcal{G}^{d,tw}\llbracket\hbar\rrbracket$ the 
    $(d-2)$-twisted BD algebra obtained by twisting $s\mathcal{G}^{d}\llbracket\hbar\rrbracket$ by the Maurer-Cartan element $G$.

Thus by the previous remark we also get a dequantization map
$$p: s\mathcal{G}^{d,tw}\llbracket\hbar\rrbracket\rightarrow \mathcal{T}^{d,tw}. $$
 Note that we can restrict the maps $\Delta_i$ and $\Delta_e$ to graphs.
 Exactly as in theorem \ref{BD_G} we can deduce that this defines a $(d-2)$-twisted BD algebra on the graph complex.
 Furthermore diagram \ref{inc_proj_2} induces a map of BD algebras
$$v:\mathcal{G}^d\llbracket\hbar\rrbracket\rightarrow s\mathcal{G}^d\llbracket\hbar\rrbracket$$
and a dequantization map, defined exactly as in \ref{cgdq}
$$p:\mathcal{G}^d\llbracket\hbar\rrbracket\rightarrow \mathcal{T}^d.$$
Summarizing, we have 
\begin{theorem}
The graph complex $\big(\mathcal{G}^d\llbracket\hbar\rrbracket,\cdot,-d+\Delta_i+\hbar\Delta_e,\{\_,\_\}\big)$
is a $(d-2)$-twisted BD algebra. Furthermore diagram \ref{inc_proj_2} induces a diagram

 \begin{equation}\label{ag_dq} \begin{tikzcd}
\mathcal{T}^d&\mathcal{G}^d\llbracket\hbar\rrbracket\arrow[swap]{l}{p}\arrow{r}{v}&
s\mathcal{G}^d\llbracket\hbar\rrbracket\arrow[bend left=25]{ll}{p}
    \end{tikzcd},
    \end{equation}
    where the arrows pointing to $\mathcal{T}^d$ denote dequantization maps and the third one denotes a map of BD algebras.
\end{theorem}
\begin{remark}
    Note that the Maurer-Cartan element G does not come from an element in $\mathcal{G}^d\llbracket\hbar\rrbracket$ under the map $v$.
\end{remark}
\subsection{`Comm to Assoc'}\label{vsgzg}
Here we explain how the objects from the previous two subsections are connected.
For this we need that $B$, the set decorating the stable ribbon graphs, is finite.
We denote by
\begin{equation}\label{sRGsG}pr: s\mathcal{RG}^B_{c,\bullet}\rightarrow s\mathcal{G}_{c,\bullet} \end{equation}
the map given by sending an oriented B-colored stable ribbon graph to the underlying oriented graph, which we endow with the structure of a stable graph
by defining the loop defect at a vertex $v$ by
$$o(v)=2g(v)+\sum_{\lambda\in B}b_\lambda(v)+c(v)-1.$$
\begin{lemma}\label{fal}
For every finite set $B$
$$pr: s\mathcal{RG}^B_{c,\bullet}\rightarrow s\mathcal{G}_{c,\bullet}$$
is a map of chain complexes.
\end{lemma}
\begin{proof}
This is a slight generalization of the standard fact that the  map \begin{equation}\label{rib_g}
\mathcal{RG}_{\bullet,c}\rightarrow \mathcal{G}_{\bullet,c}\end{equation}  
sending a ribbon graph to its underlying graph is a map of chain complexes. The lemma is proven in the same way by noting that the differential both on $s\mathcal{RG}^B_{c,\bullet}$ and on $s\mathcal{G}_{c,\bullet} $ is given by all ways of contracting an edge, independent of whether it is considered as part of a stable ribbon graph or a stable graph. The fact that the differential is compatible with the other combinatorial structures can be verified by a straightforward case-by-case analysis.
\end{proof}
Let us consider the induced map, which we denote by$$pr^\vee:s\mathcal{G}_{c,\bullet}^\vee\rightarrow {s\mathcal{RG}^B_{c,\bullet}}^\vee.$$
In the dual basis it is given by
\begin{equation}\label{ebieb}pr^\vee(\Gamma)=\sum_{Q}\Bar{\Gamma}
\end{equation}
where $Q$ indicates that we sum over all isomorphism classes of stable ribbon graphs whose underlying stable graph is $\Gamma$ and $\Bar{\Gamma}$ denotes the corresponding class.
\begin{lemma}\label{kId}
We have that  
$$\{pr^\vee(\_),pr^\vee(\_)\}=pr^\vee(\{\_,\_\})\ \
\text{and}\ \  pr^\vee\circ\Delta_i=(\delta_i+\nabla)\circ pr^\vee.$$
\end{lemma}
\begin{proof}
Given $\Gamma_1$ and $\Gamma_2$ stable graphs we have that $\{pr^\vee(\Gamma_1),pr^\vee(\Gamma_2)\}$ is the sum over all stable ribbon graphs such that
when cutting one edge the result is the disjoint union of stable ribbon graph whose underlying stable graphs (in the sense of lemma \ref{fal}) are $\Gamma_1$ respectively $\Gamma_2$.
But this is the same as the sum over all stable ribbon graphs such that when cutting one internal edge of their underlying stable graphs the result is the disjoint union of $\Gamma_1$ respectively $\Gamma_2$.
In other words `cutting an edge and projecting to the underlying stable graph commutes´.
This is the first identity and the second one follows analogously. 
\end{proof}

\begin{cons}
Let us consider an auxiliary map of degree $(d-3)$ induced by the map $pr^\vee$ and suitable shifts 
$$s\mathcal{G}_{c,\bullet}^\vee[d-3]\rightarrow {s\mathcal{RG}^B_{c,\bullet}}^\vee[2d-6].$$
It induces a map 
$$aux:Sym(s\mathcal{G}_{c,\bullet}^\vee[d-3])\rightarrow Sym({s\mathcal{RG}^B_{c,\bullet}}^\vee[2d-6])$$
which has degree $n(d-3)$ on the subspace $Sym^n(s\mathcal{G}_{c,\bullet}^\vee[d-3])$.
\end{cons}
\begin{df}
We define a map of degree zero
\begin{equation}\label{isjk}
\pi: Sym(s\mathcal{G}_{c,\bullet}^\vee[d-3])\llbracket\hbar\rrbracket\rightarrow Sym({s\mathcal{RG}^B_{c,\bullet}}^\vee[2d-6])\llbracket\gamma\rrbracket[3-d]
\end{equation}
by sending 
$$(\Gamma_1\cdot\ldots \cdot\Gamma_n){\hbar}^{n+2k-1}\mapsto aux(\Gamma_1\cdot\ldots \cdot\Gamma_n)\gamma^k,$$
composed by a shift of $(3-d)$ and zero on elements not of this form.
\end{df}
\begin{theorem}\label{GAL}
For every finite set $B$ the map from definition \ref{isjk} is a map of dg shifted Lie algebras
 $$\pi:\Big(s\mathcal{G}^d\llbracket\hbar\rrbracket,-d+\Delta_i+\hbar\Delta_e,\{\_,\_\}\Big)\rightarrow \Big( s\mathcal{RG}^d_B\llbracket\gamma\rrbracket,-d+\nabla+\delta_i+\gamma\delta_e,\{\_,\_\}\Big).$$   
   \end{theorem}
\begin{proof}
 The fact that $\pi$ is compatible with the graph complex differentials and intertwines $\Delta_i$ with $\nabla+\delta_i$ follows directly from the lemma \ref{fal} respectively \ref{kId} and by definitions.
 The fact that $\pi$ is compatible with the brackets follows as they are defined by extending the Lie brackets as a derivation to the symmetric algebra- further by noting that $n_1+2k_1-1+n_2+2k_2-1=(n_1+n_2-1)+2(k_1+k_2)-1,$
 ie. the case distinction for when $\pi$ is or is not zero checks out as well.
 Similarly it follows that $\pi$ intertwines $\hbar\Delta_e$ with $\gamma\delta_e$. 
\end{proof}
Further the map $\pi$ is the predual of a 2-weighted map of algebras, in the following sense.

 \begin{proof}

 Let $\tilde{\gamma}^k$ be a variable of degree $k(2d-6)$ and $\tilde{\hbar}^l$ a variable of degree $l(d-3)$, for all $k,l>0$ which we will understand as the dual of $\gamma^k$ respectively $\hbar^l$.
 Define a map of degree zero
$$\pi^{n,k}:Sym^n\big(s\mathcal{RG}^B_{c,\bullet}[6-2d]\big)\tilde{\gamma}^k[d-3] \rightarrow Sym^n(s\mathcal{G}_{c,\bullet}[3-d])\cdot \tilde{\hbar}^{n+2k-1} $$
by $$\pi^{n,k}(\Gamma_1\cdot\ldots \cdot \Gamma_n\tilde{\gamma}^k)=proj(\Gamma_1\cdot\ldots \cdot\Gamma_n)\widetilde{\hbar}^{n+2k-1}$$
composed by a shift of $(d-3)$.
Here $proj$ is defined as the symmetric algebra functor of the map $pr$ composed by suitable shifts.
Thus we have that
\begin{equation}\label{p2w}
\pi^{n,k}(\Gamma_1\cdot\ldots \cdot \Gamma_n\tilde{\gamma}^k)=\pi^{1,0}(\Gamma_1)\cdot\ldots \cdot \pi^{1,0}(\Gamma_n)\widetilde{\hbar}^{n+2k-1}.
\end{equation}
Considering the dual maps and their direct sum over all $n,k$ this induces a map
$$s\mathcal{G}^d\llbracket\hbar\rrbracket\rightarrow s\mathcal{RG}^d_B\llbracket\gamma\rrbracket,$$
which coincides with the map $\pi$ from theorem \ref{GAL}. We say $\pi$ is predual to a 2-weighted map of algebras (recalling definition \ref{weimap}) in the sense that it is
induced by maps satisfying \ref{p2w} via taking duals and then direct sum.
 \end{proof}
 By using the map \ref{rib_g} $$\mathcal{RG}_{\bullet,c}\rightarrow \mathcal{G}_{\bullet,c}$$ we can prove in the same way as in the previous theorem that it
 induces a map (by abuse of notation also denoted $\pi$) as follows:
\begin{theorem}\label{GALt}
For every finite set $B$
    $$\pi: \Big(\mathcal{G}^d\llbracket\hbar\rrbracket,-d+\Delta_i+\hbar\Delta_e,\{\_,\_\}\Big)\rightarrow \Big(\mathcal{RG}^d_B\llbracket\gamma\rrbracket-d+\nabla+\delta_i+\gamma\delta_e,\{\_,\_\}\Big)$$
    is a map of dg $(d-2)$-shifted Lie algebras and the predual to a map of weighted algebras in the sense of the previous theorem.  
\end{theorem}
Further we note that the map \ref{rib_g} restricts to $$\mathcal{RT}_{\bullet,c}\rightarrow \mathcal{T}_{\bullet,c}.$$
We can prove in a similar way (just omitting some previous shifts) that it induces a map (by abuse of notation also denoted $\pi$) as follows:
\begin{theorem}
For every finite set $B$
    $$\pi: \Big(\mathcal{T}^d,\cdot,-d,\{\_,\_\}\Big)\rightarrow \Big(\mathcal{RT}^d_B,\cdot,-d,\{\_,\_\}\Big)$$
     is a map of $(d-2)$-shifted Poisson dg algebras.
\end{theorem}
The diagrams \ref{rbg_comp} and \ref{ag_dq} fit with the maps provided by the three previous theorems together as follows:
    \begin{equation}\label{comp_g_rg}
\begin{tikzcd}
\mathcal{T}^d\arrow{d}{\pi}&\mathcal{G}^d\llbracket\hbar\rrbracket\arrow{d}{\pi}\arrow[swap]{l}\arrow{r}&
s\mathcal{G}^d\llbracket\hbar\rrbracket\arrow{d}{\pi}\arrow[bend right=20]{ll}\\
\mathcal{RT}^d_{B}&\mathcal{RG}^d_{B}\llbracket\gamma\rrbracket\arrow[swap]{l}\arrow{r}&
s\mathcal{RG}^d_{B}\llbracket\gamma\rrbracket\arrow[bend left=20]{ll}.
    \end{tikzcd}
    \end{equation}
    This can be verified by noting that the diagrams \ref{inc_proj} and \ref{inc_proj_2}
    fit together with the maps induced by \ref{sRGsG} into a commutative diagram and by construction. 

Using the explicit description of the maps $\pi$ from equation \ref{ebieb}
it is immediate to see that for the Maurer-Cartan elements from the previous sections we have that 
 $$\pi(G)=S\ \text{and}\ \pi(T)=D.$$

Since with the previous remark and remarks \ref{nütz} and \ref{nütz_2} we have thus seen that
all the respective Maurer-Cartan elements are compatible we can twist the outer square of \ref{comp_g_rg} by those to obtain following
commutative diagram where the horizontal maps are dequantization maps 
     $$\begin{tikzcd}
     s\mathcal{G}^{d,tw}\llbracket\hbar\rrbracket\arrow{d}{\pi}\arrow{r}{p}&\mathcal{T}^{d,tw}
\arrow{d}{\pi}\\
    s\mathcal{RG}^{d,tw}_{B}\llbracket\gamma\rrbracket\arrow{r}{p} &\mathcal{RT}^{d,tw}_{B}.    
     \end{tikzcd}$$
\newpage
\section{Generalized Kontsevich Cocycle Map}
In section \ref{ksc} we explain how to connect the ribbon graph world, section \ref{rbgs}, with the world of cyclic $A_\infty$-categories and their quantizations, section \ref{ncw}. We do so using a slight generalization of Kontsevich's cocycle construction \cite{Kon92b}, see also \cite{AmTu15}. Further we study carefully what happens when we restrict to essentially finite cyclic $A_\infty$-categories. 

In the following section \ref{cksc} we study similar ideas, but for ordinary graphs and for cyclic $L_\infty$-algebras and their quantizations.
\subsection{Non-Commutative World}\label{ksc}

    Given $V_B$ a collection of cyclic chain complexes of odd degree $d$ together with ${I^q\in MCE (\mathcal{F}^{pq}(V_B))}$ we construct a map (see lemma \ref{skadj123} below)
    $$\rho_{I^q}: s\mathcal{RG}_B^d\llbracket\gamma\rrbracket\rightarrow \mathcal{F}^{pq}(V_B).$$

We explain how to to so in multiple steps, partly inspired by \cite{AmTu15}.
We begin by associating certain tensors, denoted $I^q|_{v_i}$, to each vertex $v_i$ of a connected B-colored stable ribbon graph.
\begin{cons}
  Let $\Gamma$ be a connected B-colored stable ribbon graph and $v_i$ one of its vertices:\begin{itemize}
      \item 
 To such a vertex $v_i$ of $\Gamma$ is associated a genus defect $g_i\in\mathbb{N}_{\geq 0}$, a boundary defect $b_\lambda(v_i)\in\mathbb{N}_{\geq 0}$ for each $\lambda\in B$.
 Further the half-edges of that vertex are partitioned into cycles $c_1,\cdots, c_{n}$.
 Two consecutive half-edges in a cycle together belong to a unique free boundary of $\Gamma$ and thus are associated an element of $B$.
 Considering all the half-edges in a cycle thus defines a cyclic word in elements of $B$ and all the cycles taken together a `symmetric' word in cyclic word in letters from $B$. 
 \item  Given a vertex $v_i$ of $\Gamma$ we define a map (which is not graded), denoted
  \begin{equation}\label{iso_sc}
(\_)|_{v_i}: Sym(Cyc^*(V_B)[d-4])\llbracket\gamma\rrbracket[3-d]\rightarrow Sym(Cyc^*_+(V_B)[-1]).\end{equation}
It is given by first a shift by $(3-d)$, composed with the map  
$$Sym(Cyc^*(V_B)[d-4])\llbracket\gamma\rrbracket\rightarrow Sym(Cyc^*(V_B)[-1])\llbracket\gamma\rrbracket$$
induced by {$Cyc^*(V_B)[d-4]\rightarrow Cyc^*(V_B)[-1]$}.
Lastly we project to the subspace of $Sym(Cyc^*_+(V_B)[-1])$ determined by the symmetric word in cyclic words in letters of $B$ and 
restrict to the power $\gamma^{g_i}$ and the power of $\Pi_{\lambda\in B}\nu_\lambda^{b_\lambda(v_i)}$.
\item Note that for $I^q\in MCE(\mathcal{F}^{pq}(V_B))$, which has degree $(3-d)$, we have that 
$${|{I}^q|_{v_i}|=(3-d)(2-2g_i-c(v_i)-\sum_{\lambda\in B}b_\lambda(v_i))}.$$
\end{itemize}
\end{cons}
Next we explain how the datum of an oriented stable ribbon graph together with some auxiliary choices
defines a permutation of the half-edges of such a stable ribbon graph:
\begin{cons}\label{perm}
 Assume that $\Gamma$  has $n$ vertices, each with $c(v_i)$ cycles.
 \begin{itemize}
     \item 
Chose an ordering of its vertices, at each vertex $v_i$  an ordering of the cycles and for each such cycle an half-edge.
This gives an ordering of all the half edges at that vertex (using the cyclic orderings of the half edges within each cycle) denoted 
$$h_{v_i,1}\cdots h_{v_i,|v_i|}$$
and with the ordering of the vertices together defines an ordering of all half edges of $\Gamma$
$$h_{v_1,1}\cdots h_{v_1,|v_1|}h_{v_2,1}\cdots h_{v_r,|v_r|}.$$
\end{itemize}
Next assume that $\Gamma$ is orientated, ie. we have an ordering $(e_1\cdots e_n)$ of the internal edges.
\begin{itemize}
    \item 
Chose an ordering of the half edges belonging to a given internal edge.
Together with the orientation this gives an ordering of the half edges belonging to the internal edges
$$h_{1_x}h_{1_y}\cdots h_{n_x} h_{n_y}.$$ 
\end{itemize}
Given the choices made and given the orientation we can define a permutation of all the half-edges as follows
$$
\sigma: H(\Gamma)\rightarrow H(\Gamma) $$
$$
h_{v_1,1}\cdots h_{v_1,|v_1|}h_{v_2,1}\cdots h_{v_r,|v_r|}\mapsto h_{1_x}h_{1_y}\cdots h_{n_x} h_{n_y}l_{1,1}\cdots l_{1,m_1}l_{2,1}\cdots l_{b,m_{f_{>0}}} $$
\begin{itemize}
    \item 

We begin the ordering of the leaves on the right hand side by the first leave that appears on the left hand side.
This leafs belongs to a face. We use the cyclic ordering within that face to order the remaining half edges that belong to it.
Then we go to the first half edge that is a leaf and does not lie in that first face.
Again we use the cyclic ordering within that face to order the remaining half edges that belong to it.
And so on till we come to the $f_{>0}$-th face with non-zero open boundaries.
\end{itemize}
\end{cons}
Recall the definition of $\widetilde{(V_B)}$, the cyclic chain complex associated to the collection of cyclic chain complexes $V_B$, see equation \eqref{cyccomm}. Denote $W:=\widetilde{(V_B)}^\vee[-1]$.
We still fix the same oriented stable ribbon graph $\Gamma$ and the choices of orderings made.
We next explain how to define, given these choices, a map $$W^{\otimes h(\Gamma)}\rightarrow Sym(Cyc^*_+({V_B})[d-4])$$ of degree $e(\Gamma)(d-2)+f_{>0}(\Gamma)(d-3)$:

\begin{cons}\label{con_proj}
The permutation $\sigma$ induces a map which we denote by abuse of notation by the same letter
$${\sigma}:W^{\otimes h(\Gamma)}\rightarrow W^{\otimes h(\Gamma)}\cong W^{\otimes 2e(\Gamma)}\otimes W^{\otimes l(\Gamma)}.$$
We say an element in $W^{\otimes n}$ has `matching boundary' conditions if it is the sum of tensors having matching boundary conditions. 
\begin{itemize}
    \item 
To any edge $e$ of $\Gamma$ are associated two elements of $B$, say $i$ and $j$. Denote by 
$\langle\_,\_\rangle^{-1}_{e}: W^{\otimes 2}\rightarrow k[d-2]$
the map given by sending elements to zero which do not have the matching boundary condition $\{ij\}$ and
by sending elements which do have the matching boundary condition $\{ij\}$ to the number given by applying $\langle\_,\_\rangle^{-1}_{ij}$ to them\footnote{The ordering of $\{ij\}$ does not matter since $\langle\_,\_\rangle^{-1}_{ij}$ is symmetric.}.
\item A stable ribbon graph $\Gamma$ has $f$ faces, $f_{>0}$ of those with non-zero open boundaries lying in that face. Further say that the i-th such face 
(determined by the ordering of all half-edges given made choices, as before) has $m_i$ open boundaries. Denote by 
$$proj: W^{\otimes l(\Gamma)}\rightarrow Sym(Cyc^*_+({V_B})[d-4])$$
the map described as follows:
It is zero on elements not of following matching boundary conditions: We ask that the first $m_1$ elements have the matching boundary condition determined by 
the first face of the graph with nonzero leafs, the next $m_2$ elements have the matching boundary condition determined by the second face of the graph with nonzero leafs etc till the last $m_{f_{>0}}$ elements. 
Then the map is given by
 $$proj=sym^{f_{>0}}(cyc_{m_1}\otimes\cdots\otimes cyc_{m_{f_{>0}}}).$$
Here $cyc_{m_i}$ denotes the projection to cyclic words of length $m_i$ with boundary conditions determined by the i-th face as above,
shifted by $(d-3)$ and $sym$ denoted the projection to symmetric words in cyclic words. Counting together $proj$ has degree $f_{>0}(\Gamma)(3-d)$.
\end{itemize}
Thus we can define a map of degree $e(\Gamma)(d-2)-f_{>0}(\Gamma)(d-3)$
$$\bigotimes_{e\in E(\Gamma)}\langle\_,\_\rangle^{-1}_e\otimes proj:\ W^{\otimes 2e(\Gamma)}\otimes W^{\otimes l(\Gamma)}\rightarrow Sym(Cyc^*_+({V_B})[d-4]).$$
\end{cons}
\begin{df}
Fixing the same oriented stable ribbon graph $\Gamma$ and $I^q\in MCE (\mathcal{F}^{pq}(V_B))$ as in the previous constructions and the choices made we define 
\begin{equation}\label{main_constr_nc}
\rho_{I^q}(\Gamma)=(\bigotimes_{e\in E(\Gamma)}\langle\_,\_\rangle^{-1}_e\otimes proj)\circ{\sigma}\big(I^q|_{v_1}\otimes\ldots\otimes I^q|_{v_n}\big)\nu^{f_{tot}(\Gamma)}\gamma^{g(\Gamma)}\in Sym(Cyc^*({V_B})[d-4])\llbracket\gamma\rrbracket.
\end{equation}
Here 
$$\nu^{f_{tot}(\Gamma)}:=\prod_{\lambda\in B,v\in V(\Gamma)}\nu^{b_\lambda(v)}\prod_{l\in F_0(\Gamma)}\nu_{\lambda(l)},$$
where we recall that $F_0(\Gamma)$ denotes the set of faces without open boundaries,
each of which is thus colored by one element of $B$.
Further we used the standard isomorphism between invariants and co-invariants in characteristic zero to view $I^q|_{v_i}$ as an invariant tensors, which we can thus see as an element of $W^{\otimes r},$
for $r=H(v_i)$, the valency of the vertex $v_i$.
The order of these tensors is determined by the choice of ordering of vertices, the order of the $\nu$ does not matter since they have even degree.
\end{df}
\begin{lemma}\label{skadj123}
   Given $I^q\in MCE (\mathcal{F}^{pq}(V_B))$ the assignment \eqref{main_constr_nc} defines a map
\begin{equation}\label{Kon_gen}
\begin{split}
\rho_{I^q}:{s\mathcal{RG}^B_{c,\bullet}}^\vee[2d-6]&\rightarrow Sym(Cyc^*(V_B)[4-d])\llbracket\gamma\rrbracket,\\
\Gamma&\mapsto \rho_{I^q}(\Gamma).
\end{split}
\end{equation}
In particular, it is independent of the choices made.
\end{lemma}
\begin{proof}
The fact that the degrees work out follows by counting together \begin{equation*} 
\begin{split}&-e+(d-2)e-f_{>0}(d-3)+(d-3)(-2v+c+\sum_{\lambda\in B,v\in V}b_\lambda(v)+2\sum_{v_i}g_i-2g-\sum_{\lambda\in B,v\in V}b_\lambda(v)-f_{0})\\
&=(d-3)(-2g-2v+e+c-f+2\sum_{v_i}g_i)=6-2d
\end{split}\end{equation*}
and definition \ref{armg} of the arithmetic genus. 
This map is independent of the choices made as follows:
\begin{itemize}
    \item Let $\Gamma$ be a stable ribbon graph as in the previous construction. Assume that we chose a different ordering of vertices, of the cycles belonging to each vertex and of a half-edge in each cycle,
    but for now take the same orientation datum and ordering of half-edges belonging to each internal edge.
    Let us call the initial permutation as in construction \ref{perm} $\sigma_1:H(\Gamma)\rightarrow H(\Gamma)$ and the new one $\sigma_2:H(\Gamma)\rightarrow H(\Gamma)$.
    Then we have that 
    $${\sigma_1}\big(I^q|_{v_1}\otimes\ldots\otimes I^q|_{v_n}\big)=\tau\circ{\sigma_2}\big(I^q|_{v_{\Tilde{1}}}\otimes\ldots\otimes I^q|_{v_{\Tilde{n}}}\big),$$
    where on the right-hand-side we use the new ordering of the vertices to order the input tensors.
    Further $\tau$ is the map determined by the following permutation: It is the identity on the half-edges belonging to internal edges.
    On the half-edges belonging to leafs it can be written as a permutation on $f_{> 0}(\Gamma)$ letters applied to $f_{> 0}(\Gamma)$ cyclic permutations (with regard to the cyclic ordering of the leafs in each face of the graph).
    Further there is no sign showing up because all the $|I^q(v_i)|$ are even.
    
    Since in the definition of $\rho_{I^q}$, described in the second bullet point of construction \ref{con_proj},
    we project to the quotient determined by exactly such permutations $\tau$ we have that $\rho_{I^q}$ is independent of 
    a different ordering of vertices, of the cycles belonging to each vertex and of a half-edge in each cycle.
    \item Now assume that we chose a different ordering of the half-edges belonging to an internal edge respectively.
    Since the pairings $\langle\_,\_\rangle^{-1}_{ij}$ are symmetric it follows that this gives the same result $\rho_{I^q}$.
  \end{itemize}  
Lastly this map is compatible with the orientation: Since the pairings have odd degree if
we change the orientation by an odd permutation the result changes by a factor of $(-1)$.
\end{proof}
\begin{lemma}\label{bra_comp}
    Given $I^q\in MCE (\mathcal{F}^{pq}(V_B))$ we have that $\rho_{I^q}(\{\_,\_\})=\{\rho_{I^q}(\_)\rho_{I^q}(\_)\}.$

\end{lemma}
\begin{proof}
    Recall that the bracket on $Sym(Cyc^*({V_B})[d-4])\llbracket\gamma\rrbracket$ is defined by pairing each letter of a symmetric word in cyclic words with a letter of a symmetric word in cyclic words exactly once,
    by that joining two cyclic words and taking into account the Koszul rule. Furthermore, if two cyclic words consists of just one letter respectively we obtain
    a $\nu_\lambda$ weighted by the pairing of the two letters (which have matching boundary condition $(\lambda,\lambda)$, otherwise we get zero).

    On the other hand we can write 
    $$\{\Gamma_1,\Gamma_2\}=\sum_{l_1\in L(\Gamma_1),l_2\in L(\Gamma_1)} \frac{\Gamma_1\cup_{l_1,l_2} \Gamma_2}{c_{l_1}(\Gamma_1)c_{l_2}(\Gamma_2)},$$
    the sum over all leafs of the graphs and gluing the two graphs together along those (if they have matching boundary condition). Here $c_{l_1}(\Gamma_i)=n\in \mathbb{N}$ if there are automorphisms of $\Gamma_i$ that send $l_1$ to $n$ other leafs. 

    By definition \ref{main_constr_nc} we see that $\sum_{l_1\in L(\Gamma_1),l_2\in L(\Gamma_1)} \rho(\Gamma_1\cup_{l_1,l_2} \Gamma_2)$
    is given by pairing the tensor $(\bigotimes_{e\in E(\Gamma_1)}\langle\_,\_\rangle^{-1}_e)\circ{\sigma}\big(I^q|_{v_1}\otimes\ldots\otimes I^q|_{v_n}\big)$ and 
    the tensor $(\bigotimes_{e\in E(\Gamma_2)}\langle\_,\_\rangle^{-1}_e)\circ{\sigma}\big(I^q|_{v_1}\otimes\ldots\otimes I^q|_{v_n}\big)$
    together in all possible ways, using a letter $l_1$ and $l_2$ from each word and then projecting to symmetric words in cyclic words.
    However, compared to $\{\rho(\Gamma_1),\rho(\Gamma_2\}$ we are over-counting exactly $c_{l_1}(\Gamma_1)c_{l_2}(\Gamma_2)$ terms.
    Further, if both leafs that we glue come from a face with only one leaf respectively, we increase the number
    of empty faces by one and thus get one additional $\nu$ according to formula \ref{main_constr_nc}. This proves the result.
\end{proof}
\begin{lemma}\label{diff_comp}
Given $I^q\in MCE (\mathcal{F}^{pq}(V_B))$ we have that 
$\rho_{I^q}\circ \nabla=\nabla\circ\rho_{I^q}$ and
$\rho_{I^q}\circ \delta_i=\gamma\delta_i\circ\rho_{I^q}.$
\end{lemma}
\begin{proof}
This follows similarly to above lemma.
In the second equality the additional $\gamma$ comes from the fact that $\delta_i$ increases the genus by one.
\end{proof}
\begin{lemma}\label{diff_comp2r}
 Given $I^q\in MCE (\mathcal{F}^{pq}(V_B))$ we have that $\rho_{I^q}\circ (-d)=d\circ \rho_{I^q}$.
\end{lemma}
\begin{proof}
    We have 
    $$\rho_{I^q}(-d\Gamma)=-\rho_{I^q}(\sum_{(*)}\Gamma')=\sum_{(*)}(\bigotimes_{e\in E(\Gamma')}\langle\_,\_\rangle^{-1}_e\otimes proj)\circ\sigma\circ'(I^q|_{v_1}\otimes\ldots \otimes X\otimes\ldots I^q|_{v_n})\nu^{f_{tot}(\Gamma)}\gamma^{g(\Gamma)}.$$
    Here $(*)$ indicates that the sum is over isomorphism classes of stable ribbon graphs $\Gamma'$ such that when contracting one of their legs according to rule \ref{cont} we recover $\Gamma$.
    Say $\Gamma$ has $n$ vertices.
    Note that a graph $\Gamma'$ as before has at most n+1 vertices and $n-1$ of those actually `coincide' with $n-1$ ones of $\Gamma$; let us call the vertex of $\Gamma$ which is not matched $v_i$.
    That is how the second equality above follows, where we denote by $X$ the tensor that gets associated to the at most two vertices which are different (and which we assume wlog. to be numbered consecutively).
    Note further that $d$ preserves the arithmetic genus and the total number of empty colored faces.
    
    According to the rule how to contract \ref{cont} the sum over $(*)$ splits up into three summands: 
    \begin{enumerate}
        \item 
    Similarly to lemma \ref{bra_comp} we have that for the first summand, corresponding to rule \ref{item_contractedge} of definition \ref{cont}
    $$\sum_{(*.1)}(\bigotimes_{e\in E(\Gamma')}\langle\_,\_\rangle^{-1}_e\otimes proj)\circ\sigma'(\cdots\otimes X\otimes\cdots )=(
    \bigotimes_{e\in E(\Gamma)}\langle\_,\_\rangle^{-1}_e\otimes proj)\circ\sigma(\cdots\otimes  \frac{1}{2}\{I^q,I^q\}|_{v_i}\otimes\cdots).$$
 Note that here it is important that we had redefined the bracket on words of length one, see \ref{ce}, which corresponds to rule \ref{item_contractedge}, second case distinction.
 \item Again similarly to lemma \ref{diff_comp} we have that for the second summand, corresponding to rule \ref{item_contractloop1} of definition \ref{cont} 
 $$\sum_{(*.2)}(\bigotimes_{e\in E(\Gamma')}\langle\_,\_\rangle^{-1}_e\otimes proj)\circ\sigma'(\cdots\otimes X\otimes\cdots )=(
    \bigotimes_{e\in E(\Gamma)}\langle\_,\_\rangle^{-1}_e\otimes proj)\circ\sigma(\cdots\otimes  (\gamma\delta I^q)|_{v_i}\otimes\cdots).$$
Here we need to include a $\gamma$ since for this summand we have decreased the genus defect by one.
\item Further we have that for the third summand, corresponding to rule \ref{item_contractloop2} of definition \ref{cont} that 
$$\sum_{(*.3)}(\bigotimes_{e\in E(\Gamma')}\langle\_,\_\rangle^{-1}_e\otimes proj)\circ\sigma'(\cdots\otimes X\otimes\cdots )=(
    \bigotimes_{e\in E(\Gamma)}\langle\_,\_\rangle^{-1}_e\otimes proj)\circ\sigma(\cdots\otimes  \nabla I^q|_{v_i}\otimes\cdots).$$

 Again it is crucial how we defined $\nabla$ in definition \ref{nabla} with regards to the central extension
 when contracting consecutive letters and when a word consists of just two letters, which matches the second and third case distinction of rule \ref{item_contractloop2}.
 \end{enumerate}
 Since we have that $I^q$ is a Maurer-Cartan element, ie. in particular
 $$dI^q|_{v_i}=-(\frac{1}{2}\{I^q,I^q\}|_{v_i}+\nabla I^q|_{v_i}+\gamma\delta I^q|_{v_i})$$
 it follows that above sum is equal to
 $$=\sum_{i=1}^n(\bigotimes_{e\in E(\Gamma)}\langle\_,\_\rangle^{-1}_e\otimes proj)\circ\sigma\circ(I^q|_{v_1}\otimes\ldots dI^q|_{v_i}\cdots I^q|_{v_n})\gamma^{g(\Gamma)},$$
 which we can rewrite, since both $proj$ and $(\langle\_,\_\rangle^{-1}$ are chain maps, as 
 
 $$=d\circ\big(\bigotimes_{e\in E(\Gamma)}\langle\_,\_\rangle^{-1}_e\otimes proj\big)\circ\sigma\circ(I^q|_{v_1}\otimes\cdots \otimes I^q|_{v_n})\gamma^{g(\Gamma)}=d\rho(\Gamma).$$ \end{proof}
\begin{cons}\label{ext_sym}
    We extend the map \eqref{Kon_gen} as a map of algebras to 
$$Sym({s\mathcal{RG}^B_{c,\bullet}}^\vee[2d-6])\rightarrow Sym(Cyc^*(V_B)[d-2])\llbracket\gamma\rrbracket,$$
and finally shift both sides to obtain the map
\begin{equation}\label{sdt}
Sym({s\mathcal{RG}^B_{c,\bullet}}^\vee[2d-6])[3-d]\rightarrow Sym(Cyc^*(V_B)[d-2])\llbracket\gamma\rrbracket[3-d].\end{equation}
which by abuse of notation we denote by the same letter.
\end{cons}
\begin{theorem}\label{Kc2}
    For every set $B$ the map \eqref{sdt} defines a map of $(d-2)$-twisted BD algebras
    $$\rho_{I^q}: s\mathcal{RG}_B^d\llbracket\gamma\rrbracket\rightarrow \mathcal{F}^{pq}(V_B),$$
    which is functorial with respect to strict cyclic $A_\infty$-morphisms of quantum $A_\infty$-categories. 
\end{theorem}
\begin{proof}  \begin{itemize}
    \item 
  It is a map of algebras (whose multiplication have non-zero even degree) by construction.
\item By the first part of lemma \ref{diff_comp} and construction $\rho_{I^q}$ intertwines
the maps on both sides that we called suggestively $\nabla$, further by lemma \ref{diff_comp2r} it intertwines the internal differentials.
\item We proved that map \eqref{Kon_gen} is a map of shifted Lie algebras, in fact it is a map of a shifted Lie algebra to a shifted Poisson algebra.
Any such map determined a map of shifted Poisson algebras from the free algebra on the domain to the target and this is indeed how we defined the shifted Poisson structure on 
$s\mathcal{RG}_B^d\llbracket\gamma\rrbracket$ (up to a total shift).
\item Furthermore this implies that $\rho_{I^q}$ intertwines $\delta_e$ with the associated Chevalley-Eilenberg differential of the 
Lie algebra $\mathcal{F}^{pq}(V_B)$. This, together with the arguments in bullet point two imply
$$\rho_{I^q}\big((-d+\gamma\delta_i+\gamma\delta_e+\nabla)(\Gamma_1\cdot\ldots\cdot \Gamma_n)\big)=(d+\gamma\delta+\nabla)\rho_{I^q}(\Gamma_1\cdot\ldots\cdot \Gamma_n),$$
which finishes the first part of the proof.
\item Given $F:V_B\rightarrow V_C$ a strict cyclic $A_\infty$-morphisms of quantum  $A_\infty$-categories we have in particular a map of sets $B\rightarrow C$, which induces a map 
$$s\mathcal{RG}_C^d\llbracket\gamma\rrbracket\rightarrow s\mathcal{RG}_B^d\llbracket\gamma\rrbracket.$$
Since $F$ is strict and cyclic and we have by definition that $F^*I^q=J^q$ for the respective quantizations one shows that
$$\begin{tikzcd}
 \mathcal{F}^{pq}(V_C)\arrow{r}{F^*}&   \mathcal{F}^{pq}(V_B)\\
 s\mathcal{RG}_C^d\llbracket\gamma\rrbracket\arrow{r}\arrow{u}& s\mathcal{RG}_B^d\llbracket\gamma\rrbracket\arrow{u}
\end{tikzcd}$$
commutes, using the definition \ref{main_constr_nc} of the vertical maps.
\end{itemize}
\end{proof}
Given an $I\in MCE(\mathcal{F}^{fr}(V_B))$, we can construct a map
$$\rho_{I}:{\mathcal{RG}^B_{c,\bullet}}^\vee[2d-6]\rightarrow Sym(Cyc^*(V_B)[d-2)\llbracket\gamma\rrbracket.$$
It is given by almost the exactly the same construction, just that we modify the map \ref{iso_sc} by omitting all shifts. 
Furthermore, in the same way as in lemma \ref{diff_comp} and \ref{diff_comp2r} we can prove that this maps intertwines the algebraic structures.
In the same way as in construction \ref{ext_sym} we get a map, also denoted
\begin{equation}\label{sigs}
\rho_{I}:Sym({\mathcal{RG}^B_{c,\bullet}}^\vee[2d-6])\llbracket\gamma\rrbracket[3-d]\rightarrow Sym(Cyc^*(V_B)[d-4])\llbracket\gamma\rrbracket[3-d]\end{equation}
for which we can argue exactly the same as in theorem \ref{Kc2} that we have: 
\begin{theorem}\label{Kc2.1}
Given $I\in MCE(\mathcal{F}^{fr}(V_B))$ the map \eqref{sigs} is a map of $(d-2)$-twisted BD algebras
    $$\rho_{I}: \mathcal{RG}_B^d\llbracket\gamma\rrbracket\rightarrow \mathcal{F}^{pq}(V_B),$$
    functorial with respect to strict cyclic $A_\infty$-morphisms.
\end{theorem}
Again using the same construction, but omitting some shifts, given $I\in MCE(\mathcal{F}^{fr}(V_B))$, we can construct a map
$$\rho_{I}:{\mathcal{RT}^B_{c,\bullet}}^\vee[d-3]\rightarrow (Cyc^*_+(V_B)[-1]).$$
Indeed we land in $Cyc^*_+(V_B)[-1]$ since ribbon trees have by definition genus zero and only one face, which needs to have a non-zero number of open boundaries. 
In the same way as in lemmas \ref{bra_comp} and \ref{diff_comp2r} we can argue that this map is compatible with the algebraic structures. 
As in construction \ref{ext_sym} we get a map, also denoted
\begin{equation}\label{www}
\rho_{I}:Sym({\mathcal{RT}^B_{c,\bullet}}^\vee[3-d])\rightarrow Sym(Cyc^*_+(V_B)[-1])
\end{equation}
for which we can argue analogously to same as in theorem \ref{Kc2} that we have:
\begin{theorem}\label{Kc1}
Given $I\in MCE(\mathcal{F}^{fr}(V_B))$ the map \eqref{www} is a map of $(d-2)$-shifted Poisson algebras
    $$\rho_{I}: \mathcal{RT}_B^d\rightarrow \mathcal{F}^{fr}(V_B),$$
    functorial with respect to strict cyclic $A_\infty$-morphisms.
\end{theorem}
Theorems \ref{Kc2.1}, \ref{Kc2} and \ref{Kc1} fit together with 
diagram \eqref{rbg_comp} and map \eqref{nc_dq} together:
\begin{theorem}\label{comp_rg_al}
Given $I^q\in MCE(\mathcal{F}^{pq}(V_B))$ denote by $I=p(I^q)\in MCE(\mathcal{F}^{fr}(V_B))$ its dequantization. Then following diagram commutes:
$$\begin{tikzcd}
\mathcal{F}^{fr}(V_B)&\mathcal{F}^{pq}(V_B)\arrow{l}{p}&
\mathcal{F}^{pq}(V_B)\arrow[equal]{l}\\
\mathcal{RT}^d_{B}\arrow{u}{\rho_{I}}&\mathcal{RG}^d_{B}\llbracket\gamma\rrbracket\arrow[swap]{l}\arrow{r}\arrow{u}{\rho_{I}}&
s\mathcal{RG}^d_{B}\llbracket\gamma\rrbracket\arrow[bend left=20]{ll}\arrow{u}{\rho_{I^q}}.
    \end{tikzcd}$$
\end{theorem}
\begin{proof}
    This basically follows by construction. Note that it is essential that the dequantization maps
    project to symmetric words of length one, both for (stable) ribbon graphs and for cyclic words. 
\end{proof}
Recall the two Maurer-Cartan elements from theorems \ref{RS_mce} and \ref{sRG_mce} in the respective graph complexes. We state following lemmas, which are straightforward to prove:
\begin{lemma}\label{comp_1}
Letting $B$ be a set and given $I^q\in MCE(\mathcal{F}^{pq}(V_B))$ we have that $\rho_{I^q}(S)=I^q$.
\end{lemma}
\begin{lemma}\label{comp_2}
   Letting $B$ be a set and given $I\in MCE(\mathcal{F}^{pq}(V_B))$ we have $\rho_I(D)=I.$
\end{lemma}
Thus we can induce the following theorem:
 \begin{theorem}\label{aiaiai}
Given a quantization $I^q$ of a cyclic $A_\infty$-category $\mathcal{C}$ there is a commutative diagram as follows:        
    $$\begin{tikzcd}
      \mathcal{F}^{cl}(\mathcal{C})&  \mathcal{F}^{q}(\mathcal{C})\arrow{l}\\
      \mathcal{RT}^{d,tw}_{Ob\mathcal{C}}\arrow{u}{\rho_{I}}&s\mathcal{RG}^{d,tw}_{Ob\mathcal{C}}\llbracket\gamma\rrbracket\arrow{u}{\rho_{I^q}}\arrow{l}.
    \end{tikzcd}$$
    Here the horizontal maps are dequantization maps, the right vertical map is a map of $(d-2)$-twisted BD algebras and 
    the left vertical maps is a map of $(d-2)$-shifted Poisson algebras
    \end{theorem} 
    \begin{proof}
  Indeed this follows from the previous theorem and definitions and the consecutive two lemmas by twisting the outer diagram by the compatible Maurer-Cartan elements.      
    \end{proof}
Lastly assume we are given an essentially finite (in the sense of definition \ref{essfin}) cyclic $A_\infty$-category $\mathcal{C}$ with two choices of skeleta $(\mathcal{C},Sk\mathcal{C}_1)$ and $(\mathcal{C},Sk\mathcal{C}_2)$.
Further assume that we are given compatible quantizations of these three categories in the sense of definition \ref{qesfq}.
Let us denote by $\widebar{\mathcal{C}}$ the full subcategory on the images of $Sk\mathcal{C}_1$ and $Sk\mathcal{C}_2$.
$\widebar{\mathcal{C}}$ can be endowed with the structure of a quantum $A_\infty$-category, it has finitely many objects.
Its quantization gives another finite presentation of the quantization of $\mathcal{C}$ and it is compatible with the ones by $\mathcal{C}_1$ and $\mathcal{C}_2$ by definition.
We summarize the situation by following diagram, where all the maps are equivalences of quantum $A_\infty$-categories in the sense of \ref{BDqi}.
\begin{equation}\label{dall}\begin{tikzcd}
    &\mathcal{C}&\\
   Sk\mathcal{C}_1\arrow[hookrightarrow]{ur}\arrow[hookrightarrow]{r}&\widebar{\mathcal{C}}\arrow[hookrightarrow]{u}& Sk\mathcal{C}_2\arrow[hookrightarrow]{ul}\arrow[hookrightarrow]{l}
\end{tikzcd}\end{equation}

Then, the vertical right maps of theorem \ref{aiaiai} induced from the quantization of $\mathcal{C}_1$ and from the quantization of $\mathcal{C}_2$ are compatible as follows:
 \begin{equation}\label{unabhai}
\begin{tikzcd}
      \mathcal{F}^{q}(Sk\mathcal{C}_1)&  \mathcal{F}^{q}(\widebar{\mathcal{C}})\arrow{l}[swap]{\simeq}\arrow{r}{\simeq}&\mathcal{F}^{q}(Sk\mathcal{C}_2)\\
      s\mathcal{RG}^{d,tw}_{ObSk\mathcal{C}_1}\llbracket\gamma\rrbracket\arrow{u}&s\mathcal{RG}^{d,tw}_{Ob\widebar{\mathcal{C}}}\llbracket\gamma\rrbracket\arrow{u}\arrow{r}\arrow{l}&s\mathcal{RG}^{d,tw}_{ObSk\mathcal{C}_2}\llbracket\gamma\rrbracket\arrow{u}.
    \end{tikzcd}\end{equation}
Note that here the left and right horizontal map are just given by extending by zero on elements which are not colored by $Sk\mathcal{C}_1$ respectively $Sk\mathcal{C}_2$.
Further $s\mathcal{RG}^{d,tw}_{ObSk\mathcal{C}_1}\llbracket\gamma\rrbracket$ and 
$s\mathcal{RG}^{d,tw}_{ObSk\mathcal{C}_2}\llbracket\gamma\rrbracket$ are in fact (non-canonically) isomorphic, since $ObSk\mathcal{C}_1$ and $ObSk\mathcal{C}_2$ are isomorphic sets, both being the skeleton of the same category.
Thus diagram \ref{unabhai} tells us that its left and right vertical map are equivalent in the derived category, up to extending these maps by zero to a bigger domain.

\subsection{Commutative World}\label{cksc}
In this section we explain how to connect the  graph complex world, section \ref{grs}, with the world of cyclic $L_\infty$-algebras and their quantizations, section \ref{cw}.
We do so using a slight generalization of Penkava's cocycle construction \cite{pe96}, generalizing it to graphs with leafs.

 Given $V$ a cyclic chain complex of odd degree $d$ together with $I^q\in MCE (\mathcal{O}^{pq}(V))$ we construct a map 
 $$\theta_{I^q}: s\mathcal{G}^d\llbracket\hbar\rrbracket\rightarrow \mathcal{O}^{pq}(V).$$

We explain how to to so in multiple steps.
We begin by associating certain tensors, denoted $I^q|_{v_i}$, to each vertex $v_i$ of a connected stable graph.
\begin{cons}
  Let $\Gamma$ be a connected stable graph and $v_i$ one of its vertices:\begin{itemize}
      \item 
 To such a vertex $v_i$ of $\Gamma$ is associated a loop defect $o_i\in\mathbb{N}_{\geq0}$ by definition, see \ref{sgrdf}. 
 \item  Given a vertex $v_i$ of $\Gamma$ we define a map of degree $(d-3)o_i$, denoted
  \begin{equation}\label{iso_scg}
(\_)|_{v_i}:\mathcal{O}^{pq}(V)\rightarrow C^*(V)\end{equation}
by restricting to the power $\hbar^{o_i}$.
\item Thus for $I^q\in MCE(\mathcal{O}^{pq}(V))$, which has degree $(3-d)$, we have that $${|{I}^q|_{v_i}|=(3-d)(1-o_i)}.$$
\end{itemize}
\end{cons}
Next we explain how the datum of an oriented stable graph together with some auxiliary choices defines a permutation of the half-edges of such a stable  graph:
\begin{cons}\label{perm_g}
 Assume that $\Gamma$  has $n$ vertices.
 \begin{itemize}
     \item 
Chose an ordering of its vertices and at each vertex $v_i$  an ordering of the half-edges, denoted, 
$$h_{v_i,1}\cdots h_{v_i,|v_1|}.$$
With the ordering of the vertices together this defines an ordering of all half edges of $\Gamma$
$$h_{v_1,1}\cdots h_{v_1,|v_1|}h_{v_2,1}\cdots h_{v_r,|v_r|}.$$
\end{itemize}
Next assume that $\Gamma$ is oriented, ie. we have an ordering $(e_1\cdots e_n)$ of the internal edges.
\begin{itemize}
    \item 
Chose an ordering of the half edges belonging to a given internal edge.
Together with the orientation this gives an ordering of the half edges belonging to the internal edges 
$$h_{1_x}h_{1_y}\cdots h_{n_x} h_{n_y}.$$ 
\end{itemize}
Given the choices made and given the orientation we can define a permutation of all the half-edges as follows
$$
\sigma: H(\Gamma)\rightarrow H(\Gamma) $$
$$
h_{v_1,1}\cdots h_{v_1,|v_1|}h_{v_2,1}\cdots h_{v_r,|v_r|}\mapsto h_{1_x}h_{1_y}\cdots h_{n_x} h_{n_y}l_{1,1}\cdots l_{1,m_1}l_{2,1}\cdots l_{b,m_b} $$
\begin{itemize}
    \item 

We begin the ordering of the leaves on the right hand side by the first leave that appears on the left hand side, then continue with the next one and so on.
\end{itemize}
\end{cons}
 We still fix the same oriented stable graph $\Gamma$ and the choices of orderings made.
 Denote $U:=V^\vee[-1]$. We next explain how to define, given these choices, a map
$$U^{\otimes h(\Gamma)}\rightarrow C^*(V)$$
 of degree $e(\Gamma)(d-2)$.
\begin{cons}\label{con_proj_g}
The permutation $\sigma$ induces a map which we denote by abuse of notation by the same letter
$${\sigma}:U^{\otimes h(\Gamma)}\rightarrow U^{\otimes h(\Gamma)}\cong U^{\otimes 2e(\Gamma)}\otimes U^{\otimes l(\Gamma)}.$$ 
\begin{itemize}
    \item 
Denote by 
$$proj: U^{\otimes l(\Gamma)}\rightarrow C^*(V)$$
the projection to symmetric words, which has degree 0.
\end{itemize}
Thus we can define a map of degree $e(\Gamma)(d-2)$
$$\bigotimes_{e\in E(\Gamma)}\langle\_,\_\rangle^{-1}\otimes proj:\ U^{\otimes 2e(\Gamma)}\otimes U^{\otimes l(\Gamma)}\rightarrow C^*(V).$$
\end{cons}
\begin{df}
Fixing the same oriented stable graph $\Gamma$ as in the previous constructions and the choices made we define 
\begin{equation}\label{main_constr_c}
\theta_{I^q}(\Gamma)=(\bigotimes_{e\in E(\Gamma)}\langle\_,\_\rangle^{-1}\otimes proj)\circ{\sigma}\big(I^q|_{v_1}\otimes\ldots\otimes I^q|_{v_n}\big)\hbar^{bt(\Gamma)}\in C^*(V)\llbracket\hbar\rrbracket.
\end{equation}
Here we use the the standard isomorphism between invariants and co-invariants in characteristic zero to view $I^q|_{v_i}$
as an invariant tensors, which we can thus see as an element of $U^{\otimes r},$ for $r=|v_i|$.
Further the order of the tensors is determined by the choice of ordering of vertices. 
\end{df}
\begin{lemma}
   Given $I^q\in MCE (\mathcal{O}^{pq}(V))$ the assignment  \ref{main_constr_c} defines a degree zero map of vector spaces
\begin{equation}\label{Kon_gen_2}
\begin{split}
\theta_{I^q}:{s\mathcal{G}_{c,\bullet}}^\vee[d-3]&\rightarrow C^*(V)\llbracket\hbar\rrbracket,\\
\Gamma&\mapsto \rho_{I^q}(\Gamma).
\end{split}
\end{equation}
In particular, it is independent of the choices made.
\end{lemma}
\begin{proof}
The fact that the degrees work out follows by counting together 
$$3-d=-e+(d-2)e+(3-d)(bt+v-\sum_{v_i}o_i)$$
and definition \ref{betti} of the betti number. 
This map is independent of the choices made as follows:
\begin{itemize}
    \item Let $\Gamma$ be a stable graph as in the previous construction.
    Assume that we chose a different ordering of vertices and of the half-edges at each vertex, but for now take the same orientation datum and ordering of half-edges belonging to each internal edge.
    Let us call the initial permutation as in construction \ref{perm_g} $\sigma_1:H(\Gamma)\rightarrow H(\Gamma)$ and the new one $\sigma_2:H(\Gamma)\rightarrow H(\Gamma)$.
    Then we have that 
    $${\sigma_1}\big(I^q|_{v_1}\otimes\ldots\otimes I^q|_{v_n}\big)=\tau\circ{\sigma_2}\big(I^q|_{v_{\Tilde{1}}}\otimes\ldots\otimes I^q|_{v_{\Tilde{n}}}\big),$$
    where on the right-hand-side we use the new ordering of the vertices to order the input tensors.
    Here $\tau$ is a general permutation.
    There is no sign showing up because all the $|I^q(v_i)|$ are even.
    
    Since in the definition of $\theta_{I^q}$, described in the second bullet point of construction \ref{con_proj_g}, we project to the quotient determined by all permutations $\tau$, that is symmetric words,
    we have that $\theta_{I^q}$ is independent of a different ordering of vertices and of an ordering of the half-edge at each vertex.
    \item Now assume that we chose a different ordering of the half-edges belonging to an internal edge respectively.
    Since the pairings $\langle\_,\_\rangle^{-1}$ are symmetric it follows that this gives the same result $\theta_{I^q}$.
  \end{itemize}  
Lastly this map is compatible with the orientation: Since the pairings have odd degree if
we change the orientation by an odd permutation the result changes by a factor of $(-1)$.
    
\end{proof}
\begin{lemma}\label{bra_comp_2}
     Given $I^q\in MCE (\mathcal{O}^{pq}(V))$ we have that 
     $\theta_{I^q}(\{\_,\_\})=\{\theta_{I^q}(\_)\theta_{I^q}(\_)\}.$

\end{lemma}
\begin{proof}
    Recall that the bracket on 
    $C^*(V)\llbracket\hbar\rrbracket$
    is defined by pairing each letter of a symmetric word with a letter of a symmetric word exactly once,
    by that joining the two symmetric words and taking into account the Koszul rule. 

    On the other hand we can write 
    $$\{\Gamma_1,\Gamma_2\}=\sum_{l_1\in L(\Gamma_1),l_2\in L(\Gamma_1)} \frac{\Gamma_1\cup_{l_1,l_2} \Gamma_2}{c_{l_1}(\Gamma_1)c_{l_2}(\Gamma_2)},$$
    the sum over all leafs of the graphs and gluing the two graphs together along those.
    Here $c_{l_1}(\Gamma)=n\in \mathbb{N}$ if there are automorphisms of $\Gamma$ that send $l_1$ to $n$ other leafs. 

    By definition \ref{main_constr_c} we see that $\sum_{l_1\in L(\Gamma_1),l_2\in L(\Gamma_1)} \theta(\Gamma_1\cup_{l_1,l_2} \Gamma_2)$
    is given by pairing the tensor 
    $(\bigotimes_{e\in E(\Gamma_1)}\langle\_,\_\rangle^{-1})\circ{\sigma}\big(I^q|_{v_1}\otimes\ldots\otimes I^q|_{v_n}\big)$
    and the tensor $(\bigotimes_{e\in E(\Gamma_2)}\langle\_,\_\rangle^{-1})\circ{\sigma}\big(I^q|_{v_1}\otimes\ldots\otimes I^q|_{v_n}\big)$
    together in all possible ways, using a letter $l_1$ and $l_2$ from each word and then projecting to symmetric words.
    However, compared to $\{\theta(\Gamma_1),\theta(\Gamma_2\}$ we are overcounting exactly $c_{l_1}(\Gamma_1)c_{l_2}(\Gamma_2)$ terms.
    This proves the result. 
\end{proof}
\begin{lemma}\label{diff_comp_2}
Given $I^q\in MCE (\mathcal{O}^{pq}(V))$ we have $\theta_{I^q}\circ \Delta_i=\hbar\Delta\circ\theta_{I^q}$.
\end{lemma}
\begin{proof}
This follows similarly to above lemma.
In the second equality the additional $\hbar$ comes from the fact that $\Delta_i$ increases the loop defect by one.
\end{proof}
\begin{lemma}\label{diff_comp2}
   Given $I^q\in MCE (\mathcal{O}^{pq}(V))$ we have that 
   $\theta_{I^q}\circ (-d)=d\circ \theta_{I^q}$.
\end{lemma}
\begin{proof}
    We have 
    $$\theta_{I^q}(-d\Gamma)=-\theta_{I^q}(\sum_{(*)}\Gamma')=-\sum_{(*)}(\bigotimes_{e\in E(\Gamma')}\langle\_,\_\rangle^{-1}\otimes proj)\circ\sigma'\circ(I^q|_{v_1}\otimes\ldots \otimes X\otimes\ldots I^q|_{v_n})\hbar^{bt(\Gamma)}.$$
    Here $(*)$ indicates that the sum is over isomorphism classes of stable graphs $\Gamma'$
    such that when contracting one of their legs according to rule \ref{cont_g} we recover $\Gamma$.
    Say $\Gamma$ has $n$ vertices.
    Note that a graph $\Gamma'$ as before has at most n+1 vertices and $n-1$ of those actually `coincide' with $n-1$ ones of $\Gamma$;
    let us call the vertex of $\Gamma$ which is not matched $v_i$.
    That is how the second equality above follows, where we denote by $X$ the tensor that gets associated to the 
    at most two vertices which are different (and which we assume wlog. to be numbered consecutively).
    Note further that $d$ preserves the betti number.
    
    According to the rule how to contract \ref{cont_g} the sum over $(*)$ splits up into two summands: 
    \begin{enumerate}
        \item 
    Similarly to lemma \ref{bra_comp_2} we have that for the first summand, corresponding to rule \ref{cont_g}.1
    $$-\sum_{(*.1)}(\bigotimes_{e\in E(\Gamma')}\langle\_,\_\rangle^{-1}\otimes proj)\circ\sigma'(\cdots\otimes X\otimes\cdots )=-(
    \bigotimes_{e\in E(\Gamma)}\langle\_,\_\rangle^{-1}\otimes proj)\circ\sigma(\cdots\otimes  \frac{1}{2}\{I^q,I^q\}|_{v_i}\otimes\cdots).$$
 \item Again similarly to lemma \ref{diff_comp_2} we have that for the second summand, corresponding to rule \ref{cont_g}.2
 $$-\sum_{(*.2)}(\bigotimes_{e\in E(\Gamma')}\langle\_,\_\rangle^{-1}\otimes proj)\circ\sigma'(\cdots\otimes X\otimes\cdots )=-(
    \bigotimes_{e\in E(\Gamma)}\langle\_,\_\rangle^{-1}\otimes proj)\circ\sigma(\cdots\otimes  (\hbar\Delta I^q)|_{v_i}\otimes\cdots).$$
Here we need to include a $\hbar$ since for this summand we have decreased the betti number by one.
 \end{enumerate}
 Since we have that $I^q$ is a Maurer-Cartan element, ie. in particular
 $$dI^q|_{v_i}=-(\frac{1}{2}\{I^q,I^q\}|_{v_i}+\hbar\Delta I^q|_{v_i})$$
 it follows that above sum is equal to
 $$=\sum_{i=1}^n(\bigotimes_{e\in E(\Gamma)}\langle\_,\_\rangle^{-1}\otimes proj)\circ\sigma\circ(I^q|_{v_1}\otimes\ldots dI^q|_{v_i}\cdots I^q|_{v_n})\hbar^{bt(\Gamma)},$$
 which we can rewrite, since both $proj$ and $\langle\_,\_\rangle^{-1}$ are chain maps, as 
 
 $$=d\circ\big(\bigotimes_{e\in E(\Gamma)}\langle\_,\_\rangle^{-1}\otimes proj\big)\circ\sigma\circ(I^q|_{v_1}\otimes\cdots \otimes I^q|_{v_n})\hbar^{g(\Gamma)}=d\theta_{I^q}(\Gamma).$$ \end{proof}
\begin{cons}\label{ext_sym_g}
    We extend the map \ref{Kon_gen_2} as a map of algebras to 
\begin{equation}\label{sdnit}
Sym({s\mathcal{G}_{c,\bullet}}^\vee[d-3])\rightarrow C^*(V)\llbracket\hbar\rrbracket,\end{equation}
which by abuse of notation we denote by the same letter.
\end{cons}
\begin{theorem}\label{Kc3}
    The map \eqref{sdnit} is a map of $(d-2)$-twisted BD algebras
    $$\theta_{I^q}: s\mathcal{G}^d\llbracket\hbar\rrbracket\rightarrow \mathcal{O}^{pq}(V),$$
    which is functorial with respect to strict cyclic $L_\infty$-morphisms.
\end{theorem}
\begin{proof}  \begin{itemize}
    \item 
  It is a map of algebras by construction.
\item By the first part of lemma \ref{diff_comp_2} and construction $\theta_{I^q}$ intertwines $\Delta_i$ with $\hbar\Delta$, further by lemma \ref{diff_comp2} it intertwines the internal differentials.
\item We proved that map \ref{Kon_gen_2} is a map of shifted Lie algebras,
in fact it is a map of a shifted Lie algebra to a shifted Poisson algebra.
Any such map determines a map of shifted Poisson algebras from the free algebra on the domain to the target and this is indeed
how we defined the shifted Poisson structure on 
$s\mathcal{G}_B^d\llbracket\hbar\rrbracket$.
\item Furthermore this implies that $\theta_{I^q}$ intertwines $\Delta_e$ with the associated Chevalley-Eilenberg differential of the Lie algebra $\mathcal{O}^{pq}(V)$.
This, together with the arguments in bullet point two imply
$$\theta_{I^q}\big((-d+\hbar\Delta_e+\Delta_i)(\Gamma_1\cdot\ldots\cdot \Gamma_n)\big)=(d+\hbar\Delta)\rho_{I^q}(\Gamma_1\cdot\ldots\cdot \Gamma_n),$$
which finishes the first part of the proof.
\item Given $F:V\rightarrow W$ a strict cyclic $L_\infty$-morphism such that $F^*I^q=J^q$ one can check that 
\begin{equation}\label{hdlm}\begin{tikzcd}
 \mathcal{F}^{pq}(W)\arrow{rr}{F^*}&&   \mathcal{F}^{pq}(V)\\
&s\mathcal{G}^d\llbracket\hbar\rrbracket\arrow[swap]{ur}{\theta_{I^q}}\arrow{ul}{\theta_{J^q}}&
\end{tikzcd}\end{equation}
commutes, using the definition \ref{main_constr_c}.
\end{itemize}\end{proof}

Given an $I\in MCE(\mathcal{O}^{fr}(V))$, we can construct a map
$$\theta_{I}:{\mathcal{G}_{c,\bullet}}^\vee[d-3]\rightarrow C^*(V)\llbracket\hbar\rrbracket$$
given by exactly the same construction. Furthermore, in the same way as in lemma \ref{bra_comp_2} and \ref{diff_comp2} we can prove that this maps intertwines the algebraic structures.
In the same way as in construction \ref{ext_sym_g} we get a map, also denoted
\begin{equation}\label{hhd}
    \theta_{I}:Sym({\mathcal{G}_{c,\bullet}}^\vee[d-3])\llbracket\hbar\rrbracket\rightarrow C^*(V)\llbracket\hbar\rrbracket
    \end{equation}
for which we can argue exactly the same as in theorem \ref{Kc3} that we have: 
\begin{theorem}\label{Kc3.1}
Given $I\in MCE(\mathcal{O}^{fr}(V))$ the map \eqref{hhd} is a map of $(d-2)$-twisted BD algebras
    $$\theta_{I}: \mathcal{G}^d\llbracket\hbar\rrbracket\rightarrow \mathcal{O}^{pq}(V),$$
    functorial with respect to strict cyclic $L_\infty$-morphisms.
\end{theorem}
Again, using the same construction, given $I\in MCE(\mathcal{O}^{fr}(V))$, we can construct a map
$$\theta_{I}:{\mathcal{T}_{c,\bullet}}^\vee[d-3]\rightarrow C^*_+(V).$$
Indeed we land in $C^*_+(V)$ since trees needs to have a non-zero number of leafs.
In the same way as in lemmas \ref{bra_comp_2} and \ref{diff_comp2} we can argue that this map is compatible with the algebraic structures.
As in construction \ref{ext_sym_g} we get a map, also denoted
\begin{equation}\label{ibmd}
\theta_{I}:Sym({\mathcal{T}_{c,\bullet}}^\vee[d-3])\rightarrow C^*_+(V)\end{equation}
for which we can argue analogously to same as in theorem \ref{Kc3} that we have:
\begin{theorem}\label{Kc4}
Given $I\in MCE(\mathcal{O}^{fr}(V))$ the map \eqref{ibmd} is a map of $(d-2)$-shifted Poisson algebras
    $$\theta_{I}: \mathcal{T}^d\rightarrow \mathcal{O}^{fr}(V),$$
    functorial with respect to strict cyclic $L_\infty$-morphisms.
\end{theorem}
Theorems \ref{Kc3.1}, \ref{Kc4} and \ref{Kc3} fit together with 
diagram \eqref{ag_dq} and the dequantization map \eqref{fc_dq} together:
\begin{theorem}\label{comp_gr_al}
Given $I^q\in MCE(\mathcal{O}^{pq}(V))$ denote by $I=p(I^q)\in MCE(\mathcal{O}^{fr}(V))$ its dequantization, see equation \eqref{fc_dq}. Then following diagram commutes:
$$\begin{tikzcd}
\mathcal{O}^{fr}(V)&\mathcal{O}^{pq}(V)\arrow{l}{p}&
\mathcal{O}^{pq}(V)\arrow[equal]{l}\\
\mathcal{T}^d\arrow{u}{\theta_{I}}&\mathcal{G}^d\llbracket\hbar\rrbracket\arrow[swap]{l}\arrow{r}\arrow{u}{\theta_{I}}&
s\mathcal{G}^d\llbracket\hbar\rrbracket\arrow[bend left=20]{ll}\arrow{u}{\theta_{I^q}}.
    \end{tikzcd}$$
\end{theorem}
\begin{proof}
    This follows by unraveling the constructions. 
\end{proof}
We recall the singled-out Maurer-Cartan elements from theorems \ref{S_mce} and \ref{sG_mce} and state following lemmas, which are straightforward to prove:
\begin{lemma}\label{comp_3}
Given $I^q\in MCE(\mathcal{O}^{pq}(V))$ we have that $\theta_{I^q}(G)=I^q$.
\end{lemma}
\begin{lemma}\label{comp_4}
    Given $I\in MCE(\mathcal{O}^{fr}(V))$ we have $\theta_I(T)=I.$
\end{lemma}
Thus we can induce the following theorem:
 \begin{theorem}\label{fde}
Given a quantization $I^q$ of a cyclic $L_\infty$-algebra $L$ there is a commutative diagram as follows:        
    $$\begin{tikzcd}
      \mathcal{O}^{cl}(L)&  \mathcal{O}^{q}(L)\arrow{l}\\
      \mathcal{T}^{d,tw}\arrow{u}{\theta_{I}}&s\mathcal{G}^{d,tw}\llbracket\hbar\rrbracket\arrow{u}{\theta_{I^q}}\arrow{l}.
    \end{tikzcd}$$
    Here the horizontal maps are dequantization maps, the right vertical map is a map of $(d-2)$-twisted BD algebras and the left vertical maps is a map of $(d-2)$-shifted Poisson algebras
    \end{theorem} 
    \begin{proof}
  Indeed this follows from the previous theorem and the consecutive two lemmas by twisting the outer diagram by the compatible Maurer-Cartan elements.      
    \end{proof}

\newpage
\section{Compatibility of Quantized LQT Map, Commutative- and Non-Commutative Kontsevich Cocycle Construction}
In this section we explain how cyclic $A_\infty$-categories and their quantizations from section \ref{ncw}, cyclic $L_\infty$-algebras and their quantizations  from section \ref{cw},
(stable) ribbon graphs (section \ref{rbgs}) and (stable) graphs from section \ref{grs} fit together, using the quantized LQT map from section \ref{slqtq}, Kontsevich's cocycle construction from section \ref{cksc}, its commututive analogue (section \ref{ksc}) and the map from section \ref{vsgzg}. We do this in following main theorem:
\begin{theorem}\label{Main}
Assume we are given $I^q\in MCE(\mathcal{F}^{pq}(V_{Ob\mathcal{C}}))$,
a quantization of an essentially finite dimension $d$ cyclic $A_\infty$-category $\mathcal{C}$ with its associated hamiltonian 
$I\in MCE(\mathcal{F}^{fr}(V_{Ob\mathcal{C}}))$.
For $N\in \mathbb{N}$ denote by $\mathfrak{gl}_NA_\mathcal{C}$ the associated cyclic commutator $L_\infty$-algebra with values in $N\times N$-matrices
and by $I^q_N$ its quantization determined by theorem \ref{LQT_q}.
Then following diagram commutes.
Further it is functorial with respect to strict cyclic morphisms of quantizations of essentially finite cyclic $A_\infty$-categories
and independent up to isomorphism from the choice of a finite presentation.   
$$ \begin{tikzcd}[row sep=scriptsize, column sep=scriptsize]
& \mathcal{F}^{cl}(\mathcal{C})  \arrow[rr,"LQT^{cl}"] & & \mathcal{O}^{cl}(\mathfrak{gl}_NA_\mathcal{C})   \\ \mathcal{F}^{q}(\mathcal{C})\arrow[rr, "LQT^q"{xshift=-7pt}]  \arrow[ur] & & \mathcal{O}^{q}(\mathfrak{gl}_NA_\mathcal{C})\arrow[ur] \\
& \mathcal{RT}^{d,tw}_{Ob(Sk\mathcal{C})}\arrow[uu,dashed, "\rho_{I}^{tw}" {yshift=-10pt}]   & & \mathcal{T}^{d,tw}\arrow[ll,dashed]\arrow[ uu,"\theta_{I_N}^{tw}" {yshift=-8pt}]  \\
 s\mathcal{RG}^{d,tw}_{Ob(Sk\mathcal{C})}\llbracket\gamma\rrbracket \arrow[uu,"\rho_{I^q}^{tw}" {yshift=5pt}]\arrow[ur] & & s\mathcal{G}^{d,tw}\llbracket\hbar\rrbracket\arrow[ll]\arrow[ur]\arrow[uu,"\theta_{I^q_N}^{tw}" {yshift=-10pt}]\\
\end{tikzcd}
$$
The objects on the front face are $(d-2)$-twisted BD algebras, the vertical arrows are maps of $(d-2)$-twisted BD algebras, the upper horizontal arrow is a 2-weighted map of  BD algebras and the lower one is the predual to such a map. The back face commutes as a diagram $(d-2)$-shifted Poisson algebras. The maps in between the two faces are dequantization maps.   
\end{theorem}
\begin{proof}
Commutativity will follow from the next theorem \ref{main_compa}. More precisely, consider the outer commuting cube featuring in the next theorem. Rotating that cube around the z-axis by 90 degrees and then twisting that diagram by the compatible Maurer-Cartan in each corner (see lemma \ref{comp_1}, lemma \ref{comp_2} and lemma \ref{comp_3}, lemma \ref{comp_4}) gives the commutative cube of this theorem.

We have explained in theorem \ref{LQT_q} and \ref{LQT_cl} how the top face is functorial as claimed in the theorem. This also explains functoriality of the right face by diagram \ref{hdlm}. In the same cited theorems the independence of the top face on the choice of a finite presentation is explained. 

In theorem \ref{Kc2} and theorem \ref{Kc1} functoriality for the left face is explained. Further we have elaborated around \eqref{unabhai} how the left face is independent of the choice of finite presentation. The corresponding statements for the other faces can be deduced from these results.
\end{proof}

\begin{theorem}\label{main_compa}
Assume we are given an $I^q\in MCE(\mathcal{F}^{pq}(V_{Ob\mathcal{C}}))$, a quantization of an essentially finite dimension $d$ cyclic $A_\infty$-category $\mathcal{C}$ with its associated hamiltonian $I\in MCE(\mathcal{F}^{fr}(V_{Ob\mathcal{C}}))$. Fix the same notations as in the previous theorem. Then following diagram, where $B=ObSk\mathcal{C}$ finite, commutes.
$$\begin{tikzcd}
&\mathcal{O}^{fr}(\mathfrak{gl}_NV)&&\mathcal{O}^{pq}(\mathfrak{gl}_NV)\arrow[swap]{ll}{p}&&
\mathcal{O}^{pq}(\mathfrak{gl}_NV)\arrow[equal]{ll}\\
\mathcal{F}^{fr}(V_B)\arrow{ur}{LQT^{fr}_N}&&\mathcal{F}^{pq}(V_B)\arrow[swap,ll, "p" {xshift=10pt}]\arrow{ur}{LQT^{pq}_N}&&
\mathcal{F}^{pq}(V_B)\arrow[equal]{ll}\arrow{ur}{LQT^{pq}_N}&\\ 
&\mathcal{T}^d\arrow[dashed, uu,"\theta_{I_N}" {yshift=-8pt}]\arrow{dl}{\pi}&&\mathcal{G}^d\llbracket\hbar\rrbracket\arrow[swap,dashed]{ll}\arrow[dashed]{rr}\arrow[dashed, uu,"\theta_{I_N}" {yshift=-10pt}]\arrow{dl}&&
s\mathcal{G}^d\llbracket\hbar\rrbracket\arrow[bend left=12,dashed]{llll}\arrow[dashed, uu,"\theta_{I^q_N}" {yshift=-8pt}]\arrow{dl}{\pi}\\
\mathcal{RT}^d_{B}\arrow[uu,"\rho_{I}" {yshift=5pt}]&&\mathcal{RG}^d_{B}\llbracket\gamma\rrbracket\arrow[swap]{ll}\arrow{rr}\arrow[uu,"\rho_{I}" {yshift=5pt}]&&
s\mathcal{RG}^d_{B}\llbracket\gamma\rrbracket\arrow[bend left=12]{llll}\arrow[uu,"\rho_{I^q}" {yshift=5pt}]&
\end{tikzcd}$$
The corners of the cube on the right are $(d-2)$-twisted BD algebras, the upper arrows pointing inward are 2-weighted maps of BD algebras, whereas the lower arrows pointing outward are predual to such.
The horizontal arrows are maps of $(d-2)$-twisted BD algebras. The horizontal arrows in the left cube are dequantization maps, as are the curved arrows.
Lastly the left face commutes as a diagram $(d-2)$-shifted Poisson algebras.  
\end{theorem}
\begin{proof}
We have already seen that the upper two faces are commuting by remark \ref{pqdfr}, the lower two faces, including the curved arrows, are commuting by remark \ref{comp_g_rg}.
Further the front two faces, including the curved arrows, are commuting by theorem \ref{comp_rg_al} as are the two faces in the back, including the curved arrows, by theorem \ref{comp_gr_al}.
Thus it remains to show that the three parallel squares commute. We show that the rightmost one commutes, for the other two it follows similarly.
By definition of the prequantum LQT map this outermost diagram is equal to the following one.
\begin{equation}\label{comp_LQT_}
\begin{tikzcd}
\mathcal{F}^{pq}(V_B)\arrow{r}{\iota_B^*}&\mathcal{F}^{pq}(\widetilde{V_B})\arrow{r}{LQT_N^{pq}}&\mathcal{O}^{pq}(\mathfrak{gl}_NV)\\ \\
s\mathcal{RG}^d_{B}\llbracket\gamma\rrbracket\arrow{uu}{\rho_{I^q}}&s\mathcal{RG}^d_{\{*\}}\llbracket\gamma\rrbracket\arrow{uu}{\rho_{\widetilde{I}^q}}\arrow{l}&s\mathcal{G}^d\llbracket\hbar\rrbracket\arrow{l}{\pi}\arrow{uu}{\theta_{I^q_N}}
\end{tikzcd}
\end{equation}
Here we denoted by $\widetilde{V_B}$ the cyclic chain complex associated to $V_B$ under the cyclic commutator functor \ref{cyccomm} and  $\widetilde{I}^q:=\iota_B^*I^q\in MCE(\mathcal{F}^{pq}(\widetilde{V_B}).$
Further the lower left horizontal map is induced by the map of sets $B\rightarrow \{*\}$.
We prove commutativity of diagram \ref{comp_LQT_} in two steps:
    \begin{lemma}
The left square of diagram \ref{comp_LQT_} commutes. 
\end{lemma}
\begin{proof}
For $\Gamma$ a stable ribbon graph colored just by one element $\{*\}$ we can write
$$\rho_{\widetilde{I}^q}(\Gamma)=\sum_C \rho_{I^q}(\Gamma_C),$$
where the sum is over all isomorphism classes of $B$-colored stable ribbon graphs $\Gamma_C$ whose underlying $\{*\}$-colored stable ribbon graph is $\Gamma$.
Here we denoted by $C$ such a coloring.
Indeed, this follows by recalling the definition \ref{cyccomm} of the cyclic pairing on the cyclic chain complex associated to a collection of cyclic chain complexes and the construction \ref{main_constr_nc} of the map $\rho$.
But this equality expresses exactly commutativity of the diagram since $\iota_B^*$ is just an inclusion.
\end{proof}
\begin{lemma}
The right square of diagram \ref{comp_LQT_} commutes.
\end{lemma}
\begin{proof}
By recalling the definition of the prequantum LQT map for algebras the right square of the diagram is given by
$$
\begin{tikzcd}
\mathcal{F}^{pq}(\widetilde{V_B})\arrow{r}{tr^{pq}}&\mathcal{F}^{pq}(M_N\widetilde{V_B})\arrow{r}{m^{pq}\circ f^{pq}}&\mathcal{O}^{pq}(\mathfrak{gl}_NV)\\ \\
s\mathcal{RG}^d_{\{*\}}\llbracket\gamma\rrbracket\arrow{uu}{\rho_{\widetilde{I}^q}}\arrow{uur}{\rho_{M_N\widetilde{I}^q}}&&s\mathcal{G}^d\llbracket\hbar\rrbracket\arrow{ll}{\pi}\arrow{uu}{\theta_{I^q_N}}.
\end{tikzcd}
$$
Here $M_N\widetilde{I}^q=tr^{pq}\widetilde{I}^q$ and $I^q_N=m^{pq}\circ f^{pq}(M_N\widetilde{I}^q).$
We prove commutativity of the triangle first.

We recall from around lemma \ref{tr} the definition of $tr^{pq}$.
As explained in \cite{GGHZ21} theorem 4.9 and theorem 4.15, where $tr^{pq}$ is denoted $\mathcal{M}_{\gamma,\nu}$, one way to understand this map is by tensoring with certain tensors
$$\mu^{g,b}_{n_1\cdots n_m}\in M_N^{-\otimes n_1}\otimes\cdots \otimes M_N^{-\otimes n_m},$$
which exist for all integers $g,b\geq 0$ and $m,n_i>0$.
The properties of these tensors is described in lemma 3.10 of \cite{Ham11}, where they are denoted $\alpha^{g,b}$.
To be precise we are interested in these tensors applied to the Frobenius algebra $M_N$ of $N\times N$ matrices, in the set-up of \cite{Ham11}.

In the following we use the ambiguous notation $\mu^{g,b}_{m}$ for such a tensor above. 
We verify commutativity of the triangle for connected stable ribbon graphs $\Gamma$ from which general commutativity follows. We have by definition that
$$\rho_{M_N\widetilde{I}^q}(\Gamma)=proj \circ \bigotimes_{e\in E(\Gamma)}(\langle\_,\_\rangle^{-1}_e\langle\_,\_\rangle^{-1}_{M_N})\sigma \Big((\mu^{g_1, b_1}_{h_1}\otimes \widetilde{I}^q)|_{v_1}\otimes\cdots\otimes (\mu^{g_n, b_n}_{h_n}\otimes \widetilde{I}^q)|_{v_n}\Big)\nu^{{tot(\Gamma)}}\gamma^{g(\Gamma)}$$
Using implicitly the isomorphism $M_N^{\otimes 2}\cong {M_N^{\otimes -2}}^\vee$ and picking a good basis for $M_N$ we may write $$\langle\_,\_\rangle^{-1}_{M_N}=x_i\otimes y^i$$
where $x_i,y^i\in M_N$ and an implicit sum over $i$ is omitted. Then consecutively using properties (1)-(5) of the tensors $\mu^{g,b}_{m}$ as described in lemma 3.10 of \cite{Ham11} we have that above is equal to

$$=proj \circ \bigotimes_{e\in E(\Gamma)}(\langle\_,\_\rangle^{-1}_e\otimes\mu^{g(\Gamma), f_{tot}(\Gamma)}_{l(\Gamma)})\sigma(\widetilde{I}^q|_{v_1}\otimes\cdots\otimes \widetilde{I}^q|_{v_n})\nu^{f_{tot}}\gamma^{g(\Gamma)} =tr^{pq}\rho_{\widetilde{I}^q}$$
Indeed, since the graph is connected we end up with just one tensor $\mu^{g(\Gamma), f_{tot}(\Gamma)}_{l(\Gamma)}$.
While consecutively applying (3)-(5) of these properties (and possibly using (1)-(2) to bring elements in the good positions - no signs arise since $M_N$ is concentrated in degree zero) we pick up exactly the genus and the total number of faces of $\Gamma$, recalling definitions \ref{armg} and \ref{faces}.
For that latter part we have used the (implicit `extreme' terms) of lemma 3.10 (3)-(5) in loc. cit.
$$(3)\ \ \mu^{g,b}_{n+1}((x_iy^i)\otimes (a_{11}\cdots a_{1k_1})\otimes\cdots\otimes (a_{n1}\cdots a_{nk_n} )=\mu^{g,b+2}_{n}( (a_{11}\cdots a_{1k_1})\otimes\cdots \otimes(a_{n1}\cdots a_{nk_n} ))$$
\smallskip
$$(4)\ \ \mu^{g,b}_{n+2}((x_i)\otimes(y^i)\otimes (a_{11}\cdots a_{1k_1})\otimes\cdots\otimes (a_{n1}\cdots a_{nk_n} )=\mu^{g+1,b+1}_{n}( (a_{11}\cdots a_{1k_1})\otimes\cdots \otimes(a_{n1}\cdots a_{nk_n} ))$$
\smallskip\begin{equation*}
\begin{split}
(5)\ \
& \mu^{g_1,b_1}_{n_1+1}\big( (a_{11}\cdots a_{1k_1})\otimes\cdots\otimes (a_{n_11}\cdots a_{n_1k_{n_1}})\otimes(x_i) \big)\mu^{g_2, b_2}_{n_2+1}\big((y^i)\otimes (b_{11}\cdots b_{1k_1})\otimes\cdots\otimes (b_{n_21}\cdots b_{n_2k_{n_2}}) \big)\\
&=\mu^{g_1+g_2,b_1+b_2+1}_{n_1+n_2}\big( (a_{11}\cdots a_{1k_1})\otimes\cdots\otimes (a_{n_11}\cdots a_{n_1k_{n_1}} )\otimes(b_{11}\cdots b_{1k_1})\otimes\cdots\otimes (b_{n_21}\cdots b_{n_2k_{n_2}})\big)
\end{split}
\end{equation*}
Now we show commutativity of the odd square. This is expressed by the following equality of first and third line:
\begin{equation*}
\begin{split}
&m^{pq}f^{pq}\rho_{M_N\widetilde{I}^q}(\pi(\Gamma))=m^{pq}f^{pq}\Big(\sum_{\bar{\Gamma}}(\bigotimes_{e\in E(\bar{\Gamma})}\langle\_,\_\rangle^{-1}\otimes proj)\circ{\sigma}\big(M_N\widetilde{I}^q|_{v_1}\otimes\ldots\otimes M_N\widetilde{I}^q|_{v_n}\big)\nu^{tot}\gamma^{g(\Gamma)}\Big)\\
&=(\bigotimes_{e\in E(\Gamma)}\langle\_,\_\rangle^{-1}\otimes proj)\circ{\sigma}\big({I}^q_N|_{v_1}\otimes\ldots\otimes {I}^q_N|_{v_n}\big)\hbar^{2g(\bar{\Gamma})+f_{tot}(\bar{\Gamma})-1}\\
&=(\bigotimes_{e\in E(\Gamma)}\langle\_,\_\rangle^{-1}\otimes proj)\circ{\sigma}\big({I}^q_N|_{v_1}\otimes\ldots\otimes {I}^q_N|_{v_n}\big)\hbar^{bt(\Gamma)}=\theta_{I^q_N}(\Gamma)
\end{split}
\end{equation*}
Here $\Gamma$ denotes a connected stable graph and the sum in the first line is over all stable ribbon graphs $\bar{\Gamma}$ whose underlying connected stable graph is $\Gamma$.
The first equality follows since  pairing the tensors $\{I^q_N|_{v_i}\}$ according to the graph $\Gamma$ is the same as pairing the tensors $\{M_NI^q|_{v_i}\}$ in all possible ways, but respecting the rule imposed by $\Gamma$.
However we can reorder all these possible ways according to each way of putting the structure of a stable ribbon graph on $\Gamma$ and then pairing the tensors $M_NI^q|_{v_i}$ respectively,
which is expressed by the sum in the first line. 

The second equality follows since given a connected stable ribbon graph $\bar{\Gamma}$ with underlying stable graph $\Gamma$ (recall definition \ref{sRGsG}) we have
$$bt((\Gamma))=2g(\bar{\Gamma})+f_{tot}(\bar{\Gamma})-1,$$
as can be verified by a simple computation further; recalling the definitions of the arithmetic genus \ref{armg}, of the total number of empty faces \ref{faces} as well as of the betti number \ref{betti}.
\end{proof}

These two lemmas finish the proof of theorem \ref{main_compa}.
\end{proof}

\newpage
\section{Open-Closed Trivialization of Circle Action}\label{octosa}
In this section we prove theorem K from the introduction, which tells us that given a splitting of the non-commutative Hodge filtration (definition \ref{splitting}) there exist an $L_\infty$ quasi-isomorphism between the open-closed SFT BD algebra associated to a smooth cyclic $A_\infty$-category and a full subcategory $\Lambda\subset \mathcal{C}$ (see definition \ref{ysysys}) and the open-closed SFT BD algebra where we have trivialized most of the algebraic structure coming from the `closed side'. 
\subsection{String Field Theory Beilinson-Drinfeld Algebras}\label{setupsection}
In order to make this precise we need to carefully set up the open-closed SFT BD algebra. Let us recall respectively slightly modify the open BD algebra that we will use.
\begin{df}\label{oBD}
Let $V_B$ be a collection of cyclic chain complexes of degree $d$, recalling definition \ref{ColV}. Then we denote $$\mathcal{F}^o(V_B):=\Big(Sym\big(Cyc^*(V_B)[d-4]\big) \llbracket\gamma\rrbracket,d+\nabla+\gamma\delta,\{\_,\_\}_o\Big).$$ 
It has a $(2d-5)$-twisted Beilinson-Drinfeld algebra structure over $k\llbracket\gamma\rrbracket$, a formal variable of degree $(6-2d)$. This follows by identifying $\mathcal{F}^o(V_B)=\mathcal{F}^{pq}(V_B)[d-3],$ see around theorem \ref{nc_pq}, just that here we give $\gamma$ degree $(6-2d)$. If $V_B$ is induced from a full subcategory with objects $\Lambda$ of a given cyclic $A_\infty$-category $\mathcal{C}$, recalling remarks around \ref{ColV}, we denote $\mathcal{F}^o(V_B)=:\mathcal{F}^o(\Lambda).$
\end{df}
\begin{remark}\label{dgsmnggw}
    As written we shifted the initial $(d-2)$-twisted BD algebra from theorem \ref{nc_pq} by $(d-3)$. In the previous chapters we made the choice to work with 
    the shifting set-up from theorem \ref{nc_pq} in order to be compatible with the target of the LQT map, which naturally has a $(d-2)$-shifted BD algebra structure as well. In the current chapter it is more practical not to perform that shift and to consider $(2d-5)$-twisted Beilinson-Drinfeld algebras. We hope to find a coherent set-up in the future and apologize for potential confusions.  
\end{remark}
Furthermore we recall from definition \ref{cBD} the $(2d-5)$-twisted Beilinson-Drinfeld algebra over $k\llbracket\gamma\rrbracket$, denoted
$\mathcal{F}^{c}(\mathcal{C})$ and from definitions \ref{cBD_t} and \ref{cBD_Tr} the $(2d-5)$-twisted Beilinson-Drinfeld algebras over $k\llbracket\gamma\rrbracket$ denoted $\mathcal{F}^{c}(\mathcal{C})^{triv}$ and $\mathcal{F}^{c}(\mathcal{C})^{Triv}$.

Recall from lemma \ref{tensorBD} that we can consider the tensor product of $(2d-5)$-twisted BD algebras.
\begin{df}\label{ysysys}
    Given a cyclic $A_\infty$-category $\mathcal{C}$ and a full subcategory $\Lambda\subset \mathcal{C}$ we denote the \emph{observables of the open-closed free quantum SFT} by 
    $$\mathcal{F}^{c}(\mathcal{C})\otimes \mathcal{F}^o(\Lambda).$$
\end{df}
Of course similar definitions apply when replacing in above tensor product $\mathcal{F}^{c}(\mathcal{C})$ with $\mathcal{F}^{c}(\mathcal{C})^{triv}$ or $\mathcal{F}^{c}(\mathcal{C})^{Triv}$. In order to set-up the open-closed trivializing morphism we recall the purely closed trivializing morphism from \cite{CaTu24}.
 \subsection{Recollection on Closed Trivialization of Circle Action}
Note that by forgetting the multiplication a $(2d-5)$-twisted BD algebra becomes a dg $(2d-5)$-shifted Lie algebra, which we denote by the same letter by abuse of notation. We obtain dg Lie algebras by shifting by $-(2d-5)$ a dg $(2d-5)$ shifted Lie algebra.

We recall from \cite{CaTu24} the construction of an $L_\infty$-quasi-isomorphism 
\begin{equation}\label{inLqi}
\mathcal{K}_s:\ \mathcal{F}^{c}(\mathcal{C})[5-2d]\rightsquigarrow\mathcal{F}^{c}(\mathcal{C})^{triv}[5-2d],\end{equation}
given a splitting $s$ of the non-commutative Hodge filtration (definition \ref{splitting}). Note, however, that we have introduced a different grading as compared to \cite{CaTu24}; we carefully keep track of the $\mathbb{Z}$-grading.  The Taylor coefficients of $\mathcal{K}_s$ (compare remark \ref{nütz}) are given by maps
\begin{equation}\label{tayf}
\mathcal{K}_s^m:Sym^m(\mathcal{F}^{c}(\mathcal{C})[6-2d])\rightarrow \mathcal{F}^{c}(\mathcal{C})^{triv}[6-2d],\end{equation}
which uniquely define the morphism \ref{inLqi}. Since this morphism constitutes one of the main ingredient of this section, we recall its definition here, after \cite{CaTu24}, whose notation we follow. First we recall some
notions of graphs, introduced in \cite{geKa96}. A labeled
graph means for us a graph $G$ (possibly with leaves) endowed with a
`loop defect' labeling function $g: V_G\rightarrow \mathbb{Z}_{\geq 0}$ on the set of its
vertices $V_G$. 
We use the following notations for a labeled graph $G$:
\begin{itemize}
\item[--] $E_G$ denotes the set of edges, the cardinality of that set denoted $|e|$;
\item[--] $h_v$ denotes the set of half-edges belonging to a vertex $v$;
\item[--] The valency of a vertex $v\in V_G$ is denoted by $|h_v|$, the cardinality of half-edges belonging to that vertex;
\item[--] the number of vertices of a graph is denoted by $|v|$.
\end{itemize}
The betti number of a labeled graph is defined to be
\begin{equation}\label{betti2}
 g(G):= \sum_{v\in V_G} g(v)+{rank\,} H_1(G), \end{equation}
and we recall that for connected graph we have \begin{equation}
\label{rank}{rank\,} H_1(G)=|e|-|v|+1.\end{equation}
\begin{itemize}
\item[--] If $G$ has $m$ vertices, a marking of $G$ is a bijection
  \[ f: \{1,\cdots,m\} \rightarrow V_G. \] 
 \item[--] An isomorphism between two marked
and labeled graphs is an isomorphism of the underlying labeled graphs
that also preserves the markings.
\end{itemize}

We denote by $\widetilde{\Gamma(g,n)}_m$  the set of isomorphism classes of connected {\em marked}
graphs of betti number $g$ with $n$ leaves and with $m$ vertices. 

Next we recall that 
$$\mathcal{F}^{c}(\mathcal{C})[6-2d]=Sym\Big(CH_*(\mathcal{C})[u^{-1}][d-2]\Big)^{-*}\llbracket\gamma\rrbracket[6-2d]$$
and fix the notation $$\mathcal{H}:=(CH_*(\mathcal{C})[u^{-1}][d-2])^{-*}$$ so that 
$$\mathcal{F}^{c}(\mathcal{C})[6-2d]=Sym(\mathcal{H})\llbracket\gamma\rrbracket[6-2d].$$
Further for an element $y=:x_1x_2\cdots x_j \in Sym^{j}(\mathcal{H})$ let
$\widetilde{y} \in \mathcal{H}^{\otimes j}$ be its
desymmetrization,
\begin{equation}\label{desym}
 \widetilde{y}:=\sum_{\sigma\in \Sigma_j} \epsilon \cdot
x_{\sigma(1)}\otimes\cdots\otimes x_{\sigma(j)} \in \mathcal{H}^{\otimes j}.
\end{equation}
Here $\epsilon$ is the Koszul sign for permuting the elements
$x_1,\ldots, x_i$ with respect to the degree in $\mathcal{H}$. 
Given $n>0,g\geq0$, we denote by 
$$(\_)|_{g,n}:\mathcal{F}^{c}(\mathcal{C})[6-2d]\rightarrow (\mathcal{F}^{c}(\mathcal{C})[6-2d])|_{g,n}$$
the projection on the sub vector space given by symmetric words of length $n$ and of $\gamma$-coefficient~$g$.

Given a splitting $s$ from definition \ref{splitting} we can construct a symmetric bilinear form
  \begin{equation}\label{Hsym}
H^{sym}_s: \mathcal{H}^{\otimes 2} \rightarrow k[6-2d], \end{equation}  see page 38 of \cite{AmTu22} or in more detail proposition 7.5 of \cite{CaTu24} for the definition - carried over to our shifting and grading convention. 

Using these notations we can recall the definition of the degree zero linear maps \eqref{tayf}
$$\mathcal{K}_s^m: Sym(\mathcal{F}^{c}(\mathcal{C})[6-2d])\rightarrow \mathcal{F}^{c}(\mathcal{C})^{triv}[6-2d],$$
defined for each $m\geq 1$. We have
\[ \mathcal{K}_s^m := \sum_{g,n} \sum_{ (G,f)\in \widetilde{\Gamma(g,n)}_m}
  \frac{1}{|Aut(G,f)|}\cdot K_{(G,f)}, \]
and we describe the map $K_{(G,f)}$: Let $(G,f)\in \widetilde{\Gamma(g,n)}_m$ be a marked graph. Set
$g_i = g(f(i))$, and $n_i = |h_{f(i)}|$.  The map $K_{(G,f)}$ has only non-trivial components
\begin{equation}\label{betco}
\bigotimes_{i=1}^m  \mathcal{F}^{c}(\mathcal{C})[6-2d]|_{g_i,n_i}  \rightarrow \mathcal{F}^{c}(\mathcal{C})[6-2d]|_{g,n},
\end{equation}
which leaves us to specify for $y_i\in Sym^{n_i}(\mathcal{H})[6-2d]\ (i=1,\ldots,m)$ the expression $$K_{(G,f)} (y_1\gamma^{g_1},\ldots, y_m\gamma^{g_m})=Z\gamma^g.$$ 
\begin{ctr}\label{Feynm}
Here the element
$Z$ is computed by the following Feynman-type procedure. 
\begin{enumerate}
\item Given $y_i\in Sym^{n_i}(\mathcal{H})[6-2d]$ for $i=1,\ldots,m$ 
decorate the half-edges adjacent to each vertex $v_i$ by
$\widetilde{y}_i$, recalling the desymmetrization operation \ref{desym}. Tensoring together the results over all the
vertices yields a tensor of the form
\begin{equation}\label{abt}
 \bigotimes_{i=1}^m \widetilde{\gamma_{i}}\in \mathcal{H}^{\otimes (\sum
    n_i)}[|v(G)|(6-2d)].  \end{equation}
The order in which these elements are tensored is the one
given by the marking $f$.

\item  For each internal edge $e$ of $G$ contract the corresponding
components of the above tensor \eqref{abt} using the symmetric bilinear form $H^{sym}_s$, built from the splitting $s$. When applying the
contraction we always permute the tensors to bring the two terms
corresponding to the two half-edges to the front, and then apply the
contraction map. The ordering of the set $E_G$ does not matter since
the operator $H^{sym}_s$ is even; also the ordering of the two
half-edges of each edge does not matter since the operator $H^{sym}_s$
is (graded) symmetric.
\item Read off the remaining tensor components (corresponding to the
leaves of the graph) in any order, and regard the result as the
element $$Z\in Sym^n (\mathcal{H})[6-2d]$$ via the canonical projection map
$\mathcal{H}^{\otimes n}\rightarrow Sym^n \mathcal{H}$, followed by a shift by $(|v|-1)(6-2d)$. 
\end{enumerate}
Tracking the degrees it follows that $|Z|=\sum_i|y_i|+e(2d-6)+(1-v)(2d-6).$ This shows that $|K_{(G,f)}|=0$, since we have
$$|Z\gamma^g|=|Z|+g(6-2d)=(g-2d)\sum_ig_i+\sum_i|y_i|,$$
recalling the formula \ref{betti2} for the betti number $g$ of a graph.

\end{ctr}
\subsection{Open-Closed Trivialization Fomality Morphism}
In this section we prove theorem K from the introduction, which we state as corollary \ref{maco}. 

Let us fix a dimension $d$ cyclic $A_\infty$-category $\mathcal{C}$, a splitting $s$ of the non-commutative Hodge filtration (see definition \ref{splitting}) and an auxiliary $(2d-5)$-twisted BD algebra $W$ (which we will later take to be $\mathcal{F}^{c}(\Lambda)$, for $\Lambda\subset\mathcal{C}$ a full sub-category) for the rest of this subsection.

Recall from lemma \ref{tensorBD} that we can consider the tensor product $(2d-5)$-twisted BD algebras $\mathcal{F}^{c}(\mathcal{C})\otimes W$ and $\mathcal{F}^{c}(\mathcal{C})^{triv}\otimes W$, recalling our set-up from section \ref{setupsection}. By forgetting their multiplications and shifting we find Lie algebras $\mathcal{F}^{c}(\mathcal{C})\otimes W[5-2d]$ and  $\mathcal{F}^{c}(\mathcal{C})^{triv}\otimes W[5-2d]$.
The goal of this section is to construct an $L_\infty$ quasi-isomorphism denoted 
\begin{equation}\label{L_inf}
\mathcal{K}_s\otimes m:\ \mathcal{F}^{c}(\mathcal{C})\otimes W[5-2d]\rightsquigarrow \mathcal{F}^{c}(\mathcal{C})^{triv}\otimes W[5-2d].\end{equation}

Equivalently (see definition \ref{coalgebra}) such a map \eqref{L_inf} can be described as a map of dg coalgebras 
\begin{equation*}
C(\mathcal{K}_s\otimes m):\ Sym\big(\mathcal{F}^{c}(\mathcal{C})\otimes W[6-2d],d_{oc}\big)\rightarrow Sym\big(\mathcal{F}^{c}(\mathcal{C})^{triv}\otimes W[6-2d],d_{oc}^t\big),\end{equation*}
where the left and right hand side have the standard coalgebra structure (compare definition \ref{coal_def}). The differentials $d_{oc}$ and $d_{oc}^t$ of degree 1 are encoding the differential graded Lie algebra structures of $\mathcal{F}^{c}(\mathcal{C})\otimes W[5-2d]$ respectively $\mathcal{F}^{c}(\mathcal{C})^{triv}\otimes W[5-2d]$, see example \ref{dgla}. 
\begin{df}
We define a map of graded coalgebras, that is with zero differentials,
\begin{equation}\label{cogb}
C(\mathcal{K}_s\otimes m):\ Sym\big(\mathcal{F}^{c}(\mathcal{C})\otimes W[6-2d],0\big)\rightarrow Sym\big(\mathcal{F}^{c}(\mathcal{C})^{triv}\otimes W[6-2d],0\big).\end{equation}
As explained in remark \ref{nütz} a coalgebra morphism on such a coalgebra is uniquely determined by its Taylor coefficients. We define those by
\begin{equation*}
C(\mathcal{K}_s\otimes m)^n((x_1\otimes y_1)\odot \ldots\odot(x_n\otimes y_n)):=(-1)^*\mathcal{K}_s^n(x_1\odot\ldots\odot x_n)\otimes (y_1\cdot\ldots\cdot y_n),\end{equation*}
where $\mathcal{K}_s^n$ is the n-th Taylor coefficient of the $L_\infty$ morphism \ref{inLqi} and
where the sign arises from the Koszul rule of permuting all the $x$'s to the left and the $y$'s to the right. 
\end{df}
It remains to verify that the map \ref{cogb} commutes with the differentials, ie. defines a map of $L_\infty$-algebras and thus determines the required map \ref{L_inf}. The fact that it is an $L_\infty$ \emph{quasi-isomorphism} follows then directly since its first Taylor component will be given by the tensor product of a quasi-isomorphism (by \ref{inLqi}) and the identity, thus being a quasi-isomorphism. 
\begin{theorem}\label{cwd}
Given a dimension $d$ cyclic $A_\infty$-category $\mathcal{C}$, a splitting $s$ of the non-commutative Hodge filtration (see definition \ref{splitting}) and a $(2d-5)$-twisted BD algebra $W$. Then we have that 
    $$C(\mathcal{K}_s\otimes m)\circ d_{c}=d_{c}^t\circ C(\mathcal{K}_s\otimes m)$$
    for the map $C(\mathcal{K}_s\otimes m)$ from \ref{cogb}. That is, it defines the desired $L_\infty$ quasi-morphism
$$\mathcal{K}_s\otimes m:\ \mathcal{F}^{c}(\mathcal{C})\otimes W[5-2d]\rightsquigarrow \mathcal{F}^{c}(\mathcal{C})^{triv}\otimes W[5-2d].$$
\end{theorem}
The proof of this theorem occupies the rest of this section (see summary \ref{Zusa} below for an overview):

As explained in example \ref{wmg}, since we are actually dealing with dg Lie algebras, proving theorem \ref{cwd} is equivalent to showing equation \eqref{equ:dglaLinfty}. Translating our notation to that of equation \eqref{equ:dglaLinfty}, by recalling example \ref{dgla} and equation \eqref{tenbra}, means that  
\begin{equation*}
Q_1'=d_{c,t}+d_w,\ \ \ Q_1=d_{c}+d_w,
\end{equation*}
\begin{equation*}
 Q_2'(x_1\otimes y_1,x_2\otimes y_2)=(-1)^{|x_1|+|y_1|+|x_2|}m_c(x_1,x_2)\otimes\{y_1,y_2\}_w,
 \end{equation*}
 \begin{equation*}
    Q_2(x_1\otimes y_1,x_2\otimes y_2)=(-1)^{|x_1|}\{x_1,x_2\}_c\otimes m_w(y_1,y_2)+(-1)^{|x_1|+|y_1|+|x_2|}m_c(x_1,x_2)\otimes\{y_1,y_2\}_w,
    \end{equation*}
    
where we denote (momentarily) by $m_c$ respectively $m_w$ the commutative product on $\mathcal{F}^{c}(\mathcal{C})$ (and by abuse of notation on $\mathcal{F}^{c}(\mathcal{C})^{triv}$) respectively on $W$, by $d_w$ the BD differential on $W$; by $d_c$ the BD differential of $\mathcal{F}^{c}(\mathcal{C})$ and by $d_{c,t}$ the BD differential of $\mathcal{F}^{c}(\mathcal{C})^{triv}$. 

Summarizing, proving theorem \ref{cwd} is equivalent to showing that following equation\footnote{Here $(-1)^{*_1}$ is given by the Koszul rule of permuting the $x$'s to the left and the $y$'s to the right. $(-1)^{*_2}=(-1)^{\sum_{i}|x_i|+\sum_{i\in L_1}|y_i|}$ times the Koszul rule of permuting the factors of ${(x_1\otimes y_1)\odot\cdots \odot (x_m\otimes y_m)}$ into the respective order. $(-1)^{*_3}=\epsilon(i,j,1,..,\hat{i},..,\hat{j},.., n)$. $(-1)^{*_4}=(-1)^{(\sum_i|x_i|)+|y_i|} $ times the Koszul rule of permuting the factors. $(-1)^{*_5}=(-1)^{|x_i|}$ times the Koszul rule of permuting the factors.} holds:
\begin{align}\label{inter}
&(d_{c,t}+d_w)(-1)^{*_1}\mathcal{K}_n(x_1\odot\cdots \odot x_n)\otimes (y_1\cdot\ldots\cdot y_m)\nonumber\\
+&\sum_{\substack{L_1\sqcup L_2=[n]\nonumber \\ |L_1|,|L_2|\geq 1}}(-1)^{*_2}\big(\mathcal{K}_{|L_1|}(\bigodot_{i\in L_1} x_i)\cdot\mathcal{K}_{|L_2|}(\bigodot_{i\in L_2} x_i)\big)\otimes \{\Pi_{i\in L_1} y_i,\Pi_{j\in L_2} y_j\}_w\nonumber\\
=&\sum_{i}(-1)^{*_3}(\mathcal{K}_s\otimes m)^n\big((d_{c}+d_w)(x_i\otimes y_i)\odot(x_1\otimes y_1)\odot\cdots \widehat{(x_i\otimes y_i)}\cdots \odot (x_m\otimes y_n)\big)\nonumber\\
+&\sum_{i\neq j}(-1)^{*_4}\mathcal{K}_{m-1}\big((x_i\cdot x_j)\odot x_1\odot\ldots \hat{x_i}\ldots\hat{x_j}\ldots\odot x_m\big)\otimes\big(\{y_i,y_j\}_w\cdot  y_1\cdot\ldots \hat{y_i}\ldots\hat{y_j}\ldots\cdot y_m\big)\nonumber\\
+&\sum_{i\neq j}(-1)^{*_5}\mathcal{K}_{m-1}\big(\{x_i, x_j\}_c\odot x_1\odot\ldots \hat{x_i}\ldots\hat{x_j}\ldots\odot x_m\big)\otimes  \big((y_i\cdot y_j)\cdot y_1\cdot\ldots\hat{y_i}\ldots\hat{y_j}\ldots \cdot y_m\big)
\end{align}

Using following two facts \eqref{BDr} and \eqref{KiL} we show that proving theorem \ref{cwd}, which is equivalent to proving that equation \eqref{inter} holds, is in fact equivalent to proving lemma \ref{key} below.

The fact that $W$ is a BD algebra implies that\footnote{Here the sign $(-1)^{*_6}$ results from permuting the odd operator $d_w$ through the $y_1,\ldots,y_{i-1}.$ Further we have $(-1)^{*_7}=(-1)^{|y_i|}\epsilon (i,j,\dots,\hat{i},\dots,\hat{j},\dots, n) $}
\begin{align}\label{BDr}
\mathcal{K}_n(x_1\odot\cdots \odot x_n)\otimes d_w(y_1\cdot\ldots\cdot y_n)&\nonumber\\
=\sum_i(-1)^{*_6}\mathcal{K}_n(x_1\odot\cdots \odot x_n)\otimes (y_1\cdot&\ldots \cdot d_wy_i\cdot\ldots\cdot y_m)\\+\sum_{i\neq j}(-1)^{*_7}&\mathcal{K}_n(x_1\odot\cdots \odot x_n)\otimes\big(\gamma\{y_i,y_j\}_w\cdot  y_1\cdot\ldots \hat{y_i}\ldots\hat{y_j}\ldots\cdot y_m\big).\nonumber
\end{align}
Further by equation \eqref{equ:dglaLinfty} the fact that the maps $\mathcal{K}^n_s$ define an $L_\infty$-algebra morphism (by equation 26 of \cite{AmTu22}, following theorem 7.1 of \cite{CaTu24}) implies that\footnote{Here $(-1)^{*_8}=\epsilon(i,1,\dots,\hat{i},\dots, n)$ and $(-1)^{*_9}=\epsilon(i,j,1,..,\hat{i},..,\hat{j},.., n)(-1)^{|x_i|}$}
\begin{align}\label{KiL}
    d_{c,t}\mathcal{K}_n(x_1\odot\cdots \odot& x_n)\otimes  \big(y_1\cdot\ldots \cdot y_m\big)\nonumber\\
=  \sum_{i}(-1)^{*_8}\mathcal{K}_s^n\big(d_{c}x_i\odot x_1&\odot\cdots \widehat{x_i}\odot \cdots \odot x_m\big)\otimes  \big(y_1\cdot\ldots \cdot y_m\big)\\ 
&+\sum_{i\neq j}(-1)^{*_9}\mathcal{K}_{m-1}\big(\{x_i, x_j\}_c\odot x_1\odot\ldots \hat{x_i}\ldots\hat{x_j}\ldots\odot x_m\big)\otimes  \big(y_1\cdot\ldots \cdot y_m\big)\nonumber.
\end{align}
Using these facts, showing that equation \eqref{inter} reduces to the equation in the following lemma is now basic algebra: Multiplying equality \eqref{KiL} by the Koszul sign of permuting the $x$'s of ${(x_1\otimes y_1)\odot\cdots \odot (x_m\otimes y_m)}$ to the left and the $y$'s to the right shows that the first summand of the first line of \eqref{inter} is equal to the first summand of the third summand and the last line. Multiplying equality \eqref{BDr} by the same Koszul sign times $(-1)^{\sum_i|x_i|}$ implies that the second summand of the first line of \eqref{inter} is equal to the second summand of the third line of \eqref{inter} plus the last line of \eqref{BDr} (multiplied by that same sign). Since also the second and forth line of \eqref{inter} is multiplied by $(-1)^{\sum_i|x_i|}$ we have proven the claim. Thus to prove theorem \ref{cwd} it remains to show the following.
\begin{lemma}\label{key}
For all $m>1$ and $(x_1\otimes y_1)\odot\cdots \odot (x_m\otimes y_m)\in Sym^m(\mathcal{F}^{c}(\mathcal{C})\otimes W[6-2d])$ we have 
\begin{align*}
\sum_{i\neq j}(-1)^{*_a}\mathcal{K}_{m-1}\big(&(x_i\cdot x_j)\odot x_1\odot\ldots \hat{x_i}\ldots\hat{x_j}\ldots\odot x_m\big)\otimes\big(\{y_i,y_j\}_w\cdot  y_1\cdot\ldots \hat{y_i}\ldots\hat{y_j}\ldots\cdot y_m\big)\\
=\sum_{\substack{L_1\sqcup L_2=[m] \\ |L_1|,|L_2|\geq 1}}(-1)^{*_b}&\big(\mathcal{K}_{|L_1|}(\bigodot_{i\in L_1} x_i)\cdot\mathcal{K}_{|L_2|}(\bigodot_{i\in L_2} x_i)\big)\otimes \{\Pi_{i\in L_1} y_i,\Pi_{j\in L_2} y_j\}_w\\
+&\gamma\sum_{i\neq j}(-1)^{*_c}\mathcal{K}_m(x_1\odot\cdots \odot x_m)\otimes\big(\{y_i,y_j\}_w\cdot y_1\cdot\ldots \hat{y_i}\ldots\hat{y_j}\ldots\cdot y_m\big)
\end{align*}
The signs arise from the Koszul rule of permuting the (odd) factors of ${(x_1\otimes y_1)\odot\cdots \odot (x_m\otimes y_m)}$ into the respective order; further from moving the $y_i$ respectively the $\Pi_{i\in L_1} y_i$ into the odd bracket $\{\_,\_\}_w.$ 
\end{lemma}
In order to prove this key lemma, which is found below the proof of lemma \ref{easy}, we need to elaborate a bit: 
\subsubsection{An Interlude on Graphs}
In this subsection we prove theorem \ref{gt}, which is essential in the proof of the key lemma \ref{key}.
\begin{cons}\label{Cons1}
Given a connected marked graph $G$ with $m-1$ vertices 
and a partition $I\sqcup J=h_1$ of the half-edges $h_1$ belonging to the (for example) first vertex of the graph we can construct a new marked graph with $m$ vertices
where now only the half-edges $I$ belong to the first vertex and the half-edges $J$ are declared to belong to a new vertex, marked $m$.
We denote the new graph by $G^{I,J}$.
However, we need to be careful as in this process we may create two disconnected marked graph, denoted $G^{I}$ respectively $G^{J}$,
which for instance always happens when our starting graph is a star. See the proof of theorem \ref{gt} below for how we mark these new graphs.
\end{cons}
We will see that we are also able to reverse this procedure. To make this more precise we introduce the following notations.
\begin{df}
For a given marked graphs $(G,f)$ and two partitions of their half-edges belonging to the first vertex $I_1\sqcup J_1=h_1$ respectively $I_2\sqcup J_2=h_1$ we say $$ ( I_1,J_1)\sim ( I_2,J_2)\ \ \text{if}\
 \begin{cases} G^{I_1,J_1}\cong G^{I_2,J_2},& \text{if the resulting graphs are connected} \\
G^{I_1}\cong G^{I_2}\ \text{and}\ G^{J_1}\cong G^{J_2},& \text{otherwise.}
\end{cases}
$$
To be precise the isomorphism have to be from the category of marked graphs.
 Further given a graph $(G,f)\in \widetilde{\Gamma(g,n)}_{m-1}$ and $k_1,k_2\in\mathbb{N}_{>0}$ denote by $$P_{(G,f)}^{k_1,k_2}:=\{I\sqcup J= h_{1}\ |\ |I|=k_1,|J|=k_2\}/\sim.$$
\end{df}
\begin{df}\label{grdf}
Fix $0\leq g$, $n<0$ and $k_1,k_2\in\mathbb{N}_{>0}$. We call $A_{k_1,k_2}^{g,n}$ \textit{the set of isomorphism classes of marked graphs with a partition of the first vertex}, $B_{k_1,k_2}^{g,n}$ \textit{the set of isomorphism classes of marked graphs with fixed valencies at two vertices} and $C_{k_1,k_2}^{g,n}$ \textit{the set of 2-disconnected isomorphism classes of marked graphs with a partition}. More precisely:
\begin{align*}
A_{k_1,k_2}^{g,n}:=&\{(G,f)\in \widetilde{\Gamma(g,n)}_{m-1}, I,J\in P_{(G,f)}^{k_1,k_2}\},\\
B_{k_1,k_2}^{g,n}:=&\{(G,f)\in \widetilde{\Gamma(g,n)}_m\ |\  |h_{1}|=k_1,|h_{m}|=k_2\}\\
C_{k_1,k_2}^{g,n}:=&\{(G_1,f_1)\in \widetilde{\Gamma(g_1,n_1)}_{m_1},(G_2,f_2)\in \widetilde{\Gamma(g_2,n_2)}_{m_2},\ L_1\sqcup L_2=\{1,\cdots,m\}| 1\in L_1,m\in L_2 \\
&\ g_1+g_2=g, m_1+m_2=m, n_1+n_2=n, 
|h_1(G_1)|=k_1, |h_m(G_2)|=k_2   \}.
\end{align*}
Further given $1\leq i<j\leq m$ we denote 
\begin{align*}
\widetilde{B}_{k_1,k_2}^{g,n}:=&\{(G,f)\in \widetilde{\Gamma(g,n)}_m\ |\  |h_{i}|=k_1,|h_{j}|=k_2\}\\
\widetilde{C}_{k_1,k_2}^{g,n}:=&\{(G_1,f_1)\in \widetilde{\Gamma(g_1,n_1)}_{m_1},(G_2,f_2)\in \widetilde{\Gamma(g_2,n_2)}_{m_2},\ L_1\sqcup L_2=\{1,\cdots,m\}| i\in L_1,j\in L_2 \\
&\ g_1+g_2=g, m_1+m_2=m, n_1+n_2=n, 
|h_i(G_1)|=k_1, |h_j(G_2)|=k_2   \}
\end{align*}
\end{df}
Obviously we have that \begin{equation}\label{obv2}
{B}_{k_1,k_2}^{g,n}\cong\widetilde{B}_{k_1,k_2}^{g,n}
\end{equation}
\begin{equation}\label{obv}
{C}_{k_1,k_2}^{g,n}\cong\widetilde{C}_{k_1,k_2}^{g,n}
\end{equation}
by swapping $1$ with $i$ and $m$ with $j$.
\begin{theorem}\label{gt}
   Fix $0\leq g$, $0<n$ and $k_1,k_2\in\mathbb{N}$. Then there is a bijection 
   \begin{equation}\label{wfegtffv}
    \psi: A_{k_1,k_2}^{g,n}\cong B_{k_1,k_2}^{g,n}\cup C_{k_1,k_2}^{g,n}.
    \end{equation}
\end{theorem}
\begin{proof}
    We have already described in construction \ref{Cons1} how given an element of $A_{k_1,k_2}^{g,n}$ we produce (a) graph(s) on the right hand side. In the connected case we also described the marking. In the disconnected case we simply use the old marking, but reorder according to size. Further in the disconnected case the partition $L_1\sqcup L_2=\{1,\cdots,m\}$ is obtained by remembering which of the previous marked vertices are now part of $G_1$ (including the old first one) and which of the previous marked vertices are now part of $G_2$ (including the new last one, which was obtained by splitting the old first one into two).
    
    Given an element in the right hand side (be it connected or not), whose total number of vertices is $m$, we send it to the graph determined by declaring the first and $m$-th vertex of the initial graph(s) to coincide and which we number by $1$, then.
    In the case the graph came from $B$, i.e. was connected, it is straightforward to number the remaining vertices of the resulting new graph.
    In the case the graph comes from $C$ we use the partition $L_1\sqcup L_2=\{1,\cdots,m\}$ which satisfied the condition that $ 1\in L_1,m\in L_2$ to mark the remaining vertices of the resulting new graph. To define the partition into two sets of the half-edges of the `new' vertex of $\psi^{-1}(G,f)$ we remember which ones came from the first vertex of $(G,f)$ and which came from the $m$-th vertex of $(G,f)$. 
    
    Finally it is straightforward to check that these two assignments are inverse to each other.
\end{proof}
\subsubsection{The Proof of Theorem K continued}
In turn theorem \ref{gt} allows us to prove following key technical theorem:
\begin{theorem}\label{BVinf}
For $m>1$ and $x_1,\cdots, x_m\in \mathcal{F}^{c}(\mathcal{C})[6-2d]$ and $1\leq i\neq j\leq m$ we have that\footnote{Here $(-1)^*=\epsilon(i,j,1,..,\hat{i},..,\hat{j},.., n)$ and $(-1)^\#=\epsilon(L_1,L_2)$.}
\begin{align*}
&(-1)^*\mathcal{K}_{m-1}\big((x_i\cdot x_j)\odot x_1\odot\ldots \hat{x_i}\ldots\hat{x_j}\ldots\odot x_m\big)\\
=&\gamma\mathcal{K}_m(x_1\odot\cdots \odot x_m)+\sum_{\substack{L_1\sqcup L_2=[m] \\  i\in L_1,j\in L_2}}(-1)^\#\big(\mathcal{K}_{|L_1|}(\bigodot_{i\in L_1} x_i)\cdot\mathcal{K}_{|L_2|}(\bigodot_{i\in L_2} x_i)\big)
\end{align*}
\end{theorem}
\begin{proof}
By definition we have 
\begin{align*}
&(-1)^*\mathcal{K}_{m-1}\big((x_i\cdot x_j)\odot x_1\odot\ldots \hat{x_i}\ldots\hat{x_j}\ldots\odot x_m\big)\\=&\sum_{g,n} \sum_{ (G,f)\in \widetilde{\Gamma(g,n)}_{m-1}}(-1)^*
  \frac{1}{|Aut(G,f)|}\cdot K_{(G,f)}\big((x_i\cdot x_j),x_1,\ldots, \hat{x_i},\ldots,\hat{x_j},\ldots, x_m\big)\\
  =&:E
  \end{align*}

We denote by $\widetilde{ev}_{G,f}(x_1,\cdots,x_m)$ the operation of \textit{contracting exactly once the letters of the symmetric words} $\{x_l\}_{l=1,\cdots,m}\in\mathcal{F}^c(\mathcal{C})[6-2d]$ with each other, using $H^{sym}$ and according to the datum given by a marked graph $(G,f)$ having $m$ vertices, weighted appropriately by $\gamma$. This gives the identification 
\begin{equation}\label{id1}
\widetilde{ev}_{(G,f)}(x_1,\cdots,x_m)=\frac{1}{|Aut(G,f)|}\cdot K_{(G,f)}(x_1,\cdots,x_m),\end{equation}
recalling definition \ref{Feynm}.
We can ask whether we can re-express in a similar way
\begin{equation}\label{summe}
(-1)^*\frac{1}{|Aut(G,f)|}\cdot K_{(G,f)}(x_i\cdot x_j,x_1,\cdots,\hat{x_i},\cdots,\hat{x_j},\cdots,x_m),\end{equation}
where $x_i\in Sym^{k_1}(\cdots)$ and $x_j\in Sym^{k_2}(\cdots)$.
Indeed, again recalling the definition \eqref{Feynm}, this term equals a sum which we can naturally index by $(I,J)\in P_{(G,f)}^{(k_1,k_2)}$ and whose summands we suggestively denote by 
\begin{equation}\label{se}
(-1)^*\widetilde{ev}_{(G,f,I,J)}(x_i, x_j,x_1,\cdots,\hat{x_i},\cdots,\hat{x_j},\cdots,x_m).\end{equation}
\begin{remark}
We will give a precise meaning to \eqref{se} in equations \eqref{ia1} and \eqref{ia2}.
\end{remark}For the moment we repeat that 
\begin{align*}
\sum_{I,J\in P_{(G,f)}}(-1)^*\widetilde{ev}_{(G,f,I,J)}&(x_i, x_j,x_1\cdots,\hat{x_i},\cdots,\hat{x_j},\cdots,x_m)\label{fs}\\
&=(-1)^*\frac{1}{|Aut(G,f)|}\cdot K_{(G,f)}(x_i\cdot x_j,x_1\cdots,\hat{x_i},\cdots,\hat{x_j},\cdots,x_m)\nonumber
\end{align*}
Thus we can write
\begin{equation}\label{E}
E=\sum_{g,n} \sum_{ (G,f)\in \widetilde{\Gamma(g,n)}_{m-1}}\sum_{I,J\in P_{(G,f)}}(-1)^*\widetilde{ev}_{(G,f,I,J)}(x_i, x_j,x_1,\cdots,\hat{x_i},\cdots,\hat{x_j},\cdots,x_m).\end{equation}
\begin{obs}
We note that we can re-express the summands \eqref{se} of \eqref{summe}, applying the rule \ref{Feynm}, as
\begin{align}\label{ia1}
(-1)^*\widetilde{ev}_{(G,f,I,J)}&(x_i,x_j,x_1,\cdots,\hat{x_i},\cdots,\hat{x_j},\cdots,x_m)\nonumber\\=&\gamma\epsilon(i,1,..,\hat{i},..,\hat{j},.., n,j)\widetilde{ev}_{\psi(G,f,I,J)}(x_i,x_1\cdots,\hat{x_i},\cdots,\hat{x_j},\cdots,x_m,x_j)\end{align}
if $\psi(G,f,I,J)\in B_{(k_1,k_2)}^{g,n}$. Here the right hand side is defined in the standard way by rule \ref{id1}.
The additional $\gamma$ factors in \eqref{ia1} arise from the fact that the appearing graphs on the right hand side have the same number of total edges, but one more vertex, thus decreasing the betti number by one as compared to the appearing graphs appearing on the left hand side (recalling definition \ref{betti}).

Similarly we have 
\begin{align}
&(-1)^*\widetilde{ev}_{(G,f,I,J)}(x_i, x_j,x_1,\cdots,\hat{x_i},\cdots,\hat{x_j},\cdots,x_m)\nonumber \\
=&\widetilde{ev}_{(G_1,f_1)}( \_,\cdots,\_)\cdot \widetilde{ev}_{(G_2,f_2)}(\_,\cdots,\_,)(-1)^\sigma(L_1\sqcup L_2)_*(x_i,x_j,x_1,\cdots,\hat{x_i},\cdots,\hat{x_j},\cdots,x_m) \label{ia2}
\end{align}
if $\psi(G,f,I,J)=:\Big((G_1,f_1),(G_2,f_2),L_1,L_2\Big)\in C_{k_1,k_2}^{g,n}$, recalling the notations from definition \ref{grdf}.
In \eqref{ia2} we plug the $x_l$'s belonging to $L_1$, that is such that $l\in L_1$ (among them $x_i$ at the first place) into the first factor and the $x_l$'s belonging to $L_2$ (among them $x_j$ at the last place) into the second factor, indicated by $(L_1\sqcup L_2)_*$.
The sign $(-1)^\sigma$ arises from the resulting Koszul sign of the permutation of the $x_l$'s. There is no additional $\gamma$ factors in \eqref{ia2} since on the right hand side we both increase the number of total vertices by one, but also the number of connected components from one to two, thus the total betti number stays the same.
\end{obs}
Next we note that the indexing set of the second and third sum in \eqref{E} is in bijection with the set $A_{k_1,k_2}^{g,n}$. Theorem \ref{gt} tells us that this set is in bijection with $B_{k_1,k_2}^{g,n}\cup C_{k_1,k_2}^{g,n}$ and equations \eqref{ia1} and \eqref{ia2} imply that the evaluation maps are compatible with this bijection. Thus we have  
\begin{align*}
E=&\sum_{g,n}\sum_{B^{g,n}_{k_1,k_2}}\epsilon(i,1,..,\hat{i},..,\hat{j},.., n,j)\widetilde{ev}_{(G,f)}(x_i,x_1\cdots,\hat{x_i},\cdots,\hat{x_j},\cdots,x_m,x_j)\\
+\sum_{g,n}&\sum_{C^{g,n}_{k_1,k_2}}\widetilde{ev}_{(G_1,f_1)}(\_, ,\cdots,\_)\cdot \widetilde{ev}_{(G_2,f_2)}(\_,,\cdots,\_)(-1)^\sigma(L_I\sqcup L_J)_*(x_i,x_j,x_1,\cdots,\hat{x_i},\cdots,\hat{x_j},\cdots,x_m).
\end{align*}
Finally using the identification \eqref{obv2} and the fact that the evaluation maps are compatible with that identification we rewrite the first summand
\begin{align*}
E=&\gamma\sum_{g,n}\sum_{\widetilde{B}^{g,n}_{k_1,k_2}}\widetilde{ev}_{(G,f)}(x_1,\cdots ,x_m)\\
+&\sum_{g,n}\sum_{\widetilde{C}^{g,n}_{k_1,k_2}}\widetilde{ev}_{(G_1,f_1)}(\_,\cdots,\_)\cdot \widetilde{ev}_{(G_2,f_2)}(\_, \cdots,\_)(-1)^{\widetilde\sigma}(L_1\sqcup L_2)_*(x_1,\cdots,x_m).
\end{align*}
and where we have rewritten the second summand using identification \eqref{obv} and the fact that the evaluation maps are compatible with that. Now using again \eqref{id1} we deduce that
\begin{align*}
E=\gamma\sum_{g,n}\sum_{ (G,f)\in \widetilde{\Gamma(g,n)}_m}
  \frac{1}{|Aut(G,f)|}\cdot& K_{(G,f)}(x_1,\cdots,x_m)+\\
\sum_{g,n}\sum_{\substack{n_1+n_2=n\\g_1+g_2=g\\
m_1+m_2=m}}\sum_{ \substack{G_1\in {\Gamma(g_1,n_1)}_{m_1}\\
G_2\in {\Gamma(g_2,n_2)}_{m_2}}}&\sum_{\substack{L_1\sqcup L_2=[m] \\  i\in L_1,j\in L_2}}
  \\
  \frac{1}{|Aut(G_1)|\cdot |Aut(G_2)|}&\big(\mathcal{K}_{|L_1|}(\_,\cdots,\_)\cdot\mathcal{K}_{|L_2|}(\_,\cdots,\_)\big)(-1)^{\widetilde{\sigma}}(L_1\sqcup L_2)_*(x_1,\cdots,x_m)
  \end{align*}
which by definition is equal to  
  $$=\gamma\mathcal{K}_m(x_1\odot\cdots \odot x_m)+\sum_{\substack{L_1\sqcup L_2=[m] \\  i\in L_1,j\in L_2}}\epsilon(L_1,L_2)\big(\mathcal{K}_{|L_1|}(\bigodot_{i\in L_1} x_i)\cdot\mathcal{K}_{|L_2|}(\bigodot_{i\in L_2} x_i)\big),$$
  which is what we wanted to show.
\end{proof}
\begin{lemma}\label{easy}
For all $m>1$ and $(x_1\otimes y_1)\odot\cdots \odot (x_m\otimes y_m)\in Sym^m(\mathcal{F}^{c}(\mathcal{C})\otimes W[6-2d])$ we have\footnote{Here $(-1)^{*_1}=(-1)^{|y_i|}$ and $(-1)^{*_2}=(-1)^{\sum_{i\in L_1}|y_i|}$} 

\begin{align*}
&\sum_{i\neq j}\sum_{\substack{L_1\sqcup L_2=[m]\\ i\in L_1,j\in L_2}}\big(\mathcal{K}_{|L_1|}(\bigodot_{i\in L_1} x_i)\cdot\mathcal{K}_{|L_2|}(\bigodot_{i\in L_2} x_i)\big)\otimes(-1)^{*_1}\big(\{y_i,y_j\}_w\cdot  y_1\cdot\ldots \hat{y_i}\ldots\hat{y_j}\ldots\cdot y_m\big)\\
&=\sum_{\substack{L_1\sqcup L_2=[m] \\ |L_1|,|L_2|\geq 1}}\big(\mathcal{K}_{|L_1|}(\bigodot_{i\in L_1} x_i)\cdot\mathcal{K}_{|L_2|}(\bigodot_{i\in L_2} x_i)\big)\otimes (-1)^{*_2}\{\bigodot_{i\in L_1} y_i,\bigodot_{j\in L_2} y_j\}_w
\end{align*}
\end{lemma}
\begin{proof}
This follows directly from plugging the identity
$$(-1)^{\sum_{i\in L_1}|y_i|}\{\bigodot_{i\in L_1} y_i,\bigodot_{j\in L_2} y_j\}_w=\sum_{i\in L_1,j\in L_2}(-1)^{|y_i|}\{y_i,y_j\}_w\cdot  y_1\cdot\ldots \hat{y_i}\ldots\hat{y_j}\ldots\cdot y_m,$$
which follows from the Leibniz rule for $W$, into the second line, which gives the first line.
\end{proof}
Now we can easily give the proof of the key lemma \ref{key}:
\begin{proof}
For $m>1$ and $(x_1\otimes y_1)\odot\cdots \odot (x_m\otimes y_m)\in Sym^m(\mathcal{F}^{c}(\mathcal{C})\otimes W[6-2d])$ we have\footnote{Here the signs $a$, $b$, $c$ are as in the statement of lemma \ref{key}. Further $(-1)^d=(-1)^{\sum_{i\in L_1}|y_i|}$ times the Koszul sign of permuting the $x$'s to the left and $y$'s to right.} 
\begin{align*}
&\sum_{ i\neq j}(-1)^{*_a}\mathcal{K}_{m-1}\big((x_i\cdot x_j)\odot x_1\odot\ldots \hat{x_i}\ldots\hat{x_j}\ldots\odot x_m\big)\otimes\big(\{y_i,y_j\}_w\cdot  y_1\cdot\ldots \hat{y_i}\ldots\hat{y_j}\ldots\cdot y_m\big) \\
=&\sum_{i\neq j}(-1)^{*_b}\gamma\mathcal{K}_m(x_1\odot\cdots \odot x_m)\otimes\big(\{y_i,y_j\}_w\cdot  y_1\cdot\ldots \hat{y_i}\ldots\hat{y_j}\ldots\cdot y_m\big)\\
&+\sum_{i\neq j}\sum_{\substack{L_1\sqcup L_2=[m] \\  i\in L_1,j\in L_2}}(-1)^{*_d}\big(\mathcal{K}_{|L_1|}(\bigodot_{i\in L_1} x_i)\cdot\mathcal{K}_{|L_2|}(\bigodot_{i\in L_2} x_i)\big)\otimes\big(\{y_i,y_j\}_w\cdot  y_1\cdot\ldots \hat{y_i}\ldots\hat{y_j}\ldots\cdot y_m\big)\\
=&\sum_{i\neq j}(-1)^{*_b}\gamma\mathcal{K}_m(x_1\odot\cdots \odot x_m)\otimes\big(\{y_i,y_j\}_w\cdot  y_1\cdot\ldots \hat{y_i}\ldots\hat{y_j}\ldots\cdot y_m\big)\\
&+\sum_{\substack{L_1\sqcup L_2=[m] \\ |L_1|,|L_2|\geq 1}}(-1)^{*_c}\big(\mathcal{K}_{|L_1|}(\bigodot_{i\in L_1} x_i)\cdot\mathcal{K}_{|L_2|}(\bigodot_{i\in L_2} x_i)\big)\otimes \{\bigodot_{i\in L_1} y_i,\bigodot_{j\in L_2} y_j\}_w
\end{align*}
Here the first equality follows from theorem \ref{BVinf} (to be more precise from multiplying the equality stated by $(-1)^{|y_i|}$ times the Koszul sign from permuting in $(x_1\otimes y_1)\odot\cdots \odot (x_m\otimes y_m)$ the $x$'s to the left and $y$'s to right ). The second equality follows from lemma \ref{easy} (again multiplied by that same Koszul sign).
\end{proof}  
With this we have now proven theorem \ref{cwd}; let us quickly recall why:
\begin{Zusa}\label{Zusa}
As explained (beneath its statement) proving theorem \ref{cwd} is equivalent to proving equation \ref{inter}. By equations \ref{BDr} and \ref{KiL} proving equation  \ref{inter} is equivalent to proving lemma \ref{key}. Finally theorem \ref{BVinf} and the easy lemma \ref{easy} together give a proof of lemma \ref{key}, as we just explained above.
\end{Zusa}
A direct corollary is:
\begin{corollary}\label{ec}
Given a dimension $d$ cyclic $A_\infty$-category $\mathcal{C}$, a splitting $s$ of the non-commutative Hodge filtration and $\Lambda,$ a full subcategory of $\mathcal{C}$. Then there is an $L_\infty$ quasi-isomorphism
$$\mathcal{K}_s\otimes m:\ \mathcal{F}^{c}(\mathcal{C})\otimes \mathcal{F}^{o}(\Lambda)[5-2d]\rightsquigarrow \mathcal{F}^{c}(\mathcal{C})^{triv}\otimes \mathcal{F}^{o}(\Lambda)[5-2d].$$
\end{corollary}
\begin{proof}
    We apply theorem \ref{cwd} to the $(2d-5)$-twisted BD algebra $\mathcal{F}^{o}(\Lambda)$ from definition \ref{oBD}.
\end{proof}
Given a dimension $d$ cyclic $A_\infty$-category $\mathcal{C}$ and a splitting $s$ there is an isomorphism of induced dg Lie algebras (recalling definitions \ref{cBD_t} and \ref{cBD_Tr})
\begin{equation}\label{R}
R:\mathcal{F}^{c}(\mathcal{C})^{triv}\rightarrow \mathcal{F}^{c}(\mathcal{C})^{Triv},\end{equation}
see beginning of section 9.1 of \cite{CaTu24}, keeping in mind that we denoted eg. $\mathfrak{h}_{\mathcal{C}}^{triv}=\mathcal{F}^{c}(\mathcal{C})^{triv}.$ Thus it follows:
\begin{corollary}\label{maco}
Given a dimension $d$ cyclic $A_\infty$-category $\mathcal{C}$, a splitting $s$ of the non-commutative Hodge filtration and $\Lambda,$ a full subcategory of $\mathcal{C}$. Then there is an $L_\infty$ quasi-isomorphism
$$\Psi_s^{oc}:\ \mathcal{F}^{c}(\mathcal{C})\otimes \mathcal{F}^{o}(\Lambda)[5-2d]\rightsquigarrow \mathcal{F}^{c}(\mathcal{C})^{Triv}\otimes \mathcal{F}^{o}(\Lambda)[5-2d],$$
given by composing the isomorphism \eqref{R} with the $L_\infty$-map from corollary \ref{ec}.
\end{corollary}
Thus we have proven theorem K from the introduction.

\newpage
\section{Twisted Holography from Calabi-Yau Categories}\label{TwHoCY}
In this section we elaborate on the circle of ideas presented in the introduction, in section \ref{twhocy}, on how to formulate a version of twisted holography at the level of partition function, given as input datum a Calabi-Yau categories and an object of that category. First, we comment on further fundamental aspects of the open-closed story, citing partly the forthcoming work \cite{AmTu25}. In particular we explain how assumption $(*)$ from the introduction should be satisfied in general, using the existence of a Maurer-Cartan element in the universal open-closed BD algebra.

\smallskip
These ideas are still in preparation and we should warn the reader that for instance details regarding shifts (see remarks \ref{dgsmnggw2} and \ref{dgsmnggw}) still need to be pinned-down.
\subsection{Fundamental Open-Closed Structures}\label{FOCS}
 Our setting in this section is as follows: Let $\mathcal{C}$ be a smooth cyclic $A_\infty$-category and $\Lambda\subseteq\mathcal{C}$ some choice of objects. As explained eg. in definition \ref{ysysys} there is an associated tensor product BD algebra $$\mathcal{F}^c(\mathcal{C})\otimes \mathcal{F}^o(\Lambda).$$
Given such a datum we have seen in theorem \ref{cei} and theorem B that there are maps of induced $L_\infty$-algebras
\begin{equation}\label{isidhb1}
\rho_\mathcal{C}^c:M^c\rightarrow \mathcal{F}^c(\mathcal{C})\end{equation}
respectively 
\begin{equation}\label{isidhb2}
\rho_{{\Lambda}}^o:M^o_\Lambda\rightarrow \mathcal{F}^o(\Lambda)
\end{equation}
where $M^c$ respectively $M^o_\Lambda$ are built from the moduli space of `closed' Riemann surfaces respectively `open' Riemann surfaces (the latter colored by $\Lambda$). There is an open-closed analogue of these BD algebras: From now on we take for simplicity $\Lambda=\{\lambda\}$. The chain complex \begin{equation}\label{ocmsbla}
    \Big(\bigoplus_{stable} C_*(\mathcal{M}_{g,n,b,m_1,\cdots,m_b}^{fr,})_{hS^1},\partial\Big)\end{equation}
computes the ($S^1$-equivariant) cohomology of the moduli space of genus $g$ Riemann surfaces with $n$ framed boundaries and $b$ unframed boundaries each with $m_i\in\mathbb{Z}_{\geq 0}$ marked points, the free boundaries decorated by $\lambda$; we refer to \cite{HVZ08} for details on this and for the following. One can define a BD algebra structure
\begin{equation}\label{docms123}
M^{oc}:=\Big(\bigoplus_{stable} C_*(\mathcal{M}_{g,n,b,m_1,\cdots,m_b}^{fr})_{hS^1}\llbracket\gamma\rrbracket,\partial+\gamma\Delta_c+\gamma\delta+\nabla,\{\_,\_\}_c+\{\_,\_\}_o\Big)
\end{equation} on the chain complex \eqref{ocmsbla}, where $\{\_,\_\}_o$, $\delta$ and $\nabla$ are induced by gluing boundary marked points ($\delta$ gluing together marked points lying on different boundary components and $\nabla$ those which lie on the same boundary component), further $\{\_,\_\}_c$ and $\Delta_c$ twist-glue interior marked points.
Following theorem in writing generalizes the maps \eqref{isidhb1} and \eqref{isidhb2} to the open-closed setting: 
\begin{theorem}[\cite{AmTu25}]\label{ocRep}
Given $\mathcal{C}$ a smooth cyclic $A_\infty$-category and $\lambda\in \mathcal{C}$ there is an open-closed representing map of induced $L_\infty$-algebras
$$\begin{tikzcd}
    M^{oc}\arrow{r}{\rho_{\mathcal{C},\lambda}} &\mathcal{F}^c(\mathcal{C})\otimes \mathcal{F}^o(\lambda).
\end{tikzcd}$$ 
\end{theorem}
\begin{remark}
    Theorem \ref{ocRep} is still in preparation and the author is not aware how Amorim and Tu will set up exactly the gradings for $\mathcal{F}^c(\mathcal{C})\otimes \mathcal{F}^o(\lambda)$. Recall remarks \ref{dgsmnggw2} and \ref{dgsmnggw} for different set-ups used in this thesis. We suppress this point in the following, but ask the careful reader to keep this in mind. 
\end{remark}

In the purely closed BD algebra we could find a canonical Maurer-Cartan element, the closed string vertices from equation \eqref{csvmce}. In the purely open BD algebra this was not the case. We did study related phenomena, however (see eg. theorem G.1 from the introduction).

\smallskip
To explain the analogous statement for the open-closed set-up we need to introduce an additional map.
Denote by  
\begin{equation}\label{emjk}
\Delta_{co}:M^{oc}\rightarrow M^{oc}
\end{equation}
the degree 1 differential induced from declaring a framed boundary to be a non-framed boundary with zero marked points (see \cite{HVZ08}, just before theorem 4.1). Now we can state 
\begin{theorem}[\cite{Zw98}, \cite{HVZ08}]
There is an element $SV^{oc}\in M^{oc}$, called \emph{open-closed string vertices}, playing the role of a fundamental class. It is \emph{almost} a Maurer-Cartan element, up the the term coming from $\Delta_{co}$, that is it 
\begin{equation}\label{QME}
(\partial+\gamma\Delta_c+\gamma\delta+\nabla)SV^{oc}+\frac{1}{2}\Big(\{SV^{oc},SV^{oc}\}_{o}+\{SV^{oc},SV^{oc}\}_{c}\Big)+\Delta_{co}SV^{oc}=0. \end{equation}

\end{theorem} 
 Under the map $\rho_{\mathcal{C},\lambda}$ the differential $\Delta_{co}$ has an interesting categorical meaning, which will appear in the forthcoming work \cite{AmTu25}. Denote by $ch(\lambda)\in CH_*(\mathcal{C})\llbracket u\rrbracket$ the non-commutative Chern character of $\lambda$ induced by the class of the unit of the unital algebra $End_\mathcal{C}(\lambda)$. Further denote $$\langle ch(\lambda),\_,\rangle_{res}:\mathcal{F}^c(\mathcal{C})\rightarrow \mathcal{F}^c(\mathcal{C})$$ the degree 1 map extending as a derivation the pairing with $ch(\lambda)$ via the higher residue pairing \ref{Res} which we weight by $\nu$, the basis of cyclic words of length zero:
\begin{lemma}[\cite{AmTu25}]\label{chk}
Given $\mathcal{C}$ a smooth cyclic $A_\infty$-category and $\lambda\in \mathcal{C}$ then the differential \eqref{emjk} intertwines the map \eqref{ocRep} as follows 
    $$\rho_{\mathcal{C},\lambda}\circ \Delta_{co}=(\nu\langle ch(\lambda),\_\rangle_{res})\circ \rho.$$
\end{lemma}
Note that lemma \ref{chk} together with theorem \ref{ocRep} explain that assumption $(*)$ from the introduction is satisfied in general:
\begin{corollary}[\cite{AmTu25}]\label{snsnsn}
 Given a smooth cyclic $A_\infty$-category $\mathcal{C}$ and an object ${\lambda\in \mathcal{C}}$ there is an element $S^{oc}\in \mathcal{F}^c(\mathcal{C})\otimes \mathcal{F}^o(\lambda)$ given as
 $$S^{oc}:=\rho_{\mathcal{C},\Lambda}(SV^{oc})$$
 such that it is a Maurer-Cartan element up to a term coming from the Chern character
\begin{equation}
(d_{hoch}+uB+\gamma\Delta_{c}+d+\gamma\delta+\nabla)S^{oc}+\frac{1}{2}\{S^{oc},S^{oc}\}_{o}+\frac{1}{2}\{S^{oc},S^{oc}\}_{c}+\nu\langle ch(\lambda),S^{oc}\rangle_{res}=0
\end{equation}
 and its open tree level part\footnote{That is the projection to $Sym^0(\cdots)\otimes Sym^1(Cyc^*(\lambda)[-1])\gamma^0$.} is the hamiltonian  of the $A_\infty$-algebra, see eg. \eqref{HamfurIn}.
\end{corollary}
The second property should follow from the fact that the map $\rho_{\mathcal{C},\Lambda}$ when restricted to the purely open part coincides with the map \ref{isidhb1} and that the purely open, genus zero and one boundary component part of the open-closed string vertices is described by the disk Maurer-Cartan element from theorem \ref{RS_mce}, which gets mapped to the hamiltonian of the cyclic $A_\infty$-algebra as described in lemma \ref{comp_2}.
\subsection{Quantization of Open String Field Theory}
In this section we explain when we can find a quantization (in the sense of definition \ref{incqcyc}) of the (large $N$) open string field theory of an objects $\lambda$ (or better the objects $\lambda^{\oplus N}$ for all $N\in\mathbb{N}$, assuming that $\mathcal{C}$ is additive) using the theory laid out in the previous section \ref{FOCS} and corollary \ref{maco}. As we will see, this quantization depends on a splitting of the non-commutative Hodge filtration \ref{splitting} and a trivialization of the non-commutative Chern character.

\medskip
Dependent on theorem \ref{ocRep} and lemma \ref{chk} from \cite{AmTu25} in preparation we have:
\begin{prop}\label{ki}
Let $(\mathcal{C},\lambda)$ be a tuple of a smooth cyclic $A_\infty$-category and an object $\lambda$ of $\mathcal{C}$. Assume that $(d_{hoch}+uB)h=ch(\lambda)$, a trivialization of the Chern character of $\lambda$. Then  \begin{equation}\label{gauaw}
S^{oc,q}(\mathcal{C},\lambda,h):=e^{(\nu\langle h,\_\rangle_{res})}\rho(SV^{oc})
\end{equation}is a Maurer-Cartan element of ${\mathcal{F}^c(\mathcal{C})}\otimes \mathcal{F}^o(\lambda).$ Further if we change $h$ by an exact term $S^{oc,q}(\mathcal{C},\lambda,h)$ changes up to homotopy.  
\end{prop}
\begin{proof}
Let us denote for simplicity by $d_{oc}$ the full BD differential of ${\mathcal{F}^c(\mathcal{C})}\otimes \mathcal{F}^o(\lambda)$ and by $\{\_,\_\}_{oc}$ its shifted Poisson bracket. The proposition follows by computing
\begin{align*}
&d_{oc}\Big(e^{\nu\langle h,-\rangle_{res}}S^{oc}\Big)=e^{\nu\langle h,-\rangle_{res}}d_{oc}S^{oc}+e^{\nu\langle h,-\rangle_{res}}\nu\langle ch(\lambda),-\rangle_{res}S^{oc}\\
=&-\Big(\Delta_{oc}I^{oc}+\frac{1}{2}\{I^{oc},I^{oc}\}_{oc}+\nu\langle ch(\lambda),e^{\nu\langle h,-\rangle_{res}}S^{oc}\rangle_{res}\Big)+\nu\langle ch(\lambda),-\rangle_{res}e^{\nu\langle h,-\rangle_{res}}S^{oc}\\
=&-\big(\Delta_{oc}I^{oc}+\frac{1}{2}\{I^{oc},I^{oc}\}_{oc}\big),\end{align*}
where the first equality uses lemma 6.0.3 (3) of \cite{Cos05} and for the second equality we used corollary \ref{snsnsn} and that $e^{\nu\langle h,-\rangle_{res}}$ commutes with the brackets and BV differentials. Similarly we can compute that
$$(d_{oc}+d_t)\Big(e^{\nu\langle h+tdg+dtg,-\rangle_{res}}S^{oc}\Big)\in MCE\Big({\mathcal{F}^c(\mathcal{C})}\otimes \mathcal{F}^o(\lambda)\otimes \Omega^*([0,1])\Big),$$
where $d_t$ denotes the de Rham differential on $\Omega^*([0,1])$. This explains the independence up to homotopy if we change $I^{oc}(\lambda,h)$ by an exact terms.
\end{proof}
We collect some ingredients to prove theorem \ref{oq} below: We recall that corollary \ref{maco} tells us that given a cyclic $A_\infty$-category $\mathcal{C}$,  an object $\lambda$ and a splitting s of the non-commutative Hodge filtration there is a map of $L_\infty$-algebras
   $$ 
        \Psi^{oc}_s:\ \mathcal{F}^c(\mathcal{C})\otimes \mathcal{F}^o(\lambda) \rightsquigarrow \mathcal{F}^c(\mathcal{C})^{Triv}\otimes \mathcal{F}^o(\lambda). 
    $$
Thus from proposition \ref{ki}, given additionally a trivialization $(d_{hoch}+uB)h=ch(\lambda)$ of the Chern character, and corollary \ref{maco} it follows that
\begin{equation}\label{wbvmb}
\Psi^{oc}_s(S^{oc,q}(\mathcal{C},\lambda,h))\in MCE\Big(\mathcal{F}^c(\mathcal{C})^{Triv}\otimes \mathcal{F}^o(\lambda)\Big).
\end{equation}
Let us introduce following definition in order to reduce the notational clutter a bit.
\begin{df}
    We call a \emph{package} \begin{equation}\label{sssswmi}
\mathcal{P}:=\{\mathcal{C},\lambda,h,s\} \end{equation}
a tuple consisting of a smooth cyclic $A_\infty$ category $\mathcal{C}$, an object in that category $\lambda\in\mathcal{C}$, a trivialization $(d_{hoch}+uB)h=ch(\lambda)$ and a splitting $s$.
\end{df}
 Let us denote by 
$$ \pi:\mathcal{F}^c(\mathcal{C})^{Triv}\otimes \mathcal{F}^o(\lambda)\rightarrow Sym^0(CH_*(\mathcal{C})\llbracket u\rrbracket)\gamma^0\otimes \mathcal{F}^o(\lambda)\cong \mathcal{F}^o(\lambda)$$
the above composite, which is a map of BD algebras.  Given  a package $\mathcal{P}$ we denote
\begin{equation}\label{qoosft}
S^{o,q}(\mathcal{P}):=(\pi\circ\Psi^{oc}_s)\big(S^{oc,q}(\mathcal{C},\lambda,h)\big)\in  \mathcal{F}^o(\lambda),
\end{equation}
recalling definition \ref{gauaw}.
A consequence of the fact that $\pi$ is a map of BD algebras, equation \eqref{wbvmb}, the explicit form of the map $\Psi^{oc}_s$ and the second part of corollary \ref{snsnsn} is:
\begin{theorem}\label{oq}
Given a package $\mathcal{P}$ as in \eqref{sssswmi} the element
$S^{o,q}(\mathcal{P})$ from \eqref{qoosft} is a  quantization of the $A_\infty$-hamiltonian of the cyclic $A_\infty$-algebra $End_\mathcal{C}(\lambda)$, that is it is a solution to the quantum master equation in $\mathcal{F}^o(\lambda)$ and its tree level component is the hamiltonian of the cyclic $A_\infty$-algebra $End_\mathcal{C}(\lambda)$. 
\end{theorem}
\begin{df}
Given a package $\mathcal{P}$ we denote the \emph{partition function of the large $N$ open string field theory} by
\begin{equation}\label{fmc}
Z^{o}(\mathcal{P}):=[e^{S^{o,q}(\mathcal{P})}]\in H_*\big(\mathcal{F}^o(\lambda),d+\gamma\delta+\nabla\big).
\end{equation}
\end{df}

\medskip
To justify this name we recall the quantized LQT map, which is a map of BV algebras \cite{GGHZ21}
\begin{equation}\label{lqt}
LQT_N:\mathcal{F}^o(\lambda)\rightarrow C^*\big(\mathfrak{gl}_NEnd_\mathcal{C}(\lambda)\big). \end{equation}
 The right-hand side of \eqref{lqt} describes the observables of the \emph{open string gauge theory on a stack of $N$ branes of $\lambda$} modeled by the cyclic $L_\infty$-algebra \begin{equation}\label{diya}
    End_\mathcal{C}(\lambda^{\oplus N})\cong\mathfrak{gl}_NEnd_\mathcal{C}(\lambda),
    \end{equation}
remarking that above isomorphism \ref{diya} of course is only applicable in additive categories. Summarizing, theorem \ref{oq} tell us when to find a quantization of this $L_\infty$-algebra, and uniformly for every $N\in\mathbb{N}$. See also the discussion around equation \eqref{copex} from the introduction, which justifies the name of $Z^{o}(\mathcal{P})$ as the partition function of the quantized large $N$ open string field theory on a stack of $N$ branes of type $\lambda$, associated to a package $\mathcal{P}$ \eqref{P}.

\begin{remark}
 Another interesting element to consider should be the following. Given a package $\mathcal{P}$
 this element is defined as
$$e^{\Psi^{oc}_s(S^{oc,q}(\mathcal{C},\lambda,h))/\gamma}\in {\mathcal{F}^c(\mathcal{C})}^{Triv}\otimes \mathcal{F}^o(\lambda),$$
recalling equation \eqref{wbvmb}.
From this equation it follows that $$Z^{oc}_{\mathcal{C},\lambda,h,s}:=[e^{\Psi^{oc}_s(S^{oc,q}(\mathcal{C},\lambda,h))/\gamma}]\in H_*\Big({\mathcal{F}^c(\mathcal{C})}^{Triv}\otimes \mathcal{F}^o(\lambda),d_{Hoch}+d+\gamma\delta+\nabla\Big),$$
which may justify calling $Z^{oc}_{\mathcal{C},\lambda,h,s}$ the open-closed GW potential. We don't consider this element in this thesis, however.
 
\end{remark}

 \subsection{Backreacted Closed String Field Theory}\label{mimsss}
In physics, there is the idea that a closed string field theory, which we introduced in section \ref{wsuaderinn}, can be deformed by introducing terms sourced by branes (but no other open strings).  In the universal geometric open-closed BV algebra this picture of a brane $\lambda$ sourcing closed strings should be encoded by those surfaces (possibly with framed boundaries) which additionally have some (unframed) boundaries with no marked points which are colored by an element $\lambda$. We can ask whether the corresponding subset of the open-closed string vertices, which we denote by $SV|^{sourced}_{closed}$, still satisfies the quantum master equation. Upon examination of the quantum master equation \eqref{QME} we note that this is not true. Also for the classical part of \eqref{QME}, which we denote by $SV_{g=0}|^{sourced}_{closed}$, this is not true. However, we can ask that their image under the open-closed representing map from theorem \ref{ocRep} satisfies the master equation, leading us to the following:
\subsubsection{Classical Backreacted Closed String Field Theory}
We try to overcome the problem of $SV_{g=0}|^{sourced}_{closed}\in M^{oc}$ not satisfying the classical master equation by making a definition (which we hope to improve, see remark \ref{hopefullytobeimprov} later).
\begin{df}\label{backreactable}
    We say that a pair of a smooth cyclic category $\mathcal{C}$ and an object $\lambda\in \mathcal{C}$ are \emph{backreactable} if, recalling the open-closed representing map $\rho_{\mathcal{C},\lambda}$ of theorem \ref{ocRep}, we have that 
    \begin{equation}\label{hrmop1}
\rho_{\mathcal{C},\lambda}\big(\{SV^{oc}|^{m=1}_{g=0},SV^{oc}|^{m=1}_{g=0}\}_c\big)=0\end{equation}
    and 
    \begin{equation}\label{hrmop2}
        \rho_{\mathcal{C},\lambda} \big(\nabla SV^{oc}|^{m=2}_{g=0}\big)=0,
        \end{equation}
        referring to this footnote\footnote{Here we denote by $SV^{oc}|^{m=1}_{g=0}$ the genus zero part of the string vertices with any number of framed boundaries and unframed boundaries, of which exactly one unframed boundaries has one marked point. Similarly we denote by $SV^{oc}|^{m=2}_{g=0}$ the genus zero part of the string vertices with any number of framed boundaries and unframed boundaries, of which exactly one unframed boundaries has two marked point.} for the notation.
\end{df}

\medskip
In this section we define `classical sourced closed string field theory' (see definition \ref{brLinfalg}) and the `classical backreacted closed string field theory' (see definition \ref{wzhhsimsl}) in order to shed light on the backreaction phenomena. In short, given a category and an object that are backreactable we can always define the former, while the latter only exists additionally assuming a trivialization of the Chern character of the object. However, given such a trivialization, we can actually identify both structures, see remark \ref{dbegaj}. We further examine carefully the curvature of the underlying $L_\infty$-algebras defining these classical field theories.
\medskip

A central element for these definitions is the element $SV_{g=0}|^{sourced}_{closed}$, the genus zero open-closed string vertices with arbitrary number of framed and unframed boundaries but no marked points; compare equation \eqref{QME}. Assuming that $(\mathcal{C},\lambda)$ is backreactable it follows from the definition and corollary \ref{snsnsn} that
\begin{equation}\label{stfuiwfs}
(d_{Hoch}+uB+\nu\langle ch(\lambda),\_\rangle_{res})\rho\big(SV_{g=0}|^{sourced}_{closed}\big)+\frac{1}{2}\{\rho\big(SV_{g=0}|^{sourced}_{closed}\big),\rho\big(SV_{g=0}|^{sourced}_{closed}\big)\}_{c}\Big)=0, \end{equation}
which implies that the derivation 
\begin{equation}\label{ouoibh}
d_{Hoch}+uB+\nu\langle ch(\lambda),\_\rangle_{res}+\{\rho\big(SV_{g=0}|^{sourced}_{closed}\big),\_\}_c
\end{equation}
defines a differential on $\mathcal{F}^c(\mathcal{C})|_{\gamma=0}[\nu]$. By setting $\nu=N$ the induced differential on the algebra
$$\mathcal{F}^c(\mathcal{C})|_{\gamma=0}=Sym(CH_*(\mathcal{C})[ u^{-1}])$$
dually induces a dg coalgebra structure on the graded vector space
$$Sym\big(CH_*(\mathcal{C})[ u^{-1}]\big)^\vee,$$
that is we recover a possibly \emph{curved $L_\infty$-algebra} structure on the graded vector space $\big(CH_*(\mathcal{C})[ u^{-1}]\big)^\vee$. This is the first piece of data allowing us to define the notion of the `classical sourced closed string field theory at level $N$' (see definition \ref{brLinfalg} below):
\begin{df}\label{scbcsftN}
Given $(\mathcal{C},\lambda)$ a smooth cyclic $A_\infty$-category and an object $\lambda\in\mathcal{C}$ that are backreactable we denote by $L_N(\mathcal{C},\lambda)$ the curved $L_\infty$-algebra on cyclic cochains of $\mathcal{C}$ induced from the differential \eqref{ouoibh} and by setting $\nu=N$.
\end{df}
\begin{lemma}\label{isidfrz}
    
Given $(\mathcal{C},\lambda)$ a smooth cyclic $A_\infty$-category and an object $\lambda\in\mathcal{C}$ that are backreactable  then the curvature of $L_N(\mathcal{C},\lambda)$ is given to leading order in $N$ by 
\begin{equation}\label{focdsl}
N\langle ch(\lambda),\_\rangle_{res}\in \big(CH_*(\mathcal{C})[ u^{-1}]\big)^\vee.\end{equation}
\end{lemma}
\begin{proof}
    
The fact that $$N\langle ch(\lambda),\_\rangle_{res}\in \big(CH_*(\mathcal{C})[ u^{-1}]\big)^\vee$$ contributes to the leading order in $N$ of the curvature is clear. The projection of $\rho\big(SV_{g=0}|^{sourced}_{closed}\big)$ onto $Sym^1(CH_*(\mathcal{C})[ u^{-1}])[\nu]$ is in general non-zero and these terms do also contribute to the curving of $L_N(\mathcal{C},\lambda)$ under the differential induced from 
$$\{\rho\big(SV_{g=0}|^{sourced}_{closed},\_\}_c.$$
However, because of the stability condition on the open-closed string vertices (see around equation 1 of \cite{HVZ08}) the sub-collection of the open-closed string vertices which have $n=1$ framed boundaries need to have $b\geq 2$ unframed boundaries in the genus zero case. Thus the image of the projection of  $\rho\big(SV_{g=0}|^{sourced}_{closed}\big)$ onto $Sym^1(CH_*(\mathcal{C})[ u^{-1}])[\nu]$ actually lands in the part with higher or equal to 2 powers of $\nu$; thus the contribution to the curvature coming from $\{\rho\big(SV_{g=0}|^{sourced}_{closed},\_\}_c$ is of higher order in $N$.
\end{proof}
By definition \ref{backreactable} if $(\mathcal{C},\lambda)$ is backreactable we have that   $$\rho\big(SV_{g=0}|^{sourced}_{closed}\big)\in MCE\Big(\mathcal{F}^c(\mathcal{C})|_{\gamma=0}[\nu],d_{Hoch}+uB+\nu\langle ch(\lambda),\_\rangle_{res},\{\_,\_\}_c\Big)$$ and by setting $\nu=N$ we get $$\rho\big(SV_{g=0}|^{sourced}_{closed}\big)_N\in MCE\Big(\mathcal{F}^c(\mathcal{C})|_{\gamma=0},d_{Hoch}+uB+N\langle ch(\lambda),\_\rangle_{res},\{\_,\_\}_c\Big).$$ Thus we can twist by $\rho\big(SV_{g=0}|^{sourced}_{closed}\big)_N$ leading to following definition, which we contrast with definition \ref{cisft} of classical interacting closed string field theory:
\begin{df}\label{brLinfalg}
   Given a smooth cyclic $A_\infty$-category  $\mathcal{C}$ together with  $\lambda\in\mathcal{C}$ that is backreactable we call \emph{classical sourced closed string field theory at level $N$} the datum of the shifted Poisson algebra
$$\Big(\mathcal{F}^c(\mathcal{C})|_{\gamma=0},\cdot,d_{Hoch}+uB+N\langle ch(\lambda),\_\rangle_{res}+\{\rho\big(SV_{g=0}|^{sourced}_{closed}\big)_N,\_\}_c,\{\_,\_\}_c\Big).$$
and the curved $L_\infty$-algebra $L_N(\mathcal{C},\lambda)$ from definition \ref{scbcsftN}, whose curvature (to leading order in $N$) is given by the Chern character; see equation \eqref{focdsl}.
\end{df}
\begin{remark}
     In the setting of the preceding definition we can define a shifted Poisson algebra at level $N=\infty$ 
$$\Big(\mathcal{F}^c(\mathcal{C})|_{\gamma=0}[\nu],\cdot,d_{Hoch}+uB+\nu\langle ch(\lambda),\_\rangle_{res}+\{\rho\big(SV_{g=0}|^{sourced}_{closed}\big),\_\}_c,\{\_,\_\}_c\Big),$$
but it is not clear to us what is the good analogue of the $L_\infty$-algebra structure on cyclic cochains.
\end{remark}

\medskip
Now we begin to collect some ideas to define the 'classical backreacted closed string field theory', see definition \ref{wzhhsimsl} below. For that we assume that we are given a trivialization of the Chern character of an object $\lambda\in\mathcal{C}$ of a cyclic $A_\infty$-category, which gives us: 
\begin{lemma}\label{gauawamap}
    Given $(\mathcal{C},\lambda)$ a cyclic $A_\infty$-category and an object $\lambda\in\mathcal{C}$ together with a trivialization $ch(\lambda)=(d_{hoch}+uB)h$ then
    \begin{equation}\label{ishissh}
        e^{\nu\langle h,\_\rangle_{res}}:\Big(\mathcal{F}^c(\mathcal{C})|_{\gamma=0}[\nu],\cdot,d_{Hoch}+uB+\nu\langle ch(\lambda),\_\rangle_{res},\{\_,\_\}_c\Big)\rightarrow \Big(\mathcal{F}^c(\mathcal{C})|_{\gamma=0}[\nu],\cdot,d_{Hoch}+uB,\{\_,\_\}_c\Big)
        \end{equation}
    defines an isomorphism of shifted Poisson algebras with inverse given by the map $$e^{-\nu\langle h,\_\rangle_{res}}:\mathcal{F}^c(\mathcal{C})|_{\gamma=0}[\nu]\rightarrow \mathcal{F}^c(\mathcal{C})|_{\gamma=0}[\nu].$$
\end{lemma}
\begin{proof}
    This follows similarly to the proof of proposition \ref{gauaw}.
\end{proof}
Let us now additionally assume that $(\mathcal{C},\lambda)$ is backreactable. Since the element $\rho\big(SV_{g=0}|^{sourced}_{closed}\big)$ is then a Maurer-Cartan element of the left hand side of map \ref{ishissh} by lemma \ref{gauawamap} its image is also a Maurer-Cartan element. From this it follows that the derivation
\begin{equation}\label{ihhhh}
d_{Hoch}+uB+\{e^{\nu\langle h,\_\rangle_{res}}\rho\big(SV_{g=0}|^{sourced}_{closed}\big),\_\}_c:\mathcal{F}^c(\mathcal{C})|_{\gamma=0}[\nu]\rightarrow \mathcal{F}^c(\mathcal{C})|_{\gamma=0}[\nu]
\end{equation}
squares to zero. In the same way as in the discussion around equation \ref{ouoibh} we conclude that by setting $\nu=N$ the differential \eqref{ihhhh} defines an $L_\infty$-algebra structure on cyclic cochains:
\begin{df}\label{brlinalsb}
   Given $(\mathcal{C},\lambda)$ a cyclic $A_\infty$-category and an object $\lambda\in\mathcal{C}$ which are backreactable together with a trivialization $ch(\lambda)=(d_{hoch}+uB)h$ we denote by  $L_N^{br}(\mathcal{C},\lambda,h)$ the $L_\infty$-algebra structure on the graded vector space $\big(CH_*(\mathcal{C})[ u^{-1}]\big)^\vee$ induced from the differential \eqref{ihhhh} by setting $\nu=N$ and by dualizing.
\end{df}
Lemma \ref{gauawamap} and the preceding discussion imply that following definition is well-defined:
\begin{df}\label{wzhhsimsl}
  Given $(\mathcal{C},\lambda)$ a smooth cyclic $A_\infty$-category and an object $\lambda\in\mathcal{C}$ that are backreactable together with a trivialization $ch(\lambda)=(d_{hoch}+uB)h$. We call
\emph{classical closed string field theory at level $N$ backreacted in the direction $(\lambda,h)$} the datum of the shifted Poisson algebra
$$\Big(\mathcal{F}^c(\mathcal{C})|_{\gamma=0},\cdot,d_{Hoch}+uB+\{e^{\nu\langle h,\_\rangle_{res}}\rho\big(SV_{g=0}|^{sourced}_{closed}\big)_N,\_\}_c,\{\_,\_\}_c\Big).$$
and the $L_\infty$-algebra $L_N^{br}(\mathcal{C},\lambda,h)$ from definition \ref{brlinalsb}.
  \end{df}
\begin{remark}\label{dbegaj}
    The map \ref{gauawamap} induces an
    isomorphism from `the classical sourced closed string field theory at level $N$' (from definition \ref{brLinfalg}) to `the classical closed string field theory at level $N$ backreacted in the direction~$(\lambda,h)$' from definition \ref{wzhhsimsl}; that is we have an isomorphism of the shifted Poisson algebras featuring in these definitions (which follows by twisting the map \ref{gauawamap} with the Maurer-Cartan element $\rho\big(SV_{g=0}|^{sourced}_{closed}\big)$) and an isomorphism of $L_\infty$-algebras
    $$L_N(\mathcal{C},\lambda)\rightsquigarrow L_N^{br}(\mathcal{C},\lambda,h),$$
    (which follows since the map from lemma \ref{gauawamap} twisted by the Maurer-Cartan element $\rho\big(SV_{g=0}|^{sourced}_{closed}\big)$ is still a map of algebras, inducing a map of coalgebras by dualizing). The $L_\infty$-algebra $L_N^{br}(\mathcal{C},\lambda,h)$ is possibly curved, however its curving is of higher than first order in $N$; this follows by arguments as in lemma \ref{isidfrz}. 
 As an upshot, 
    this tells us that if $ch(\lambda)=(d_{hoch}+uB)h$ the curving \eqref{focdsl} of $L_N(\mathcal{C},\lambda)$ is in fact trivial to leading order in $N$ (up to $L_\infty$ isomorphism). 
    \end{remark}
    \begin{remark}

    Assuming additionally the existence of a splitting $s$ the loop zero part of the map from corollary \ref{maco} induces (as it is a map of algebras, which follows from the $\gamma=0$ part of theorem \ref{BVinf}) an $L_\infty$ quasi-isomorphism
    $$L_N(\mathcal{C},\lambda,h)\rightsquigarrow \Big(\big(CH_*(\mathcal{C})[ u^{-1}]\big)^\vee,d_{hoch}+uB\Big) ,$$
    that is it implies, combining this result with the previous remark, that (up to $L_\infty$ quasi-isomorphism) the $L_\infty$-algebra structure $L_N(\mathcal{C},\lambda)$ is abelian and in particular that its curving  \eqref{focdsl}  is (up to $L_\infty$ quasi-isomorphism) trivial.
\end{remark}
Recall also the discussion in the introduction in remark \ref{sihdrdnls} about backreaction for BCOV theory on $\mathbb{C}^3$ for a conjectural (almost) example of these considerations.
\subsubsection{Quantization of Backreacted Closed String Field Theory}
Next we introduce a condition that will allow us to find a quantization of the classical closed backreacted string field theory. Furthermore we phrase ideas for what its partition function encodes via conjecture \ref{diaestf} below.
\begin{df}\label{HolDef}
     We say that a pair $(\mathcal{C},\lambda)$ of a smooth cyclic $A_\infty$-category and an object $\lambda\in\mathcal{C}$ is \emph{holographic} if it is backreactable (recalling definition \ref{backreactable}) and in addition
   \begin{equation}\label{deltaterm}
    \delta \rho(SV^{oc}|_{g=0}^{m=1+1})=0,\end{equation}
    referring to this footnote\footnote{Here we denote by $SV^{oc}|_{g=0}^{m=1+1}$ the part of the genus zero open-closed string vertices with arbitrarily many framed boundaries and arbitrarily many unframed boundaries out of which exactly two have one marked point each.} for the notation.
    \end{df} 

\begin{remark}\label{hopefullytobeimprov}
It seems likely that a tuple of a dimension $d$ Calabi-Yau category $\mathcal{C}$ and an object $\lambda\in\mathcal{C}$ are both `backreactable' and `holographic' if $End_{\mathcal{C}}(\lambda)$ is isomorphic to the $CW$ complex of a $d$-sphere. This is a condition that has been used in the geometric contexts of open Gromov-Witten theory, see
 \cite{Fuk11}.
 
 Roughly speaking, in our set-up this should follow as one can use models for cyclic cohomology where one quotients out by the units; compare section 2.2.13 and 2.2.14 of \cite{Lo13}. However if we can 'set the units to zero' and work under the assumption that $End_{\mathcal{C}}(\lambda)$ is spanned by the unit and its dual (which is the case for the $CW$ complex of a $d$-sphere), then one can reasonably expect that the terms \eqref{hrmop1}, \eqref{hrmop2} and \eqref{deltaterm} are zero. 
 To be a bit more precise, one may need to relax conditions \eqref{hrmop1}, \eqref{hrmop2} and \eqref{deltaterm} suitably for the above argument to work, likely in the sense that the featuring terms are exact but not necessarily zero. But then proving an analogue of equations \eqref{stfuiwfs}  and \eqref{sbqmep13} seems more difficult, which we left for the future.
\end{remark}  
In order to find in theorem \ref{bgc} below a quantization of classical closed string field theory backreacted in the direction~$(\lambda,h)$ we need following higher genus version of lemma \ref{gauawamap}.
\begin{lemma}
    Given $(\mathcal{C},\lambda)$ a cyclic $A_\infty$-category and an object $\lambda\in\mathcal{C}$ together with a trivialization $ch(\lambda)=(d_{hoch}+uB)h$ then
    \begin{equation}\label{ishisshg}
    \hspace{-0.5cm}
        e^{\nu\langle h,\_\rangle_{res}}:\Big(\mathcal{F}^c(\mathcal{C})|[\nu],\cdot,d_{Hoch}+uB+\nu\langle ch(\lambda),\_\rangle_{res}+\gamma\Delta_{c},\{\_,\_\}_c\Big)\rightarrow \Big(\mathcal{F}^c(\mathcal{C})|[\nu],\cdot,d_{Hoch}+uB+\gamma\Delta_{c},\{\_,\_\}_c\Big)
        \end{equation}
    defines an isomorphism of BD algebras.
    \end{lemma}
    \begin{proof}
    This follows similarly to the proof of proposition \ref{gauaw}.
    \end{proof}
Denote by $SV|^{sourced}_{closed}$ the (all genus) open-closed string vertices with arbitrary number of framed and unframed boundaries but no marked points; compare equation \eqref{QME}. Assuming that  $\lambda\in\mathcal{C}$ is holographic a direct consequence is then that 
\begin{equation}\label{sbqmep13}
(d_{Hoch}+uB+\nu\langle ch(\lambda),\_\rangle_{res}+\gamma\Delta_{c})\rho\big(SV|^{sourced}_{closed}\big)+\frac{1}{2}\{\rho\big(SV|^{sourced}_{closed}\big),\rho\big(SV|^{sourced}_{closed}\big)\}_{c}=0, \end{equation}
that is it is a Maurer-Cartan element in the left hand side of map \eqref{ishisshg}; thus its image under that map is as well.

\medskip
Now we can argue that the element \eqref{bcsftaf} below defines a \emph{quantization of the classical closed string field theory backreacted in the direction $(\lambda,h)$}: 
\begin{theorem}\label{bgc}
{Let $(\mathcal{C},\lambda)$ be holographic and pick a trivialization of the Chern character $(d_{hoch}+uB)h=ch(\lambda)$. Then
\begin{equation}\label{bcsftaf}
S^{c,q}_{br}(\mathcal{C},\lambda,h):=e^{\nu\langle h,\_\rangle_{res}}\rho\big(SV|^{sourced}_{closed}\big)
\end{equation}
is a Maurer-Cartan element of  $\mathcal{F}^c(\mathcal{C})[\nu]$
and its genus zero part induces via
$$\{S^{c,q}_{br}(\mathcal{C},\lambda,h)|_{\gamma=0},\_\}$$
 the square zero odd vector field defining the $L_\infty$-algebra of classical closed string field theory backreacted in the direction $(\lambda,h)$ from definition \ref{wzhhsimsl}.}
\end{theorem}
\begin{remark}
 Theorem \ref{bgc} follows easily from corollary \ref{snsnsn} and the definition of being holographic. However, one may hope to generalize what it means to be holographic in the sense of definition \ref{HolDef}, like asking that the term in equation \ref{deltaterm} is exact and not directly zero. It seems that a proof of the analog of theorem \ref{bgc} in this situation is more complicated, which we will study in the future.
\end{remark}

To explain definition \ref{yangning} further below we now additionally assume that we are given a splitting $s$ of the non-commutative Hodge filtration, that is all together a package
\begin{equation}
\mathcal{P}_{\mathcal{C},\lambda,h,s}=\{\mathcal{C},\lambda,h,s\} \end{equation}
of a smooth cyclic $A_\infty$-category $\mathcal{C}$ and an object $\lambda\in\mathcal{C}$ which are holographic, a trivialization ${(d_{hoch}+uB)h=ch(\lambda)}$ and a splitting $s$. Then, recalling the map $\Psi_s^{oc}$ from corollary \ref{maco} it follows that $$\Psi_s^{oc}S^{c,q}_{br}(\mathcal{C},\lambda,h)\in H_*(\mathcal{F}^{c}(\mathcal{C})^{Triv}\llbracket\nu\rrbracket)$$
as is its exponential
\begin{equation}\label{itsexp}
e^{\Psi_s^{oc}S^{c,q}_{br}(\mathcal{C},\lambda,h)/\gamma}\in H_*(\mathcal{F}^{c}(\mathcal{C})^{Triv}\llbracket\nu\rrbracket).
\end{equation}
Let us recall the element 
$$Z_{\mathcal{C},s}^c\in H_*(\mathcal{F}^{c}(\mathcal{C})^{Triv}$$
from definition \ref{cgw}, which  also coincides the summand of the element from \eqref{itsexp} that has zero powers of $\nu$. Then we make following definition   
\begin{df}\label{yangning}
    Given a package
$\mathcal{P}_{\mathcal{C},\lambda,h,s}=\{\mathcal{C},\lambda,h,s\}$
of a smooth cyclic $A_\infty$-category $\mathcal{C}$ and an object $\lambda\in\mathcal{C}$ which are holographic, a trivialization ${(d_{hoch}+uB)h=ch(\lambda)}$ and a splitting~$s$. Then denote the \emph{partition function of the backreacted closed string field theory associated to $\mathcal{C}$ in the direction of $\lambda$ associated to the package $\mathcal{P}$}
\begin{equation}\label{brcsftsnsn}
Z^{c,br}(\mathcal{P}_{\mathcal{C},\lambda,h,s}):=[e^{\Psi_s^{oc}S^{c,q}_{br}(\mathcal{C},\lambda,h)/\gamma}]-Z_{\mathcal{C},s}^c\in H_*\big(\mathcal{F}^{c}(\mathcal{C})^{Triv}\llbracket\nu\rrbracket\big).\end{equation}

\end{df}
\begin{remark}
In definition \ref{yangning} we made the ad-hoc choice of subtracting the term ${Z_{\mathcal{C},s}^c\in H_*(\mathcal{F}^{c}(\mathcal{C})^{Triv}}$
from definition \ref{cgw}. We have no direct mathematical or physical motivation for doing so. However, because of this definition the element $Z^{c,br}(\mathcal{P}_{\mathcal{C},\lambda,h,s})\in H_*(\mathcal{F}^{c}(\mathcal{C})^{Triv}$ lives in the subspace of strictly positive powers of $\nu$. This is a necessary condition for conjecture \ref{conj_ma} below to hold, which is the reason we made that choice in definition \ref{yangning}. We hope to better understand this ad-hoc choice in the future.
\end{remark}
\medskip
Recall that
$$H_*\big(\mathcal{F}^{c}(\mathcal{C})^{Triv}\llbracket\nu\rrbracket\big)=Sym (HH_*(\mathcal{C})[ u^{-1}])\llbracket\nu\rrbracket,$$
thus we can read off the coefficients (as a power series in $\nu$) of the element \eqref{brcsftsnsn} with respect to homology class insertions and powers of $u$, as explained around equation \eqref{ccei} in the introduction. \textbf{What do these coefficients describe?} 
\begin{Conj}\label{diaestf}[Backreacted Calabi-Yau Category]
Given a holography package $\mathcal{P}_{\mathcal{C},\lambda,h,s}$ as in definition \ref{yangning} there is a family of  Calabi-Yau categories $\mathcal{C}^{br}_N$, depending on $N\in \mathbb{N}$, together with a splitting $s^{br}$ such that 
\begin{equation}\label{dbwuszm}
HH_*(\mathcal{C})\cong HH_*(\mathcal{C}^{br}_N)\end{equation}
and 
\begin{equation}\label{wmsaanm}
{Z_{\mathcal{C}^{br}_N,s^{br}}^c}=Z^{c,br}(\mathcal{P}_{\mathcal{C},\lambda,h,s})|_{\nu=N}\in Sym (HH_*(\mathcal{C})[ u^{-1}]).\end{equation}
Here on the left hand side we refer to definition \ref{cgw} applied to the conjectural Calabi-Yau category $\mathcal{C}^{br}_N$ and the splitting $s^{br},$ whereas on the right we refer to element \eqref{brcsftsnsn} evaluated at $\nu=N$; further we need to assume equation \eqref{dbwuszm} to make sense of equation \eqref{wmsaanm}. 
\end{Conj}
Let us explain the motivation for conjecture \ref{diaestf}. The genus zero part of the element \eqref{brcsftsnsn}, more precisely the induced $L_\infty$-algebra(s) $L_N^{br}(\mathcal{C},\lambda,h)$ from definition \ref{wzhhsimsl} should describes some deformed geometry, as compared to the $L_\infty$-algebra $L(\mathcal{C})$ described around equation \eqref{cych}.

For instance\footnote{Note that we cannot apply our methods to the following setting as $\mathbb{C}^3$ is not proper Calabi-Yau; so the following can only serve for motivational purposes.} taking $\mathcal{C}=\mathcal{D}^{b}_{dg}(\mathbb{C}^3)$, its object(s) $\lambda^{\oplus N}$ determined by the stack of $N$ branes of $\mathbb{C}\subset \mathbb{C}^3$ and $h$ given by the Mauer-Cartan element described in lemma 1 of \cite{cg21}, Costello-Gaiotto find a dg-Lie algebra (for every $N\in\mathbb{N}$) describing the (deformations of the complex structure of the) deformed conifold with period $N$ denoted $X_N$, which is obtained by perturbing the dg-Lie algebra describing the (deformations of the complex structure of) $\mathbb{C}^3$ in `the direction' of $\lambda^{\oplus N}$  and using $h$ (section 4 and 4.1 of \cite{cg21}); see also the discussion in the introduction in remark \ref{sihdrdnls}. However the  deformed conifold with period $N$ is again a Calabi-Yau manifold with its associated Calabi-Yau category $\mathcal{D}^{b}_{dg}(X_N)$. Conjecture \ref{diaestf} postulates in particular that such a phenomena should hold in general and at a purely categorical level. 

In a similar vein Gopakumar-Vafa \cite{GoVa99} describe how the cotangent bundle of a three manifold (to which we can assign the Calabi-Yau category of the Fukaya category of the cotangent bundle) gets deformed in the presence of the zero section (which is an object of the Fukaya category) to the resolved conifold. This is again a Calabi-Yau manifold with associated Calabi-Yau category of quasi-coherent sheaves (also here a subtle $N$-dependence is present).
\smallskip
  Questions around conjecture \ref{diaestf} have also recently been considered in the physics literature; see \cite{Gai24}, especially section 11.

\subsection{Holography from a Calabi Yau Category $\mathcal{C}$ and an object $\lambda\in\mathcal{C}$ at the Level of Partition Functions}
In physics, the so called top-down approach to holography proposes a duality between the large $N$ open string field theory on a stack of $N$ branes of a string theory and the backreacted closed string field theory in the direction of that same brane  \cite{Ma99}, \cite{cg21}. This should in particular imply that the partition functions of the two theories can be identified. This and the discussion in the previous sections leads us to formulate:
\begin{Conj}[Conjecture H]\label{conj_ma}
Given a package $\mathcal{P}_{\mathcal{C},\lambda,h,s}=\{\mathcal{C},\lambda,h,s\}$
of a smooth cyclic $A_\infty$-category $\mathcal{C}$ and an object $\lambda\in\mathcal{C}$ which are holographic, a trivialization ${(d_{hoch}+uB)h=ch(\lambda)}$ and a splitting~$s$. Then
the partition function of the backreacted closed SFT \eqref{brcsftsnsn} is homologous to the partition function of large $N$ open SFT \eqref{fmc}. That is
$$\iota_c(Z^{c,br}(\mathcal{P}_{\mathcal{C},\lambda,h,s}))\simeq \iota_o(Z^o(\mathcal{P}_{\mathcal{C},\lambda,h,s})),$$
as elements of $H_*\Big({\mathcal{F}^c(\mathcal{C})}^{triv}\otimes \mathcal{F}^o(\lambda),d_{Hoch}+d+\Delta_o\Big)$
under the natural inclusion maps 
$$\begin{tikzcd}
   & \mathcal{F}^{c}(\mathcal{C})^{Triv}\otimes \mathcal{F}^{o}(\lambda)&\\
   \mathcal{F}^{c}(\mathcal{C})^{Triv}\llbracket\nu\rrbracket\arrow{ur}{\iota_c}&&\mathcal{F}^{o}(\lambda)\arrow{ul}[swap]{\iota_o}.
\end{tikzcd}$$
\end{Conj}
    We refer to the last part of the introduction, section \ref{twhocy}, for further motivation and conjectural examples of conjecture H. There we explain its relevance in trying to give a new and mathematical perspective on known and proposed relations \cite{Kon92a, GoVa99,cg21} between large $N$ gauge theories and enumerative geometry.

\begin{remark}\label{DepOnPar}
In physics, both sides of the conjecture, the partition function of large $N$ open SFT and the partition function of closed backreacted SFT depend on certain parameters, identified under holography. Already for the (non-backreacted) GW potential \eqref{cgw} one expects a  dependence on such parameters. We imagine that in the categorical approach this could be encoded as follows:

Theorem \ref{imdjns} below implies that given a smooth cyclic $A_\infty$-category $\mathcal{C}$ and taking $\Lambda\subseteq\mathcal{C}$ to be $\Lambda=\mathcal{C}$, together with a splitting $s$
there is a map 
\begin{equation}\label{fam}
 B_{\mathcal{C},\lambda,s}: HC^*(\mathcal{C})[-1]\rightarrow MCE\Big(Cyc^*_+(\mathcal{C})[-1],d_{cyc},\{\_,\_\}_o\Big),\end{equation}
 mapping the zero vector to the zero Maurer-Cartan element. It follows similar to lemma \ref{Ham1} that (up to finiteness aspects) we can identify elements of the right hand side of map \eqref{fam} with cyclic $A_\infty$-categories, the zero Maurer-Cartan element corresponding to our initial category $\mathcal{C}$. 
 Thus the map $B_{\mathcal{C},\lambda,s}$ produces a family of cyclic $A_\infty$-categories parametrized by cyclic cohomology, which at zero gives our initial category. We may have to reformulate conjecture \ref{conj_ma} with respect to this family induced by \eqref{fam}. One major caveat is that some elements of this family may be possibly curved, which we do not know how to treat with our formalism. We hope to work on this in the future.
\end{remark}

\medskip
We finish this section by explaining the origin of the map \ref{fam}, which we use to obtain a family version of conjecture H. This map is connected to the cyclic variant of Kontsevich's formality result \cite{kon03}; see remark \ref{iewdalm}.

\medskip
For that recall from definition \ref{cisft} the $L_\infty$-algebra $L(\mathcal{C}),$ depending on $\mathcal{C}$ a smooth cyclic $A_\infty$-category. Further denote by 
$$CC^*(\mathcal{C}):=\Big(CH_*(\mathcal{C})[u^{-1}]^\vee[-1],d_{hoch}+uB\Big),$$
the abelian $L_\infty$-algebra with the same underlying graded vector space as  $L(\mathcal{C})$. Given a collection of object $\Lambda\subseteq\mathcal{C}$, since the full subcategory on $\Lambda$ is a cyclic $A_\infty$-category, we have the dg shifted Lie algebra (see around definition \ref{F_cl}) 
$$O(\Lambda):=\Big(Cyc^*_+(\Lambda)[-1],d_{cyc},\{\_,\_\}_o\Big).$$
\begin{theorem}\label{imdjns}
    Given a smooth cyclic $A_\infty$-category $\mathcal{C}$ and $\Lambda\subseteq \mathcal{C}$ then there exist a (non-trivial) $L_\infty$-morphism $F_{D^{oc}}$ on the right of diagram \ref{sddgnngi}. Assuming further the existence of a splitting $s$ there exists a (non-trivial) $L_\infty$-morphism $F_s$ on the left of diagram \ref{sddgnngi}.
\begin{equation}\label{sddgnngi}
\begin{tikzcd}
CC^*(\mathcal{C}) & L(\mathcal{C}) \arrow{r}{F_{D^{oc}}}\arrow{l}[swap]{F_s} & O(\Lambda).   
   \end{tikzcd}
    \end{equation}
If $\Lambda \cong\mathcal{C}$ these maps are in fact quasi-isomorphisms.   
\end{theorem}
\begin{proof}[Sketch of Proof]
Let us focus on the right part of diagram \ref{sddgnngi} first. Recall, eg. from proposition 6.1 of \cite{krsn22}, that we can identify an $L_\infty$-map $$F_{D^{oc}}:L(\mathcal{C}) \rightarrow  O(\Lambda)$$ with a Maurer-Cartan element in the convolution dg Lie algebra (see \cite{krsn22} for details)
\begin{equation}\label{ngdgzg}
Hom\big(Sym(L(\mathcal{C})[1]), O(\Lambda)\big).
\end{equation}
There is an obvious map of graded vector spaces
\begin{equation}\label{qweqweqwe}
Sym(L(\mathcal{C})[1])^\vee\otimes O(\lambda)\rightarrow Hom(Sym(L(\mathcal{C})[1]), O(\Lambda))\end{equation}
On the left hand side of equation \eqref{qweqweqwe} we use that the tensor product of a dg algebra (given as the linear dual of the dg colagebra defining the $L_\infty$-algebra $L(\mathcal{C})$, recalling definition \ref{coalgebra}) and a dg shifted Lie algebra is again dg shifted Lie algebra in a natural way. Further one can directly show that the map \ref{qweqweqwe} is a map of dg Lie algebras.

Next we recall that given a smooth cyclic $A_\infty$-category $\mathcal{C}$ and a collection of objects $\Lambda\subseteq\mathcal{C}$ that by definition \ref{cisft} (but ignoring the closed shifted Poisson bracket) we can define a dg-algebra
$$C(\mathcal{C}):=\Big({\mathcal{F}^c(\mathcal{C})}|_{\gamma=0},\cdot, d_{hoch}+uB+\{\rho(S^c_{g=0}),\_\}_c,\Big).$$
and that there is a natural map (induced from including a vector space into its double dual) 
\begin{equation}\label{werwerwer}
C(\mathcal{C})\otimes O(\Lambda)\rightarrow Sym(L(\mathcal{C})[1])^\vee\otimes O(\lambda).
\end{equation}
Thus if we can show that there is a (non-trivial) Maurer-Cartan element in the left hand side of map \eqref{werwerwer} it follows after applying the (injective) maps \eqref{werwerwer} and \eqref{qweqweqwe} that we find a non-zero element of the dg-Lie algebra \eqref{ngdgzg}, which is equivalent to providing the $L_\infty$-map from the right hand side of diagram \eqref{sddgnngi}.

For that let us denote by $D^{oc}$ the genus zero, one unframed boundary with arbitrary boundary marked points and arbitrary interior framed marked points part of string vertices, compare \eqref{QME}. Further recall that we denoted by $S^c|_{g=0}$ the genus zero part of the closed string vertices and by $D$ the part of the open-closed string vertices corresponding to the disk with arbitrary many boundary marked points but no interior marked points. Then equation \eqref{QME} implies that 
\begin{equation}\label{ohaiaiho}
\partial D^{oc}+\{D^{oc},S^c|_{g=0}\}_c+\frac{1}{2}\{D^{oc},D^{oc}\}_o+\{D^{oc},D\}_o=0 
\end{equation}

Since we have 
$$d_{cyc}=d+\{\rho(D),\_\}_o$$
(compare the comment after corollary \ref{snsnsn}) it follows from theorem \ref{ocRep} that  
\begin{equation}\label{omgfvssb}
\rho_{\mathcal{C},\Lambda}(D^{oc})\in MCE(C(\mathcal{C})\otimes O(\Lambda)),
\end{equation}
which is the Maurer-Cartan element that we asked for.

The fact that this Maurer-Cartan element is non-trivial should follow from the results of section 7.4 of \cite{Cos07a}, which imply that the component of $\rho_{\Lambda,\Lambda}(D^{oc})$ with just one interior marked point exactly induces (as the linear part of the $L_\infty$-morphism induced by \eqref{omgfvssb}) the quasi-isomorphism
$$\big(CH_*(\Lambda)[u^{-1}]^\vee[-1],d_{hoch}+uB\big)\rightarrow \big(Cyc^*_+(\Lambda)[-1],d_{cyc}\big)$$
of two different models computing cyclic cohomology (see section 2.1.11 of \cite{Lo13}). We leave making this more precise to future work.

\medskip
Assuming additionally the existence of a splitting $s$ the loop zero part of the map from theorem \ref{splmap} induces (as it is a map of algebras, which follows from the $\gamma=0$ part of theorem \ref{BVinf}) an $L_\infty$ quasi-isomorphism
    $$L(\mathcal{C})\rightsquigarrow \Big(\big(CH_*(\mathcal{C})[ u^{-1}]\big)^\vee,d_{hoch}+uB\Big),$$
    which explains the left arrow in diagram \ref{sddgnngi}.
    \end{proof}
\begin{remark}\label{iewdalm}
    The element $D^{oc}$ (in fact only its one marked point part) plays, at least heuristically, a crucial part in Costello-Li's analytic approach to quantizing open-closed BCOV theory \cite{CL15}, which we introduced in section 3.1. They further note the relation to the cyclic formality version \cite{wi11} of Kontsevich's formality theorem \cite{kon03}. 
    
   In another direction, the map \eqref{fam} should be related to what is called bulk-boundary deformation in the symplectic context, compare \cite{fooo11}.
\end{remark}

\newpage

\bibliography{refs}
\bibliographystyle{alpha}
\end{document}